\documentclass[10pt, leqno]{article}
\usepackage{amsmath,amsfonts,amsthm,amscd,amssymb}
\numberwithin{equation}{section}
\theoremstyle{plain}
\newtheorem{theo}{Theorem}[section]
\newtheorem{prop}[theo]{Proposition}
\newtheorem{coro}[theo]{Corollary} 
\newtheorem{lemm}[theo]{Lemma}
\theoremstyle{definition}
\newtheorem{defi}[theo]{Definition}
\newtheorem{rema}[theo]{Remark}
\newtheorem{theo-defi}[theo]{Theorem-Definition}
\newtheorem{prop-defi}[theo]{Proposition-Definition}
\newtheorem{rema-defi}[theo]{Remark-Definition}

\def \al{\alpha}
\def \bet{\beta}
\def \bul{\bullet}
\def \col{\colon}
\def \Del{\Delta}

\def \del{\delta}
\def \eps{\epsilon}
\def \Gam{\Gamma}
\def \gam{\gamma}

\def \kap{\kappa}
\def \Lam{\Lambda}
\def \lam{\lambda}

\def \Lo{\Longrightarrow}
\def \lo{\longrightarrow}
\def \lom{\longmapsto}
\def \mab{\mathbb}

\def \Om{\Omega}
\def \om{\omega}
\def \ol{\overline}
\def \os{\overset}
\def \parno{\par\noindent}
\def \part{\partial}

\def \sus{\subset}

\def \ul{\underline}
\def \us{\underset}

\def \vp{\varpi}

\def \vil{\varinjlim}
\def \vpl{\varprojlim}
\def \we{\wedge}
\def \wh{\widehat}
\def \wt{\widetilde}
\bigskip

\newcommand{\getsfrom}
{\ensuremath{\longleftarrow\kern-.
52em\lower-.1ex\hbox%
{$\shortmid\,$}}}

\begin{document}
\title{Theory of weights for log convergent cohomologies I: 
the case of a proper smooth scheme with an SNCD in characteristic $p>0$}
\author{Yukiyoshi Nakkajima and Atsushi Shiho
\date{}
\thanks{  
The authors were supported by 
JSPS Core-to-Core Program 18005 whose representative 
is Makoto Matsumoto. 
The first-named author was and is supported from JSPS
Grant-in-Aid for Scientific Research (C)'s
(Grant No.~20540025, 18K03224). 
The second-named author was supported from JSPS
Grant-in-Aid for Young Scientists (B)
(Grant No.~21740003). 
\endgraf}}
\maketitle

\begin{flushright}
%Preliminary version 2.
\end{flushright}

$${\rm ABSTRACT}$$
\bigskip 
\parno
Using log convergent topoi,  
%In the derived category of filtered complexes of 
%sheaves of modules over 
%an isostructure 
we define two fundamental filtered complexes 
$(E_{\rm conv},P)$ and $(C_{\rm conv},P)$ 
for the log scheme obtained by a smooth scheme 
with a relative simple normal crossing divisor 
over a scheme of characteristic $p>0$. 
Using $(C_{\rm conv},P)$, we prove the 
$p$-adic purity. As a corollary of this purity, 
we prove that $(E_{\rm conv},P)$ and $(C_{\rm conv},P)$ 
are canonically isomorphic.  
These filtered complexes produce 
the weight spectral sequence of 
the log convergent cohomology sheaf 
of the log scheme. 
We also give the comparison theorem 
between the projections of $(E_{\rm conv},P)$ and 
$(C_{\rm conv},P)$   
to the derived category of 
bounded below filtered complexes of sheaves of modules in 
the Zariski topos of the log scheme and 
the weight-filtered isozariskian filtered complex 
$(E_{\rm zar},P)\otimes^L_{\mab Z}{\mab Q}$ 
of the log scheme defined in \cite{nh2}. 
\bigskip

$${\bf Contents}$$
\bigskip 
\parno
\S\ref{sec:intro}. Introduction 
\medskip 
\parno 
\S\ref{sec:logcd}. 
Log convergent topoi and the systems of 
log universal enlargements 
\medskip 
\parno
\S\ref{sec:lll}. Log convergent linearization functors.~I 
%\medskip 
%\parno
%\S\ref{sec:sdtt}. Log convergent linearization functors.~II 
\medskip 
\parno
\S\ref{sec:llf}. Log convergent linearization functors.~II 
\medskip 
\parno
\S\ref{sec:vflvc}. Vanishing cycle sheaves in log convergent topoi 
\medskip 
\parno
\S\ref{sec:lcs}. 
Log convergent linearization functors 
of smooth schemes with relative SNCD's 
\medskip 
\parno
\S\ref{sec:rlct}. 
(Bi)simplicial log convergent topoi 
\medskip 
\parno
\S\ref{sec:wfcipp}. Weight-filtered convergent complexes  
and $p$-adic purity 
\medskip 
\parno
\S\ref{sec:fpw}.
The functoriality of weight-filtered convergent 
and isozariskian complexes 
\medskip 
\parno
\S\ref{sec:bd}. Boundary morphisms 
\medskip 
\parno
\S\ref{sec:ct}. 
Comparison theorem 
\medskip 
\parno
\S\ref{sec:cptsc}.  
Log convergent cohomologies  
with compact supports 
\medskip
\parno
References

\section{Introduction}\label{sec:intro} 
This paper is a continuation of \cite{nh2}, a companion of 
\cite{s3} and \cite{nh3}. 
%and a precedence of \cite{nrh}. 
\par 
First we recall results in \cite{nh2}. 
\par 
Let ${\cal V}$ be a complete discrete valuation ring 
of mixed characteristics $(0,p)$ 
whose residue field $\kap$ is perfect. 
Let $\pi$ be a nonzero element of the maximal ideal of 
${\cal V}$. 
Let $S$ be a $p$-adic formal ${\cal V}$-scheme in the sense of \cite{of}, 
that is, a noetherian formal scheme over ${\rm Spf}({\cal V})$ 
with the $p$-adic topology which is topologically of finite type 
over ${\rm Spf}({\cal V})$. 
Set $S_n:=\ul{\rm Spec}_S({\cal O}_S/\pi^n{\cal O}_S)$ 
$(n\in {\mab Z}_{\geq 1})$. 
Let $X$ be a smooth scheme over $S_1$ and 
let $D$ be a relative SNCD(=simple normal crossing divisor) 
on $X/S_1$. 
Stimulated by \cite{klog1} and \cite{fao}, 
we have constructed an 
fs(=fine and saturated) log structure 
$M(D)$ on $X_{\rm zar}$ 
with morphism $M(D) \lo ({\cal O}_X,*)$ in \cite{nh2}. 
If there exists the following cartesian diagram 
\begin{equation*}  
\begin{CD} 
D\cap V @>{\subset}>> V \\ 
@VVV @VVV \\ 
\ul{\rm Spec}_{S_1}
({\cal O}_{S_1}[x_1,\ldots,x_d]/(x_1\cdots x_a))
@>{\subset}>> 
\ul{\rm Spec}_{S_1}({\cal O}_{S_1}[x_1,\ldots,x_d]) 
\end{CD}
\end{equation*}
($a,d\in {\mab Z}_{\geq 1}$, $a\leq d$)
for an open subscheme $V$ of $X$, 
where the right vertical morphism 
in the diagram above is \'{e}tale, 
then 
$$(M(D)\vert_V \lo ({\cal O}_V,*))
=(({\cal O}_V^*x_1^{\mab N}\cdots x_a^{\mab N},*) 
\os{\sus}{\lo} ({\cal O}_V,*)).$$ 
By abuse of notation, we denote 
the log scheme $(X,M(D))$ by $(X,D)$.  
\par 
Let $f\col X \lo S_1 \os{\sus}{\lo} S$ be 
the structural morphism. 
In this Introduction, we assume 
that the ideal sheaf $\pi{\cal O}_S$ of 
${\cal O}_S$ has a PD-structure $\gam$ 
when we consider (an object of) 
a (restricted) (log) crystalline site or 
(that of) a (restricted) (log) crystalline topos. 
Let ${\rm Crys}(X/S)$ and $(X/S)_{\rm crys}$ 
be the crystalline site of $X/(S,\pi {\cal O}_S,\gam)$ 
and the crystalline topos of $X/(S,\pi {\cal O}_S,\gam)$ defined in 
\cite{bb} and \cite{bob}, respectively. 
Let ${\rm Rcrys}(X/S)$ be the restricted crystalline site of 
$X/(S,\pi {\cal O}_S,\gam)$: 
${\rm Rcrys}(X/S)$ is a subsite of ${\rm Crys}(X/S)$
whose objects are isomorphic to 
$(U,\mathfrak{D}_{U,\gam}({\cal U}),[~])$'s, 
where $U$ is an open subscheme of $X$ and 
${\cal U}$ is a smooth scheme over $S_n$ for some $n$ 
which contains $U$ as a closed subscheme over $S_n$ 
and $\mathfrak{D}_{U,\gam}({\cal U})$ is 
the PD-envelope of $U$ in ${\cal U}$ 
over $(S_n,\pi{\cal O}_{S_n},\gam)$. 
(The definition of ${\rm Rcrys}(X/S)$ here 
is slightly different from that  in \cite{bb} 
and it is a special case of the definition in \cite{tsp}.) 
Let 
$(X/S)_{\rm Rcrys}$ be the restricted crystalline topos of 
$X/(S,\pi {\cal O}_S,\gam)$ associated to ${\rm Rcrys}(X/S)$. 
Let $Q^{\rm crys}_{X/S} \col (X/S)_{\rm Rcrys} \lo (X/S)_{\rm crys}$ 
be a morphism of topoi defined in \cite{bb}: 
for an object $E$ of 
$(X/S)_{\rm crys}$, $Q^{{\rm crys}*}_{X/S}(E)$ 
is the natural restriction of $E$ to 
${\rm Rcrys}(X/S)$. 
Let ${\cal O}_{X/S}$ be the structure sheaf in $(X/S)_{\rm crys}$. 
Then we have a natural morphism 
$$Q^{\rm crys}_{X/S} \col 
(({X/S})_{\rm Rcrys},Q^{{\rm crys}*}_{X/S}({\cal O}_{X/S})) 
\lo (({X/S})_{\rm crys},{\cal O}_{X/S})$$ 
of ringed topoi. 
Let $((X,D)/S)_{\rm crys}$ be 
the log crystalline topos of $(X,D)$ over 
$(S,\pi{\cal O}_S,\gam)$ defined in \cite{klog1} and 
let ${\cal O}_{(X,D)/S}$ be 
the structure sheaf in $((X,D)/S)_{\rm crys}$. 
Let 
$$\eps^{\rm crys}_{(X,D)/S} \col (((X,D)/S)_{\rm crys},{\cal O}_{(X,D)/S}) 
\lo ((X/S)_{\rm crys},{\cal O}_{X/S})$$ 
be the morphism of topoi induced by   
the natural morphism 
$(X,M(D)) \lo (X,{\cal O}^*_X)$  of log schemes over $S_1$.  
Let 
$$u^{\rm crys}_{X/S} \col 
((X/S)_{\rm crys},{\cal O}_{X/S}) 
\lo (X_{\rm zar},f^{-1}({\cal O}_{S}))$$ 
be the natural projection. Set 
$\ol{u}^{\rm crys}_{X/S}:=u^{\rm crys}_{X/S}\circ Q^{\rm crys}_{X/S}$ 
as in \cite{bb}. 
Let $\tau$ be the canonical filtration on a complex. 
Let ${\rm D}^+{\rm F}({\cal O}_{X/S})$ be the derived category of bounded below 
filtered complexes of ${\cal O}_{X/S}$-modules and let 
${\rm D}^+{\rm F}(f^{-1}({\cal O}_S))$ be the derived category of bounded below 
filtered complexes of $f^{-1}({\cal O}_S)$-modules. 
In \cite{nh2} 
we have defined the following four filtered complexes:  
$$(E_{\rm crys}({\cal O}_{(X,D)/S}),P):=
(R\eps^{\rm crys}_{(X,D)/S*}({\cal O}_{(X,D)/S}),\tau)\in 
{\rm D}^+{\rm F}({\cal O}_{X/S}),$$ 
$$(E_{\rm zar}({\cal O}_{(X,D)/S}),P)
:=Ru^{\rm crys}_{X/S*}((E_{\rm crys}({\cal O}_{(X,D)/S}),P))
\in 
{\rm D}^+{\rm F}(f^{-1}({\cal O}_S)),$$ 
$$(C_{\rm Rcrys}({\cal O}_{(X,D)/S}),P)\in 
{\rm D}^+{\rm F}(Q^{{\rm crys}*}_{X/S}({\cal O}_{X/S})),$$ 
$$(C_{\rm zar}({\cal O}_{(X,D)/S}),P)
:=R\ol{u}^{\rm crys}_{X/S*}((C_{\rm Rcrys}({\cal O}_{(X,D)/S}),P))
\in {\rm D}^+{\rm F}(f^{-1}({\cal O}_S)).$$  
The definition of $(C_{\rm Rcrys}({\cal O}_{(X,D)/S}),P)$ 
is as follows (the definition below 
is essentially the same as that in 
\cite{nh2} and \cite{nh3}; but it is 
a more general description than that). 
\par  
Let $X=\bigcup_iX_i$ be an open covering of $X$. 
Set $X_0:=\coprod_iX_i$.  
Let $\pi_0 \col X_0 \lo X$ be 
the natural morphism. 
Set $D_0:=\pi^{-1}_0(D)$ and 
$(X_{\bul},D_{\bul})_{\bul \in {\mab N}}
:={\rm cosk}_0^{(X,D)}(X_0,D_0)$.  
Let $\pi \col X_{\bul} \lo X$ be 
the augmentation morphism.
Assume that there exists 
a simplicial immersion 
$(X_{\bul},D_{\bul})\os{\sus}{\lo} {\cal P}_{\bul}$ 
into a formally log smooth 
simplicial log $p$-adic formal scheme over $S$. 
(If each member of the open covering of $X$ is 
an affine scheme, we can easily show the existence of 
this immersion.) 
Assume that $\os{\circ}{\cal P}_{\bul}$ is 
topologically of finite type over $S$. 
For an immersion $Y\os{\sus}{\lo} {\cal Q}$ 
of fine log formal schemes over $S$, 
let $Y\os{\sus}{\lo} {\cal Q}^{\rm ex}$ 
be the exactification of this immersion 
defined in \cite{s3}. 
(In (\ref{prop:exad}) below we shall recall the exactification.) 
Then, by the universality of the exactification, we have 
the simplicial exact closed immersion 
$(X_{\bul},D_{\bul}) 
\os{\sus}{\lo} {\cal P}^{\rm ex}_{\bul}$ over $S$.  
Let 
\begin{align*} 
L^{\rm crys}_{X_{\bul}/S} \col 
\{{\cal O}_{{\cal P}^{\rm ex}_{\bul}}\text{-modules}\} 
\lo \{{\cal O}_{X_{\bul}/S}\text{-modules}\} 
\tag{1.0.1}\label{ali:lxbs} 
\end{align*}
be the linearization functor for 
${\cal O}_{{\cal P}^{\rm ex}_{\bul}}$-modules in \cite{bob}. 
(Note that, though the topology of 
$\os{\circ}{\cal P}{}^{\rm ex}_{\bul}$ 
is not necessarily $p$-adic, the topology of 
the simplicial PD envelope of the simplicial immersion 
$X_{\bul} \os{\sus}{\lo} \os{\circ}{\cal P}{}^{\rm ex}_{\bul}$ 
is $p$-adic; we can use the theory in [loc.~cit.].) 
%Following Friedman (\cite{fr1}), 
For a log formal scheme ${\cal Q}$ over $S$, let us denote by 
$\Om^{\bul}_{{\cal Q}/S}$ 
the log de Rham complex of relative log differential forms on 
${\cal Q}/S$. 
Let $P=\{P_k\}_{k\in {\mab Z}}$ be the filtration 
on $\Om^{\bul}_{{\cal Q}/S}$ defined by the following formula:  
for $0\leq k\leq i$, 
$P_k\Om^i_{{\cal Q}/S}:=
{\rm Im}(\Om^k_{{\cal Q}/S}
{\otimes}_{{\cal O}_{\cal Q}}
\Om^{i-k}_{\os{\circ}{\cal Q}/S} \lo 
\Om^i_{{\cal Q}/S})$. 
By the essentially same proof as that of \cite[(2.2.17)]{nh2}, 
we can prove that the natural morphism 
\begin{equation*}
Q^{{\rm crys}*}_{X_{\bul}/S}L^{\rm crys}_{X_{\bul}/S}
(P_k\Om^{\bul}_{{\cal P}^{\rm ex}_{\bul}/S}) \lo 
Q^{{\rm crys}*}_{X_{\bul}/S}L^{\rm crys}_{X_{\bul}/S}
(\Om^{\bul}_{{\cal P}^{\rm ex}_{\bul}/S})
\tag{1.0.2}\label{eqn:qli}
\end{equation*}
is injective in the category of 
$Q^{{\rm crys}*}_{X_{\bul}/S}({\cal O}_{X_{\bul}/S})
\text{-modules}$.  
Here we need to take good care of the following related fact  with this injectivity 
and the following related question with it. 
In \cite[(2.7.11)]{nh2} we have also proved that 
the natural morphism 
\begin{equation*}
L^{\rm crys}_{X_{\bul}/S}
(P_k\Om^{\bul}_{{\cal P}^{\rm ex}_{\bul}/S}) 
\lo 
L^{\rm crys}_{X_{\bul}/S}(\Om^{\bul}_{{\cal P}^{\rm ex}_{\bul}/S})
\tag{1.0.3}\label{eqn:lni}
\end{equation*}
is not injective in the category of 
${\cal O}_{X_{\bul}/S}\text{-modules}$ in general. 
Let 
\begin{align*} 
L^{\rm crys}_{X_{\bul}/S} \col 
\{{\cal O}_{{\cal P}_{\bul}}\text{-modules}\} 
\lo \{{\cal O}_{X_{\bul}/S}\text{-modules}\} 
\tag{1.0.4}\label{ali:lnbi}
\end{align*}
be the linearization functor for 
${\cal O}_{{\cal P}_{\bul}}$-modules in \cite{bob}. 
(Though the sources of the linearization functors (\ref{ali:lxbs}) and (\ref{ali:lnbi}), 
we use the same notation $L^{\rm crys}_{X_{\bul}/S}$ by abuse of notation.)
Though one may think about whether the natural morphism 
\begin{equation*}
Q^{{\rm crys}*}_{X_{\bul}/S}L^{\rm crys}_{X_{\bul}/S}
(P_k\Om^{\bul}_{{\cal P}_{\bul}/S}) \lo 
Q^{{\rm crys}*}_{X_{\bul}/S}L^{\rm crys}_{X_{\bul}/S}
(\Om^{\bul}_{{\cal P}_{\bul}/S})
\end{equation*}
is injective in the category of 
$Q^{{\rm crys}*}_{X_{\bul}/S}({\cal O}_{X_{\bul}/S})
\text{-modules}$, we are not interested in this question 
%we do not use this injectivity in this paper 
because the filtration $P$ on 
$\Om^{\bul}_{{\cal P}_{\bul}/S}$ is useless in general   
(see (\ref{rema:unusefil}) below for this). 
\par 
Let 
$$\pi_{{\rm Rcrys}} \col 
(({X_{\bul}/S})_{\rm Rcrys}, 
Q^{{\rm crys}*}_{X_{\bul}/S}({\cal O}_{X_{\bul}/S}))_{\bul \in {\mab N}}
\lo 
(({X/S})_{\rm Rcrys},Q^{{\rm crys}*}_{X/S}({\cal O}_{X/S}))$$  
be the natural morphism of ringed topoi (\cite[(1.6)]{nh2}). 
Then, set   
\begin{align*} 
(C_{\rm Rcrys}({\cal O}_{(X,D)/S}),P)
:= R\pi_{{\rm Rcrys}*}
(Q^{{\rm crys}*}_{X_{\bul}/S}
L^{\rm crys}_{X_{\bul}/S}(\Om^{\bul}_{{\cal P}^{\rm ex}_{\bul}/S}),
\{Q^{{\rm crys}*}_{X_{\bul}/S}L^{\rm crys}_{X_{\bul}/S}
(P_k\Om^{\bul}_{{\cal P}^{\rm ex}_{\bul}/S})\}_{k\in {\mab N}}).
\end{align*} 
By the essentially same proof as that of the well-definedness 
of $(C_{\rm Rcrys}({\cal O}_{(X,D)/S}),P)$ in \cite{nh2} or 
by the log crystalline version of 
the proof of (\ref{theo:ifc}) below 
or (\ref{eqn:qpc}) below, we can prove that  
$(C_{\rm Rcrys}({\cal O}_{(X,D)/S}),P)$ depends only on 
$(X,D)/(S,\pi{\cal O}_S,\gam)$. 
\par 
In \cite[(2.7.3)]{nh2} we have proved that 
there exists an isomorphism
\begin{equation*}
Q^{{\rm crys}*}_{X/S}(E_{\rm crys}({\cal O}_{(X,D)/S}),P)
\os{\sim}{\lo}(C_{\rm Rcrys}({\cal O}_{(X,D)/S}),P)
\tag{1.0.5}\label{eqn:qpc}
\end{equation*}
by showing the following $p$-adic purity
\begin{equation*}
Q^{{\rm crys}*}_{X/S}
R^k\eps^{\rm crys}_{(X,D)/S*}({\cal O}_{(X,D)/S})
\os{\sim}{\lo} 
Q^{{\rm crys}*}_{X/S}a^{(k)}_{{\rm crys}*}({\cal O}_{D^{(k)}/S}
\otimes_{\mab Z}
\vp^{(k)}_{\rm crys}(D/S))
\quad (k\in {\mab N}),
\tag{1.0.6}\label{eqn:vpuros}
\end{equation*}
where $D^{(k)}$ is a scheme over $S_1$ 
defined in \cite[(2.2.13.2)]{nh2} 
and will be recalled in \S\ref{sec:lcs} below, 
$a^{(k)} \col D^{(k)} \lo X$ is the natural  
morphism  of schemes over $S_1$ and 
$\vp^{(k)}_{\rm crys}(D/S)\in ({D^{(k)}/S})_{\rm crys}$ 
is the crystalline orientation sheaf defined in 
\cite[(2.2.18)]{nh2}, 
which is isomorphic to ${\mab Z}$ non-canonically 
(cf.~\S\ref{sec:lcs} below). 
%(See also \S\ref{sec:lcs} below for the analogue of 
%$\vp^{(k)}_{\rm crys}(D/S)$  
%in the convergent topos $({D^{(k)}/S})_{\rm conv}$.) 
As a corollary of the existence of 
the canonical isomorphism (\ref{eqn:qpc}), 
we obtain the following canonical isomorphism 
\begin{equation*}
(E_{\rm zar}({\cal O}_{(X,D)/S}),P)
\os{\sim}{\lo} 
(C_{\rm zar}({\cal O}_{(X,D)/S}),P) 
\tag{1.0.7}\label{eqn:upuros}
\end{equation*} 
by using the formula 
$R\ol{u}_{X/S*}^{\rm crys}Q^{{\rm crys}*}_{X/S}
=Ru_{X/S*}^{\rm crys}$ 
in \cite{bb}.  
\par 
However, unfortunately,   
\begin{equation*}
R^k\eps^{\rm crys}_{(X,D)/S*}({\cal O}_{(X,D)/S})\not=
a^{(k)}_{{\rm crys}*}({\cal O}_{D^{(k)}/S}
\otimes_{\mab Z}
\vp^{(k)}_{\rm crys}(D/S))
\quad (k\in {\mab N})
\tag{1.0.8}\label{eqn:nvpcrs}
\end{equation*}
in general; the natural morphism 
${\cal O}_{X/S} \lo 
R^0\eps^{\rm crys}_{(X,D)/S*}({\cal O}_{(X,D)/S})$
is not an isomorphism in general (\cite[(2.7.11)]{nh2}). 
%(However we do not know whether 
%the following equality 
%\begin{equation*}
%R^k\eps^{\rm crys}_{(X,D)/S*}({\cal O}_{(X,D)/S})
%\otimes_{\mab Z}{\mab Q}=
%a^{(k)}_{{\rm crys}*}({\cal O}_{D^{(k)}/S}
%\otimes_{\mab Z}\vp^{(k)}_{\rm crys}(D/S))
%\otimes_{\mab Z}{\mab Q} \quad (k\in {\mab N})
%\tag{1.0.7}\label{eqn:nvpcktrs}
%\end{equation*}
%holds.)
The non-equality (\ref{eqn:nvpcrs}) tells us that 
we have to consider the undesirable pull-back 
$Q^{{\rm crys}*}_{X/S}$ in the isomorphism 
(\ref{eqn:vpuros}). 
To be worse, the non-equality gives us 
an undesirable consequence. 
Indeed, set $S_n:=\ul{\rm Spec}_S({\cal O}_S/\pi^n{\cal O}_S)$ 
$(n\in {\mab Z}_{\geq 2})$ and endow 
$S_n$ with the induced PD-structure  
on $\pi{\cal O}_S/\pi^n{\cal O}_S$ by $\gam$. 
In \cite{bb} Berthelot has pointed out that, 
for a commutative diagram 
\begin{equation*} 
\begin{CD} 
Y @>{g}>> Z \\ 
@VVV @VVV \\ 
S_1 @>>> S_n
\end{CD} 
\end{equation*} 
of schemes, 
there does not exist a morphism 
$g_{\rm Rcrys} \col ({Y/S_1})_{\rm Rcrys} 
\lo  ({Z/S_n})_{\rm Rcrys}$ of topoi fitting into 
the following commutative diagram 
\begin{equation*} 
\begin{CD} 
({Y/S_1})_{\rm Rcrys} 
@>{g_{\rm Rcrys}}>> ({Z/S_n})_{\rm Rcrys} \\ 
@V{Q_{Y/S_1}}VV @VV{Q_{Z/S_n}}V \\ 
({Y/S_1})_{\rm crys} @>{g_{\rm crys}}>> ({Z/S_n})_{\rm crys}
\end{CD} 
\end{equation*}  
in general. As a result, 
we do not know whether 
$(C_{\rm Rcrys}({\cal O}_{(X,D)/S}),P)$ is 
contravariantly functorial. 
(However we know the contravariant functoriality of 
$(C_{\rm zar}({\cal O}_{(X,D)/S}),P)$ 
by the canonical isomorphism (\ref{eqn:upuros}).) 
\par 
Let ${\rm A}^{\geq 0}{\rm F}({\cal O}_{X/S})$ 
be the category of 
filtered positively-graded graded commutative dga's 
of ${\cal O}_{X/S}$-algebras 
and let ${\rm D}({\rm A}^{\geq 0}{\rm F}({\cal O}_{X/S}))$ 
be the localized category of  
${\rm A}^{\geq 0}{\rm F}({\cal O}_{X/S})$  
inverting the weakly equivalent morphisms 
in ${\rm A}^{\geq 0}{\rm F}({\cal O}_{X/S})$ 
(cf.~\cite{gelma}).   
Let 
$U_{\rm crys} \col 
{\rm D}({\rm A}^{\geq 0}{\rm F}({\cal O}_{X/S})) 
\lo {\rm D}^+{\rm F}({\cal O}_{X/S})$ be 
the natural forgetful functor. 
By the non-injectivity of the morphism  
(\ref{eqn:lni}), it seems that there does not exist 
a functorial filtered object 
$(\wt{C},\wt{P})\in 
{\rm D}({\rm A}^{\geq 0}{\rm F}({\cal O}_{X/S}))$ 
such that 
$Q^{{\rm crys}*}_{X/S}U_{\rm crys}((\wt{C},\wt{P})) 
\simeq (C_{\rm Rcrys}({\cal O}_{(X,D)/S}),P)$.  
\par 
To avoid these undesirable consequences, 
we ignore torsions and 
use the framework of (log) convergent topoi by Ogus
(\cite{oc}, \cite{ofo}); 
our approach in this paper turns out a success as follows.   
\par 
Let $({X/S})_{\rm conv}$ be the convergent topos 
of $X/S$ which is the relative version of 
the convergent topos in \cite{oc}. 
Let ${\cal K}_{X/S}$ be the isostructure sheaf in 
$({X/S})_{\rm conv}$.  
Let $({(X,D)/S})_{\rm conv}(=
({(X,D)/S})_{{\rm conv},{\rm zar}})$ be the log convergent topos 
of $(X,D)/S$ which is the relative version of 
the log convergent topos in \cite{s2}.  
(In \S\ref{sec:logcd} below we shall define 
the topoi $({X/S})_{\rm conv}$ and 
$({(X,D)/S})_{\rm conv}$.) 
Let ${\cal K}_{(X,D)/S}$ be the isostructure sheaf in 
$({(X,D)/S})_{\rm conv}$.  
Let 
$$\eps^{\rm conv}_{(X,D)/S} \col 
(({(X,D)/S})_{\rm conv},{\cal K}_{(X,D)/S}) 
\lo (({X/S})_{\rm conv},{\cal K}_{X/S})$$ 
be the morphism of ringed topoi forgetting the log structure along $D$.  
%(cf.~\cite{ofo}). 
Set ${\cal K}_S:={\cal O}_S\otimes_{\mab Z}{\mab Q}$. 
Let 
$$u^{\rm conv}_{X/S} \col 
(({X/S})_{\rm conv},{\cal K}_{X/S})
\lo ({X}_{\rm zar},f^{-1}({\cal K}_S))$$ 
be the canonical projection. 
Let ${\rm D}^+{\rm F}({\cal K}_{(X,D)/S})$ be the derived category of bounded below 
filtered complexes of ${\cal K}_{(X,D)/S}$-modules and let 
${\rm D}^+{\rm F}(f^{-1}({\cal K}_S))$ be the derived category of bounded below 
filtered complexes of $f^{-1}({\cal K}_S)$-modules. 
In this paper we define  the following four filtered complexes 
$$(E_{\rm conv}({\cal K}_{(X,D)/S}),P):=
(R\eps^{\rm conv}_{(X,D)/S*}({\cal K}_{(X,D)/S}),\tau)\in 
{\rm D}^+{\rm F}({\cal K}_{X/S}),$$ 
$$(E_{\rm isozar}({\cal K}_{(X,D)/S}),P)
:=Ru^{\rm conv}_{X/S*}
((E_{\rm conv}({\cal K}_{(X,D)/S}),P))
\in 
{\rm D}^+{\rm F}(f^{-1}({\cal K}_S)),$$ 
$$(C_{\rm conv}({\cal K}_{(X,D)/S}),P)\in 
{\rm D}^+{\rm F}({\cal K}_{X/S})$$   
and 
$$(C_{\rm isozar}({\cal K}_{(X,D)/S}),P)
:=
Ru^{\rm conv}_{X/S*}((C_{\rm conv}({\cal K}_{(X,D)/S}),P))
\in {\rm D}^+{\rm F}(f^{-1}({\cal K}_S)).$$ 
%in the case where $X$ is separated over $S_1$. 
In the text we shall give the definition of 
$(C_{\rm conv}({\cal K}_{(X,D)/S}),P)$,  
which is analogous to that of 
$(C_{\rm Rcrys}({\cal O}_{(X,D)/S}),P)$.  
We call $(C_{\rm conv}({\cal K}_{(X,D)/S}),P)$ 
the {\it weight-filtered convergent complex} of $(X,D)/S$, 
which plays a central role in this paper.  
\medskip
\par 
In this paper we prove the following three theorems:

\begin{theo}[{\bf $p$-adic purity}]\label{theo:pap} 
%Assume that $X$ is separated over $S_1$. Then 
There exists a canonical isomorphism 
\begin{equation*} 
R^k\eps^{\rm conv}_{(X,D)/S*}({\cal K}_{(X,D)/S}) \os{\sim}{\lo} 
a^{(k)}_{{\rm conv}*}({\cal K}_{D^{(k)}/S}
\otimes_{\mab Z}
\vp^{(k)}_{\rm conv}(D/S))
\quad (k\in {\mab N}),
\tag{1.1.1}\label{eqn:vpcs}
\end{equation*}
where $\vp^{(k)}_{\rm conv}(D/S)$ is the convergent  
orientation sheaf which will be defined in 
\S{\rm \ref{sec:lcs}} below. 
\end{theo}

\begin{theo}[{\bf Comparison theorem between $(E_{\rm conv},P)$ 
and $(C_{\rm conv},P)$}]\label{theo:cteck} 
%Assume that $X$ is separated over $S_1$. Then 
There exists a canonical isomorphism 
\begin{equation*}
(E_{\rm conv}({\cal K}_{(X,D)/S}),P)\os{\sim}{\lo} 
(C_{\rm conv}({\cal K}_{(X,D)/S}),P).
\tag{1.2.1}\label{eqn:excrs}
\end{equation*}
Consequently there exists a canonical isomorphism
\begin{equation*}
(E_{\rm isozar}({\cal K}_{(X,D)/S}),P)\os{\sim}{\lo} 
(C_{\rm isozar}({\cal K}_{(X,D)/S}),P), 
\tag{1.2.2}\label{eqn:uec}
\end{equation*} 
and $(C_{\rm conv}({\cal K}_{(X,D)/S}),P)$ and 
$(C_{\rm isozar}({\cal K}_{(X,D)/S}),P)$ are 
contravariantly functorial.
\end{theo} 

%\begin{theo}[Representability of 
%$(C_{\rm conv}({\cal K}_{(X,D)/S}),P)$ 
%as a sheaf of 
%filtered positively-graded graded commutative dga's]
%\label{theo:asfp}  
%There exists an object  
%in ${\rm A}^{\geq 0}{\rm F}({\cal K}_{X/S})$ 
%which represents 
%$(C_{\rm conv}({\cal K}_{(X,D)/S}),P)$.
%\end{theo} 

\begin{theo}[{\bf Comparison theorem 
between $(C_{\rm isozar},P)$ 
and $(C_{\rm zar},P)$}]\label{theo:izz} 
%Assume that $X$ is separated over $S_1$. Then 
There exists a canonical isomorphism 
\begin{equation*}
(C_{\rm isozar}({\cal K}_{(X,D)/S}),P) 
\os{\sim}{\lo} 
(C_{\rm zar}({\cal O}_{(X,D)/S}),P)\otimes_{\mab Z}{\mab Q}.
\tag{1.3.1}\label{eqn:crs}
\end{equation*} 
\end{theo}
\par 
Let ${\rm D}({\rm A}^{\geq 0}{\rm F}({\cal K}_{X/S}))$ 
be the analogous localized category to   
${\rm D}({\rm A}^{\geq 0}{\rm F}({\cal O}_{X/S}))$ 
which is obtained by the replacement of 
${\cal O}_{X/S}$ by ${\cal K}_{X/S}$.  
Let 
$U_{\rm conv} \col 
{\rm D}({\rm A}^{\geq 0}{\rm F}({\cal K}_{X/S})) 
\lo {\rm D}^+{\rm F}({\cal K}_{X/S})$ be 
the natural forgetful functor. 
By the construction of 
$(C_{\rm conv}({\cal K}_{(X,D)/S}),P)$, 
by the derived direct image of Thom-Whitney (\cite{nav})
and by the single complex of Thom-Whitney ([loc.~cit.]), 
we see that there exists  
a functorial filtered object 
$(\wt{C},\wt{P})\in 
{\rm D}({\rm A}^{\geq 0}{\rm F}({\cal K}_{X/S}))$ 
such that 
$U_{\rm conv}((\wt{C},\wt{P})) 
\simeq (C_{\rm conv}({\cal K}_{(X,D)/S}),P)$.  
Thus we need not consider the undesirable consequences 
in the log crystalline case mentioned above if we ignore torsions 
and use log convergent topoi. 
%(However, in the text, we also remark that 
%the following equality 
%\begin{equation*}
%R^k\eps^{\rm crys}_{(X,D)/S*}({\cal O}_{(X,D)/S})
%\otimes_{\mab Z}{\mab Q}=
%a^{(k)}_{{\rm crys}*}({\cal O}_{D^{(k)}/S}
%\otimes_{\mab Z}
%\vp^{(k)}_{\rm crys}(D/S))
%\otimes_{\mab Z}{\mab Q}
%\quad (k\in {\mab N})
%\tag{1.3.2}\label{eqn:nvpcktrs}
%\end{equation*}
%also holds.)
\par 
As in the crystalline case in \cite{bb}, 
we give a well-known remark about 
(\ref{eqn:vpcs}) as follows.  
\par 
Set $U:=X\setminus D$. Let $j \col U \os{\sus}{\lo} X$ 
be the open immersion and let $j_{\rm conv} \col 
({U/S})_{\rm conv} \lo ({X/S})_{\rm conv}$ 
be the induced morphism of topoi by $j$. 
We  remark that   
\begin{equation*} 
R^kj_{{\rm conv}*}({\cal K}_{U/S}) =0 \quad (\forall k\in {\mab Z}_{\geq 1}). 
\tag{1.3.2}\label{eqn:jnvp}
\end{equation*}
In particular, 
\begin{equation*} 
R^kj_{{\rm conv}*}({\cal K}_{U/S})
\not\simeq
a^{(k)}_{{\rm conv}*}({\cal K}_{D^{(k)}/S}
\otimes_{\mab Z}\vp^{(k)}_{\rm conv}(D/S))
\quad (\exists k\in {\mab N})
\tag{1.3.3}\label{eqn:jnnsvp}
\end{equation*}
in general ((\ref{rema:wnjp}) below). 
Thus $R\eps^{\rm conv}_{(X,D)/S*}$ is better 
than $Rj_{{\rm conv}*}$ 
in the $p$-adic case, which is different from 
the $l$-adic case in \cite{fu}, \cite{fk} and \cite{illl}. 
\par  
Let $(X,D)$ be as above or 
a smooth scheme with an SNCD 
over the complex number field ${\mab C}$. 
Set $U:=X\setminus D$ and let $j \col U \os{\sus}{\lo} X$ 
be the open immersion. 
Then, as in  \cite[(0.0.0.22)]{nh2}, 
we can give the following translation, which fills blanks (modulo torsion) in [loc.~cit.]: 
\begin{equation*}
\begin{tabular}{|l|l|} \hline 
$/{\mab C}$  & (log) convergent topoi\\ \hline    
$U_{\rm an}$, $(X_{\rm an},D_{\rm an})$ & 
${((X,D)/S)}_{\rm conv}$\\ \hline 
$X_{\rm an}$ & ${(X/S)}_{\rm conv}$ \\ \hline 
$j_{\rm an} \col U_{\rm an} \os{\subset}{\lo} 
X_{\rm an}$, & {} \\
$\eps_{\rm top} \col (X_{\rm an},D_{\rm an})^{\log} 
\lo X_{\rm an}$ & 
$\eps^{\rm conv}_{(X,D)/S} \col {((X,D)/S)}_{\rm conv} \lo 
{(X/S)}_{\rm conv}$ \\ \hline 
$Rj_{{\rm an}*}({\mab Q})=R\eps_{{\rm top}*}({\mab Q})~$ 
(\cite{kn})  & $
R\eps^{\rm conv}_{(X,D)/S*}({\cal K}_{(X,D)/S})$\\ \hline 
$X_{\rm an} \lo X$ & 
$u^{\rm conv}_{X/S} \col ({X/S})_{\rm conv} \lo 
{X}_{\rm zar}$ \\ \hline
${\mab Q}_{(X_{\rm an},D_{\rm an})}$ & ${\cal K}_{(X,D)/S}$ \\ \hline
${\mab Q}_{X_{\rm an}}$ 
& ${\cal K}_{X/S}$ \\ \hline 
$(\Om^{\bul}_{X_{\rm an}/{\mab C}}(\log D_{\rm an}),\tau)$ 
& $(C_{\rm conv}({\cal K}_{(X,D)/S}),\tau)$\\ 
$=(\Om^{\bul}_{X_{\rm an}/{\mab C}}(\log D_{\rm an}),P)$ & 
$=(C_{\rm conv}({\cal K}_{(X,D)/S}),P)$ \\ \hline
$(\Om^{\bul}_{X/{\mab C}}(\log D),P)$  
& $(C_{\rm isozar}({\cal K}_{(X,D)/S}),P)$ \\ \hline 
\end{tabular}
\tag{1.3.4}\label{tab:intop}
\end{equation*}
Here $(X_{\rm an},D_{\rm an})^{\log}$ is 
the real blow up of $(X_{\rm an},D_{\rm an})$ defined 
in \cite{kn} and 
$\eps_{\rm top}$ 
is the natural morphism of topological spaces 
which is denoted by $\tau$ in [loc.~cit.]. 
(See \cite[(2.7.0.1)]{nh2} for the analogous objects 
in the $l$-adic case which can be inserted 
in (\ref{tab:intop}).) 
\par
The contents of this paper are as follows. 
\par  
In \S\ref{sec:logcd} (resp.~\S\ref{sec:lll})
we give the definition of a log convergent topos  
and the system of log universal enlargements 
(resp.~a log convergent linearization functor) 
and give basic properties of them. 
%\par 
%In \S\ref{sec:sdtt} we give the definition of 
%the log convergent linearization functor following 
%the method in \cite{bob} 
%for the crystalline linearization functor. 
\par 
In \S\ref{sec:llf} we give the compatibility of 
the log convergent linearization functor with 
a closed immersion of log schemes. 
\par 
In \S\ref{sec:vflvc}   
we prove the Poincar\'{e} lemma of the vanishing cycle sheaf 
$R\eps^{\rm conv}_{(X,D)/S*}
(\eps^{{\rm conv}*}_{(X,D)/S}(E))$  
for a coherent crystal $E$ of ${\cal K}_{X/S}$-modules. 
\par 
In \S\ref{sec:lcs} we  give fundamental results for  
log convergent linearization functors 
of smooth schemes with relative SNCD's.  
\par 
In \S\ref{sec:rlct} we review 
the (bi)simplicial topoi of log convergent topoi. 
\par 
In \S\ref{sec:wfcipp} 
we give the definition of 
$(C_{\rm conv}({\cal K}_{(X,D)/S}),P)$  
and prove (\ref{theo:pap}) and (\ref{theo:cteck}).  
\par 
Let $f^{(k)}\col D^{(k)}\lo S$ be the structural morphism. 
Set $f^{\rm conv}_{D^{(k)}/S}
:=f^{(k)}\circ u^{\rm conv}_{D^{(k)}/S}$  
and $f^{\rm conv}_{(X,D)/S}:=f\circ u^{\rm conv}_{(X,D)/S}$.  
In \S\ref{sec:fpw} 
we construct the following weight spectral sequence
\begin{align*}
E_1^{-k,h+k} & = 
R^{h-k}f^{\rm conv}_{D^{(k)}/S*}({\cal K}_{D^{(k)}/S}
\otimes_{\mab Z}
\vp^{(k)}_{{\rm conv}}(D/S))(-k) \tag{1.3.5}\label{ali:wtwtc}\\
{} & \Lo  R^hf^{\rm conv}_{(X,D)/S*}({\cal K}_{(X,D)/S}).  
\end{align*} 
\par 
In \S\ref{sec:fpw} we give the contravariant functoriality of 
the weight-filtered convergent and isozariskian complexes. 
\par 
In \S\ref{sec:bd} we give an explicit description of 
the boundary morphism of the weight spectral sequence 
(\ref{ali:wtwtc}). 
\par 
In \S\ref{sec:ct} we prove (\ref{theo:izz}). 
Set 
$f^{\rm crys}_{D^{(k)}/S}
:=f^{(k)}\circ u^{\rm crys}_{D^{(k)}/S}$ 
and 
$f^{\rm crys}_{(X,D)/S}:=f\circ u^{\rm crys}_{(X,D)/S}$. 
As a corollary of (\ref{theo:izz}),  
(\ref{ali:wtwtc}) turns out to be canonically isomorphic to 
the following weight spectral sequence 
\begin{align*}
E_1^{-k,h+k} & = 
R^{h-k}f^{\rm crys}_{D^{(k)}/S*}({\cal O}_{D^{(k)}/S}
\otimes_{\mab Z}
\vp^{(k)}_{{\rm crys}}(D/S))(-k)\otimes_{\mab Z}{\mab Q} 
\tag{1.3.6}\label{ali:wtwcc}\\
{} & \Lo  R^hf^{\rm crys}_{(X,D)/S*}({\cal O}_{(X,D)/S})
\otimes_{\mab Z}{\mab Q}, 
\end{align*}
which has been constructed in \cite{nh2}. 
As a corollary of the comparison of (\ref{ali:wtwtc}) 
with (\ref{ali:wtwcc}), 
if $X$ is proper over $S_1$, 
then we obtain the $E_2$-degeneration of (\ref{ali:wtwtc}) 
by reducing it to that of (\ref{ali:wtwcc}) 
which we have already obtained in \cite{nh2}.  
In the same section, as another corollary of (\ref{theo:izz}), 
we deduce the filtered base change theorem 
(resp.~filtered K\"{u}nneth formula) 
of $(E_{\rm conv}({\cal K}_{(X,D)/S}),P)$ from 
the filtered base change theorem 
(resp.~filtered K\"{u}nneth formula) of 
$(E_{\rm crys}({\cal O}_{(X,D)/S}),P)$ which 
we have already obtained in [loc.~cit.]. 
\par 
In fact, we give the theorems above 
for a truncated simplicial smooth scheme with 
a truncated simplicial SNCD over $S_1$ 
satisfying a certain condition. 
(A split truncated simplicial smooth scheme 
with a split truncated simplicial SNCD over $S_1$ 
satisfies this condition.) 
%and construct 
%the split simplicial version of (\ref{ali:wtwtc}) 
%and give the comparison of it with 
%the split simplicial version of (\ref{ali:wtwcc}), 
%which has been constructed in \cite{nh3}.  
%We also obtain $E_2$-degeneration of 
%the split simplicial version of 
%the spectral sequence (\ref{ali:wtwtc}).  
\par 
In \S\ref{sec:cptsc} we define the 
filtered complex
$(E_{\rm conv,c}({\cal K}_{(X,D)/S}),P)$ 
in ${\rm D}^+{\rm F}({\cal K}_{X/S})$ 
which can be considered as 
the compact support version 
of $(E_{\rm conv}({\cal K}_{(X,D)/S}),P)$.   
We prove a filtered Poincar\'{e} duality 
by using $(E_{\rm conv,c}({\cal K}_{(X,D)/S}),P)$ 
and $(E_{\rm conv}({\cal K}_{(X,D)/S}),P)$. 
\par 
In a future paper we would like to 
construct an analogous theory 
for a proper SNCL(=simple normal crossing log) scheme 
over a family of log points.

\bigskip
\parno
{\bf Notations.}
%(1) For a complex $(E^{\bul}, d^{\bul})$ and for an 
%integer $n$, $(E^{\bul}\{n\}, d^{\bul}\{n\})$ 
%denotes the following complex:
%$$\cdots \lo \us{q-1}{E^{q-1+n}}\os{d^{q-1+n}}{\lo} 
%\us{q}{E^{q+n}} \os{d^{q+n}}{\lo} \us{q+1}{E^{q+1+n}} 
%\os{d^{q+1+n}}{\lo} 
%\cdots.$$
%Here the numbers under $E$'s mean the degrees. 
%\par 
%(2) 
(1) The log structure of a log scheme in this paper is 
a sheaf of monoids in the Zariski topos. 
For a log (formal) scheme $S$, 
$\os{\circ}{S}$ denotes 
the underlying (formal) scheme of $S$ and 
${\cal K}_S$ denotes ${\cal O}_S\otimes_{\mab Z}{\mab Q}$. 
More generally, for a ringed topos $({\cal T},{\cal O}_{\cal T})$,  
${\cal K}_{\cal T}$ denotes 
${\cal O}_{\cal T}\otimes_{\mab Z}{\mab Q}$. 
For a morphism $f\col X\lo S$ of log (formal) schemes, 
$\os{\circ}{f}$ denotes the underlying morphism 
$\os{\circ}{X} \lo \os{\circ}{S}$ of (formal) schemes. 
\par
%(3) 
(2) SNCD=simple normal crossing divisor.
\par 
%(4) 
(3) For a morphism $X \lo S$ of log (formal) schemes, 
we denote by 
$\Om^i_{X/S}$ ($=\om^i_{X/S}$ in 
\cite{klog1}) the sheaf of 
(formal) relative logarithmic differential forms on $X/S$
of degree $i$ $(i\in {\mab N})$. 
\par 
%(5) 
(4) For a commutative monoid $P$ with unit element 
and for a formal scheme $S$ 
with an ideal of definition of ${\cal I}\subset {\cal O}_S$, 
we denote by ${\cal O}_S\{P\}$ 
the sheaf $\vpl_n({\cal O}_S[P]/{\cal I}^n{\cal O}_S[P])$ 
in the Zariski topos ${S}_{\rm zar}$ of $S$. 
We denote by $\ul{\rm Spf}_S({\cal O}_S\{P\})$ 
the log scheme whose underlying scheme is 
$\ul{\rm Spf}_{S}({\cal O}_S\{P\})$ 
(with an ideal of definition ${\cal I}{\cal O}_S\{P\}$) 
and whose log structure 
is associated to the natural morphism 
$P \lo {\cal O}_S\{P\}$ of sheaves of monoids in $S_{\rm zar}$. 
When ${\cal I}=0$, we denote 
$\ul{\rm Spf}_{S}({\cal O}_S\{P\})$ by 
$\ul{\rm Spec}_{S}({\cal O}_S\{P\})$.  
\par 
For a log formal scheme $S$ 
with an ideal of definition of ${\cal I}\subset {\cal O}_S$, 
we denote by $\ul{\rm Spec}_S({\cal O}_S/{\cal I})$ 
the log scheme whose underlying scheme is 
$\ul{\rm Spec}_{\os{\circ}{S}}({\cal O}_S/{\cal I})$ 
and whose log structure 
is the inverse image of the log structure of $S$ 
by the natural closed immersion 
$\ul{\rm Spec}_{\os{\circ}{S}}({\cal O}_S/{\cal I})\os{\sus}{\lo} S$. 
%we denote by $\ul{\rm Spec}_S({\mab Z}[R])$ 
%the log scheme whose underlying scheme is 
%${\rm Spec}({\mab Z}[R])$ and whose log structure 
%is associated to the natural morphism $R \lo {\mab Z}[R]$ 
%of monoids. For a prime number $p$, we denote by 
%${\rm Spf}({\mab Z}_p\{R\})$ 
%the $p$-adic completion of 
%${\rm Spec}({\mab Z}[R])$. 
\par 
(5) 
%(6) 
For a morphism $f\col X\lo S$ and $i\in {\mab Z}_{\geq 1}$, 
we denote 
$\underset{(i+1)~{\rm pieces}}{\underbrace{
X\times_SX\times_S
\times \cdots \times_SX}}$ by $X(i)$. 

\bigskip
\par\noindent
{\bf Conventions.}
We use the conventions about signs in 
\cite{nh2} and \cite{nh3}.

\section{Log convergent topoi and the systems of 
log universal enlargements}\label{sec:logcd}
\par 
In this section we define the relative versions of basic notions 
about the log convergent topos in \cite{ofo} and \cite{s2} 
which is a log version of the convergent topos in \cite{oc}. 
In \cite{s3} the second named author 
has also developed an analogous theory. 
As in \cite{oc},  \cite{s2} and  \cite{s3}, 
we define a (pre)widening and 
the log universal enlargement of it and 
show basic properties of 
the log universal enlargement. 
\par 
Let ${\cal V}$ be a complete discrete valuation ring 
of mixed characteristics $(0,p)$ 
with perfect residue field $\kap$. 
Let $S$ be a fine log formal scheme 
whose underlying formal scheme is a 
$p$-adic formal ${\cal V}$-scheme 
in the sense of \cite{of}: 
$\os{\circ}{S}$ is a noetherian formal scheme 
over ${\rm Spf}({\cal V})$ 
with the $p$-adic topology 
which is topologically of finite type over ${\rm Spf}({\cal V})$.  
In this paper we always assume that $\os{\circ}{S}$ is 
flat over ${\rm Spf}({\cal V})$. 
We call such a log formal scheme $S$ a fine log flat $p$-adic formal ${\cal V}$-scheme.  
Let $\pi$ be a non-zero element of the maximal ideal 
of ${\cal V}$.  Following \cite{oc}, we denote by $S_1$ 
the exact closed log subscheme of $S$ defined by 
the ideal sheaf $\pi{\cal O}_S$. 
%and $S_0$ denotes $(S_1)_{\rm red}$ 
%(with the pull-back of the log structure of $S_1$). 
Let $Y$ be a fine log scheme over $S_1$ 
whose underlying scheme is of finite type 
over $\os{\circ}{S}_1$. 
As in \cite[p.~212]{ofo} and \cite[Definition 2.1.8]{s2} 
(see also \cite[I (2.1)]{s3}), 
we consider a quadruple $(U,T,\iota,u)$, 
where $u \col U \lo Y$ is a strict morphism of 
fine log schemes over $S_1$ such that 
$\os{\circ}{u}\col \os{\circ}{U}\lo \os{\circ}{Y}$ is a morphism of schemes 
of finite type,  
$T$ is a fine log noetherian (not necessarily $p$-adic) 
formal scheme over $S$ 
whose underlying formal scheme 
is topologically of finite type over $\os{\circ}{S}$ 
and $\iota \col U \os{\sus}{\lo} T$ is 
a closed immersion over $S$. 
Here we have assumed the condition 
``the strictness of $u$'',  
which has not been assumed 
in \cite[Definition 2.1.1, 2.1.9]{s2} 
and \cite[I (2.1), (2.2)]{s3}. 
We shall use this strictness 
essentially in \S\ref{sec:vflvc} below. 
We call the $(U,T,\iota,u)$  
a {\it prewidening} of $Y/S$. 
If $\os{\circ}{T}$ is $p$-adic, we say that 
$(U,T,\iota,u)$ is a $p$-{\it adic} {\it prewidening} of $Y/S$.  
Let $T_0$ be the maximal reduced log subscheme of 
the exact closed log subscheme of $T$ 
defined by $\pi{\cal O}_T$.  
Let $\wt{T}$ be a log formal scheme whose underlying scheme 
is $\ul{\rm Spf}_{\os{\circ}{T}}
({\cal O}_T/(\pi{\textrm -}{\rm torsion}))$ 
and whose log structure is 
the inverse image of the log structure of $T$. 
By imitating \cite[(1.1)]{oc} and 
\cite[Definition 2.1.1, 2.1.9]{s2}, 
we obtain the notion of a widening, an exact (pre)widening 
and an enlargement of $Y/S$ as follows (see also \cite[I (2.2)]{s3}). 
\par 
(1) We say that 
a ($p$-adic) prewidening 
$(U,T, \iota,u)$ of $Y/S$
is a ($p$-{\it adic}) {\it widening} of $Y/S$ if 
$\os{\circ}{\iota} \col \os{\circ}{U} \os{\sus}{\lo} \os{\circ}{T}$ 
is defined by an ideal sheaf of definition of $\os{\circ}{T}$. 
%\par 
%$(2)'$ A $p$-adic {\it widening} of $Y/S$ is 
%a $p$-adic prewidening and a widening. 
\par 
(2)  We say that a ($p$-adic) (pre)widening 
$(U,T, \iota,u)$ of $Y/S$
is {\it exact} if $\iota$ is exact. 
\par 
(3) We say that an exact widening 
$(U,T,\iota,u)$ of $Y/S$ is 
an {\it enlargement} of $Y/S$ if $\os{\circ}{T}$ is flat over 
${\rm Spf}({\cal V})$ and if $T_0 
%\ul{\rm Spf}_T(({\cal O}_T/\pi{\cal O}_T)_{\rm red}) 
\subset U$; 
the latter condition is equivalent to the following condition  
\begin{equation*} 
({\rm Ker}({\cal O}_T \lo {\cal O}_U))^N \subset \pi{\cal O}_T  
\quad (\exists N\in {\mab N}).  
\tag{2.0.1}\label{eqn:kotu} 
\end{equation*} 
(As noted in \cite[(1.2)]{oc}, the topology of $\os{\circ}{T}$ is $p$-adic.) 
\par 
We define a {\it morphism of {\rm (}$p$-adic{\rm )} 
{\rm (}exact{\rm )} {\rm (}pre{\rm )}widenings and 
a morphism of  enlargements} in an obvious way. 
For a morphism 
$f\col (U,T,\iota,u) \lo (U',T',\iota',u')$ of 
($p$-adic) (exact) (pre)widenings 
or enlargements of $Y/S$, we denote the morphism 
$T\lo T'$ of fine log ($p$-adic) formal schemes over $S$ 
by $f\col T \lo T'$ by abuse of notation. 
As in the log crystalline case in \cite[(5.2)]{klog1},  
for a morphism $(U,T,\iota,u) \lo (U',T',\iota',u')$ 
of exact widenings of $Y/S$, we can easily show 
that the morphism $T \lo T'$ is strict by using 
the strictness of the morphisms $U\lo Y$ and $U'\lo Y$. 
As usual, we denote $(U,T,\iota,u)$ simply by $T$ 
if there is no risk of confusion. 
\par  
By imitating \cite[p.~135]{oc} and 
\cite[Definition (2.1.3)]{s2}, we define 
the log convergent site 
${\rm Conv}(Y/S)(={\rm Conv}_{\rm zar}(Y/S))$ 
as follows (see also \cite[I (2.4)]{s3}):  
the objects of ${\rm Conv}(Y/S)$
are the enlargements of $Y/S$; the morphisms 
of ${\rm Conv}(Y/S)$ are the morphisms of enlargements of 
$Y/S$; for an object $(U,T,\iota,u)$ of 
${\rm Conv}(Y/S)$, a family 
$\{(U_{\lam},T_{\lam},\iota_{\lam},u_{\lam})
\lo (U,T,\iota,u)\}_{\lam}$ 
of morphisms in  ${\rm Conv}(Y/S)$ is defined to 
be a covering family of $(U,T,\iota,u)$ if 
$T=\bigcup_{\lam}T_{\lam}$ is a Zariski open covering and if 
the natural morphism $U_{\lam}\lo T_{\lam}\times_TU$ is an isomorphism.  
Let $(Y/S)_{\rm conv}$ 
be the topos associated to the site ${\rm Conv}(Y/S)$.  
An object $E$ in $(Y/S)_{\rm conv}$ 
gives a Zariski sheaf $E_T$ 
in the Zariski topos $T_{\rm zar}$ of $T$ 
for $(U,T, \iota,u)\in {\rm Ob}({\rm Conv}(Y/S))$ 
and a morphism 
$\rho_f \col f^{-1}(E_{T'}) \lo E_T$ 
for a morphism 
$f\col T \lo T'$ of enlargements of $Y/S$. 
As in \cite{oc} and \cite{s2}, 
for a prewidening $T$ of $Y/S$, 
we denote simply by $T$ 
the following sheaf 
\begin{equation*} 
h_T \col  {\rm Conv}(Y/S)\owns T' \lom 
\{\text{the morphisms $T' \lo T$ of prewidenings 
of $Y/S$}\} \in ({\rm Sets})  
\end{equation*} 
on ${\rm Conv}(Y/S)$.  
Let ${\cal K}_{Y/S}$ be an isostructure sheaf in 
$(Y/S)_{\rm conv}$ defined by 
the following formula:  
$$\Gam((U,T,\iota,u),{\cal K}_{Y/S}):=
\Gam(T, {\cal O}_T)\otimes_{\mab Z}{\mab Q}\quad ((U,T, \iota,u)\in {\rm Ob}({\rm Conv}(Y/S)).$$ 
\par 
Let $f \col Y \lo S_1\os{\sus}{\lo} S$ 
be the structural morphism. 
Let $Y_{\rm zar}$ be the Zariski topos of $Y$ and 
let 
\begin{equation*}
u^{\rm conv}_{Y/S} \col (Y/S)_{\rm conv} \lo Y_{\rm zar} 
\tag{2.0.2}\label{eqn:uwks}
\end{equation*}
be a morphism of topoi 
characterized by a formula 
$u^{{\rm conv}*}_{Y/S}(E)(T)=
\Gam(U,u^{-1}(E))$ for a sheaf  
$E$ in $Y_{\rm zar}$ and an enlargement $T=(U,T,\iota,u)$ of $Y/S$. 
This morphism is 
the obvious relative version of 
the natural projection in 
\cite[p.~146]{oc} and \cite[pp.90--91]{s2}.  
Set $f^{\rm conv}_{Y/S}:=f\circ u^{\rm conv}_{Y/S}$. 
The morphism (\ref{eqn:uwks}) induces the following 
morphism of ringed topoi: 
\begin{equation*}
u^{\rm conv}_{Y/S} \col 
((Y/S)_{\rm conv},{\cal K}_{Y/S})
\lo (Y_{\rm zar}, f^{-1}({\cal K}_S)). 
\tag{2.0.3}\label{eqn:uwksr}
\end{equation*} 
\par
Let ${\cal V}'$ be another complete discrete valuation ring 
of mixed characteristics $(0,p)$ with perfect residue field. 
Let $\pi'$ be a non-zero element of the maximal ideal 
of ${\cal V}'$. Let ${\cal V} \lo {\cal V}'$ be a morphism of 
commutative rings with unit elements inducing a morphism 
${\cal V}/\pi{\cal V} \lo {\cal V}'/\pi' {\cal V}'$. 
Let $S'$ be a fine log flat $p$-adic formal ${\cal V}'$-scheme.  
Set $S'_1:=\ul{\rm Spec}_{S'}({\cal O}_{S'}/\pi'{\cal O}_{S'})$. 
Let $Y'$ be a fine log scheme over $S'_1$ such that 
the underlying structural morphism 
$\os{\circ}{Y}{}'\lo \os{\circ}{S}{}'_1$ is of finite type. 
For a commutative diagram 
\begin{equation*} 
\begin{CD}
Y' @>>> S'_1 @>{\subset}>> S' @>>>{\rm Spf}({\cal V}')\\
@V{g}VV @VVV  @VVV @VVV \\ 
Y    @>>> S_1 @>{\subset}>> S  @>>>{\rm Spf}({\cal V}) 
\end{CD}
\tag{2.0.4}\label{cd:yypssp}
\end{equation*}
of fine log (formal) schemes, we have a natural morphism 
\begin{equation*} 
g_{\rm conv} \col ((Y'/S')_{\rm conv},{\cal K}_{Y'/S'}) 
\lo ((Y/S)_{\rm conv},{\cal K}_{Y/S})  
\end{equation*} 
of ringed topoi. 
(We can check that $g^*_{\rm conv}$ commutes with the projective limit of 
projective system of finitely many sheaves in $(Y/S)_{\rm conv}$ as in the crystalline case 
in \cite{bob}.)
We remark 
that there is a point of difference in the nontrivial log case 
from the trivial log case in 
\cite[(2.2.1)]{of}: 
for an enlargement $(U',T',\iota',u')$ of  $Y'/S'$, 
the quadruple $(U',T',\iota',g\circ u')$ is 
not necessarily an enlargement of  $Y/S$ 
since $g\circ u'$ is not necessarily strict. 
\par 
The third statement in the following proposition 
tells us that the choice of $\pi$ is not important 
for the log convergent cohomology of 
the trivial coefficient as in the crystalline case 
(\cite[III Th\'{e}or\`{e}me 2.3.4 
(Th\'{e}or\`{e}me d'invariance)]{bb}).  

\begin{prop}\label{prop:toi} 
Assume that ${\cal V}'={\cal V}$ and that 
the morphism ${\cal V} \lo {\cal V}'$ is the identity of ${\cal V}$ 
$(\pi'$ may be different from $\pi)$. 
Assume also that $S'=S$ and $Y'=Y\times_{S_1}S'_1$. 
Denote $g$ in {\rm (\ref{cd:yypssp})} by $i$. 
Let $f' \col Y' \lo S$ 
be the structural morphism. 
Then the following hold$:$ 
\par 
$(1)$ The direct image $i_{{\rm conv}*}$ 
from the category of abelian sheaves in 
$(Y'/S')_{\rm conv}$ to 
the category of abelian sheaves in 
$(Y/S)_{\rm conv}$ 
is exact. 
\par 
$(2)$ For a bounded below filtered complex 
$(E^{\bul},P)$ of ${\cal K}_{Y'/S}$-modules, 
\begin{equation*} 
Rf^{\rm conv}_{Y/S*}(i_{{\rm conv}*}((E^{\bul},P)))
=Rf'{}^{\rm conv}_{Y'/S*}((E^{\bul},P)) 
\tag{2.1.1}\label{eqn:ysep}
\end{equation*} 
in ${\rm D}^+{\rm F}({\cal K}_S)$. 
\par 
$(3)$ 
\begin{equation*} 
Rf^{\rm conv}_{Y/S*}({\cal K}_{Y/S})=
Rf'{}^{\rm conv}_{Y'/S*}({\cal K}_{Y'/S}) 
\tag{2.1.2}\label{eqn:ksep}
\end{equation*} 
in $D^+({\cal K}_S)$. 
\end{prop} 
\begin{proof} 
%(The proof is simpler than that of [loc.~cit.].) 
%\par 
(1): For an object $(U,T,\iota,u)$ of ${\rm Conv}(Y/S)$, 
set $U':=U\times_{Y}Y'$, and 
let $u'\col U'\lo Y'$ be the projection and 
$\iota' \col U' \lo T$ the composite morphism 
$U' \lo U \os{\iota}{\lo} T$. 
Then it is immediate to check that $(U',T,\iota',u')$ 
is an object of ${\rm Conv}(Y'/S')$ and  that 
\begin{equation*} 
i^*_{\rm conv}((U,T,\iota,u))= (U',T,\iota',u'). 
\tag{2.1.3}\label{eqn:utiu} 
\end{equation*} 
Hence, for an abelian sheaf $E$ in 
$(Y'/S')_{\rm conv}$, 
\begin{equation*} 
i_{{\rm conv}*}(E)_{(U,T,\iota,u)}= E_{(U',T,\iota',u')}. 
\tag{2.1.4}\label{eqn:iuiue} 
\end{equation*}
In particular, $i_{{\rm conv}*}$ is exact. 
\par 
(2) By (1) we see that 
$Ri_{{\rm conv}*} =i_{{\rm conv}*}$. 
Thus we obtain (\ref{eqn:ysep}). 
\par 
(3): By (\ref{eqn:iuiue}) we have the following equality: 
\begin{equation*} 
i_{{\rm conv}*}({\cal K}_{Y'/S'})= {\cal K}_{Y/S}. 
\tag{2.1.5}\label{eqn:iuoe} 
\end{equation*} 
Hence (\ref{eqn:ksep}) follows from (\ref{eqn:ysep}). 
\end{proof} 

\par 
For a fine log flat $p$-adic formal ${\cal V}$-scheme $T$, 
let  ${\rm Coh}({\cal O}_T)$ be 
the category of coherent ${\cal O}_T$-modules. 
Let  ${\rm Coh}({\cal K}_T)$ be the full subcategory 
of the category of ${\cal K}_T$-modules 
whose objects are isomorphic to $K\otimes_{\cal V}{\cal F}$ 
for some ${\cal F}\in {\rm Coh}({\cal O}_T)$
(\cite[(1.1)]{of}). 
Following \cite{of}, \cite{oc}, \cite{ofo} and \cite{s2},  
we say that a ${\cal K}_{Y/S}$-module $E$ is 
a coherent isocrystal 
if $E_T\in {\rm Coh}({\cal K}_T)$ 
for any enlargement $T$ of $Y/S$  and if 
$\rho_f \col f^*(E_{T'}) \lo E_T$ is an isomorphism 
for any morphism $f\col T\lo T'$ of enlargements of  $Y/S$. 
\par 
Next we recall the relative version of 
the system of universal enlargements in \cite[(2.1)]{oc} 
and \cite[Definition 2.1.22]{s2} 
with an immediate generalization(=without using the 
existence of an exactification of a (pre)widening). 
(See \cite[I (2.10), (2.11), (2.12)]{s3} 
for a slightly different formulation using the existence of the exactification; 
we recall it in (\ref{prop:exad}) and 
(\ref{rema:rfite}) below.)
\par  
Let $(U,T,\iota,u)$ be 
an enlargement of $Y/S$. 
We say that $(U,T,\iota,u)$ is 
an enlargement of $Y/S$ {\it with radius} 
$\leq \vert \pi \vert$
if ${\rm Ker}({\cal O}_T \lo {\cal O}_U) \subset \pi{\cal O}_T$. 
The following is a log version of \cite[(2.1)]{oc} 
and an immediate generalization of 
\cite[Definition 2.1.22, Lemma 2.1.23]{s2} 
for the case $n=1$ in [loc.~cit.], 
which is a log version of \cite[(2.3)]{of} and \cite[(2.1)]{oc}:
%the log convergent version of \cite[(5.4)]{klog1}: 

\begin{prop}\label{prop:exue} 
Let $T=(U,T,\iota,u)$ be 
a $(p$-adic$)$ $($pre$)$widening of  
$Y/S$. Then there exists an enlargement 
$T_1=(U_1,T_1,\iota_1,u_1)$ of $Y/S$ 
with radius $\leq \vert \pi \vert$ 
with a morphism $T_1 \lo T$ of 
$(p$-adic$)$ $($pre$)$widenings of $Y/S$ 
satisfying the following universality$:$ 
any morphism $`T:=(`U,`T,`\iota,`u) \lo T$ 
of $(p$-adic$)$ $($pre$)$widenings of $Y/S$ 
from an enlargement of $Y/S$ 
with radius $\leq \vert \pi \vert$ 
factors through 
a unique morphism   
$`T \lo T_1$ of enlargements of $Y/S$. 
\end{prop}
\begin{proof} 
Because the question is local, we may assume 
that there exists an exactification of $\iota$ 
(cf.~\cite[(5.6)]{klog1}), that is, 
we may have the following commutative diagram 
\begin{equation*} 
\begin{CD} 
U@>{\sus}>> T''\\
@| @VVV \\
U@>{\sus}>> T, 
\end{CD}
\tag{2.2.1}\label{cd:utut}
\end{equation*}
where the upper horizontal morphism is 
an exact closed immersion and 
the right vertical morphism is formally log \'{e}tale. 
In this case, (\ref{prop:exue}) follows from 
the relative version of   
\cite[Definition 2.1.22, Lemma 2.1.23]{s2} 
(cf.~\cite[(2.3)]{of}, \cite[(2.1)]{oc}). 
For the completeness of this paper, we give 
the construction of $(U_1,T_1,\iota_1,u_1)$ as follows. 
\par 
Set $\os{\circ}{\wt{T''}}:=
\ul{\rm Spf}_{\os{\circ}{T}{}''}
({\cal O}_{T''}/(\pi{\textrm -}{\rm torsion}))$ 
and $\os{\circ}{U}{}'':=\os{\circ}{U}
\times_{\os{\circ}{T}}\os{\circ}{\wt{T''}}$. 
Let ${\cal I}''$ be the defining ideal sheaf of the immersion 
$\os{\circ}{U}{}'' \os{\sus}{\lo}  \os{\circ}{\wt{T''}}$. 
Take a formal blow up of $\os{\circ}{\wt{T''}}$ with respect to 
the ideal sheaf ${\cal I}''+\pi {\cal O}_{\os{\circ}{\wt{T''}}}$. 
Let $\os{\circ}{B}(U,T)$ be the $p$-adic completion of 
this formal blowing up and let $\os{\circ}{B}(U,T) \lo \os{\circ}{T''}$ 
be the natural morphism. 
Endow $\os{\circ}{B}(U,T)$ with 
the inverse image of the log structure of $T''$. 
Let $B(U,T)$ be the resulting log formal scheme. 
Let $T_1$ be the maximal log formal open subscheme of 
$B(U,T)$ such that 
${\cal I}''{\cal O}_{B(U,T)} \subset \pi {\cal O}_{B(U,T)}$. 
Set $U_1:=U\times_TT_1$. 
Then we have a natural immersion 
$\iota_1 \col U_1 \os{\sus}{\lo} T_1$ and 
a natural morphism $u_1 \col U_1 \lo Y$. 
The quadruple $(U_1,T_1,\iota_1,u_1)$ 
is a desired enlargement of $Y/S$ 
with radius $\leq \vert \pi \vert$. 
\end{proof}

\par 
Now, as in \cite{of}, \cite{oc} and \cite{s2}, 
by replacing the ideal sheaf ${\cal I}''$ in the proof of (\ref{prop:exue})
by ${\cal I}''{}^n$ 
$(n\in {\mab Z}_{\geq 1})$, 
we have the inductive system $\{T_n\}_{n=1}^{\infty}$ 
($T_n=(U_n, T_n,\iota_n,u_n)$). 
We call $\{T_n\}_{n=1}^{\infty}$ 
the {\it system of the universal enlargements} 
of $(U,T,\iota,u)$.  
By using the condition (\ref{eqn:kotu}) and 
by the definition of $T_n$,  
we have the following equality 
\begin{equation*}
h_T=\vil_{n}h_{T_n} 
\tag{2.2.2}\label{eqn:chynhn}
\end{equation*}
as in the proof of 
\cite[(2.1)]{oc} and \cite[Lemma 2.1.23]{s2}.  
We denote the log formal scheme $T_n$ 
by ${\mathfrak T}_{U,n}(T)$. 
By abuse of notation, we also denote 
the enlargement $(U_n,T_n,\iota_n,u_n)$ 
by ${\mathfrak T}_{U,n}(T)$.  
(In \cite{oc} 
(the log structures of $Y$ and $T$ are trivial there), 
Ogus has distinguished the quadruple $T_{Y,n}((U,T,u)):=(U_n,T_n,\iota_n,u_n)$ 
($T_{Y,n}((T,U,u)):=(U_n,T_n,\iota_n,u_n)$ in his notation)
from the formal scheme $T_{U,n}(T):=T_n$.) 

\begin{rema}\label{rema:exdras}
Let the notations be as above. 
Assume that $\os{\circ}{T}$ is an affine formal scheme 
${\rm Spf}(A)$ and that $\iota$ is an exact closed immersion. 
Let $I\subset A$ be the defining ideal of $\os{\circ}{\iota}$.
Let $f_1,\ldots, f_r$ $(r\in {\mab Z}_{\geq 1})$ 
be a system of generators of $I$. 
For $\ul{m}:=(m_1,\ldots,m_r)\in {\mab N}^r$, set 
$\vert \ul{m} \vert :=\sum_{j=1}^{r}m_j$ and  
$f^{\ul{m}}:=f^{m_1}_1\cdots f^{m_r}_r$. 
For $n\in {\mab Z}_{\geq 1}$ and 
$\ul{m}\in {\mab N}^r$ 
with $\vert \ul{m}\vert=n$, 
let $t_{\ul{m}}$ be independent variables.  
Then, in the proof of \cite[(2.3)]{of}, Ogus has (essentially) given  
the following description: 
\begin{align*}
\Gam({\mathfrak T}_{U,n}(T),
{\cal O}_{{\mathfrak T}_{U,n}(T)})
&=(A[t_{\ul{m}}~\vert~\ul{m} 
\in {\mab N}^r,\vert\ul{m}\vert =n]
/((f^{\ul{m}}-\pi t_{\ul{m}})+
(p{\textrm -}{\rm torsion})))^{\wh{}}\tag{2.3.1}\label{eqn:ldra} \\
&=\wh{A}\{t_{\ul{m}}~\vert~\ul{m} 
\in {\mab N}^r,\vert\ul{m}\vert =n\}
/((f^{\ul{m}}-\pi t_{\ul{m}})+
(p{\textrm -}{\rm torsion})), 
\end{align*}  
where the upper $\,\,\wh{~}\,\,$ on the right hand side 
means the $p$-adic completion.  
\end{rema}

\par 
The following is a relative version of 
\cite[Lemma 2.1.25]{s2} 
which is a log version of \cite[(2.4.4)]{of} 
(see also \cite[I (2.15)]{s3}): 

\begin{prop}\label{prop:fc} 
Let $(U,T,\iota,u)$ be a $(p$-adic$)$ prewidening of $Y/S$.  
Let $\wh{T}$ be the formal completion of $T$ along $U$ 
$(\wh{T}$ is, by definition, the log scheme whose underlying scheme 
is the formal completion of $\os{\circ}{T}$ along $\os{\circ}{U}$ 
and whose log structure is 
the inverse image of that of $T)$. 
Then the natural morphism 
${\mathfrak T}_{U,n}(\wh{T}) \lo {\mathfrak T}_{U,n}(T)$ 
$(n\in {\mab Z}_{\geq 1})$ 
of enlargements of $Y/S$ is an isomorphism. 
\end{prop} 
\begin{proof}(We need an additional slight care to 
\cite[Lemma 2.1.25]{s2} 
since $\iota$ is not necessarily exact.) 
The question is local. 
Let the notations be as in the proof of (\ref{prop:exue}). 
Let $\os{\circ}{\wh{T}{}''}$ be the formal completion of 
$\os{\circ}{T}{}''$ along $U$. 
Endow $\os{\circ}{\wh{T}{}''}$ with 
the inverse image of the log structure of $T''$ 
and let $\wh{T}{}''$ be the resulting log formal scheme. 
Then we have to prove that 
the natural morphism 
${\mathfrak T}_{U,n}(\wh{T}{}'') \lo {\mathfrak T}_{U,n}(T'')$ 
is an isomorphism, which is proved by 
the same proof as that of \cite[Lemma 2.1.25]{s2} 
using the description (\ref{eqn:ldra}): 
we can easily construct the inverse morphism 
${\mathfrak T}_{U,n}(T'')\lo {\mathfrak T}_{U,n}(\wh{T}{}'')$ 
of the morphism ${\mathfrak T}_{U,n}(\wh{T}{}'') \lo {\mathfrak T}_{U,n}(T'')$. 
Indeed, set ${\cal J}:={\rm Ker}({\cal O}_{T''}\lo {\cal O}_U)$. 
Since ${\cal J}^n{\cal O}_{{\mathfrak T}_{U,n}(T'')}\subset 
\pi{\cal O}_{{\mathfrak T}_{U,n}(T'')}$, 
we have a natural morphism 
${\cal O}_{\wh{T}{}''}\lo {\cal O}_{{\mathfrak T}_{U,n}(T'')}$ 
which induces a natural morphism 
${\cal O}_{{\mathfrak T}_{U,n}(\wh{T}{}'')}\lo {\cal O}_{{\mathfrak T}_{U,n}(T'')}$. 
\end{proof}

\begin{lemm}\label{lemm:uys1} 
Let $u \col U \lo Y$ be a strict morphism of 
fine log schemes over $S_1$ 
such that $\os{\circ}{u}$ is of finite type. 
Let $T$ be a fine log formal scheme over $S$ 
whose underlying formal scheme is a 
noetherian formal scheme 
which is topologically of finite type over $S$.   
Let $\iota \col U \os{\sus}{\lo} T$ be a $($not necessarily closed$)$ 
immersion over $S$. 
Let $`T$ be a log formal open subscheme of $T$ 
such that $`\os{\circ}{T}$ contains $\os{\circ}{U}$ 
as a closed subscheme. Then the enlargement 
${\mathfrak T}_{U,n}(`T)$ is independent of the choice of 
the fine log formal open subscheme $`T$ of $T$. 
\end{lemm} 
\begin{proof} 
(cf.~\cite[I Proposition 4.2.1]{bb}) 
Let $``T$ be another log formal open subscheme of $T$ 
such that $``\os{\circ}{T}$ contains $\os{\circ}{U}$ 
as a closed subscheme. We may assume that there exists 
an open immersion $``T\os{\sus}{\lo} `T$ such that the composite morphism 
$U\os{\sus}{\lo} ``T\os{\sus}{\lo} `T$ is the given immersion 
$U\os{\sus}{\lo} `T$. Let $\wh{?}$ be the formal completion along $U$.  
Then $\wh{``T}=\wh{`T}$ since they are supported on $U$. 
Hence (\ref{lemm:uys1}) follows from (\ref{prop:fc}).  
%By using (\ref{prop:fc}),  the proof is the same as that of \cite[I Proposition 4.2.1]{bb}. 
%Let  $`T\os{\sus}{\lo} ``T$ be an open immersion in $T$ 
%such that $``\os{\circ}{T}{}$ contains $\os{\circ}{U}$ 
%as a closed subscheme. Let $`\wh{T}{}$ and $``\wh{T}{}$ 
%be the formal completions of $`T$ and $``T$ along $U$, 
%respectively. 
%Then, by (\ref{prop:fc}), we have 
%the following natural commutative diagram of 
%enlargements of $Y/S$: 
%\begin{equation*} 
%\begin{CD} 
%{\mathfrak T}_{U,n}(`\wh{T}{}) 
%@>>> {\mathfrak T}_{U,n}(``\wh{T}{}) \\ 
%@V{\simeq}VV @VV{\simeq}V \\ 
%{\mathfrak T}_{U,n}(`T) @>>> {\mathfrak T}_{U,n}(``T). 
%\end{CD} 
%\end{equation*} 
%The morphism ${\mathfrak T}_{U,n}(`\wh{T}{}) 
%\lo {\mathfrak T}_{U,n}(``\wh{T}{})$ 
%an isomorphism over $`\wh{T}{}=``\wh{T}{}(=U)$ 
%as topological spaces. 
%Hence the morphism 
%${\mathfrak T}_{U,n}(`T)  \lo {\mathfrak T}_{U,n}(``T)$ 
%is an isomorphism. 
\end{proof}

\begin{defi}\label{defi:qpw} 
We call the quadruple $(U,T,\iota,u)$ in 
(\ref{lemm:uys1}) a {\it quasi-prewidening} 
of $Y/S$. We define a morphism of quasi-prewidenings 
of $Y/S$ in an obvious way. We often denote 
$(U,T,\iota,u)$ simply by $T$. 
We denote ${\mathfrak T}_{U,n}(`T)$ by 
${\mathfrak T}_{U,n}(T)$ and call 
$\{{\mathfrak T}_{U,n}(T)\}_{n=1}^{\infty}$ 
the {\it system of the universal enlargements} of 
$(U,T,\iota,u)$.  
\end{defi}

\par 
We recall the following(=the exactification) 
which has been proved in \cite{s3}.

\begin{prop}
[{\rm {\bf \cite[Proposition 2.10]{s3}}}]\label{prop:exad}
Let $T$ be a fine log formal scheme.    
Let ${\cal C}^{\rm ex}_{\rm hom}$ $($resp.~${\cal C})$ 
be the category of homeomorphic exact immersions 
$($resp.~the category of immersions$)$ of 
fine log formal schemes over $T$.  
Let $\iota$ be the inclusion functor 
${\cal C}^{\rm ex}_{\rm hom} \os{\sus}{\lo} {\cal C}$. 
Then $\iota$ has a right adjoint functor 
$(\quad)^{\rm ex}\col 
{\cal C}\lo {\cal C}^{\rm ex}_{\rm hom}$. 
\end{prop} 
Because we need the following notations  later, 
we recall the construction of $(\quad)^{\rm ex}$. 
By the universality, the problem is local. 
Let $Y\os{\sus}{\lo} {\cal Q}$ be an object 
of ${\cal C}$.
We may assume that the immersion above is closed and 
that it has a global chart $(P\lo Q)$ such that the morphism 
$P^{\rm gp}\lo Q^{\rm gp}$ is surjective.  
Let $P^{\rm ex}$ be the inverse image of 
$Q$ by the morphism $P^{\rm gp} \lo Q^{\rm gp}$. 
Then we have a fine log formal ${\cal V}$-scheme 
${\cal Q}^{{\rm loc,ex}}:=
{\cal Q}\wh{\times}_{\ul{\rm Spf}_T({\cal O}_T\{P\})}
\ul{\rm Spf}_T({\cal O}_T\{P^{\rm ex}\})$ 
over ${\cal Q}$ with an exact closed immersion 
$Y\os{\sus}{\lo} {\cal Q}^{{\rm loc,ex}}$. 
(Recall $\ul{\rm Spf}_T$ in the {\bf Notations} (4).) 
We take the formal completion of 
${\cal Q}^{{\rm loc,ex}}$ along $Y$, which 
we denote by ${\cal Q}^{\rm ex}$. 
Since $\os{\circ}{\cal Q}{}^{\rm ex}=\os{\circ}{Y}$ 
and 
$M_{{\cal Q}^{\rm ex},x}/{\cal O}^*_{{\cal Q}^{\rm ex},x}
=M_{Y,x}/{\cal O}^*_{Y,x}$ $(x\in \os{\circ}{Y})$, 
we see that the morphism $Y\os{\sus}{\lo} {\cal Q}^{\rm ex}$ 
is a desired morphism over $T$. 
By the construction above, we see that the natural morphism 
${\cal Q}^{\rm ex}\lo {\cal Q}$ is affine.  

\begin{rema}\label{rema:rfite} 
We also recall the formulation in 
\cite[I (2.10), (2.11), (2.12)]{s3}.   
Let the notations be as in (\ref{prop:exue}), 
before (\ref{rema:exdras}) and in (\ref{lemm:uys1}). 
Let $\iota^{\rm ex} \col U\os{\sus}{\lo} T^{\rm ex}$ 
be the exactification of $\iota$. 
Set 
$T^{\rm ex}:=(U,T^{\rm ex},\iota^{\rm ex},u)$ 
by abuse of notation. 
Then $T_n=(T^{\rm ex})_n$ $(n\in {\mab Z}_{\geq 1})$. 
Indeed 
%let ${\cal I}^{\rm ex}$ be the defining ideal of 
%the exact closed immersion $\iota^{\rm ex}$.  
we can obtain $T_n$ by replacing $(`T)''$ with $T^{\rm ex}$. 
\end{rema}

The following is a relative log version of 
a slight modification of \cite[Lemma 2.4]{oc} 
and \cite[Lemma 2.1.26]{s2} (see also \cite[I (2.16)]{s3}):

\begin{lemm}\label{lemm:bcue} 
Let the notations be as in {\rm (\ref{lemm:uys1})}. 
Let $(U',T',\iota',u')\lo (U,T,\iota,u)$ be a morphism 
of quasi-prewidenings of $Y/S$ such that $U'=U\times_TT'$.  
Then the natural morphism 
${\mathfrak T}_{U',n}(T') \lo 
\wt{{\mathfrak T}_{U,n}(T)\times_{T}{T'}}$ 
$(n\in {\mab Z}_{\geq 1})$ is an isomorphism. 
If $T' \lo T$ is flat, then the natural morphism 
${\mathfrak T}_{U',n}(T') \lo 
{\mathfrak T}_{U,n}(T)\times_{T}T'$ 
$(n\in {\mab Z}_{\geq 1})$ is an isomorphism. 
\end{lemm}  
\begin{proof} 
Indeed, the natural morphism ${\mathfrak T}_{U',n}(T') \lo 
{\mathfrak T}_{U,n}(T)\times_{T}{T'}$ induces the morphism 
${\mathfrak T}_{U',n}(T') \lo \wt{{\mathfrak T}_{U,n}(T)\times_{T}{T'}}$.  
By the assumption, 
${\rm Ker}({\cal O}_{T'}\lo {\cal O}_{U'})=
{\rm Ker}({\cal O}_T\lo {\cal O}_U)\otimes_{{\cal O}_T}{\cal O}_{T'}$. 
Hence the first claim follows from the local descriptions of 
${\mathfrak T}_{U',n}(T')$ and $\wt{{\mathfrak T}_{U,n}(T)\times_{T}{T'}}$. 
By the assumption and the definition of 
${\mathfrak T}_{U,n}(T)$, 
${\mathfrak T}_{U,n}(T)\times_{T}T'$  is ${\cal V}$-flat. 
Hence 
$\wt{{\mathfrak T}_{U,n}(T)\times_{T}{T'}}={\mathfrak T}_{U,n}(T)\times_{T}{T'}$.  
\end{proof} 
\par 
Let the notations be as in 
(\ref{lemm:uys1}) and (\ref{defi:qpw}). 
Set $T_n:={\mathfrak T}_{U,n}(T)$. 
Let $\bet_n \col T_n \lo T$ $(n\in {\mab Z}_{\geq 1})$ 
be the natural morphism. 
%Because $`\wh{T}=``\wh{T}$,  
Because $h_{`T}=\vil_{n}h_{T_n}$ 
is independent of the choice of $`T$, 
we can set 
\begin{equation*}
h_T=h_{`T}  
\tag{2.9.1}\label{eqn:hynhn}
\end{equation*} 
on ${\rm Conv}(Y/S)$. 
Set also $(Y/S)_{\rm conv}\vert_T
:=(Y/S)_{\rm conv}\vert_{`T}$. 
Let $E$ be a sheaf in $(Y/S)_{\rm conv}$. 
Let  $E_T$ be a sheaf in $T_{\rm zar}$ defined by 
a formula $E_T(T')={\rm Hom}_{(Y/S)_{\rm conv}}(h_{T'},E)$
%{\cal H}{\it om}_{(Y/S)_{\rm conv}}(h_T,E)$ 
as in \cite[p.~140]{oc} for a log formal open subscheme $T'$ of $T$. 
By (\ref{eqn:hynhn}) and (\ref{eqn:chynhn}), 
we see that $E_T={\rm Hom}_{(Y/S)_{\rm conv}}(h_{T},E)=
\vpl_n\bet_{n*}(E_{T_n})$ 
as in [loc.~cit.]. 
Let $\os{\to}{T}$ be the associated topos 
to the site defined in \cite[Definition 2.1.28]{s2},  
which is the log version of the site in \cite[\S3]{oc}. 
An object of this site is 
a log formal open subscheme $V_n$ of $T_n$ 
for a positive integer  $n$; let $V_m$ be 
a log formal open subscheme of $T_m$;  
for $n>m$, ${\rm Hom}(V_n,V_m):= \emptyset$, 
and, for $n\leq m$, 
${\rm Hom}(V_n,V_m)$ 
is the set of morphisms $V_n \lo V_m$'s of 
log formal subschemes over the natural morphism $T_n \lo T_m$.
A covering of $V_n$ is a Zariski open covering of $V_n$.  
Let $\phi_n \col T_n \lo T_{n+1}$ be the natural morphism 
of enlargements of $Y/S$. 
Then an object ${\cal F}$ in $\os{\to}{T}$
is a family $\{({\cal F}_n,\psi_n)\}_{n=1}^{\infty}$ of pairs, 
where ${\cal F}_n$ is a sheaf in $(T_n)_{\rm zar}$ 
and $\psi_n \col \phi^{-1}_n({\cal F}_{n+1}) \lo {\cal F}_n$ 
is a morphism of sheaves in $(T_n)_{\rm zar}$. 
Let ${\cal O}_{\os{\to}{T}}
=\{({\cal O}_{T_n},\phi_n^{-1})\}_{n=1}^{\infty}$ 
(resp.~${\cal K}_{\os{\to}{T}}
=\{({\cal K}_{T_n},\phi_n^{-1})\}_{n=1}^{\infty}$)
be the structure sheaf (resp.~isostructure sheaf) of $\os{\to}{T}$, 
where $\phi_n^{-1}$ is the natural pull-back morphism
$\phi_n^{-1}({\cal O}_{T_{n+1}}) \lo {\cal O}_{T_n}$ 
(resp.~$\phi_n^{-1}({\cal K}_{T_{n+1}}) \lo {\cal K}_{T_n}$).  
Following \cite[p.~141]{oc}, 
we say that an ${\cal O}_{\os{\to}{T}}$-module 
${\cal F}$ is {\it coherent}  
if ${\cal F}_n \in {\rm Coh}({\cal O}_{T_n})$ 
for all $n\in {\mab Z}_{>0}$ 
and 
we say that ${\cal F}$ is {\it admissible} (resp.~{\it crystalline})
if the natural morphism 
${\cal O}_{T_n}\otimes_{\phi^{-1}_{n+1}({\cal O}_{T_{n+1}})}
\phi^{-1}_{n+1}({\cal F}_{n+1}) \lo {\cal F}_n$ 
is surjective (resp.~isomorphic). 
As in [loc.~cit.], we can give the definitions of 
a coherent ${\cal K}_{\os{\to}{T}}$-module, 
an admissible ${\cal K}_{\os{\to}{T}}$-module 
and a crystalline ${\cal K}_{\os{\to}{T}}$-module. 
For an object $E$ of 
$(Y/S)_{\rm conv}\vert_{T}$ and 
for a morphism $T' \lo T$ of quasi-prewidenings of $Y/S$, 
then we have an associated object 
${\cal E}'=\{{\cal E}'_n\}_{n=1}^{\infty}$ 
$({\cal E}'_n:=E_{T'_n})$ in $\os{\to}{T'}$. 
Following \cite[p.~147]{oc}, we say that 
a ${\cal K}_{Y/S}\vert_T$-module $E$ is 
{\it admissible} (resp.~{\it crystalline})
if ${\cal E}'$ is coherent and admissible (resp.~crystalline) for any morphism $T' \lo T$ 
of quasi-prewidenings.  
When $E$ is crystalline, 
we also say that $E$ is a crystal of 
${\cal K}_{Y/S}\vert_T$-module. 
Let $\bet \col \os{\to}{T} \lo T_{\rm zar}$ 
be the natural morphism of topoi 
which is denoted by $\gam$ in \cite[p.~55]{s2} 
(=the log version of $\gam$ in \cite[p.~141]{of}).  
\par 

\par 
Let ${\cal Z}$, ${\cal Z}'$ and $T$ 
be fine log noetherian formal schemes over $S$ 
whose underlying formal schemes are 
topologically of finite type over $S$. 
Assume that $\os{\circ}{T}$ is $p$-adic 
and that $\os{\circ}{T}$ is flat 
over ${\rm Spf}({\cal V})$.  
Let ${\cal Z} \os{\sus}{\lo} {\cal Z}'$ be a closed immersion 
of fine log formal schemes over $S$. 
Let $T \lo {\cal Z}'$ and 
let $u\col T\otimes_{\cal V}({\cal V}/\pi) \lo {\cal Z}$ 
be morphisms over $S$. 
We say that $(T,u)$ is an {\it enlargement} of ${\cal Z}/S$ 
in ${\cal Z}'$ {\it with radius} $\leq \vert \pi \vert$  
if the following diagram 
\begin{equation*} 
\begin{CD} 
T\otimes_{\cal V}({\cal V}/\pi) @>{\subset}>> T \\
@V{u}VV @VVV \\ 
{\cal Z} @>{\subset}>> {\cal Z}'
\end{CD} 
\end{equation*} 
is commutative (cf.~\cite[p.~784]{of}). 
%if the composite morphism 
%$T_1 \lo {\cal Z}'_1 \lo {\cal Z}'$ 
%is equal to the composite morphism 
%$T_1 \os{u}{\lo} {\cal Z}  \os{\sus}{\lo} {\cal Z}'$. 
We define a morphism of 
enlargements of ${\cal Z}/S$ in ${\cal Z}'$ 
with radius $\leq \vert \pi \vert$ in an obvious way. 
\par 
The following is a log version of \cite[(2.3)]{of}: 

\begin{prop}\label{prop:tzu} 
There exists an enlargement 
$({\mathfrak T}_{{\cal Z},1}({\cal Z}'),u_{\mathfrak T})$ 
of ${\cal Z}/S$ in ${\cal Z}'$ with radius $\leq \vert \pi \vert$  
such that, for any enlargement $(T,u)$ 
of ${\cal Z}/S$ in ${\cal Z}'$ with radius 
$\leq \vert \pi \vert$, 
there exists a unique morphism 
$(T,u) \lo 
({\mathfrak T}_{{\cal Z},1}({\cal Z}'),u_{\mathfrak T})$ 
such that the morphism $T\lo {\cal Z}'$ factors through 
$T\lo {\mathfrak T}_{{\cal Z},1}({\cal Z}')$. 
\end{prop}
\begin{proof} 
The proof is the same as that of \cite[(2.3)]{of} and (\ref{prop:exue}). 
\end{proof} 

As in \cite[(2.5)]{of}, replacing 
${\rm Ker}({\cal O}_{{\cal Z}'} \lo {\cal O}_{\cal Z})$ 
by $\{{\rm Ker}({\cal O}_{{\cal Z}'} \lo 
{\cal O}_{\cal Z})^n\}_{n=1}^{\infty}$,  
we obtain the system 
$\{{\mathfrak T}_{{\cal Z},n}({\cal Z}')\}_{n=1}^{\infty}$ 
of enlargements of ${\cal Z}/S$ in ${\cal Z}'$. 
We call $\{{\mathfrak T}_{{\cal Z},n}({\cal Z}')\}_{n=1}^{\infty}$
the {\it system of the universal enlargements} of ${\cal Z}/S$ in ${\cal Z}'$. 
\par

\begin{prop}\label{prop:fci} 
The following hold$:$ 
\par 
$(1)$ Let the notations be as above. 
Let $\wh{\cal Z}{}'$ be the formal completion of ${\cal Z}'$ 
along ${\cal Z}$. 
Then the natural morphism 
${\mathfrak T}_{{\cal Z},n}(\wh{\cal Z}') \lo 
{\mathfrak T}_{{\cal Z},n}({\cal Z}')$ is an isomorphism. 
\par 
$(2)$ 
Let ${\cal Z} \os{\sus}{\lo} {\cal Z}''$ be  an immersion 
of fine log noetherian formal schemes over $S$ 
whose underlying formal schemes are 
topologically of finite type over $S$. 
Let ${\cal Z}'$ be a log formal open subscheme of ${\cal Z}''$ 
such that $\os{\circ}{\cal Z}{}'$ contains $\os{\circ}{\cal Z}$ 
as a closed formal subscheme. Then the enlargement 
${\mathfrak T}_{{\cal Z},n}({\cal Z}')$ $(n\in {\mab N})$ 
is independent of the choice of 
the fine log formal open subscheme ${\cal Z}'$ of ${\cal Z}''$. 
$($We denote ${\mathfrak T}_{{\cal Z},n}({\cal Z}')$ 
by ${\mathfrak T}_{{\cal Z},n}({\cal Z}'')$.$)$
\end{prop} 
\begin{proof} 
(1): The proof of (1) is the same as that of 
(\ref{prop:fc}) and \cite[Lemma 2.1.25]{s2}. 
\par 
(2): (2) follows from (1) as in (\ref{lemm:uys1}).  
\end{proof}

\par 
Next we would like to prove (\ref{prop:lef}) below. 
To prove it, we need the following: 

\begin{lemm}[{\bf Weak fibration theorem for log formal schemes 
(cf.~\cite[(1.3.2)]{bpre})}]\label{lemm:eisd}  
Let $Z$ be a fine log scheme over $S_1$ such that 
$\os{\circ}{Z}$ is of finite type over $\os{\circ}{S}_1$. 
Let 
\begin{equation*}
\begin{CD}
Z @>{\subset}>> {\cal U}' \\ 
@| @VV{g}V \\
Z @>{\subset}>> {\cal U}
\end{CD}
\tag{2.12.1}\label{cd:zuup}
\end{equation*}
be a commutative diagram over $S$, 
where the horizontal morphisms are 
closed immersions into 
fine log noetherian $($not necessarily $p$-adic$)$ 
formal schemes over $S$ whose underlying formal schemes 
are topologically of finite type over $\os{\circ}{S}$. 
Let $g^{\rm ex}\col {\cal U}'{}^{\rm ex}\lo {\cal U}^{\rm ex}$ 
be the morphism obtained by $(\ref{cd:zuup})$. 
Then the following hold$:$
\par 
$(1)$ Assume that 
$\os{\circ}{g}{}^{\rm ex} 
\col \os{\circ}{\cal U}{}'{}^{\rm ex}\lo 
\os{\circ}{\cal U}{}^{\rm ex}$ is formally \'{e}tale.  
Let ${\mathfrak T}_n(g) \col {\mathfrak T}_{Z,n}({\cal U}') 
\lo {\mathfrak T}_{Z,n}({\cal U})$ $(n \in {\mab Z}_{\geq 1})$  
be the induced morphism by $g$. 
Then the natural morphism 
${\cal O}_{{\mathfrak T}_{Z,n}({\cal U})} \lo 
{\mathfrak T}_n(g)_*({\cal O}_{{\mathfrak T}_{Z,n}({\cal U}')})$ 
is an isomorphism.  
\par 
$(2)$ If 
$\os{\circ}{g}{}^{\rm ex}\col 
\os{\circ}{\cal U}{}'{}^{\rm ex}\lo \os{\circ}{\cal U}{}^{\rm ex}$ 
is formally smooth,  
then 
${\mathfrak T}_n(g)_*({\cal K}_{{\mathfrak T}_{Z,n}({\cal U}')})$ 
$(n\in {\mab Z}_{\geq 1})$ is a flat 
${\cal K}_{{\mathfrak T}_{Z,n}({\cal U})}$-algebra. 
\end{lemm}
\begin{proof} In the proof of \cite[(1.3.1)]{bpre} 
(resp.~\cite[(1.3.2)]{bpre}: Th\'{e}or\`{e}me de fibration faible), 
Berthelot has essentially proved  (1) 
(resp.~an analogous theorem to (2)) 
in the terms of rigid analytic spaces. Here we also give a proof of (1) (resp.~(2)) 
for the completeness of this paper. 
\par 
(1): The question is local; we may assume that the 
horizontal closed immersions in (\ref{cd:zuup}) have global charts 
and that $\os{\circ}{\cal U}{}'$ and $\os{\circ}{\cal U}$ 
are affine.  
By replacing ${\cal U}$ with ${\cal U}^{\rm ex}$ 
and ${\cal U}'$ with ${\cal U}'{}^{\rm ex}$, 
respectively,  
%(recall the notation before (\ref{rema:rfite})), 
we may assume that the horizontal closed immersions are exact 
and that $\os{\circ}{g}$ is formally \'{e}tale. 
We may also assume that the immersions 
induce isomorphisms 
$\os{\circ}{Z} \os{\sim}{\lo} \os{\circ}{\cal U}$ 
and  
$\os{\circ}{Z} \os{\sim}{\lo} \os{\circ}{\cal U}{}'$ 
as topological spaces ((\ref{prop:fc})).  
\par 
Set $A':=\Gam({\cal U}',{\cal O}_{{\cal U}'})$ 
and $A:=\Gam({\cal U},{\cal O}_{\cal U})$. 
Let $I\subset A$ be the defining ideal of 
the closed immersion $Z \os{\subset}{\lo} {\cal U}$.  
Since a chart of the log structure of ${\cal U}$ gives 
those of ${\cal U}'$ and $Z$, we see that 
the underlying scheme of 
$Z\times_{\cal U}{\cal U}'$ is equal to 
$\os{\circ}{Z}\times_{\os{\circ}{\cal U}}\os{\circ}{\cal U}{}'$. 
The natural morphism 
$\os{\circ}{Z} \lo \os{\circ}{Z}
\times_{\os{\circ}{\cal U}}\os{\circ}{\cal U}{}'$ 
is a closed immersion and \'{e}tale, 
and hence an open immersion by \cite[(17.9.1)]{ega4}. 
Consequently the morphism above is an 
isomorphism of schemes. 
%Let $Z^{\rm c}$ be the complement of $Z$ in 
%$Z\times_{\cal U}{\cal U}'$. 
%Identify $Z^{\rm c}$ with the image of 
%$Z^{\rm c}$ in ${\cal U}'$. 
%Replacing ${\cal U}'$ by ${\cal U}'\setminus Z^{\rm c}$, 
%we may assume that $Z=Z\times_{\cal U}{\cal U}'$  
%by (\ref{lemm:uys1}). 
Hence $IA'$ is the defining ideal of 
the closed immersion 
$Z \os{\subset}{\lo} {\cal U}'$ and 
the morphism $A/I \lo A'/IA'$ is an 
isomorphism. 
Let the notations be as in (\ref{rema:exdras}). 
Then, by (\ref{eqn:ldra}), 
%we have the descriptions 
%of $\Gam({\mathfrak T}_{Z,n}({\cal U}),
%{\cal O}_{{\mathfrak T}_{Z,n}({\cal U})})$ and 
%$\Gam({\mathfrak T}_{Z,n}({\cal U}'),
%{\cal O}_{{\mathfrak T}_{Z,n}({\cal U}')})$. 
%Hence the morphism 
%$\Gam({\mathfrak T}_{Z,n}({\cal U}),
%{\cal O}_{{\mathfrak T}_{Z,n}({\cal U})}) \lo 
%\Gam({\mathfrak T}_{Z,n}({\cal U}'),
%{\cal O}_{{\mathfrak T}_{Z,n}({\cal U}')})$ 
%is the ring morphism 
%\begin{equation*}
%(A[t_{\ul{m}}~\vert~\ul{m} \in {\mab N}^r,\vert\ul{m}\vert =n]
%/((f^{\ul{m}}-\pi t_{\ul{m}})+
%(p{\textrm -}{\rm torsion})))^{\wh{}}
%\lo \tag{2.12.3}\label{eqn:rini}
%\end{equation*}  
%$$(A'[t_{\ul{m}}~\vert~\ul{m} \in {\mab N}^r,
%\vert\ul{m}\vert =n]
%/((f^{\ul{m}}-\pi t_{\ul{m}})+(p{\textrm -}{\rm torsion})))^{\wh{}}.$$ 
it suffices to prove that the morphism 
\begin{equation*}
A[t_{\ul{m}}~\vert~\ul{m} \in {\mab N}^r,\vert\ul{m}\vert =n]
/(f^{\ul{m}}-\pi t_{\ul{m}},\pi^l) \lo 
A'[t_{\ul{m}}~\vert~\ul{m} \in {\mab N}^r,\vert\ul{m}\vert =n]
/(f^{\ul{m}}-\pi t_{\ul{m}},\pi^l)
\tag{2.12.2}\label{eqn:axfpi}
\end{equation*}  
is an isomorphism for any $l \in {\mab Z}_{>0}$. 
However this is clear because 
the morphism (\ref{eqn:axfpi}) is formally \'{e}tale, 
because $I$ is nilpotent in 
$A[t_{\ul{m}}~\vert~\ul{m} \in {\mab N}^r,\vert\ul{m}\vert =n]
/(f^{\ul{m}}-\pi t_{\ul{m}},\pi^l)$ 
and because 
the reduction mod $I$ of the morphism 
(\ref{eqn:axfpi}) is an isomorphism. 
\par 
(2): 
Let $\wh{\mab A}^r_S$ $(r\in {\mab N})$ 
be the formal affine space of 
relative dimension $r$ over $S$ 
($\os{\circ}{\wh{\mab A}}{}^r_S
=\ul{\rm Spf}_{\os{\circ}{S}}({\cal O}_S[[x_1,\ldots,x_r]])$; 
the log structure of $\wh{\mab A}^r_S$ is, 
by definition, the inverse image of 
the log structure of $S$). 
%Let $\hat{\mab A}^r_S$ 
%be the formal completion along the zero-section 
%$S \os{\sus}{\lo} {\mab A}^r_S$. 
By (1) and (\ref{prop:fc}), we may assume that 
${\cal U}'={\cal U}\times_S\hat{\mab A}^r_S$.  
The zero-section 
$S \os{\sus}{\lo} \hat{\mab A}^r_S$ 
induces the closed immersion 
${\cal U} \os{\sus}{\lo} {\cal U}'$. 
%Let $x_1,\ldots, x_r$ be 
%the standard coordinates of $\hat{\mab A}^r_S$ over $S$. 
For $\ul{m}:=(m_1,\ldots,m_r)\in 
{\mab N}^r$, set  
$x^{\ul{m}}:=x^{m_1}_1\cdots x^{m_r}_r$. 
Let $\{{\mathfrak T}_{{\cal U},n}({\cal U}')\}_{n=1}^{\infty}$ 
be the system of the universal enlargements 
of the closed immersion 
${\cal U} \os{\sus}{\lo} {\cal U}'$. 
Set 
$\wt{\cal U}:=\ul{\rm Spf}_{\cal U}({\cal O}_{\cal U}/
(\pi{\textrm -}{\rm torsion}))$. 
Then  
\begin{align*}
{\cal O}_{{\mathfrak T}_{{\cal U},n}({\cal U}')}  
& =
({\cal O}_{{\cal U}'}[t_{\ul{m}}~\vert~\ul{m} \in 
{\mab N}^r,\vert\ul{m}\vert =n]
/((x^{\ul{m}}-\pi t_{\ul{m}})+(p{\textrm -}{\rm torsion})))^{\wh{}} 
\tag{2.12.3}\label{ali:xte}\\ 
{} & =({\cal O}_{\wt{\cal U}}[[x_1,\ldots, x_r]]
[t_{\ul{m}}~\vert~\ul{m} \in {\mab N}^r,\vert\ul{m}\vert =n]
/(x^{\ul{m}}-\pi t_{\ul{m}}))^{\wh{}}. 
\end{align*}
Set 
$$B_n:=
({\cal O}_{\wt{\cal U}}[[x_1,\ldots, x_r]]
[t_{\ul{m}}~\vert~\ul{m} \in {\mab N}^r,\vert\ul{m}\vert =n]
/(x^{\ul{m}}-\pi t_{\ul{m}}))^{\wh{}}. $$
Since 
$B_n \! \! \mod \pi^k$ $(k\in {\mab N})$
is a free ${\cal O}_{\wt{\cal U}}
/\pi^k{\cal O}_{\wt{\cal U}}$-modules 
$(x^{\ul{l}}t_{\ul{m}}^{\ul{e}}$ 
$(\vert \ul{l} \vert <n$, 
$\vert \ul{m} \vert =n$, $\ul{e}\in {\mab N}^r)$ 
form a basis), since 
$\os{\circ}{\wt{\cal U}}$ is noetherian 
and since 
$B_n$ is separated for the $\pi$-adic topology, 
the sheaf of rings ${\cal O}_{{\mathfrak T}_{{\cal U},n}({\cal U}')}  
\simeq B_n$ 
is a flat 
${\cal O}_{\wt{\cal U}}$-module by 
\cite[Chapitre III \S5, Th\'{e}or\`{e}me 1]{bou2}. 
Let $\wt{\cal U}{}^{\wedge}$ 
be the $p$-adic completion of $\wt{\cal U}$. 
Then  
${\mathfrak T}_n(g)_*
({\cal O}_{{\mathfrak T}_{{\cal U},n}({\cal U}')})$ 
is a flat 
${\cal O}_{\wt{\cal U}{}^{\wedge}}=
{\cal O}_{{\mathfrak T}_{{\cal U},n}({\cal U})}$-module 
for any $n\in {\mab Z}_{\geq 1}$. 
By the universality, we have the following natural morphism 
\begin{equation*} 
\alpha_n \col  
{\mathfrak T}_{Z,n}({\cal U}') \lo 
{\mathfrak T}_{Z,n}({\cal U})
\times_{\wt{\cal U}{}^{\wedge}}
{\mathfrak T}_{{\cal U},n}({\cal U}').  
\tag{2.12.4}\label{eqn:tzuu}
\end{equation*} 
To prove (2), it suffices to prove that the following induced morphism 
\begin{equation*} 
\alpha_n^* \col {\cal K}_{{\mathfrak T}_{Z,n}({\cal U})
\times_{\wt{\cal U}{}^{\wedge}}
{\mathfrak T}_{{\cal U},n}({\cal U}')}\lo 
\alpha_{n*}({\cal K}_{{\mathfrak T}_{Z,n}({\cal U}')})  
\tag{2.12.5}\label{eqn:taltz}
\end{equation*} 
is an isomorphism. 
(Though the proof for this claim is essentially the same as the proof of \cite[(2.6.2)]{of}, 
we give the proof for this claim. See also (\ref{rema:ntb}) (2) below.) 
We may assume that 
$\os{\circ}{\cal U}$ is an affine formal scheme ${\rm Spf}(A)$. 
Let $y_1,\ldots, y_q$ $(q\in {\mab Z}_{\geq 1})$ 
be a system of generators of the defining ideal sheaf 
of the closed immersion $Z\os{\sus}{\lo} {\cal U}$. 
Then 
\begin{equation*} 
{\cal O}_{{\mathfrak T}_{Z,n}({\cal U})}
=
({\cal O}_{\cal U}
[s_{\ul{m}}~\vert~\ul{m} \in {\mab N}^q,\vert \ul{m}\vert =n]
/((y^{\ul{m}}-\pi s_{\ul{m}})+(p{\textrm -}{\rm torsion})))^{\wh{}}.
\tag{2.12.6}\label{eqn:tatz}  
\end{equation*} 
The ideal 
$(y_1,\ldots,y_q,x_1,\ldots,x_r)$ in 
$\Gam({\cal U}',{\cal O}_{{\cal U}'})$ gives the defining ideal sheaf 
of the closed immersion $Z\os{\sus}{\lo}{\cal U}'$. 
Set $(z_1,\ldots,z_{q+r})=(y_1,\ldots, y_q,x_1,\ldots, x_r)$. 
Then 
\begin{equation*} 
{\cal O}_{{\mathfrak T}_{Z,n}({\cal U}')}
=({\cal O}_{{\cal U}'}
[u_{\ul{m}}~\vert~\ul{m} \in {\mab N}^{q+r},\vert \ul{m}\vert =n]
/((z^{\ul{m}}-\pi u_{\ul{m}})+(p{\textrm -}{\rm torsion})))^{\wh{}}.
\tag{2.12.7}\label{eqn:tutz}  
\end{equation*}  
Set 
$$A':=A[[x_1,\ldots, x_r]],$$ 
$$C:=A'[u_{\ul{m}}~\vert~\ul{m} \in {\mab N}^{q+r},\vert \ul{m}\vert =n]
/((z^{\ul{m}}-\pi u_{\ul{m}})+(p{\textrm -}{\rm torsion})),$$  
$$D:=A[s_{\ul{m}}~\vert~\ul{m} \in {\mab N}^q,\vert\ul{m}\vert =n]
/((y^{\ul{m}}-\pi s_{\ul{m}})+(p{\textrm -}{\rm torsion}))$$ 
and 
$$E:=A'[t_{\ul{m}}~\vert~\ul{m} \in {\mab N}^r,\vert\ul{m}\vert =n]
/((x^{\ul{m}}-\pi t_{\ul{m}})+(p{\textrm -}{\rm torsion})).$$ 
Then $C$ is of finite type over $D\otimes_AE$. 
The ring $C$ is integral over $D\otimes_AE$. 
Indeed,  express $z^{\ul{m}}=y^{\ul{k}}x^{\ul{l}}$ 
such that $\vert \ul{k}\vert +\vert \ul{l}\vert =\vert \ul{m}\vert=n$. 
Then 
\begin{align*} 
\pi^nu^n_{\ul{m}}&=(z^{\ul{m}})^n
=y^{n{\ul{k}}}x^{n{\ul{l}}}=
y_1^{nk_1}\cdots y_q^{nk_q}x_1^{nl_1}\cdots x_r^{nl_r}\\
&=
(\pi^{k_1} s_{(n,0,\ldots,0)}^{k_1})
\cdots (\pi^{k_q} s_{(0,\ldots,0,n)}^{k_q})
(\pi^{l_1} t_{(n,0,\ldots,0)}^{l_1})\cdots (\pi^{l_r} t_{(0,\ldots,0,n)}^{l_r}) \\
&=\pi^ns_{(n,0,\ldots,0)}^{k_1}\cdots  s_{(0,\ldots,0,n)}^{k_q}
t_{(n,0,\ldots,0)}^{l_1}\cdots t_{(0,\ldots,0,n)}^{l_r}.
\end{align*} 
Hence $u^n_{\ul{m}}=
s_{(n,0,\ldots,0)}^{k_1}
\cdots s_{(0,\ldots,0,n)}^{k_q}
t_{(n,0,\ldots,0)}^{l_1}\cdots t_{(0,\ldots,0,n)}^{l_r}$ in $C$ 
and this tells us that $C$ is integral over  $D\otimes_AE$. 
Consequently $C$ is finite over  $D\otimes_AE$ 
and 
$\wh{C}=
(D\otimes_AE)^{\wh{}}\otimes_{D\otimes_AE}C$. 
By this equality we obtain 
\begin{align*} 
\wh{C}\otimes_{\cal V}K=
(D\otimes_AE)^{\wh{}}\otimes_{\cal V}K
\otimes_{(D\otimes_AE)\otimes_{\cal V}K}
(C\otimes_{\cal V}K)
=(D\otimes_AE)^{\wh{}}\otimes_{\cal V}K
\end{align*} 
since $(D\otimes_AE)\otimes_{\cal V}K=
C\otimes_{\cal V}K$. 
%we see that 
%the morphism $\os{\circ}{\alpha}_n$ is a finite morphism 
This shows that the morphism (\ref{eqn:taltz}) 
is an isomorphism. We complete the proof of (2). 
%As in the explanation in \cite[(2.6.2)]{of}, we see that 
%the morphism $\os{\circ}{\alpha}_n$ is a finite morphism 
%and the pull-back 
%\begin{equation*} 
%\alpha_n^* \col {\cal K}_{{\mathfrak T}_{Z,n}({\cal U})
%\times_{\wt{\cal U}{}^{\wedge}}
%{\mathfrak T}_{{\cal U},n}({\cal U}')}\lo 
%\alpha_{n*}({\cal K}_{{\mathfrak T}_{Z,n}({\cal U}')})  
%\tag{2.12.4}\label{eqn:taltz}
%\end{equation*} 
%is an isomorphism. We complete the proof of (2). 
\end{proof}

\begin{rema}\label{rema:ntb} 
(1) Let the notations be as in \cite[(2.6.2)]{of}. We find a claim 
that there exists a map 
$\al \col T_{I^{(n)}_Z}(Y) \lo T_{I^{n}_Z}(Y)$ 
in [loc.~cit.]. However, since $I^{n}_Z$ contains $I^{(n)}_Z$, 
the source and the target of $\al$ in [loc.~cit.] should be in reverse, that is, there exists a map 
$\al \col T_{I^{n}_Z}(Y)\lo T_{I^{(n)}_Z}(Y)$.  
\par 
(2) Though we find a claim 
in the proof of \cite[(2.12)]{of} 
that ``It suffices to prove this for $n=1$.'', 
the reduction for the case of any $n\in{\mab Z}_{\geq 1}$ 
to the case $n=1$ seems impossible: in the notation of 
the proof of [loc.~cit.], we have to consider 
the structure sheaf of $T_{Z,n}(Y(1))$.   
Note that, though the natural morphism 
\begin{equation*} 
T_{Z}(Y(1))\lo T_{Z}(Y)\times_YT_Y(Y(1)) 
\end{equation*} 
is an isomorphism 
as claimed in [loc.~cit.], 
while the natural morphism 
\begin{equation*} 
T_{Z,n}(Y(1)) \lo T_{Z,n}(Y)\times_YT_{Y,n}(Y(1)) \quad (n\geq 2)
\end{equation*} 
is not an isomorphism in general. 
In conclusion, we do not know whether the claim in [loc.~cit.] 
that $p_i \col T_{Z,n}(Y(1)) \lo T_{Z,n}(Y)$ is flat is correct; 
by (\ref{lemm:eisd}) (2), 
we know only that $p_{i*}({\cal K}_{T_{Z,n}(Y(1))})$ 
is a flat ${\cal K}_{T_{Z,n}(Y)}$-algebra.  
%Only this weak statement is enough 
%for the proof of \cite[(3.1)]{of} which claims the existence of 
%the $p$-adically convergent isocrystal for the higher direct image 
%of the trivial coefficient in crystalline cohomology for 
%a (proper) smooth morphism.  
\end{rema}

\begin{prop}\label{prop:lef} 
Let $Y/S_1$ be as in the beginning of this section. 
Let $\iota_Y \col Y \os{\subset}{\lo} {\cal Y}$ 
be an immersion into 
a formally log smooth log noetherian formal scheme over $S$ 
which is topologically of finite type over $S$. 
Set ${\cal Y}(1):={\cal Y}\times_S{\cal Y}$. 
Let $\{{\mathfrak T}_{Y,n}({\cal Y}(1))\}_{n=1}^{\infty}$ 
be the system of the universal enlargements of 
the diagonal immersion 
$Y \os{\sus}{\lo} {\cal Y}(1)$. 
Let 
$p_{i,n} \col {\mathfrak T}_{Y,n}({\cal Y}(1)) \lo 
{\mathfrak T}_{Y,n}({\cal Y})$ $(i=1,2)$ 
be the natural morphism induced by the  
$i$-th projection 
${\cal Y}(1) \lo {\cal Y}$. 
Then $p_{i,n*}({\cal K}_{{\mathfrak T}_{Y,n}({\cal Y}(1))})$ 
$(i=1,2)$ is a flat 
${\cal K}_{{\mathfrak T}_{Y,n}({\cal Y})}$-algebra.   
\end{prop}
\begin{proof} 
Let ${\cal Y}^{\rm ex}$ and ${\cal Y}(1)^{\rm ex}$ 
be the exactifications of 
$\iota_Y$ and the diagonal immersion 
$Y \os{\sus}{\lo} {\cal Y}(1)$, respectively.  
Let 
$p_1, p_2\col {\cal Y}(1)^{\rm ex} \lo {\cal Y}^{\rm ex}$ 
be the induced morphism 
by the first and the second projections 
${\cal Y}(1) \lo {\cal Y}$, respectively. 
As in \cite[(6.6)]{klog1}, we see that 
the morphisms $\os{\circ}{p}_1$ and $\os{\circ}{p}_2$
are formally smooth over $\os{\circ}{S}$. 
Now (\ref{prop:lef}) follows from (\ref{lemm:eisd}) (2). 
\end{proof}

\section{Log convergent linearization functors.~I}\label{sec:lll}
Let $S$ and $Y/S_1$ be as in \S\ref{sec:logcd}.  
Let $\iota_Y \col Y \os{\subset}{\lo} {\cal Y}$ 
be an immersion into 
a log noetherian $p$-adic formal scheme over $S$ 
which is topologically of finite type over $S$. 
Let ${\mathfrak T}_Y({\cal Y})$ be 
the quasi-prewidening 
$(Y,{\cal Y},\iota_Y, {\rm id}_{Y})$ of $Y/S$. 
Let 
$\{{\mathfrak T}_{Y,n}({\cal Y})\}_{n=1}^{\infty}$ 
(${\mathfrak T}_{Y,n}({\cal Y})
:=(Y_n,{\mathfrak T}_{Y,n}({\cal Y}),Y_n\os{\sus}{\lo}
{\mathfrak T}_{Y,n}({\cal Y}),Y_n\lo Y)$,  
$Y_n:=Y\times_{\cal Y}{\mathfrak T}_{Y,n}({\cal Y})$)
be the system of the universal enlargements of 
${\mathfrak T}_Y({\cal Y})$. 
We call $\{{\mathfrak T}_{Y,n}({\cal Y})\}_{n=1}^{\infty}$ 
the {\it system of the universal enlargements of} 
$\iota_Y$ for short.  By (\ref{rema:rfite}) 
we obtain the equality 
${\mathfrak T}_{Y,n}({\cal Y})=
{\mathfrak T}_{Y,n}({\cal Y}^{\rm ex})$.  
Let 
$((Y/S)_{\rm conv}
\vert_{{\mathfrak T}_Y({\cal Y})},
{\cal K}_{Y/S}\vert_{{\mathfrak T}_Y({\cal Y})})$ 
be the localized ringed topos 
of $((Y/S)_{\rm conv},{\cal K}_{Y/S})$ 
at ${\mathfrak T}_Y({\cal Y})$ 
and let  
$$j_{{\mathfrak T}_Y({\cal Y})} 
\col 
((Y/S)_{\rm conv}\vert_{{\mathfrak T}_Y({\cal Y})},
{\cal K}_{Y/S}\vert_{{\mathfrak T}_Y({\cal Y})}) 
\lo ((Y/S)_{\rm conv}, {\cal K}_{Y/S})$$ 
be the natural morphism of the ringed topoi and let 
\begin{align*}
\varphi^*_{\os{\to}{\mathfrak T}_Y({\cal Y})} 
\col & \{\text{the category of coherent crystals of }
{\cal K}_{\os{\to}{\mathfrak T}_Y({\cal Y})}\text{-modules}\} 
\tag{3.0.1}\label{eqn:phyt} \\
{} & \lo \{\text{the category of coherent crystals of }
{\cal K}_{Y/S}\vert_{{\mathfrak T}_Y({\cal Y})}\text{-modules}\}
\end{align*}
be a natural functor which is 
a relative log version (with $\otimes_{\mab Z}{\mab Q}$) 
of the functor in \cite[p.~147]{oc}.  
We also have the following natural functor 
\begin{equation*}
\varphi_{\os{\to}{\mathfrak T}_Y({\cal Y})*} 
\col 
((Y/S)_{\rm conv}
\vert_{{\mathfrak T}_Y({\cal Y})},
{\cal K}_{Y/S}\vert_{{\mathfrak T}_Y({\cal Y})}) 
\lo 
({\os{\to}{\mathfrak T}_Y({\cal Y})}_{\rm zar}, 
{\cal K}_{\os{\to}{\mathfrak T}_Y({\cal Y})})
\tag{3.0.2}\label{eqn:ltts}
\end{equation*} 
by setting $\varphi_{\os{\to}{\mathfrak T}_Y({\cal Y})*}(E)(U) 
:=E(U\os{\subset}{\lo} {\mathfrak T}_{Y,n}({\cal Y})\lo {\mathfrak T}_Y({\cal Y}))$. 
For a coherent crystal of 
${\cal K}_{\os{\to}{\mathfrak T}_Y({\cal Y})}$-module 
${\cal E}$, set 
\begin{equation*} 
L^{\rm UE}_{Y/S}({\cal E})
:=j_{{\mathfrak T}_Y({\cal Y})*}
\varphi^*_{\os{\to}{\mathfrak T}_Y({\cal Y})}({\cal E})
\tag{3.0.3}\label{eqn:luys}
\end{equation*} 
by abuse of notation (cf.~\cite[p.~152]{oc}, \cite[p.~95]{s2}) 
(``UE'' is the abbreviation of the universal enlargement). 
%We also have a pull-back 
%\begin{align*}
%\varphi^*_{{\mathfrak T}_n} 
%\col & \{\text{the category of }
%{\cal K}_{{\mathfrak T}_{Y,n}({\cal Y}}\text{-modules}\} \\ 
%{} & \lo \{\text{the category of crystals of }
%{\cal K}_{Y/S}\vert_{T_n}\text{-modules}\}
%\end{align*}
%for each positive integer $n$.  
\par

\begin{defi}\label{defi:lf}  
%We call the functor 
%\begin{equation*}
%L^{\rm UE}_{Y/S} \col 
%\{{\rm crystals}~{\rm of}~{\cal K}_{\os{\to}{T}}
%{\textrm -}{\rm modules}\} 
%\lo \{{\rm crystals}~{\rm of}~{\cal K}_{Y/S} 
%\text{-}{\rm modules}\} 
%\tag{2.6.1}\label{eqn:lvue}
%\end{equation*} 
%the {\it log convergent linearization functor 
%with respect to the closed immersion} 
%$\iota_Y \col Y \os{\subset}{\lo} {\cal Y}$. 
(1) We call the functor 
\begin{equation*}
L^{\rm UE}_{Y/S} \col 
\{{\rm coherent}~{\rm crystals}~{\rm of}~
{\cal K}_{\os{\to}{\mathfrak T}_Y({\cal Y})}
{\textrm -}{\rm modules}\} 
\lo \{{\cal K}_{Y/S}\text{-}{\rm modules}\} 
\tag{3.1.1}\label{eqn:lvnc}
\end{equation*} 
the {\it log convergent linearization functor with respect to} 
$\iota_Y$ for coherent crystals of
${\cal K}_{\os{\to}{\mathfrak T}_Y({\cal Y})}$-modules. 
\par 
(2) Let 
$\bet \col \os{\to}{\mathfrak T}_{Y}({\cal Y}) \lo {\cal Y}$ 
%$(n\in {\mab Z}_{\geq 1})$ 
be the natural morphism. 
Let $\bet^{\rm ex} \col 
\os{\to}{\mathfrak T}_{Y}({\cal Y}) \lo {\cal Y}^{\rm ex}$ 
%$(n\in {\mab Z}_{\geq 1})$ 
be also the natural morphism. 
Let ${\cal Z}$ be ${\cal Y}$ or ${\cal Y}^{\rm ex}$ 
and let 
$\gam \col \os{\to}{\mathfrak T}_{Y}({\cal Y})\lo {\cal Z}$ 
be $\bet$ or $\bet^{\rm ex}$. 
We call the composite functor 
\begin{equation*}
L^{\rm conv}_{Y/S}:= L^{\rm UE}_{Y/S} \circ \gam^* 
\col \{{\rm coherent}~
{\cal K}_{\cal Z}{\textrm -}{\rm modules}\} 
\lo 
\{{\cal K}_{Y/S}\text{-}{\rm modules}\} 
%\{{\rm crystals}~{\rm of}~{\cal K}_{Y/S} 
%\text{-}{\rm modules}\} 
\tag{3.1.2}\label{eqn:lue}
\end{equation*} 
the {\it log convergent linearization functor 
with respect to} $\iota_Y$ {\it for coherent} 
${\cal K}_{\cal Z}$-{\it modules}. 
%We also call the composite functor 
%\begin{equation*}
%L^{\rm conv}_{Y/S}:= 
%L^{\rm UE}_{Y/S}\circ \bet^{{\rm ex}*}_n 
%\col \{{\rm coherent}~
%{\cal K}_{\cal Y}{\textrm -}{\rm modules}\} 
%\lo 
%\{{\cal K}_{Y/S}\text{-}{\rm modules}\} 
%\tag{3.0.2}\label{eqn:lue}
%\end{equation*} 
%the {\it log convergent linearization functor 
%with respect to} $\iota_Y$ {\it for coherent} 
%${\cal K}_{\cal Y}$-{\it modules}. 
\end{defi}

%\begin{prop}\label{prop:fre} 
%The functor {\rm (\ref{eqn:lvnc})} is right exact. 
%\end{prop}
%\begin{proof} 
%Let ${\cal E}$ be a crystal of ${\cal K}_{\os{\to}{T}}$-modules. 
%Then the value of $L^{\rm UE}_{Y/S}({\cal E})$ 
%at $T_n$ is ${\cal O}_{T_n(1)}\otimes_{{\cal O}_{T_n}}{\cal E}$ 
%by (\ref{eqn:lueyset}). 
%Hence the functor {\rm (\ref{eqn:lvnc})} is right exact. 
%\end{proof} 

\begin{lemm}\label{lemm:rex} 
Let $T=(U,T,\iota,u)$ be 
an exact quasi-$($pre$)$widening of  
$Y/S$ and let $\{T_n\}_{n=1}^{\infty}$ be 
the system of the universal enlargements of $T$. 
Let $\bet_n \col T_n \lo T$ be the natural morphism. 
Then, for each $n$, 
the pull-back functor 
\begin{equation*} 
\beta^*_n \col 
\{{\rm coherent}~{\cal K}_T{\textrm -}{\rm modules}\} 
\lo 
\{{\rm coherent }~{\cal K}_{T_n}{\textrm -}{\rm modules}\} 
\end{equation*} 
is exact. 
\end{lemm}  
\begin{proof} 
We may assume that $\iota$ is an exact closed immersion. 
Let the notations be as in (\ref{rema:exdras}). 
We may assume that $\os{\circ}{T}$ is 
an affine formal scheme ${\rm Spf}(A)$ 
over ${\cal V}$. 
Set $A_n:=
(A[t_{\ul{m}}~\vert~\ul{m}\in {\mab N}^k,\vert \ul{m} \vert =n]
/(f^{\ul{m}}-\pi t_{\ul{m}}))^{\wh{}}$. 
Then $\Gam(T_n,{\cal O}_{T_n})\otimes_{\mab Z}{\mab Q}
=A_n\otimes_{\mab Z}{\mab Q}$ by (\ref{eqn:ldra}). 
It suffices to prove that, for a finitely generated 
$A$-module $E$,  
\begin{equation*}   
{\rm Tor}_1^A(A_n,E\otimes_{\mab Z}{\mab Q})=0. 
%\tag{3.2.1}\label{eqn:aeq}
\end{equation*} 
In fact, we claim that $\pi \cdot{\rm Tor}_1^A(A_n,E)=0$. 
(This implies $p \cdot{\rm Tor}_1^A(A_n,E)=0$, and hence 
${\rm Tor}_1^A(A_n,E\otimes_{\mab Z}{\mab Q})=0$.)  
Indeed,  
let $\iota \col A \lo A_n$ be the natural inclusion 
and let $\rho \col A_n \lo A$ be 
a morphism of $A$-modules defined by $\rho(1_A)=\pi$   
and $\rho(t_{\ul{m}})=f^{\ul{m}}$. 
Then 
%$\rho \circ \iota = \pi \times {\rm id}_{A}$ and 
$\iota \circ \rho =\pi \times {\rm id}_{A_n}$. 
The two morphisms $\rho$ and $\iota$ induce
the following two zero morphisms:   
\begin{equation*} 
\rho_* \col {\rm Tor}_1^A(A_n,E) \lo {\rm Tor}_1^A(A,E)=0
\quad  
\iota_* \col 0={\rm Tor}_1^A(A,E) 
\lo  {\rm Tor}_1^A(A_n,E).  
\end{equation*} 
Since
$\iota_*\circ \rho_*=\pi \cdot {\rm id}_{{\rm Tor}_1^A(A_n,E)}$, 
%and $\rho_*\circ \iota_*=\pi$. 
$\pi \cdot {\rm Tor}_1^A(A_n,E)=0$. 
%Since $p$ is invertible on 
%${\rm Tor}_1^A(A_n,E\otimes_{\mab Z}{\mab Q})$, 
%we have 
%${\rm Tor}_1^A(A_n,E\otimes_{\mab Z}{\mab Q})=0$. 
We finish the proof of (\ref{lemm:rex}). 
\end{proof}

\begin{coro}\label{coro:ckt} 
Let the notations be as in {\rm (\ref{lemm:rex})}. 
The category of coherent crystals of 
${\cal K}_{\os{\to}{T}}$-modules 
is an abelian category. The kernel and the cokernel in  
this category  
are equal to those in the category of 
coherent ${\cal K}_{\os{\to}{T}}$-modules.  
%The category $\{\text{coherent crystals of 
%${\cal K}_{\os{\to}{T}}$-modules}\}$ 
%is an abelian category. The kernel and the cokernel in  
%$\{\text{coherent crystals of 
%${\cal K}_{\os{\to}{T}}$-modules}\}$  
%are equal to those in the category 
%$\{\text{coherent ${\cal K}_{\os{\to}{T}}$-modules}\}$.  
\end{coro}
\begin{proof} 
For each $m\in {\mab Z}_{>0}$, 
the system of the universal enlargements of $T_m$ is 
$$\{T_1,T_2,\ldots, T_{m-1}, T_m, T_m, T_m,\ldots\}.$$ 
Let $\iota_{m,m'} \col T_{m'} \lo T_m$ $(m' \leq m)$ 
be the natural morphism. 
Because $T_m$ $(\forall m\in {\mab Z}_{\geq 1})$ is 
an exact enlargement of $Y/S$, 
the pull-back functor 
\begin{equation*} 
\iota_{m,m'}^*  \col 
\{\text{coherent ${\cal K}_{T_m}$-modules}\}
\lo \{\text{coherent ${\cal K}_{T_{m'}}$-modules}\}
\end{equation*}
is exact by (\ref{lemm:rex}). 
Now (\ref{coro:ckt}) follows from easy calculations. 
\end{proof}

\begin{rema}\label{rema:utne} 
If a quasi-prewidening $(U,T,\iota,u)$ of $Y/S$
is not exact, the obvious analogue of 
(\ref{lemm:rex}) does not hold in general.  
Indeed, set $S:=({\rm Spf}({\cal V}),{\cal V}^*)$ 
and let $Y$ 
(resp.~${\cal Y}$) 
be the following log scheme over $S_1$ (resp.~$S$): 
$Y=
({\rm Spec}(({\cal V}/\pi)[x]),
({\mab N}\owns 1\lom x\in ({\cal V}/\pi)[x])^a)$ 
(resp.~${\cal Y}=({\rm Spf}({\cal V}\{x\}),
({\mab N}\owns 1\lom x\in {\cal V}\{x\})^a)$. 
Consider the case where $U=Y$,  
$T:={\cal Y}(i):= 
\underset{(i+1)~{\rm pieces}}{\underbrace{
{\cal Y}\times_S{\cal Y}\times_S
\times \cdots \times_S{\cal Y}}}$, 
$\iota$ is the diagonal embedding 
$Y \os{\sus}{\lo} {\cal Y}(i)$ 
$(i\in {\mab Z}_{\geq 1})$ 
and $u={\rm id}_Y$. 
Then we claim that 
the pull-back functor 
\begin{equation*} 
\{{\rm coherent}~{\cal K}_T{\textrm -}{\rm modules}\} 
\lo 
\{{\rm coherent }~{\cal K}_{T^{\rm ex}}
{\textrm -}{\rm modules}\} 
\end{equation*} 
is not exact. 
Indeed, set 
$A_i=\vpl_n{\cal V}\{u^{\pm 1}_1,\ldots 
u^{\pm 1}_{i},x_0,x_1,\ldots,x_{i}\}/I^n$ 
$(I:=(x_1-u_1x_0,\ldots,x_{i}-u_{i}x_0,
u_1-1,\ldots,u_{i}-1))$. 
It is easy to check that 
$T^{\rm ex}=
({\rm Spf}(A_i),({\mab N}\owns 1 \lom x_0\in A_i)^a)$. 
We have a natural isomorphism
$A_i\os{\sim}{\lo} {\cal V}\{x_0\}
[[u_1-1,\ldots u_{i}-1]]$. 
Set $K\{x_1,\ldots,x_{i}\}:=K\otimes_{\cal V}{\cal V}\{x_1,\ldots,x_{i}\}$. 
Consider the exact sequence 
\begin{equation*} 
0\lo K\{x_1,\ldots,x_{i}\} \os{x_{i}\times}{\lo}  
K\{x_1,\ldots,x_{i}\}
\tag{3.4.1}\label{eqn:xel}
\end{equation*} 
and consider 
$K\{x_1,\ldots, x_{i}\}$ as $K\{x_0,\ldots, x_{i}\}/(x_0)$ 
with natural $K\{x_0,\ldots, x_{i}\}$-module strucuture. 
Apply the tensorization 
$\otimes_{K\{x_0,\ldots, x_{i}\}}
(A_i\otimes_{\cal V}K)$ to (\ref{eqn:xel}). 
Then we have the following sequence 
\begin{equation*} 
0\lo K\otimes_{\cal V}{\cal V}[[u_1-1,\ldots,u_{i}-1]] \os{0}{\lo}  
K\otimes_{\cal V}{\cal V}[[u_1-1,\ldots,u_{i}-1]],
\tag{3.4.2}\label{eqn:unei}
\end{equation*}  
which is not exact. 
Thus we have shown that the claim holds. 
In fact, set $B_n:=\Gam(T_n,{\cal O}_{T_n})$ 
$(n\in {\mab Z}_{\geq 1})$. 
Then 
$(\ref{eqn:xel})\otimes_{K\{x_0,\ldots, x_{i}\}}
(B_n\otimes_{\cal V}K)$ is not exact because 
$(B_n/(x_0))\otimes_{\cal V}K\not=0$ similarly shown as 
in the proof of (\ref{lemm:eisd}) (1) 
and 
$(\ref{eqn:xel})\otimes_{K\{x_0,\ldots, x_{i}\}}(B_n\otimes_{\cal V}K)$ 
is equal to the following sequence 
$$0\lo (B_n/(x_0))\otimes_{\cal V}K\os{0}{\lo} (B_n/(x_0))\otimes_{\cal V}K.$$ 
%that $\os{\circ}{T}$ is an affine formal scheme 
%${\rm Spf}(A_i)$ 
%over ${\cal V}$. 
%Set $K:={\rm Ker}(P^{\rm gp}\lo Q^{\rm gp})$. 
%Then $P^{\rm ex}=PK$ in $P^{\rm gp}$. 
%Hence the functor 
%\begin{equation*} 
%\wh{\otimes}_{{\mab Z}_p\{P\}}{\mab Z}_p\{P^{\rm ex}\}
%=
%\wh{\otimes}_{{\mab Z}_p\{P\}}{\mab Z}_p\{PK\}
%\end{equation*} 
%is exact. As a result, we may assume that $\iota$ is exact. 
\end{rema}

\begin{prop}\label{prop:afex} 
If the morphism $\os{\circ}{\cal Y} \lo \os{\circ}{S}$ is affine, 
then the functor $L^{\rm UE}_{Y/S}$ in {\rm (\ref{eqn:lvnc})} is right exact.   
\end{prop} 
\begin{proof} 
Let $T:=(U,T,\iota,u)$ be an object of 
${\rm Conv}(Y/S)$. 
Then we have the natural immersion 
$(\iota,\iota_Y\circ u)\col U\lo T\times_S{\cal Y}$ over $S$. 
We may assume that $\os{\circ}{T}{}$ is affine. 
Let 
$\{{\mathfrak T}_n\}_{n=1}^{\infty}$ 
be the system of the universal enlargements 
of this immersion.  
Then we have the following diagram 
of the inductive systems of enlargements of $Y/S$:  
\begin{equation*}
\begin{CD}
\{{\mathfrak T}_n\}_{n=1}^{\infty} 
@>{\{p_n\}_{n=1}^{\infty}}>> 
\{{\mathfrak T}_{Y,n}({\cal Y})\}_{n=1}^{\infty}\\ 
@V{\{p'_n\}_{n=1}^{\infty}}VV \\
T@. @. .
\end{CD}
\tag{3.5.1}\label{cd:ppp}
\end{equation*} 
Let ${\cal E}=\{{\cal E}_n\}_{n=1}^{\infty}$ 
be a coherent crystal of 
${\cal K}_{\os{\to}{\mathfrak T}_Y({\cal Y})}$-modules.  
Then, as in the proof of \cite[(5.4)]{oc} and 
\cite[Theorem 2.3.5]{s2},   
we have  
\begin{equation*}
L^{\rm UE}_{Y/S}({\cal E})_{T}
=\vpl_np'_{n*}p_n^*({\cal E}_n) 
\tag{3.5.2}\label{eqn:lueyset}
\end{equation*} 
by the definition of $L^{\rm UE}_{Y/S}$ ((\ref{eqn:luys})). 
\par 
Let 
\begin{equation*} 
{\cal E}' \lo {\cal E} \lo {\cal E}'' \lo 0 \quad 
({\cal E}'=\{{\cal E}'_n\}_{n=1}^{\infty},
{\cal E}''=\{{\cal E}''_n\}_{n=1}^{\infty}) 
\tag{3.5.3}\label{eqn:eee0}
\end{equation*} 
be an exact sequence of coherent crystals of 
${\cal K}_{\os{\to}{\mathfrak T}_Y({\cal Y})}$-modules. 
Set  $E_n:=\Gam({\mathfrak T}_n,p_n^*({\cal E}_n))$, 
$E'_n:=\Gam({\mathfrak T}_n,p_n^*({\cal E}'_n))$
and $E''_n:=\Gam({\mathfrak T}_n,p_n^*({\cal E}''_n))$. 
We have to prove that 
\begin{equation*} 
\vpl_n{E}'_n \lo \vpl_n{E}_n \lo 
\vpl_n{E}''_n \lo 0
\tag{3.5.4}\label{eqn:lmee0}
\end{equation*}  
is exact. 
\par 
Because the morphism 
$\os{\circ}{T}{}\times_{\os{\circ}{S}}
\os{\circ}{\cal Y}\lo \os{\circ}{T}{}$ is affine and 
because the integralization of a log structure is affine 
(\cite[(2.7)]{klog1}), 
the underlying scheme of $T\times_S{\cal Y}$ is affine. 
Because $\os{\circ}{\cal Y}\lo \os{\circ}{S}$ is affine, 
it is separated: the diagonal immersion 
$\os{\circ}{\cal Y}\lo 
\os{\circ}{\cal Y}\times_{\os{\circ}{S}}\os{\circ}{\cal Y}$ 
is a closed immersion. 
Factorize this immersion into the following composite morphism: 
$\os{\circ}{\cal Y}\lo 
({\cal Y}\times_{S}{\cal Y})^{\circ}
\lo 
\os{\circ}{\cal Y}\times_{\os{\circ}{S}}\os{\circ}{\cal Y}$. 
Then, since 
$({\cal Y}\times_{S}{\cal Y})^{\circ}\lo 
\os{\circ}{\cal Y}\times_{\os{\circ}{S}}\os{\circ}{\cal Y}$ 
is affine, it is separated. 
Hence the morphism 
$\os{\circ}{\cal Y}\lo 
({\cal Y}\times_{S}{\cal Y})^{\circ}$ 
is a closed immersion. 
Because $T=(T\times_S{\cal Y})
\times_{{\cal Y}\times_{S}{\cal Y}}{\cal Y}$, 
the morphism $T\lo T\times_S{\cal Y}$ is a closed immersion. 
Consequently the immersion 
$U\os{\sus}{\lo} T\times_S{\cal Y}$ is closed, 
and ${\mathfrak T}_n$ is also affine. 
\par   
We have the following sequence of projective systems 
of coherent crystals of 
$\{\Gam({\mathfrak T}_n,
{\cal K}_{{\mathfrak T}_n})\}_{n=1}^{\infty}$-modules: 
$$\{{E}'_n\}_{n=1}^{\infty} \lo \{{E}_n\}_{n=1}^{\infty} \lo 
\{{E}''_n\}_{n=1}^{\infty} \lo 0.$$
Because ${\mathfrak T}_n$ is affine, 
the sequence 
$E'_n \lo E_n \lo E''_n \lo 0$ for each $n$ 
is exact. Set $F_n:={\rm Ker}(E_n\lo E''_n)$
and $G_n:={\rm Ker}(E'_n\lo E_n)$. 
Since the sequences 
\begin{equation*} 
0\lo F_n \lo E_n \lo E''_n \lo 0 
\tag{3.5.5;$n$}\label{eqn:fee}
\end{equation*}  and 
\begin{equation*} 
0\lo G_n \lo E'_n \lo F_n \lo 0
\tag{3.5.6;$n$}\label{eqn:gef}
\end{equation*}  
are exact, 
the families $\{F_n\}_{n=1}^{\infty}$ and 
$\{G_n\}_{n=1}^{\infty}$ are coherent crystals of 
$\{\Gam({\mathfrak T}_n,
{\cal K}_{{\mathfrak T}_n})\}_{n=1}^{\infty}$-modules by  
(\ref{coro:ckt}). 
For a coherent crystal $\{H_n\}_{n=1}^{\infty}$ of 
$\{\Gam({\mathfrak T}_n,
{\cal K}_{{\mathfrak T}_n})\}_{n=1}^{\infty}$-modules, 
$R^1\vpl_nH_n=0$ by the log version of \cite[(3.8)]{oc} 
which follows from the proof of [loc.~cit.]. 
Hence the sequences $\vpl_n$(\ref{eqn:fee}) and 
$\vpl_n$(\ref{eqn:gef}) are exact.  
Thus we have proved the exactness of (\ref{eqn:lmee0}). 
%\par 
%Now assume that ${\cal Y}/S$ is log smooth. 
%Then we see that $L^{\rm UE}_{Y/S}$ is exact by (\ref{prop:lef}) 
%(cf.~the proof of (\ref{lemm:eisd})). 
\par 
We finish the proof of (\ref{prop:afex}).  
\end{proof}

\par 
From now on, assume that ${\cal Y}/S$ is formally log smooth. 
Let $i$ and $m$ be nonnegative integers. 
Set 
${\cal Y}(i):= 
\underset{(i+1)~{\rm pieces}}{\underbrace{
{\cal Y}\times_S{\cal Y}\times_S\times \cdots \times_S{\cal Y}}}$.  
Let ${\cal Y}_m(i)$ 
be the $m$-th infinitesimal neighborhood of 
the diagonal ${\cal Y}$ in ${\cal Y}(i)$ defined in 
\cite[(5.8)]{klog1}.   
%\par 
%Recall that, for a log formal scheme 
%${\cal Z}$ over ${\rm Spf}({\cal V})$, 
%$\wt{\cal Z}$ denotes a log scheme 
%whose underlying scheme is $\ul{\rm Spf}_{\os{\circ}{\cal Z}}
%({\cal O}_{\cal Z}/(\pi\textrm{-}{\rm torsion}))$ and 
%whose log structure is the pull-back of the log structure of ${\cal Z}$. 
\par 
Let $p_j(i) \col {\cal Y}_m(i) \lo {\cal Y}$ 
$(1\leq j \leq i+1)$ be the $j$-th projection. 
Let $p_{ij} \col {\cal Y}_m(2) \lo {\cal Y}_m(1)$ be 
the $(i,j)$-th projection $(1\leq i < j \leq 3)$. 
\par 
For an object $(U,T,\iota,u)$ of ${\rm Conv}(Y/S)$, 
set $U_j(i):=
U\times_{\iota_Y \circ u,{\cal Y},p_j(i)}{\cal Y}_m(i)$. 
Then we have immersions  
$$\Del(i):= \Del_m(i) \col U 
\os{\iota \times (\iota_Y \circ u)}{\lo}
T\times_S{\cal Y} 
\os{{\rm id}_T \times {\rm diag}.}{\lo} 
T\times_S{\cal Y}_m(i)$$  
and 
$$\iota_j(i) \col U_j(i) 
\os{\iota \times {\rm id}_{{\cal Y}_m(i)}}{\lo}
T\times_S{\cal Y}_m(i).$$

\begin{prop}\label{prop:lcl} 
Let ${\cal E}:=\{{\cal E}_n\}_{n=1}^{\infty}$ be 
a coherent crystal of 
${\cal K}_{\os{\to}{\mathfrak T}_Y({\cal Y})}$-modules. 
Then there exists a family $\{\nabla^i\}_{i=0}^{\infty}$ 
of natural morphisms  
\begin{equation*} 
\nabla^i  \col 
L^{\rm UE}_{Y/S}({\cal E}\otimes_{{\cal O}_{\cal Y}}
\Om^i_{{\cal Y}/S}) 
\lo 
L^{\rm UE}_{Y/S}({\cal E}\otimes_{{\cal O}_{\cal Y}}
\Om^{i+1}_{{\cal Y}/S}) 
\end{equation*} 
which makes $\{(L^{\rm UE}_{Y/S}({\cal E}\otimes_{{\cal O}_{\cal Y}}
\Om^i_{{\cal Y}/S}),\nabla^i)\}_{i=0}^{\infty}$ 
a complex of ${\cal K}_{Y/S}$-modules. 
\end{prop} 
\begin{proof}
(The following proof is the log version of that of \cite[(5.4)]{oc}.) 
Let the notations be as in the proof of (\ref{prop:afex}). 
%Let $\{{\mathfrak T}_n\}_{n=1}^{\infty}$ be 
%the system of the universal enlargements in the proof of (\ref{prop:afex}).  
Let $u_j(i) \col U_j(i) \lo Y$ $(i\in {\mab N})$ 
be the composite morphism 
$U_j(i) \os{\rm 1st~proj.}{\lo} U \os{u}{\lo} Y$. 
Then, for $1\leq j\leq i+1$, 
the following two commutative diagrams 
\begin{equation*} 
\begin{CD} 
U @>{({\rm id}_U,{\rm diag}.\circ \iota_Y \circ u)}>> U_j(i) \\ 
@V{\Del(i)}VV @VV{\iota_j(i)}V  \\ 
T\times_S{\cal Y}_m(i) @= T\times_S{\cal Y}_m(i) 
\end{CD} 
\end{equation*} 
and 
\begin{equation*} 
\begin{CD} 
U_j(i) @>{{\rm 1st~proj.}}>> U  \\ 
@V{\iota_j(i)}VV @VV{\Del(0)}V  \\ 
T\times_S{\cal Y}_m(i) 
@>{{\rm id}_T\times p_j(i)}>> T\times_S{\cal Y} 
\end{CD} 
\end{equation*} 
give us the following two morphisms 
\begin{equation*} 
\Del_j(i) \col  
(U,T\times_S{\cal Y}_m(i),\Del(i),u) 
\lo 
(U_j(i),T\times_S{\cal Y}_m(i),\iota_j(i),u_j(i))   
\end{equation*} 
and 
\begin{equation*} 
p_j(i) \col (U_j(i),T\times_S{\cal Y}_m(i),\iota_j(i),u_j(i))  \lo 
(U,T\times_S{\cal Y},\Del(0),u)
\end{equation*} 
of prewidenings of $Y/S$, respectively.  
Hence we have the following natural morphisms of 
log formal schemes for $n\in {\mab Z}_{\geq 1}$: 
\begin{equation*} 
\Del_{j,n}(i) \col {\mathfrak T}_{U,n}(T\times_S{\cal Y}_m(i)) \lo 
{\mathfrak T}_{U_j(i),n}(T\times_S{\cal Y}_m(i)),  
\end{equation*} 
\begin{equation*} 
p_{j,n}(i) \col {\mathfrak T}_{U_j(i),n}(T\times_S{\cal Y}_m(i)) 
\lo 
{\mathfrak T}_{U,n}(T\times_S{\cal Y})
={\mathfrak T}_n.    
\end{equation*} 
\par 
Set $\del_j(i):=({\rm id}_U,{\rm diag}.\circ \iota_Y \circ u)\col U \os{\sus}{\lo}U_j(i)$. 
Let ${\cal I}_U$ and ${\cal I}_{U_j(i)}$ be 
the defining ideal sheaves of the immersions 
$\Del(i) \col U \os{\sus}{\lo} T\times_S{\cal Y}_m(i)$ and 
$\iota_j(i) \col 
U_j(i) \os{\sus}{\lo} T\times_S{\cal Y}_m(i)$, respectively. 
We have the following commutative diagram: 
%\begin{equation*} 
%U =U\times_{\cal Y}{\cal Y} \os{\sus}{\lo} 
%U_j(i)=U\times_{{\cal Y},p_j(i)}{\cal Y}_m(i)
%\os{\sus}{\lo}
%T\times_{{\cal Y},p_j(i)}{\cal Y}_m(i),   
%\tag{3.6.1}\label{cd:csu} 
%\end{equation*} 
\begin{equation*} 
\begin{CD} 
U @>{\sus}>> T\times_S U @>{\rm proj.}>> U 
@>{\iota_Y\circ u}>> {\cal Y} \\ 
@V{\bigcap}V{\del_j(i)}V @V{\bigcap}VV @V{\bigcap}VV
@VV{{\rm diag}.}V  \\ 
U_j(i) @>{\sus}>> T\times_S U_j(i) 
@>{\rm proj.}>> U_j(i) 
@>{\rm proj.}>> {\cal Y}_m(i),   
\end{CD} 
\tag{3.6.1}\label{cd:utsu} 
\end{equation*} 
where each square in this diagram is cartesian. 
Let ${\cal J}$ be the defining ideal sheaf of 
the diagonal immersion ${\cal Y} \os{\sus}{\lo} {\cal Y}_m(i)$. 
Set ${\cal J}(T):={\cal J}{\cal O}_{T\times_S{\cal Y}_m(i)}$. 
Because each square in (\ref{cd:utsu}) is cartesian,  
%and since the base change of the following diagram 
%\begin{equation*} 
%\begin{CD} 
%U  @>>> {\cal Y} \\  
%@VVV @VVV   \\ 
%U_j(i) @>>> {\cal Y}_m(i) 
%\end{CD} 
%\end{equation*} 
%is 
%\begin{equation*} 
%\begin{CD} 
%T\times_SU  @>>> T\times_S{\cal Y} \\  
%@VVV @VVV   \\ 
%T\times_SU_j(i) @>>> T\times_S{\cal Y}_m(i),  
%\end{CD} 
%\end{equation*} 
${\cal J}(T){\cal O}_{U_j(i)}$ is the defining ideal sheaf of 
the immersion $U \os{\sus}{\lo} U_j(i)$. 
Hence ${\cal I}_U={\cal I}_{U_j(i)}+{\cal J}(T)$. 
Since ${\cal J}(T)^{m+1}=0$,  
${\cal I}_U^{n+m} \subset {\cal I}_{U_j(i)}^n$.  
Therefore we have
the following natural morphism of log formal schemes  
\begin{equation*} 
h_{j,n+m,n}(i)\col 
{\mathfrak T}_{U_j(i),n}(T\times_S{\cal Y}_m(i)) \lo 
{\mathfrak T}_{U,n+m}(T\times_S{\cal Y}_m(i)) \quad 
(n\in {\mab Z}_{\geq 1}) 
\tag{3.6.2}\label{eqn:hjmn}
\end{equation*} 
as in the proof of \cite[(5.4)]{oc} using [loc.~cit., (2.5)]. 
Then $\Del_{j,n+m}(i)\circ h_{j,n+m,n}(i)$ 
(resp.~$h_{j,n+m,n}(i)\circ \Del_{j,n}(i)$) 
is the transition morphism 
${\mathfrak T}_{U_j(i),n}(T\times_S{\cal Y}_m(i)) 
\lo {\mathfrak T}_{U_j(i),n+m}(T\times_S{\cal Y}_m(i))$
(resp.~${\mathfrak T}_{U,n}(T\times_S{\cal Y}_m(i)) 
\lo {\mathfrak T}_{U,n+m}(T\times_S{\cal Y}_m(i))$). 
As stated in [loc.~cit., (2.5)], 
we have the following isomorphism of inductive objects: 
\begin{equation*} 
\{{\mathfrak T}_{U,n}
(T\times_S{\cal Y}_m(i))\}_{n=1}^{\infty}
\os{\sim}{\lo} 
\{{\mathfrak T}_{U_j(i),n}
(T\times_S{\cal Y}_m(i))\}_{n=1}^{\infty}.  
\tag{3.6.3}\label{eqn:hjimn}
\end{equation*} 
We identify 
$\{{\mathfrak T}_{U_j(i),n}(T\times_S{\cal Y}_m(i))\}_{n=1}^{\infty}$ 
with 
$\{{\mathfrak T}_{U,n}(T\times_S{\cal Y}_m(i))\}_{n=1}^{\infty}$ 
by this isomorphism. 
For a morphism 
$\alpha \col T'\lo {\mathfrak T}_Y({\cal Y})$ 
of quasi-prewidenings of $Y/S$, 
let $\alpha_n \col T'_n \lo 
{\mathfrak T}_{Y,n}({\cal Y})$ be 
the induced morphism by $\alpha$. 
Set ${\cal E}_{T'_n}:=\alpha^*_n({\cal E}_n)$. 
Let ${\cal I}_1$ be the defining ideal sheaf of 
the immersion 
${\cal Y} \os{\sus}{\lo} {\cal Y}_1(1)$. 
Then there exists a functorial isomorphism 
\begin{equation*} 
{\cal I}_1 \simeq 
\Om^1_{{\cal Y}/S} 
\tag{3.6.4}\label{eqn:iom1} 
\end{equation*}  
(\cite[(5.8.1)]{klog1}, \cite[Proposition 3.2.5]{s1}).  
By using this fact, 
we see that ${\cal O}_{{\cal Y}_m(1)}$ is 
a locally free ${\cal O}_{\cal Y}$-module of finite rank.  
%Then 
By (\ref{lemm:bcue}) 
we have the following isomorphism 
\begin{align*} 
\eps^m  
\col & {\cal O}_{{\cal Y}_m(1)}\otimes_{{\cal O}_{\cal Y}}
\vpl_n{\cal E}_{{\mathfrak T}_n}
=\vpl_np_{2,n}(1)^*
{\cal E}_{{\mathfrak T}_n}
= 
\vpl_n
{\cal E}_{{\mathfrak T}_{U_2(1),n}(T\times_S{\cal Y}_m(1))} 
\tag{3.6.5}\label{eqn:emcoym}\\
& \os{\sim}{\longleftarrow}  
\vpl_n{\cal E}_{{\mathfrak T}_{U,n}(T\times_S{\cal Y}_m(1))}
\os{\sim}{\lo}  
\vpl_n{\cal E}_{{\mathfrak T}_{U_1(1),n}
(T\times_S{\cal Y}_m(1))} 
=\vpl_np_{1,n}(1)^*
{\cal E}_{{\mathfrak T}_n}
\\
& = (\vpl_n
{\cal E}_{{\mathfrak T}_n})
\otimes_{{\cal O}_{\cal Y}}{\cal O}_{{\cal Y}_m(1)}. 
\end{align*} 
%Here we have used the locally freeness of 
%${\cal O}_{{\cal Y}_m(1)}$ of finite rank 
%as ${\cal O}_{\cal Y}$-modules 
%(cf.~\cite[(5.8.1)]{klog1}) 
%to obtain the first and last equalities of (\ref{eqn:emcoym}). 
Set ${\mathfrak E}:=
\vpl_n({\cal E}_{{\mathfrak T}_n})$ and 
$\nabla^0(x):=\eps^1(1\otimes x)-x\otimes 1$  
for a local section $x$ of ${\mathfrak E}$. 
%Let ${\cal Y}^{\rm ex}_m(i)$ 
%be the exactification of the diagonal immersion 
%${\cal Y} \os{\sus}{\lo} {\cal Y}_m(i)$ (\cite{s3}). 
%Then we also have an isomorphism  
%${\mathfrak T}_{U_j(i),n}(T\times_S{\cal Y}_m(i))
%\simeq {\mathfrak T}_n
%\times_{\wt{\cal Y},p_j(i)}\wt{\cal Y}^{\rm ex}_m(i)$. 
%Hence  the morphism 
%$\eps^m$ is equal to the morphism 
%${\cal O}_{{\cal Y}^{\rm ex}_m(1)}
%\otimes_{{\cal O}_{\cal Y}}{\mathfrak E}\lo 
%{\mathfrak E}\otimes_{{\cal O}_{\cal Y}}
%{\cal O}_{{\cal Y}^{\rm ex}_m(1)}$. 
%Let ${\cal I}_1$ be the defining ideal sheaf of 
%the local immersion 
%${\cal Y} \os{\sus}{\lo} {\cal Y}_1^{\rm ex}(1)$. 
%Then there exists a functorial isomorphism 
%\begin{equation*} 
%{\cal I}_1 \simeq \Om^1_{{\cal Y}/S} 
%\tag{3.6.6}\label{eqn:iom1} 
%\end{equation*}  
%(\cite[(5.8.1)]{klog1}, \cite[Proposition 3.2.5]{s1}).   
The morphism $\eps^1~{\rm mod}~ {\cal I}_1$ 
is the identity.  
Thus $\nabla^0$ induces the following morphism  
\begin{equation*} 
\nabla^0 \col {\mathfrak E}
\lo  {\mathfrak E}
\otimes_{{\cal O}_{\cal Y}}\Om^1_{{\cal Y}/S}. 
\tag{3.6.6}\label{eqn:nbqe} 
\end{equation*}  
%For two local exactifications 
%${\cal Y}_1^{\rm ex}(1)$ and ${\cal Y}'_1{}^{\rm ex}(1)$ 
%of the diagonal immersion 
%${\cal Y} \os{\sus}{\lo} {\cal Y}_1(1)$, 
%there exists a local exactification ${\cal Y}^{\rm ex}_{12}(1)$ 
%fitting into the following commutative diagram 
%\begin{equation*} 
%\begin{CD} 
%{\cal Y} @= {\cal Y} @= {\cal Y} \\ 
%@V{\bigcap}VV @V{\bigcap}VV @V{\bigcap}VV \\ 
%{\cal Y}_1^{\rm ex}(1) @<<< {\cal Y}_{12}^{\rm ex}(1) 
%@>>> {\cal Y}'_1{}^{\rm ex}(1) \\ 
%@VVV @VVV @VVV \\ 
%{\cal Y}_1(1) @= {\cal Y}_1(1) @= {\cal Y}_1(1). 
%\end{CD} 
%\end{equation*} 
%Hence, by the functoriality of the isomorphism 
%(\ref{eqn:iom1}), 
%the morphism (\ref{eqn:nbqe}) is independent of the choice 
%of the local exactification, and consequently  
%the morphism (\ref{eqn:nbqe}) is globalized. 
%\par 
%We have the following diagram 
%of the inductive systems of 
%the universal enlargements of $Y/S$:  
%\begin{equation*}
%\begin{CD}
%\{{\mathfrak T}_n\}_{n=1}^{\infty} 
%@>{\{p_n\}_{n=1}^{\infty}}>> 
%\{{\mathfrak T}_{Y,n}({\cal Y})\}_{n=1}^{\infty}\\ 
%@VV{\{p'_n\}_{n=1}^{\infty}}V  \\
%T,   
%\end{CD}
%\end{equation*} 
%where $\{p'_n\}_{n=1}^{\infty}$ 
%is the family of the natural morphisms. 
By (\ref{eqn:lueyset}) we obtain the following equality: 
\begin{equation*}
L^{\rm UE}_{Y/S}({\cal E})_{T}
=\vpl_np'_{n*}
({\cal E}_{{\mathfrak T}_n}).  
\tag{3.6.7}\label{eqn:etu}
\end{equation*}
Hence we have the following morphism  
\begin{align*} 
\nabla^0 \col L^{\rm UE}_{Y/S}({\cal E})_{T}
= 
\vpl_np'_{n*}
({\cal E}_{{\mathfrak T}_n}) 
& \lo \vpl_n p'_{n*}
({\cal E}_{{\mathfrak T}_n})
\otimes_{{\cal O}_{\cal Y}}\Om^1_{{\cal Y}/S} 
\tag{3.6.8}\label{eqn:nlqe}\\
{} & = L^{\rm UE}_{Y/S}({\cal E}\otimes_{{\cal O}_{\cal Y}}
\Om^1_{{\cal Y}/S})_{T}. 
\end{align*} 
The morphism $\nabla^0$ induces $\nabla^i$ $(i\in {\mab N})$. 
\par 
The rest is to prove that 
$\nabla^{i+1} \circ \nabla^i=0$ $(i\in {\mab N})$. 
This is a routine work. For the completeness of this paper, 
we give the following proof of this fact.
\par  
As usual, it suffices to prove that 
$\nabla^{1} \circ \nabla^0=0$.  
The natural isomorphism 
${\cal Y}(1)\times_{\cal Y}{\cal Y}(1)\os{\sim}{\lo} {\cal Y}(2)$ 
induces a morphism 
${\cal Y}_m(1)\times_{\cal Y}{\cal Y}_n(1)\lo {\cal Y}_{m+n}(2)$ 
(\cite[p.~561]{s1}). 
%By the local descriptions of ${\cal Y}_1(1)$ and ${\cal Y}_2(2)$, 
%we see that this morphism is an isomorphism. 
Hence we have the following morphism as usual: 
\begin{align*} 
\del^{m,n}  \col {\cal O}_{{\cal Y}_{m+n}(1)} \os{p_{13}^*}{\lo} 
{\cal O}_{{\cal Y}_{m+n}(2)}\lo {\cal O}_{{\cal Y}_m(1)\times_{\cal Y}{\cal Y}_n(1)}. 
\end{align*}
The rest of the proof is well-known. 
Indeed, 
let 
\begin{align*} 
\theta^m\col {\mathfrak E} \lo {\cal O}_{{\cal Y}_m(1)}\otimes_{{\cal O}_{\cal Y}}
{\mathfrak E} \os{\eps^m}{\lo}  
{\mathfrak E}\otimes_{{\cal O}_{\cal Y}}{\cal O}_{{\cal Y}_m(1)}
\end{align*} 
be the usual composite morphism obtained by $\eps^m$. 
Because ${\cal E}$ is a coherent crystal of 
${\cal K}_{\os{\to}{\mathfrak T}_Y({\cal Y})}$-modules, 
$\{\eps^m\}_{m\in {\mab N}}$ satisfies the 
usual cocycle condition.
Hence the following diagram is commutative as usual: 
\begin{equation*} 
\begin{CD}
{\mathfrak E}
@>{\theta^{n}}>> {\mathfrak E}\otimes_{{\cal O}_{\cal Y}}{\cal O}_{{\cal Y}_n(1)}\\
@V{\theta^{m+n}}VV @VV{\theta^m\otimes{\rm id}}V \\
{\mathfrak E}\otimes_{{\cal O}_{\cal Y}}{\cal O}_{{\cal Y}_{m+n}(1)}
@>{\del^{m,n}}>> {\mathfrak E}\otimes_{{\cal O}_{\cal Y}}{\cal O}_{{\cal Y}_m(1)}
\otimes_{{\cal O}_{\cal Y}}{\cal O}_{{\cal Y}_n(1)}. 
\end{CD}
\tag{3.6.9}\label{cd:ooym}
\end{equation*} 
For a ${\cal K}_{\cal Y}$-module ${\cal F}$, 
set 
$${\rm Diff}^m_{{\cal Y}/S}({\cal F}):=
{\rm Hom}_{{\cal K}_{\cal Y}}({\cal K}_{{\cal Y}_m(1)}\otimes_{{\cal K}_{\cal Y}}
{\cal F},{\cal F}).$$
As usual, we have the following morphism 
\begin{align*} 
\nabla \col {\rm Diff}^m_{{\cal Y}/S}({\cal K}_{\cal Y})\lo {\rm Diff}^m_{{\cal Y}/S}({\mathfrak E})
\end{align*} 
by the following composite morphism 
$$\nabla(h) \col {\cal K}_{{\cal Y}_m(1)}\otimes_{{\cal K}_{\cal Y}}{\mathfrak E}\os{\eps^m}{\lo} 
{\mathfrak E}\otimes_{{\cal K}_{\cal Y}}{\cal K}_{{\cal Y}_m(1)}
\os{{\rm id}\otimes h}{\lo} {\mathfrak E}\otimes_{{\cal K}_{\cal Y}}{\cal K}_{{\cal Y}}
={\mathfrak E}$$
for $h\in {\rm Diff}^m_{{\cal Y}/S}({\cal K}_{\cal Y})$. 
Hence, for $h_1\in  {\rm Diff}^m_{{\cal Y}/S}({\cal K}_{\cal Y})$ and  
$h_2\in  {\rm Diff}^n_{{\cal Y}/S}({\cal K}_{\cal Y})$, 
$\nabla(h_1\circ h_2)=\nabla(h_1)\circ \nabla(h_2)$. 
Indeed, we see that 
$\nabla(h_1\circ h_2)$ is equal to the induced morphism by the following morphism: 
\begin{align*} 
\nabla(h_1\circ h_2) \col & 
{\mathfrak E}\os{\theta^{m+n}}{\lo} 
{\mathfrak E}\otimes_{{\cal K}_{\cal Y}}{\cal K}_{{\cal Y}_{m+n}(1)}
\os{{\rm id}\otimes \del^{m,n}}{\lo} 
{\mathfrak E}\otimes_{{\cal K}_{\cal Y}}{\cal K}_{{\cal Y}_{m}(1)}
\otimes_{{\cal K}_{\cal Y}}{\cal K}_{{\cal Y}_{n}(1)}\\
& \os{{\rm id}\circ h_2}{\lo} 
{\mathfrak E}\otimes_{{\cal K}_{\cal Y}}{\cal K}_{{\cal Y}_{m}(1)}
\os{{\rm id}\circ h_1}{\lo} {\mathfrak E}.
\end{align*} 
On the other hand, we see that 
$\nabla(h_1)\circ \nabla(h_2)$ is equal to the induced morphism by 
the following morphism: 
\begin{align*} 
\nabla(h_1)\circ \nabla(h_2) \col & 
{\mathfrak E}\os{\theta^n}{\lo} 
{\mathfrak E}\otimes_{{\cal K}_{\cal Y}}{\cal K}_{{\cal Y}_{n}(1)}
\os{\theta^m\otimes {\rm id}}{\lo} 
{\mathfrak E}\otimes_{{\cal K}_{\cal Y}}{\cal K}_{{\cal Y}_{m}(1)}
\otimes_{{\cal K}_{\cal Y}}{\cal K}_{{\cal Y}_{n}(1)}\\
& \os{{\rm id}\circ h_2}{\lo} 
{\mathfrak E}\otimes_{{\cal K}_{\cal Y}}{\cal K}_{{\cal Y}_{m}(1)}
\os{{\rm id}\circ h_1}{\lo} {\mathfrak E}.
\end{align*} 
By (\ref{cd:ooym}) we obtain the equality 
$\nabla(h_1\circ h_2)=\nabla(h_1)\circ \nabla(h_2)$. 
%
%
%
%
%
%
%
%
%
%
%
%
%
%
%For coherent ${\cal K}_{\cal Y}$-modules ${\cal F}$ and ${\cal G}$ and a positive integer $m$, 
\par 
Now (\ref{prop:lcl}) follows from the lemma (\ref{lemm;k}) below 
since $\part_i \circ \part_j=\part_j \circ \part_i$ in the notation in (\ref{lemm;k}) 
(cf.~\cite[Lemma 3.2.7]{s1}). 
%As in \cite[(1.1.7), (1.1.8)]{ofgt}, 
%we can check this by the usual cocycle condition of $\{\eps^m\}_{m\in {\mab N}}$.    
%\cite[II Th\'{e}or\`{e}me 4.2.11]{bb} 
%(cf.~\cite[Proposition 2.11, Construction 2.14, Theorem 4.8, Theorem 4.12]{bob}), 
\end{proof}

\begin{lemm}\label{lemm;k}
Let $\{d\log t_i\}_{i=1}^n$ be a local basis of 
$\Om^1_{{\cal Y}/S}$, where 
$t_i$ is a local section of the log structure of ${\cal Y}$. 
Let $\{\part_i(:=t_i\dfrac{\part}{\part t_i})\}_{i=1}^n\in {\rm Diff}^1_{{\cal Y}/S}({\cal O}_{\cal Y})$ 
be the dual basis of $\{d\log t_i\}_{i=1}^n$. 
Set 
$$K:=\nabla^1 \circ \nabla^0 \col 
{\mathfrak E}\lo {\mathfrak E}\otimes_{{\cal O}_{\cal Y}}\Om^2_{{\cal Y}/S}.$$ 
Then 
$$K=\sum_{1\leq i<j\leq n}[\nabla(\part_i), \nabla(\part_j)]\otimes (d\log t_i\wedge d\log t_j).$$
\end{lemm}
\begin{proof} 
By noting $d(d\log t_i)=0$, we obtain (\ref{lemm;k}) by a straight calculation. 
\end{proof}

The following is a relative version of 
\cite[(5.4)]{oc} and \cite[Theorem 2.3.5 (2)]{s2}, 
though the formulation of it is rather different from them:

\begin{theo}[{\bf Log convergent Poincar\'{e} lemma}]\label{theo:pl}
Let $E$ be a coherent crystal of ${\cal K}_{Y/S}$-modules. 
Set ${\cal E}_n:=E_{{\mathfrak T}_{Y,n}({\cal Y})}$ and 
${\cal E}:=\{{\cal E}_n\}_{n=1}^{\infty}$. 
Then the natural morphism 
\begin{equation*}
E \lo
L^{\rm UE}_{Y/S}({\cal E}\otimes_{{\cal O}_{\cal Y}}
\Om^{\bul}_{{\cal Y}/S})
\tag{3.8.1}\label{eqn:epdlym}
\end{equation*}
is a quasi-isomorphism. 
\end{theo}
\begin{proof} 
Set ${\mathfrak T}:={\mathfrak T}_Y({\cal Y})
=(Y,{\cal Y},\iota,{\rm id}_Y)$ and 
${\mathfrak T}_n:={\mathfrak T}_{Y,n}({\cal Y})$ 
$(n\in {\mab Z}_{>0})$. 
Following \cite[p.~152]{oc} and 
\cite[p.~95]{s2}, set 
$$\Om^i_{{\mathfrak T}/S}(E)
:=j_{{\mathfrak T}*}
(j^*_{\mathfrak T}(E)\otimes_{{\cal K}_{Y/S}
\vert_{\mathfrak T}}
\varphi^*_{\os{\to}{\mathfrak T}}
({\cal K}_{\os{\to}{\mathfrak T}}
\otimes_{{\cal O}_{\cal Y}}\Om^i_{{\cal Y}/S})) 
\quad (i\in {\mab N}).$$ 
Let $T$ be an enlargement of $Y/S$ with a morphism 
$T \lo {\mathfrak T}$ of prewidenings. 
Then, by (\ref{eqn:chynhn}), 
the morphism $T\lo {\mathfrak T}$ factors through 
$T \lo {\mathfrak T}_n$ for some $n\in {\mab N}$. 
Hence $(\varphi_{\os{\to}{\mathfrak T}}^*({\cal E}))_{T}=
{\cal K}_{T}\otimes_{{\cal K}_{{\mathfrak T}_n}}
E_{{\mathfrak T}_n}=E_T$. 
On the other hand, 
$j_{\mathfrak T}^*(E)_{(T\lo {\mathfrak T})}=E_T$. 
Thus 
\begin{align*}
\Om^i_{{\mathfrak T}/S}(E)& 
=j_{{\mathfrak T}*}
(\varphi^*_{\os{\to}{\mathfrak T}}({\cal E})
\otimes_{{\cal K}_{Y/S}\vert_{{\mathfrak T}}}
\varphi^*_{\os{\to}{\mathfrak T}}
({\cal K}_{\os{\to}{\mathfrak T}}
\otimes_{{\cal O}_{\cal Y}}\Om^i_{{\cal Y}/S}))
\\
&
=j_{{\mathfrak T}*}
\varphi^*_{\os{\to}{\mathfrak T}}({\cal E}
\otimes_{{\cal O}_{\cal Y}}\Om^i_{{\cal Y}/S})  
=L^{\rm UE}_{Y/S}
({\cal E}\otimes_{{\cal O}_{\cal Y}}\Om^i_{{\cal Y}/S}).
\end{align*}
Moreover, by the proof of 
\cite[(5.4)]{oc} and \cite[Theorem 2.3.5 (2)]{s2} 
and (\ref{prop:lcl}),  
$\Om^{\bul}_{{\mathfrak T}/S}(E)=
L^{\rm UE}_{Y/S}
({\cal E}\otimes_{{\cal O}_{\cal Y}}\Om^{\bul}_{{\cal Y}/S})$ 
as complexes. 
Hence, by the proof of 
\cite[(5.4)]{oc} and \cite[Theorem 2.3.5 (3)]{s2}, 
$E=L^{\rm UE}_{Y/S}
({\cal E}\otimes_{{\cal O}_{\cal Y}}\Om^{\bul}_{{\cal Y}/S})$ 
in $D^+({\cal K}_{Y/S})$. 
\end{proof}

%\begin{prop}\label{prop:fre} 
%The functor {\rm (\ref{eqn:lvnc})} is right exact. 
%\end{prop}
%\begin{proof} 
%Let ${\cal E}$ be a crystal of ${\cal K}_{\os{\to}{T}}$-modules. 
%Then the value of $L^{\rm UE}_{Y/S}({\cal E})$ 
%at $T_n$ is ${\cal O}_{T_n(1)}\otimes_{{\cal O}_{T_n}}{\cal E}$ 
%by (\ref{eqn:lueyset}). 
%Hence the functor {\rm (\ref{eqn:lvnc})} is right exact. 
%\end{proof} 

The following has been essentially proved 
in \cite[(6.6)]{oc} and \cite[Corollary 2.3.8]{s2}: 

\begin{prop}\label{prop:cdfza}
Let the notations be as in {\rm (\ref{theo:pl})}. 
%Let 
%\begin{equation*}
%u^{\rm conv}_{Y/S} \col 
%(\wt{(Y/S})_{\rm conv},{\cal K}_{Y/S}) 
%\lo (Y_{\rm zar}, f^{-1}({\cal K}_S)) 
%\tag{4.5.1}\label{eqn:uwksr}
%\end{equation*}
%be the obvious relative version of 
%the natural projection in {\rm \cite[pp.90--91]{s2}}. 
Set $Y_n:=Y\times_{\cal Y}{\mathfrak T}_{Y,n}({\cal Y})
=Y\times_{\wh{\cal Y}}{\mathfrak T}_{Y,n}(\wh{\cal Y})$, 
where $\wh{\cal Y}$ is 
the formal completion of ${\cal Y}$ along $Y$.   
Let 
$\bet_n \col Y_n \lo Y$ be 
the projection and identify 
${\mathfrak T}_{Y,n}({\cal Y})_{\rm zar}$ with 
$Y_{n,{\rm zar}}$. 
Then 
\begin{equation*}
Ru^{\rm conv}_{Y/S*}(E)=
\vpl_n\bet_{n*}({\cal E}_n{\otimes}_{{\cal O}_{{\cal Y}}} 
\Om^{\bul}_{{\cal Y}/S})
\tag{3.9.1}\label{eqn:uybo}
\end{equation*}
in $D^+(f^{-1}({\cal K}_S))$. 
\end{prop}
\begin{proof} 
(The proof is the same as that of \cite[Corollary 2.3.4]{s2} 
which is the log version of \cite[Corollary 4.4]{oc}.) 
By (\ref{eqn:epdlym}) we have the following isomorphism 
\begin{equation*} 
Ru^{\rm conv}_{Y/S*}(E) \os{\sim}{\lo}
Ru^{\rm conv}_{Y/S*}
(L^{\rm UE}_{Y/S}({\cal E}\otimes_{{\cal O}_{\cal Y}}
\Om^{\bul}_{{\cal Y}/S})). 
\tag{3.9.2}\label{eqn:uymse} 
\end{equation*} 
By the obvious relative version of 
\cite[Corollary 2.3.4]{s2} which is the log version of 
\cite[Corollary 4.4]{oc},  
for a coherent crystal $F$ of ${\cal K}_{Y/S}\vert_{\mathfrak T}$-modules, 
$j_{{\mathfrak T}*}(F)$ is $u^{\rm conv}_{Y/S*}$-acyclic. 
Hence the right hand side of (\ref{eqn:uymse})
is equal to 
$u^{\rm conv}_{Y/S*}
L^{\rm UE}_{Y/S}({\cal E}{\otimes}_{{\cal O}_{{\cal Y}}} 
\Om^{\bul}_{{\cal Y}/S})$. 
\par
By the relative version of 
the commutativity of the diagram of topoi 
in \cite[p.~91]{s2} which is a log version of that 
of topoi in \cite[p.~147]{oc}, we have 
$(\vpl_n\bet_{n*}){\varphi}_{\os{\to}{{\mathfrak T}}*}
=u^{\rm conv}_{Y/S*}j_{{\mathfrak T}*}$. 
Hence, for a ${\cal K}_{\os{\to}{{\mathfrak T}}}$-module 
${\cal F}=\{{\cal F}_n\}_{n=1}^{\infty}$, we have the following: 
\begin{equation*}
u^{\rm conv}_{Y/S*}L^{\rm UE}_{Y/S}({\cal F})
= u^{\rm conv}_{Y/S*}j_{{\mathfrak T}*}
{\varphi}^*_{\os{\to}{{\mathfrak T}}}({\cal F})
= 
(\vpl_n\bet_{n*})
{\varphi}_{\os{\to}{{\mathfrak T}}*}{\varphi}^*_{\os{\to}{{\mathfrak T}}}({\cal F}) 
= \vpl_n\bet_{n*}({\cal F}_n).  
\end{equation*} 
We complete the proof of (\ref{prop:cdfza}). 
\end{proof}

\section{Log convergent linearization functors.~II}\label{sec:llf}
In this section we prove that  
log convergent linearization functors are 
compatible with closed immersions of log schemes. 
This is the log convergent version of the compatibility of 
log crystalline linearization functors with closed immersions of log schemes 
in \cite[(2.2)]{nh2}.   
\par 
Let ${\cal V}$, $\pi$, $S$ and $S_1$ 
be as in \S\ref{sec:logcd}. 

The following is the log convergent version of 
\cite[(2.2.6)]{nh2} whose proof is 
much simpler than that of [loc.~cit.]:

\begin{lemm}[{\bf {cf.~\cite[(2.2.6)]{nh2}}}]\label{lemm:increp}
Let $\iota_Y \col Y_1 \os{\sus}{\lo} Y_2$ be 
a closed immersion of 
fine log schemes over $S_1$ 
whose underlying schemes are of finite type 
over $\os{\circ}{S}_1$.   
Let $T_j:=(U_j,T_j,\iota_j,u_j)$ 
$(j=1,2)$ be a $($pre$)$widening of $Y_j/S$.
Let ${\iota}_U \col U_1 \os{\sus}{\lo} U_2$ and 
$\iota_T \col T_1 \os{\sus}{\lo} T_2$ 
be closed immersions  
of fine log formal schemes over $S$ 
whose underlying formal schemes 
are noetherian formal schemes 
which are topologically of finite type over $\os{\circ}{S}$.  
Assume that ${\iota}_U$ and $\iota_T$
induce a morphism 
$(\iota_U,\iota_T) \col (U_1,T_1,\iota_1,u_1) 
\lo (U_2,T_2,\iota_2,u_2)$ 
of $($pre$)$widenings of $Y_1$ and $Y_2$. 
Let 
$$\iota^{\rm loc}_{\rm conv} \col 
(Y_1/S)_{\rm conv}\vert_{T_1} 
\lo 
(Y_2/S)_{\rm conv}\vert_{T_2}$$ 
be the induced morphism of topoi by $\iota_Y$ and 
$(\iota_U,\iota_T)$.
Let $(U,T,\iota,u)$ be a $($pre$)$widening of  $Y_2/S$ 
over $(U_2,T_2,\iota_2,u_2)$. 
Then $\iota^{{\rm loc}*}_{\rm conv}((U,T,\iota,u))=
(U\times_{U_2}U_1,T\times_{T_2}T_1,
\iota \times_{\iota_2}\iota_1,u\times_{u_2}u_1)$ 
as sheaves. 
\end{lemm}
\begin{proof}
We prove (\ref{lemm:increp}) for the case of widenings. 
It is evident that the closed immersion 
$U\times_{U_2}U_1 \os{\sus}{\lo} T\times_{T_2}T_1$ 
is defined by an ideal sheaf of definition of  
$T\times_{T_2}T_1$.  
Hence 
$(U\times_{U_2}U_1,T\times_{T_2}T_1,
\iota \times_{\iota_2} \iota_1,u\times_{u_2}u_1)$ 
is a widening of $Y_1/S$. 
It is straightforward to see that
$\iota^{{\rm loc}*}_{\rm conv}
((U,T,\iota,u))=
(U\times_{U_2}U_1,T\times_{T_2}T_1,\iota \times_{\iota_2}\iota_1,
u\times_{u_2}u_1)$. 
\end{proof}

\par 
Let $\iota_{Y,Z} \col Z \os{\subset}{\lo} Y$ 
be a closed immersion of fine log 
schemes over $S_1$.
Assume that there exists the following cartesian diagram
\begin{equation*}
\begin{CD}
Z  @>{\iota_{\cal Z}}>> {\cal Z} \\ 
@V{\iota_{Y,Z}}VV @VV{\iota_{{\cal Y},{\cal Z}}}V \\
Y @>{\iota_{\cal Y}}>> {\cal Y},
\end{CD}
\tag{4.1.1}\label{cd:zyectn}
\end{equation*}
where $\iota_{\cal Z}$ and $\iota_{\cal Y}$ 
are closed immersions into 
fine log flat $p$-adic formal schemes over $S$ and 
$\iota_{{\cal Y},{\cal Z}}$ is a closed immersion of 
fine log flat $p$-adic formal schemes over $S$. 
Assume that $\os{\circ}{\cal Y}$ and $\os{\circ}{\cal Z}$ 
are noetherian formal schemes which 
are topologically of finite type over $\os{\circ}{S}$.  
Then we have the natural following closed immersion  
by the universality of the exactification: 
\begin{equation*} 
\iota_{{\cal Y}^{\rm ex},{\cal Z}^{\rm ex}}\col 
{\cal Z}^{\rm ex} \os{\sus}{\lo} {\cal Y}^{\rm ex}. 
\tag{4.1.2}\label{eqn:iexyz} 
\end{equation*} 
Set ${\mathfrak T}_Z({\cal Z}):=
(Z,{\cal Z},\iota_{\cal Z},{\rm id}_Z)$ 
and  
${\mathfrak T}_Y({\cal Y})
:=(Y, {\cal Y},\iota_{\cal Y},{\rm id}_Y)$.  
Let 
$\{{\mathfrak T}_{Z,n}({\cal Z})\}_{n=1}^{\infty}$ 
and  
$\{{\mathfrak T}_{Y,n}({\cal Y})\}_{n=1}^{\infty}$ 
be the systems of the universal enlargements 
of $\iota_{\cal Z}$ and $\iota_{\cal Y}$ 
with natural morphisms 
$g_{{\cal Y},n} \col {\mathfrak T}_{Y,n}({\cal Y}) \lo {\cal Y}$ 
and 
$g_{{\cal Z},n} \col {\mathfrak T}_{Z,n}({\cal Z}) \lo {\cal Z}$,  
respectively. 
Let 
$g_{{\cal Y}^{\rm ex},n}
\col {\mathfrak T}_{Y,n}({\cal Y}) \lo {\cal Y}^{\rm ex}$ 
and 
$g_{{\cal Z}^{\rm ex},n} 
\col {\mathfrak T}_{Z,n}({\cal Z}) \lo {\cal Z}^{\rm ex}$ 
be also the natural morphisms. 
Because the diagram (\ref{cd:zyectn}) is cartesian, 
we have the following equality by (\ref{lemm:bcue}):
\begin{equation*} 
{\mathfrak T}_{Z,n}({\cal Z})
=\wt{{\mathfrak T}_{Y,n}({\cal Y})
\times_{\cal Y}{\cal Z}}. 
\tag{4.1.3}\label{eqn:tfib}
\end{equation*} 
By the universality of the exactification, 
we see that the following equality holds$:$
\begin{equation*} 
{\cal Y}^{\rm ex}\times_{\cal Y}{\cal Z}
={\cal Z}^{\rm ex}. 
\tag{4.1.4}\label{eqn:tzcyb}
\end{equation*}
By using this equality and (\ref{lemm:bcue}), 
we obtain the following equality: 
\begin{equation*} 
{\mathfrak T}_{Z,n}({\cal Z})
=\wt{{\mathfrak T}_{Y,n}({\cal Y})
\times_{{\cal Y}{}^{\rm ex}}{\cal Z}{}^{\rm ex}}. 
\tag{4.1.5}\label{eqn:tzyb}
\end{equation*}
Set $Z_n:=Z\times_{\cal Z}{\mathfrak T}_{Z,n}({\cal Z})$ and 
$Y_n:=Y\times_{\cal Y}{\mathfrak T}_{Y,n}({\cal Y})$. 
Let $\iota_{{\cal Z},n}$ 
and $\iota_{{\cal Y},n}$ 
be the natural closed immersions 
$Z_n \os{\sus}{\lo} {\mathfrak T}_{Z,n}({\cal Z})$ 
and 
$Y_n \os{\sus}{\lo} {\mathfrak T}_{Y,n}({\cal Y})$, respectively.  
Let $\iota_{{\cal Y},{\cal Z},n} \col 
{\mathfrak T}_{Z,n}({\cal Z}) \os{\sus}{\lo} 
{\mathfrak T}_{Y,n}({\cal Y})$ be also the natural closed immersion. 
Let $u_{Z,n}$ and $u_{Y,n}$ be the first projections 
$Z_n \lo Z$ and $Y_n \lo Y$. 
Let $\star$ be nothing or ${\rm ex}$. 
Let 
$$g^*_{{\cal Z}^{\star}} 
\col \{{\cal K}_{{\cal Z}^{\star}}\text{-modules}\}
\lo \{{\cal K}_{\os{\to}{\mathfrak T}_Z({\cal Z)}}\text{-modules}\}$$ 
and 
$$g^*_{{\cal Y}^{\star}} \col 
\{{\cal K}_{{\cal Y}^{\star}}\text{-modules}\}
\lo \{{\cal K}_{\os{\to}{\mathfrak T}_Y({\cal Y)}}\text{-modules}\}$$ 
be the natural pull-backs obtained by   
$\{g_{{\cal Z}^{\star},n}\}_{n=1}^{\infty}$ and 
$\{g_{{\cal Y}^{\star},n}\}_{n=1}^{\infty}$, respectively. 
%Let 
%$$g_{{\cal Z}^{\rm ex}}^*\col \{{\cal K}_{{\cal Z}^{\rm ex}}\text{-modules}\}
%\lo \{{\cal K}_{\os{\to}{\mathfrak T}_Z({\cal Z)}}\text{-modules}\}$$ 
%and 
%$$g_{{\cal Y}^{\rm ex}}^* \col \{{\cal K}_{{\cal Y}^{\rm ex}}\text{-modules}\}
%\lo \{{\cal K}_{\os{\to}{\mathfrak T}_Y({\cal Y)}}
%\text{-modules}\}$$ 
%be the natural pull-backs obtained from  
%$\{g_{{\cal Y}^{\rm ex},n}\}_{n=1}^{\infty}$ and 
%$\{g_{{\cal Z}^{\rm ex},n}\}_{n=1}^{\infty}$, respectively. 
Let 
$$\os{\to}{\iota}{}_{{\cal Y},{\cal Z}*} \col 
\{{\cal K}_{\os{\to}{\mathfrak T}_Z({\cal Z})}\text{-modules}\} 
\lo \{{\cal K}_{\os{\to}{\mathfrak T}_Y({\cal Y})}\text{-modules}\}$$ 
and 
$${\iota}_{{\cal Y},{\cal Z},n*} \col 
\{{\cal K}_{{\mathfrak T}_{Z,n}({\cal Z})}\text{-modules}\} 
\lo \{{\cal K}_{{\mathfrak T}_{Y,n}({\cal Y})}\text{-modules}\} \quad 
(n \in {\mab Z}_{\geq 1})$$
be the natural direct images. 

\begin{lemm}[{\bf {cf.~\cite[(2.2.9)]{nh2}}}]\label{lemm:git}
$(1)$ 
%Let ${\cal E}$ be a quasi-coherent ${\cal O}_{{\cal Z}^{\star}}$-module. 
%For the quasi-coherent ${\cal K}_{{\cal Z}^{\star}}$-module 
%${\cal E}\otimes_{\cal V}K$, 
The following diagram 
\begin{equation*}
\begin{CD} 
\{\text{coherent } {\cal K}_{{\cal Z}^{\star}}\text{-modules}\}@>{g^*_{{\cal Z}^{\star}}}>> 
\{\text{coherent crystals of }{\cal K}_{\os{\to}{\mathfrak T}_Z({\cal Z})}\text{-modules}\} \\
@V{\iota_{{\cal Y}^{\star},{\cal Z}^{\star}*}}VV 
@VV{\os{\to}{\iota}{}_{{\cal Y},{\cal Z}*}}V \\
\{\text{coherent }  {\cal K}_{{\cal Y}^{\star}}\text{-modules}\}@>{g^*_{{\cal Y}^{\star}}}>>  
\{\text{coherent crystals of }{\cal K}_{\os{\to}{\mathfrak T}_Y({\cal Y})}\text{-modules}\}. 
\end{CD}
\tag{4.2.1}
\end{equation*}
%is commutative, that is, the natural morphism 
%$g_{{\cal Y}^{\star}}^*\iota_{{\cal Y}^{\star},{\cal Z}^{\star}*}({\cal E}\otimes_{\cal V}K) {\lo}
%\os{\to}{\iota}_{{\cal Y},{\cal Z}*}
%g^*_{{\cal Z}^{\star}}({\cal E}\otimes_{\cal V}K)$ 
%is an isomorphism.
\end{lemm}
\begin{proof}
(The proof is the same as that of \cite[(2.2.9)]{nh2}.)
Since the morphisms  
${\cal Z}\lo {\cal Y}$ and 
${\cal Z}^{\rm ex}\lo {\cal Y}^{\rm ex}$ 
are affine,  (\ref{lemm:git}) immediately follows 
from (\ref{eqn:tfib}), (\ref{eqn:tzyb}) and 
the affine base change theorem (\cite[(1.5.2)]{ega2}).
\end{proof}

\begin{lemm}
[{\bf {cf.~\cite[(2.2.10)]{nh2}}}]\label{lemm:phiit}  
Let 
$\iota^{\rm loc}_{\rm conv} \col 
(Z/S)_{\rm conv}
\vert_{{\mathfrak T}_{Z,n}({\cal Z})} 
\lo 
(Y/S)_{\rm conv}
\vert_{{\mathfrak T}_{Y,n}({\cal Y})}$ 
be the natural morphism of topoi. 
Then the following diagram 
\begin{equation*}
\begin{CD}
 \{\text{coherent crystals of }
{\cal K}_{\os{\to}{\mathfrak T}_Z({\cal Z})}
\text{-modules}\}
@>{\varphi_{\os{\to}{\mathfrak T}_Z({\cal Z})}^*}>> \\
@V{\os{\to}{\iota}{}_{{\cal Y},{\cal Z}*}}VV  \\
\{\text{coherent crystals of }
{\cal K}_{\os{\to}{\mathfrak T}_{Y}({\cal Y})}
\text{-modules}\} 
@>{\varphi_{\os{\to}{\mathfrak T}_Y({\cal Y})}^*}>> 
\end{CD}
\tag{4.3.1}
\end{equation*}
\begin{equation*}
\begin{CD}
\{\text{coherent crystals of }
{\cal K}_{Z/S}\vert_{{\mathfrak T}_Z({\cal Z})} 
\text{-modules}\}  \\
@VV{\iota^{\rm loc}_{{\rm conv}*}}V\\
\{\text{coherent crystals of }{\cal K}_{Y/S}
\vert_{{\mathfrak T}_Y({\cal Y})}\text{-modules}\}
\end{CD}
\end{equation*}
is commutative.
\end{lemm}
\begin{proof}
Let ${\cal E}=\{{\cal E}_n\}_{n=1}^{\infty}$ 
be a coherent crystal of 
${\cal K}_{\os{\to}{\mathfrak T}_{Z}({\cal Z})}$-modules.
Let $T=(U,T,\iota,u)$ be an enlargement of $Y/S$ 
with a morphism 
$h_{{\cal Y},n} \col  T \lo {\mathfrak T}_{Y,n}({\cal Y})$ 
of enlargements of $Y/S$. 
Let $\iota' \col 
U\times_{Y_n}{Z_n} \os{\sus}{\lo} 
T\times_{{\mathfrak T}_{Y,n}({\cal Y})}
{\mathfrak T}_{Z,n}({\cal Z})$ 
and $u' \col U\times_{Y_n}{Z_n} \lo Z$
be the natural morphisms. 
By abuse of notation, 
we denote the enlargement 
$(U\times_{Y_n}{Z_n}, T\times_{{\mathfrak T}_{Y,n}({\cal Y})}
{\mathfrak T}_{Z,n}({\cal Z}),\iota',u')$ 
simply by 
$T\times_{{\mathfrak T}_{Y,n}({\cal Y})}
{\mathfrak T}_{Z,n}({\cal Z})$. 
Then we have the natural morphism 
$h_{{\cal Z},n}  \col 
T\times_{{\mathfrak T}_{Y,n}({\cal Y})}
{\mathfrak T}_{Z,n}({\cal Z})  
\lo {\mathfrak T}_{Z,n}({\cal Z})$. 
By (\ref{lemm:increp}) we have the following: 
\begin{align*}
\iota_{{\rm conv}*}^{\rm loc}
\varphi_{{\mathfrak T}_{Z,n}({\cal Z})}^*({\cal E}_n)(T) 
& =\varphi_{{\mathfrak T}_{Z,n}({\cal Z})}^*({\cal E}_n)
(T\times_{{\mathfrak T}_{Y,n}({\cal Y})}{\mathfrak T}_{Z,n}({\cal Z})) \tag{4.3.2}\label{eqn:ilpe}\\ 
{} & =\Gam(T\times_{{\mathfrak T}_{Y,n}({\cal Y})}{\mathfrak T}_{Z,n}({\cal Z}),
h^*_{{\cal Z},n}({\cal E}_n)). 
\end{align*} 
On the other hand,  
\begin{equation*}
\varphi_{{\mathfrak T}_{Y,n}({\cal Y})}^*
{\iota_{{\cal Y},{\cal Z},n*}}({\cal E}_n)(T)
=\Gam(T,h^*_{{\cal Y},n}{\iota_{{\cal Y},{\cal Z},n*}}
({\cal E}_n)).
\tag{4.3.3}\label{eqn:phdpe}
\end{equation*}
Since the morphism 
$\iota_{{\cal Y},{\cal Z},n}\col 
{\mathfrak T}_{Z,n}({\cal Z}) \lo  {\mathfrak T}_{Y,n}({\cal Y})$ 
is a closed immersion by (\ref{eqn:tfib}), in 
particular, an affine morphism, 
the affine base change theorem tells us 
that the right hand sides  of 
(\ref{eqn:ilpe}) and (\ref{eqn:phdpe}) are the same.
This completes the proof of (\ref{lemm:phiit}).
\end{proof}

\begin{lemm}[{\bf cf.~\cite[(2.2.11)]{nh2}}]\label{lemm:jilt}
%Let $n$ be a positive integer. 
%Assume that 
%$\wt{{\cal Z}\times_{{\cal Y}}{\mathfrak T}_{Y,n}({\cal Y})}
%={\cal Z}\times_{{\cal Y}}{\mathfrak T}_{Y,n}({\cal Y})$. 
The following diagram of topoi
\begin{equation*}
\begin{CD}
(Z/S)_{\rm conv} \vert_{{\mathfrak T}_Z({\cal Z})} 
@>{j_{{\mathfrak T}_{Z}({\cal Z})}}>> (Z/S)_{\rm conv}\\ 
@V{\iota_{{\rm conv}}^{\rm loc}}VV 
@VV{\iota_{\rm conv}}V \\
(Y/S)_{\rm conv}\vert_{{\mathfrak T}_Y({\cal Y})} 
@>{j_{{\mathfrak T}_Y({\cal Y})}}>> 
(Y/S)_{\rm conv} \\  
\end{CD}
\tag{4.4.1}\label{cd:loctop}
\end{equation*}
is commutative. 
\end{lemm}
\begin{proof} 
Let $(U,T,\iota,u)$ be an object of ${\rm Conv}({Y/S})$. 
Let $\wh{T}$ be the formal completion of $T$ 
along the closed log subscheme 
$U\times_YZ$ of $U$ 
(we endow $\wh{T}$ with 
the inverse image of the log structure of $T$). 
Let $\hat{\iota}$ be the closed immersion 
$U\times_YZ \os{\sus}{\lo} \wh{T}$ and 
let $u_Z \col U\times_YZ \lo Z$ be the second projection. 
We claim that 
\begin{align*} 
\iota^{*}_{\rm conv}((U,T,\iota,u))
=(U\times_YZ,\wh{T},\hat{\iota},u_Z)
\tag{4.4.2}\label{cd:uit}
\end{align*} 
as a sheaf in $({Z/S})_{\rm conv}$. 
Indeed, let $(U',T',\iota',u')$ 
be an object of ${\rm Conv}({Z/S})$ 
fitting into the following commutative diagram over $S$:  
\begin{equation*}
\begin{CD}
Z @<{u'}<< U' @>{\iota'}>> T'\\ 
@V{\iota_{Y,Z}}VV @VVV @VV{g}V \\
Y @<{u}<< U @>{\iota}>> T.  
\end{CD} 
\tag{4.4.3}\label{cd:uat}
\end{equation*}
Obviously we have the following commutative 
diagram 
\begin{equation*}
\begin{CD}
U' @>{\iota'}>> T'\\ 
@VVV @VV{g}V \\
U\times_YZ @>{\subset}>> T.    
\end{CD} 
\tag{4.4.4}\label{cd:ut}
\end{equation*} 
Let ${\cal J}$ (resp.~${\cal J}'$) be 
the ideal sheaf of 
the lower horizontal closed immersion (resp.~$\iota'$) in (\ref{cd:ut}).  
Then the commutative diagram (\ref{cd:ut}) tells us that 
the natural morphism 
$g^*({\cal O}_T) \lo {\cal O}_{T'}$ induces a morphism 
$g^*({\cal J}) \lo {\cal J}'$. 
Since ${\cal O}_{T'}=\vpl_{n}{\cal O}_{T'}/{\cal J}'{}^n$, 
we have a natural morphism 
$g^*({\cal O}_{\wh{T}}) \lo {\cal O}_{T'}$ and 
hence a morphism 
$T' \lo \wh{T}$ of log formal schemes. 
Now it is easy to check that our claim holds.    
\par 
Let $(T\times_S{\mathfrak T}_{Y,n}({\cal Y}))^{\wh{}}$ 
and $(\wh{T}\times_S{\mathfrak T}_{Z,n}({\cal Z}))^{\wh{}}$ 
be the formal completions of $T\times_S{\mathfrak T}_{Y,n}({\cal Y})$ 
and $\wh{T}\times_S{\mathfrak T}_{Z,n}({\cal Z})$
along $U\times_YY_n$ and $U\times_YZ_n=(U\times_YZ)\times_ZZ_n$, 
respectively. 
Then, by (\ref{cd:uit}), we obtain  the following equalities: 
\begin{align*}
j_{{\mathfrak T}_Z({\cal Z})}^*
\iota^{*}_{\rm conv}((U,T,\iota,u)) 
& = (U\times_YZ,\wh{T},\hat{\iota},u_Z)\times  
(Z,{\cal Z},\iota_{\cal Z},{\rm id}_Z)
\tag{4.4.5}\label{ali:jwic}\\ 
{} & =(U\times_YZ,(T\times_S{\cal Z})^{\wh{}}),
\end{align*} 
where we omit to write the closed immersion 
$U\times_YZ\os{\sus}{\lo}(T\times_S{\cal Z})^{\wh{}}$ and 
the natural morphism $U\times_YZ \lo Z$ in the last object in (\ref{ali:jwic}).  
On the other hand, by (\ref{lemm:increp}), 
\begin{align*} 
\iota_{\rm conv}^{{\rm loc}*}
j_{{\mathfrak T}_Y({\cal Y})}^*((U,T,\iota,u))  
{} & =\iota_{\rm conv}^{{\rm loc}*}((U,T,\iota,u)\times 
(Y,{\cal Y},\iota_{\cal Y},{\rm id}_Y))
\tag{4.4.6}\label{eqn:jtit}\\ 
&= \iota_{\rm conv}^{{\rm loc}*}(U,(T\times_S{\cal Y})^{\wh{}})\\
&=(U\times_YZ,{\cal Z}\times_{\cal Y}(T\times_S{\cal Y})^{\wh{}})\\
&=(U\times_YZ,(T\times_S{\cal Z})^{\wh{}}). 
\end{align*}
By (\ref{ali:jwic}) and  (\ref{eqn:jtit}), 
we have 
\begin{equation*} 
j_{{\mathfrak T}_Z({\cal Z})}^*\iota^{*}_{\rm conv}
= \iota_{\rm conv}^{{\rm loc}*}
j_{{\mathfrak T}_Y({\cal Y})}^*.  
\end{equation*} 
Therefore the diagram (\ref{cd:loctop}) is commutative.
\end{proof}

\begin{coro}[{\bf cf.~\cite[(2.2.12)]{nh2}}]\label{linc}
%Assume that $\wt{\cal Y}={\cal Y}$ and $\wt{\cal Z}={\cal Z}$. 
There exists a canonical 
isomorphism of functors
\begin{equation*}
L^{\rm UE}_{Y/S}\circ 
\os{\to}{\iota}_{{\cal Y}, {\cal Z}*} \lo 
\iota_{{\rm conv}*}\circ L^{\rm UE}_{Z/S} 
\tag{4.5.1}\label{eqn:lyiilz}
\end{equation*}
for coherent crystals of 
${\cal K}_{\os{\to}{\mathfrak T}_Z({\cal Z})}$-modules. 
%Let 
%${\cal U}$ $($resp.~${\cal W})$ be 
%${\cal Y}$ or ${\cal Y}^{\rm ex}$
%$($resp.~${\cal Z}$ or ${\cal Z}^{\rm ex})$. 
Set $L^{\rm conv}_{Y/S}:= 
L^{\rm UE}_{Y/S}\circ \{g^*_{{\cal Y}^{\star},n}\}_{n=1}^{\infty}$ 
and $L^{\rm conv}_{Z/S}:=
L^{\rm UE}_{Z/S} \circ \{g^*_{{\cal Z}^{\star},n}\}_{n=1}^{\infty}$. 
Then there also exists the following canonical 
isomorphism of functors
\begin{equation*}
L^{\rm conv}_{Y/S} \circ 
\iota_{{\cal Y}^{\star}, {\cal Z}^{\star}*}
\lo 
\iota_{{\rm conv}*}\circ L^{\rm conv}_{Z/S} 
\tag{4.5.2}\label{coh:com}
\end{equation*}
for coherent ${\cal K}_{{\cal Z}^*}$-modules.
\end{coro}
\begin{proof}
The former statement of 
(\ref{linc}) immediately follows from 
(\ref{lemm:phiit}) and (\ref{lemm:jilt}).  
The latter follows from the former and (\ref{lemm:git}). 
\end{proof}

\section{Vanishing cycle sheaves in log convergent topoi}\label{sec:vflvc}
In this section we study several properties of 
the morphisms forgetting the log structures  
of log convergent topoi.  
See \cite[(2.3)]{nh2} for the analogues 
in the log crystalline case.
\par 
Let the notations be as in \S\ref{sec:logcd}. 
Let $M$ be the log structure of $Y$. 
Let $N$ be a fine sub log structure of $M$ 
on $\os{\circ}{Y}$. 
Set $Y_M:=(\os{\circ}{Y},M)(=Y)$ 
and $Y_N:=(\os{\circ}{Y},N)$. 
The inclusion $N\subset M$ induces 
a natural morphism 
\begin{equation*} 
\eps_{(\os{\circ}{Y},M,N)/S_1} 
\col Y_M \lo Y_N 
\tag{5.0.1}\label{eqn:esfl}
\end{equation*} 
of log schemes over $S_1$. 
The morphism $\eps_{(\os{\circ}{Y},M,N)/S} $ 
induces the following morphism of 
log convergent topoi: 
\begin{equation*}
\eps^{\rm conv}_{(\os{\circ}{Y},M,N)/S} \col 
(Y_M/S)_{\rm conv} 
\lo 
(Y_N/S)_{\rm conv}. 
\tag{5.0.2}\label{eqn:tsfl}
\end{equation*}
For simplicity of notations, we denote 
$\eps_{(\os{\circ}{Y},M,N)/S}$ and $\eps^{\rm conv}_{(\os{\circ}{Y},M,N)/S}$ 
simply by $\eps$ and $\eps_{\rm conv}$ for the time being. 
Let 
$(U_M,T_M,\iota_M,u_M):=
((\os{\circ}{U},M_U),(\os{\circ}{T},M_T),\iota_M,u_M)$ 
be an object of the log convergent site
${\rm Conv}(Y_M/S)$. 
Set $N_U:=\os{\circ}{u}{}^*_M(N)$ 
and $U_N:=(\os{\circ}{U},N_U)$. 
Then we easily see that $N_U \subset M_U$ since 
$u_M$ is strict. 
We have the natural morphism
$u_N \col U_N \lo Y_N$  
of log schemes over $S_1$. 
%The explicit description of 
%$\eps_{\rm conv}^*$ is as follows: 
%first, for an object $F$ of  
%$(Y_N/S)_{\rm conv}$, 
%set 
%$$\eps_{\rm conv}^{p*}(F)((U_M,T_M, \iota_M,u_M))
%=\vil_{T'_N}F((U_N,T'_N,\iota_N,u_N)),$$ 
%where $(U_N,T'_N,\iota_N,u_N)$ 
%is an object of $(Y_M/S)_{\rm conv}$ 
%fitting into the following commutative diagram 
%\begin{equation*}
%\begin{CD}
%Y_M @>{\eps^{\rm conv}_{(\os{\circ}{Y},M,N)/S}}>> Y_N \\ 
%@A{u_M}AA @AA{u_N}A  \\
%U_M @>{\eps^{\rm conv}_{
%(\os{\circ}{U},M\vert_U,N\vert_U)/S}}>> U_N\\ 
%@V{\iota_M}V{\bigcap}V @V{\bigcap}V{\iota_N}V \\ 
%T_M @>>> T'_N \\
%@VVV @VVV \\
%S @= S. 
%\end{CD} 
%\end{equation*} 
%Then $\eps_{\rm conv}^*(F)$ is the associated sheaf of 
%the presheaf $\eps_{\rm conv}^{p*}(F)$.

\begin{defi}\label{defi:forg}
(1) We call the morphism 
$\eps_{(\os{\circ}{Y},M,N)/S}$ in 
(\ref{eqn:esfl}) 
the {\it morphism of log schemes over} $S_1$ 
{\it forgetting the structure} $M\setminus N$,
and call the morphism 
$\eps^{\rm conv}_{(\os{\circ}{Y},M,N)/S}$ in 
(\ref{eqn:tsfl}) the {\it morphism of 
log convergent topoi forgetting the structure} $M\setminus N$. 
When $N$ is trivial, 
we call $\eps_{(\os{\circ}{Y},M,N)/S}$ and 
$\eps^{\rm conv}_{(\os{\circ}{Y},M,N)/S}$ the 
{\it morphism forgetting the log structure} of $Y/S$. 
\par 
(2) (See \cite[(2.1)]{nh2} for the definition 
of the fs(=fine and saturated) log structure $M(D)$ of 
a relative SNCD $D$ on a smooth scheme $X$ over $S_1$.)  
When $\os{\circ}{Y}$ is a smooth scheme $X$ over $S_1$,  
$M=M({D\cup Z})$ and $N=M(Z)$, 
where $D$ and $Z$ are transversal relative 
SNCD's on $X/S_1$, we call 
$\eps_{(\os{\circ}{Y},M,N)/S}$ 
and $\eps^{\rm conv}_{(\os{\circ}{Y},M,N)/S}$ 
the {\it morphism of ringed topoi 
forgetting the log structure along}  
$D$ and denote them by 
$\eps_{(X,D\cup Z,Z)/S}$ and 
$\eps^{\rm conv}_{(X,D\cup Z,Z)/S}$, respectively.
\end{defi}

\par
The morphism $\eps_{\rm conv}=
\eps^{\rm conv}_{(\os{\circ}{Y},M,N)/S}$ also induces a morphism  
\begin{equation*}
\eps_{\rm conv} \col 
((Y_M/S)_{\rm conv}, {\cal K}_{Y_M/S}) \lo 
((Y_N/S)_{\rm conv}, {\cal K}_{Y_N/S}) 
\tag{5.1.1}\label{eqn:kmono}
\end{equation*} 
of ringed topoi. 
Let $u^{\rm conv}_{Y_L/S}
\col 
((Y_L/S)_{\rm conv},{\cal K}_{Y_L/S}) 
\lo (Y_{\rm zar}, f^{-1}({\cal K}_S))$ $(L:=M, N)$ 
be the projection in (\ref{eqn:uwksr}). 
Then we have the following equality:  
\begin{equation*}
u^{\rm conv}_{Y_N/S}\circ \eps_{\rm conv}= 
u^{\rm conv}_{Y_M/S}. 
\tag{5.1.2}\label{eqn:nbepu}
\end{equation*}
As in \cite[(2.3.0.3)]{nh2}, 
we assume that, locally on $\os{\circ}{Y}$, 
there exists a finitely generated commutative monoid $P$ 
with unit element such that $P^{\rm gp}$ has no $p$-torsion 
and that there exists a chart $P \lo N$.  
\par
Next we define the localizations of 
$\eps_{\rm conv}$'s in (\ref{eqn:tsfl}) and 
(\ref{eqn:kmono}). 
Let 
$$T_M=(U_M, T_M,\iota_M,u_M)=
((\os{\circ}{U},M_U),(\os{\circ}{T},M_T),\iota_M,u_M)$$ 
be an exact widening of $Y_M/S$. 
Then we have an isomorphism 
$M_T/{\cal O}_T^* 
\os{\os{\iota^*_M}{\sim}}{\lo} M_U/{\cal O}_U^*$ 
on $U_{\rm zar}=T_{\rm zar}$. 
Set $N_U:=\os{\circ}{u}{}^*_M(N)$ 
and $U_N:=(\os{\circ}{U},N_U)$.  
Let $N_T$ be the inverse image of 
$N_U/{\cal O}_{U}^*$ by the morphism
$M_T  \lo M_{U}/{\cal O}_{U}^*$. 
Then, as in \cite[(2.3.1)]{nh2}, 
$N_T$ is shown to be fine. 
Set $T_N:=(\os{\circ}{T},N_T)$. 
Let $u_N \col U_N \lo Y_N$ 
(resp.~$\iota_N \col U_N \os{\subset}{\lo} T_N$)
be the induced morphism 
by $\os{\circ}{u}_M$ (resp.~$\os{\circ}{\iota}_M$).
Set $T_N:=(U_N,T_N,\iota_N,u_N)$.
Then we have a morphism 
\begin{equation*}
\eps_{\rm conv} \vert_T \col 
(Y_M/S)_{\rm conv}\vert_{T_M} 
\lo 
(Y_N/S)_{\rm conv}\vert_{T_N} 
\tag{5.1.3}
\end{equation*}
of topoi and a morphism 
\begin{equation*}
\eps_{\rm conv} \vert_T \col 
((Y_M/S)_{\rm conv}\vert_{T_M}, 
{\cal K}_{Y_M/S}\vert_{T_M}) 
\lo 
((Y_N/S)_{\rm conv}\vert_{T_N}, 
{\cal K}_{Y_N/S}\vert_{T_N}) 
\tag{5.1.4}
\end{equation*}
of ringed topoi.

\begin{lemm}[{\bf {cf.~\cite[(2.3.3)]{nh2}}}]\label{lemm:efstex}
Let the notations be as above. 
Then the functor 
$\eps_{\rm conv} \vert_{T*}$ is exact. 
\end{lemm}
\begin{proof}
Let $\phi_N \col (U'_N,T'_N,\iota'_N,u'_N)  
\lo T_N=(U_N,T_N,\iota_N,u_N)$ 
($U'_N=(\os{\circ}{U}{}',N_{U'})$,
$T'_N=(\os{\circ}{T}{}',N_{T'})$) 
be a morphism in $(Y_N/S)_{\rm conv}$. 
Let $\phi_U \col \os{\circ}{U}{}' \lo \os{\circ}{U}$ and 
$\phi \col \os{\circ}{T}{}' \lo \os{\circ}{T}$ 
be the underlying morphisms of 
(formal) schemes obtained by $\phi_N$. 
Set $M_{U'}:=\phi_U^*(M_U)$, $M_{T'}:=\phi^*(M_T)$,  $U'_M:=(\os{\circ}{U}{}',M_{U'})$ and 
$T'_M:=(\os{\circ}{T}{}',M_{T'})$. 
Let $\iota'_M \col {U}'_M \os{\sus}{\lo} T'_M$ 
and $u'_M \col U'_M \lo Y_M$ 
be the natural morphisms. 
Let 
$$\phi_M \col (U'_M,T'_M,\iota'_M,u'_M)
\lo (U_M,T_M,\iota_M,u_M)$$ 
be the natural morphism. 
Then $(U'_M,T'_M,\iota'_M,u'_M;\phi_M)$ 
is an object of $(Y_M/S)_{\rm conv} \vert_{T_M}$.  
Let $T_M\times_{T_N}T'_N$ 
be the fiber product of 
$T_M$ and $T'_N$ over $T_N$ in 
the category of fine log formal schemes. 
By the same proof as that of \cite[(2.3.3)]{nh2}, we have
\begin{equation*}
T_M\times_{T_N}T'_N=T'_M. 
\tag{5.2.1}\label{eqn:fibes}
\end{equation*} 
%Indeed, we have the following: 
%\begin{align*} 
%(\phi^*(M_T)\oplus_{\phi^*(N_T)}N_{T'})/{\cal O}_{T'}^*& = 
%\phi^*(M_T)/{\cal O}_{T'}^*
%\oplus_{\phi^*(N_T)/{\cal O}_{T'}^*}N_{T'}/{\cal O}_{T'}^* \\
%{} & =
%\phi^{-1}(M_T/{\cal O}_{T}^*)
%\oplus_{\phi^{-1}(N_T/{\cal O}_{T}^*)}N_{T'}/{\cal O}_{T'}^* \\
%{} & \simeq 
%\phi^{-1}_U(M_{U}/{\cal O}_{U}^*)
%\oplus_{\phi^{-1}_U(N_{U}/{\cal O}_{U}^*)}
%N_{U'/}{\cal O}_{U'}^*  \\ 
%{} & = M_{U'}/{\cal O}_{U'}^* 
%\simeq \phi^*(M_T)/{\cal O}_{T'}^*. 
%\end{align*} 
%Hence the natural morphism 
%$\phi^*(M_T) \lo \phi^*(M_T)\oplus_{\phi^*(N_T)}N_{T'}$ 
%is an isomorphism and we have shown the claim. 
We also have 
\begin{equation*}
U_M\times_{U_N}U'_N=U'_M. 
\tag{5.2.2}\label{eqn:fibu}
\end{equation*}
%Denote $((U,L_U),(T,L_T), \iota,u_L;\phi_L)$ $(L:=M, N)$ by 
%$(T,L_T;\phi_L)$ for simplicity of notation. 
By the formula (\ref{eqn:fibes}) and 
(\ref{eqn:fibu}), $(\eps_{\rm conv} \vert_T)^*(U'_N,
T'_N,\iota_N',u_N';\phi_N)$ 
is represented by 
$(U'_M,T'_M,\iota'_M,u'_M;\phi_M)$.  
Therefore, 
for an object $E$ in 
$(Y_M/S)_{\rm conv} \vert_{T_M}$, we have 
\begin{align*}
{} & \Gamma((U'_N,T'_N,\iota'_N,u'_N;\phi_N), 
(\eps_{\rm conv} \vert_T)_*(E)) 
\tag{5.2.3}\label{eqn:ephomep} \\ 
{} & =
{\rm Hom}_{(Y_M/S)_{\rm conv}\vert_{T_M}}
((\eps_{\rm conv} \vert_T)^*(U'_N,T'_N,\iota'_N,u'_N;\phi_N), E) \\
{} & =E(U'_M,T'_M,\iota'_M,u'_M;\phi_M).
\end{align*}
Using this formula, we see that the functor 
$\eps_{\rm conv} \vert_{T*}$ is exact. 
\end{proof}

\begin{lemm}[{\bf {cf.~\cite[(2.3.4)]{nh2}}}]\label{lemm:ymjeps}
Let the notations be as above.
Then the following diagram of 
topoi is commutative$:$
\begin{equation*}
\begin{CD}
(Y_M/S)_{\rm conv} 
\vert_{T_M} @>{j_{T_M}}>> 
(Y_M/S)_{\rm conv} \\ 
@V{\eps_{\rm conv} \vert_T}VV @VV{\eps_{\rm conv}}V  \\
(Y_N/S)_{\rm conv} \vert_{T_N} 
@>{j_{T_N}}>> 
(Y_N/S)_{\rm conv}.  \\ 
\end{CD}
\tag{5.3.1}\label{cd:ymnj}
\end{equation*}
The obvious analogue of 
{\rm (\ref{cd:ymnj})} for ringed topoi 
$((\wt{Y_L/S})_{\rm conv},{\cal K}_{Y_L/S})$ 
$(L=M,N)$ also holds.
\end{lemm}
\begin{proof}
%Let $G$ be an object of 
%$(Y_N/S)_{\rm conv}$. 
By (\ref{eqn:fibes}) and (\ref{eqn:fibu}),  
$(\eps_{\rm conv}\vert_T)^*(T_N)=T_M$.
The rest of the proof 
is the same as that of \cite[(2.3.4)]{nh2}. 
%Hence $(\eps_{\rm conv} \vert_{T})^*j_{T_N}^*(G)
%=(\eps_{\rm conv} \vert_{T})^*(G\times T_N)=
%\eps_{\rm conv}^*(G)\times T_M
%=j_{T_M}^*\eps_{\rm conv}^*(G)$. 
%Hence the former statement follows. 
%The latter statement immediately follows. 
\end{proof}

\par 
From now on, assume that $\os{\circ}{T}$ is affine. 

\begin{lemm}[{\bf {cf.~\cite[(2.3.5)]{nh2}}}]\label{lemm:epsex}
Let $E$ be an admissible sheaf of 
${\cal K}_{Y_M/S}\vert_{T_M}$-modules in  
$(Y_M/S)_{\rm conv}\vert_{T_M}$. 
Then the canonical morphism 
\begin{equation*} 
\eps_{{\rm conv}*}j_{T_M*}(E) 
\lo R\eps_{{\rm conv}*}j_{T_M*}(E) 
\tag{5.4.1}\label{eqn:ecjme} 
\end{equation*} 
is an isomorphism in the derived category 
$D^+({\cal K}_{Y_N/S})$. 
\end{lemm}
\begin{proof} 
By the proof of \cite[Proposition 2.3.3]{s2} which is 
a log version of \cite[(4.3)]{oc}, $R^qj_{T_M*}(E)=0$ $(q >0)$. 
By (\ref{eqn:ephomep}), 
$(\eps_{\rm conv} \vert_T)_*(E)$ is an admissible sheaf of 
${\cal K}_{Y_N/S}\vert_{T_N}$-modules. 
Hence, by the proof of [loc.~cit.] again, 
$R^qj_{T_N*}((\eps_{\rm conv} \vert_T)_*(E))=0$ $(q >0)$.
Therefore we obtain the following equalities: 
\begin{align*}
\eps_{{\rm conv}*}j_{T_M*}(E) 
& \os{(\ref{lemm:ymjeps})}{=}~j_{T_N*}(\eps_{\rm conv} \vert_T)_*(E) 
=Rj_{T_N*}(\eps_{\rm conv} \vert_T)_*(E)~\os{(\ref{lemm:efstex})}{=}~
R(j_{T_N}\eps_{\rm conv} \vert_{T})_*(E) \\ 
{} & ~\os{(\ref{lemm:ymjeps})}{=}~R(\eps_{\rm conv} j_{T_M})_*(E)  
= R\eps_{{\rm conv}*}Rj_{T_M*}
(E)  =R\eps_{{\rm conv}*}j_{T_M*}(E).  
\end{align*}
\end{proof}

\begin{coro}[{\bf {cf.~\cite[(2.3.6)]{nh2}}}]\label{coro:epnr}
Let the notations be as above. 
Let $\iota_{\cal M} \col Y_M \os{\subset}{\lo} {\cal Y}_{\cal M}:=(\os{\circ}{\cal Y},{\cal M})$ 
be an immersion into a log noetherian $p$-adic formal scheme 
which is topologically of finite type over $S$. 
Assume that $\os{\circ}{\cal Y}$ is affine. 
Let ${\mathfrak T}_{\cal M}$ 
be a prewidening 
$(Y_M,{\cal Y}_{\cal M},\iota_{\cal M},{\rm id}_{Y_M})$.  
Let ${\cal E}$ be a coherent crystal of 
${\cal K}_{\os{\to}{\mathfrak T}_{\cal M}}$-modules.
Let $L^{\rm UE}_{Y_M/S}({\cal E})$ be 
the log convergent linearization of 
${\cal E}$ with respect to $\iota_{\cal M}$. 
Then the canonical morphism 
\begin{equation*}
\eps_{{\rm conv}*}L^{\rm UE}_{Y_M/S}
({\cal E}) 
\lo R\eps_{{\rm conv}*}
L^{\rm UE}_{Y_M/S}({\cal E}) 
\tag{5.5.1}\label{eqn:epsre}
\end{equation*}
is an isomorphism in $D^+({\cal K}_{Y_N/S})$. 
\end{coro}
\begin{proof}
(\ref{coro:epnr}) immediately follows from 
(\ref{lemm:epsex}) and the definition of 
$L^{\rm UE}_{Y_M/S}({\cal E})$ 
((\ref{eqn:luys})).
\end{proof}

\begin{lemm}[{\bf {cf.~\cite[(2.3.7)]{nh2}}}]\label{lemm:itmne}
Let the notations and the assumption be as in {\rm (\ref{coro:epnr})}. 
Let ${\cal N}$ be a fine sub log structure of ${\cal M}$ 
and set ${\cal Y}_{\cal N}:=(\os{\circ}{\cal Y},{\cal N})$. 
%Let ${\cal N}$ be a fine sub log structure of ${\cal M}$ 
%such that ${\cal Y}_{\cal N}:=(\os{\circ}{\cal Y},{\cal N})$
%is also formally log smooth over $S$.
Denote ${\cal O}_{\os{\circ}{\cal Y}}$ simply by ${\cal O}_{\cal Y}$. 
Let
\begin{equation*}
\begin{CD}
Y_M @>{\iota_{\cal M}}>> {\cal Y}_{\cal M}  \\ 
@V{\eps_{(\os{\circ}{Y},M,N)/S}}VV 
@VV{\eps_{(\os{\circ}{\cal Y},{\cal M}, {\cal N})/S}}V  \\
Y_N @>{\iota_{\cal N}}>> {\cal Y}_{\cal N}  \\ 
\end{CD}
\tag{5.6.1}\label{cd:cchgmn}
\end{equation*}
be a commutative diagram whose horizontal morphisms
are immersions. 
For $L=M$ or $N$ and ${\cal L}={\cal M}$ or ${\cal N}$, 
set ${\mathfrak T}_{\cal L}:=
((\os{\circ}{Y},L),(\os{\circ}{\cal Y},{\cal L}),
(\os{\circ}{Y},L)\os{\sus}{\lo} (\os{\circ}{\cal Y},{\cal L}), 
{\rm id}_{(\os{\circ}{Y},L)})$.  
Let $\{{\mathfrak T}_{{\cal L},n}\}_{n=1}^{\infty}$ be 
the system of the universal enlargements 
of $\iota_{\cal L}$ 
with the following natural commutative diagram 
for $n\in {\mab Z}_{\geq 1}$ 
$({\star}={\rm ex}$ or nothing$):$
\begin{equation*}
\begin{CD}
{\mathfrak T}_{{\cal M},n}
@>{g^{\star}_{{\cal M},n}}>> {\cal Y}^{\star}_{\cal M} \\ 
@V{h_n}VV @VVV  \\
{\mathfrak T}_{{\cal N},n} 
@>{g^{\star}_{{\cal N},n}}>> {\cal Y}^{\star}_{\cal N}.  \\ 
\end{CD}
\tag{5.6.2}\label{cd:zmnhg}
\end{equation*}
Assume that the underlying 
morphism $\os{\circ}{h}_n\col \os{\circ}{\mathfrak T}_{{\cal M},n}\lo 
\os{\circ}{\mathfrak T}_{{\cal N},n}$ of $h_n$ 
is the identity of $\os{\circ}{\mathfrak T}_{{\cal M},n}$ 
for any $n\in {\mab Z}_{\geq 1}$. 
Let $L^{\rm UE}_{Y_L/S}$ 
$($resp.~$L^{\rm conv}_{Y_L/S})$ 
be the log convergent linearization functor  
for coherent crystal of 
${\cal K}_{\os{\to}{\mathfrak T}_{\cal L}}$-modules 
$($resp.~coherent ${\cal K}_{{\cal Y}^{\star}_{\cal L}}$-modules$)$.    
Then there exist  natural isomorphisms 
\begin{equation*}
L^{\rm UE}_{Y_N/S} \os{\sim}{\lo} 
\eps_{\rm conv*}
L^{\rm UE}_{Y_M/S} 
\tag{5.6.3}\label{eqn:lnelmn}
\end{equation*} 
and 
\begin{equation*}
L^{\rm conv}_{Y_N/S} \os{\sim}{\lo} 
\eps_{\rm conv*}
L^{\rm conv}_{Y_M/S}
\tag{5.6.4}\label{eqn:lgeslg}
\end{equation*} 
of functors. Moreover, the functors 
{\rm (\ref{eqn:lnelmn})} and {\rm (\ref{eqn:lgeslg})} 
are compatible with respect to the integrable connection 
$($cf.~{\rm (\ref{prop:lcl})}$)$. 
That is, the following diagrams are commutative$:$ 
\begin{equation*}
\begin{CD}
L^{\rm UE}_{Y_N/S}({\cal E}\otimes_{{\cal O}_{{\cal Y}}}
\Om^i_{{\cal Y}_{\cal M}/S}) 
@>{\nabla^i}>> L^{\rm UE}_{Y_N/S}({\cal E}\otimes_{{\cal O}_{{\cal Y}}}
\Om^{i+1}_{{\cal Y}_{\cal M}/S})  \\ 
@V{\simeq}VV @VV{\simeq}V  \\
\eps_{\rm conv*}
L^{\rm UE}_{Y_M/S}({\cal E}\otimes_{{\cal O}_{{\cal Y}}}
\Om^i_{{\cal Y}_{\cal M}/S}) 
@>{\eps_{\rm conv*}(\nabla^i)}>> \eps_{\rm conv*}
L^{\rm UE}_{Y_M/S}({\cal E}\otimes_{{\cal O}_{{\cal Y}}}
\Om^{i+1}_{{\cal Y}_{\cal M}/S}) \\ 
\end{CD}
\tag{5.6.5}\label{cd:lnhg}
\end{equation*}
for a coherent crystal 
${\cal E}:=\{{\cal E}_n\}_{n=1}^{\infty}$ of 
${\cal K}_{\os{\to}{\mathfrak T}_{\cal M}({\cal Y})}$-modules 
and 
\begin{equation*}
\begin{CD}
L^{\rm conv}_{Y_N/S}({\cal E}\otimes_{{\cal O}_{{\cal Y}}}
\Om^i_{{\cal Y}_{\cal M}/S}) 
@>{\nabla^i}>> L^{\rm conv}_{Y_N/S}({\cal E}\otimes_{{\cal O}_{{\cal Y}}}
\Om^{i+1}_{{\cal Y}_{\cal M}/S})  \\ 
@V{\simeq}VV @VV{\simeq}V  \\
\eps_{\rm conv*}
L^{\rm conv}_{Y_M/S}({\cal E}\otimes_{{\cal O}_{{\cal Y}}}
\Om^i_{{\cal Y}_{\cal M}/S}) 
@>{\eps_{\rm conv*}(\nabla^i)}>> \eps_{\rm conv*}
L^{\rm conv}_{Y_M/S}({\cal E}\otimes_{{\cal O}_{{\cal Y}}}
\Om^{i+1}_{{\cal Y}_{\cal M}/S}) \\ 
\end{CD}
\tag{5.6.6}\label{cd:lnhlg}
\end{equation*}
for a coherent ${\cal K}_{\os{\circ}{\cal Y}}$-module
${\cal E}$.
\end{lemm}
\begin{proof}
Let $\eps_{\rm conv} \vert_{{\mathfrak T}_n} 
\col  
(Y_M/S)_{\rm conv}\vert_{{\mathfrak T}_{{\cal M},n}} 
\lo 
(Y_N/S)_{\rm conv}\vert_{{\mathfrak T}_{{\cal N},n}}$ 
$(n\in {\mab Z}_{\geq 1})$ 
be the localized morphism of topoi 
forgetting the log structure.  
Using the formula (\ref{eqn:ephomep}), 
we can immediately check that 
$(\eps_{\rm conv} \vert_{{\mathfrak T}_n})_*
\varphi_{{\mathfrak T}_{{\cal M},n}}^{*}
=\varphi_{{\mathfrak T}_{{\cal N},n}}^{*}$. 
%Let ${\cal Z}_{\cal M}$ (resp.~${\cal Z}_{\cal N}$) 
%be ${\cal Y}_{\cal M}$ or ${\cal Y}^{\rm ex}_{\cal M}$ 
%(resp.~${\cal Y}_{\cal N}$ or ${\cal Y}^{\rm ex}_{\cal N}$). 
%Let $h_{{\cal M},n}$ be $g_{{\cal M},n}$ or 
%$g^{\rm ex}_{{\cal M},n}$ and 
%let $h_{{\cal N},n}$ be $g_{{\cal N},n}$ or 
%$g^{\rm ex}_{{\cal N},n}$. 
By (\ref{lemm:ymjeps}) 
we obtain the following commutative diagram: 
\begin{equation*} 
\begin{CD} 
\{\text{coherent}
{\cal K}_{{\cal Y}^{\star}_{\cal M}}\text{-modules}\}  
@>{\{\os{\circ}{g}{}^*_{{\cal M},n}\}_{n=1}^{\infty}}>>  
\{\text{coherent crystals of }
{\cal K}_{\os{\to}{\mathfrak T}_{\cal M}}\text{-modules}\}  \\ 
@\vert @\vert \\
\{\text{coherent}
{\cal K}_{{\cal Y}^{\star}_{\cal N}}\text{-modules}\} 
@>{\{\os{\circ}{g}{}^*_{{\cal N},n}\}_{n=1}^{\infty}}>> 
\{\text{coherent crystals of }
{\cal K}_{\os{\to}{\mathfrak T}_{\cal N}}\text{-modules}\}   
\end{CD}
\tag{5.6.7}\label{cd:longcom}
\end{equation*} 
\begin{equation*} 
\begin{CD} 
@>{\varphi^*_{\os{\to}{\mathfrak T}_{\cal M}}}>> 
\{\text{crystals of }
{\cal K}_{Y_M/S}\vert_{{\mathfrak T}_{{\cal M}}}\text{-modules}\} \\
@.
@V{(\eps_{\rm conv} \vert_{{\mathfrak T}})_*}VV \\
@>{\varphi^*_{\os{\to}{\mathfrak T}_{\cal N}}}>> 
\{\text{crystals of }
{\cal K}_{Y_N/S}\vert_{{\mathfrak T}_{{\cal N}}}
\text{-modules}\}   
\end{CD} 
\end{equation*}
\begin{equation*} 
\begin{CD}  
@>{j_{{\mathfrak T}_{M*}}}>> 
\{\text{crystals of }
{\cal K}_{Y_M/S}\text{-modules}\}   \\ 
@. @V{\eps_{\rm conv*}}VV  \\
@>{j_{{\mathfrak T}_{N*}}}>> 
\{\text{crystals of }
{\cal K}_{Y_N/S}\text{-modules}\}.
\end{CD}
\end{equation*} 
Hence we have the isomorphisms 
(\ref{eqn:lnelmn}) and (\ref{eqn:lgeslg}). 
\par 
The rest we have to prove is the compatibility. 
It suffices to prove that 
\begin{equation*}
\eps^{{\rm conv}*}_{(\os{\circ}{Y},M,N)/S}
L^{\rm UE}_{Y_N/S} \lo L^{\rm UE}_{Y_M/S} 
\end{equation*} 
is compatible with respect to integrable connections. 
Let $T_M=(U_M,T_M,\iota_M,u_M)$ be an object of $(Y_M/S)_{\rm conv}$ 
and let $T_N=(U_N,T_N,\iota_N,u_N)$ be an exact widening constructed after (\ref{defi:forg}). 
Then $T_N$ is an object of $(Y_N/S)_{\rm conv}$.  
Let 
$$h_{{\mathfrak T},n}\col {\mathfrak T}_{U_M,n}
(T_M\times_S{\cal Y}_{\cal M})\lo 
{\mathfrak T}_{U_N,n}
(T_N\times_S{\cal Y}_{\cal N})$$ 
and 
$$h_{{\mathfrak T}(1),n}\col {\mathfrak T}_{U_{M},n}
(T_M\times_S{\cal Y}_{\cal M}(1))\lo 
{\mathfrak T}_{U_{N},n}
(T_N\times_S{\cal Y}_{\cal N}(1))$$ 
be the natural morphisms. 
Then we have the following (commutative) diagrams 
of the inductive systems of enlargements of $Y/S$:  
\begin{equation*}
\begin{CD}
\{{\mathfrak T}_{U_{M,j}(1),n}(T_M\times_S{\cal Y}_{M,m}(1)\}_{n=1}^{\infty} 
@>{\{p_{j,n}(1)\}_{n=1}^{\infty}}>>
\{{\mathfrak T}_{U_M,n}(T_M\times_S{\cal Y}_{\cal M})\}_{n=1}^{\infty} 
@>{\{p_n\}_{n=1}^{\infty}}>> 
\{{\mathfrak T}_{{\cal M},n}\}_{n=1}^{\infty}\\ 
@. @V{\{p'_n\}_{n=1}^{\infty}}VV \\
@. T_M,@. @. 
\end{CD} 
\end{equation*} 
\begin{equation*}
\begin{CD}
\{{\mathfrak T}_{U_{N,j}(1),n}(T_N\times_S{\cal Y}_{N,m}(1))\}_{n=1}^{\infty} 
@>{\{q_{j,n}(1)\}_{n=1}^{\infty}}>>
\{{\mathfrak T}_{U_N,n}
(T_N\times_S{\cal Y}_{\cal N})\}_{n=1}^{\infty} 
@>{\{q_n\}_{n=1}^{\infty}}>> 
\{{\mathfrak T}_{{\cal N},n}\}_{n=1}^{\infty}\\ 
@. @V{\{q'_n\}_{n=1}^{\infty}}VV \\
@. T_N,@. @. 
\end{CD} 
\end{equation*} 
\begin{equation*}
\begin{CD}
\{{\mathfrak T}_{U_M,n}
(T_M\times_S{\cal Y}_{\cal M})\}_{n=1}^{\infty} 
@>{\{p_n\}_{n=1}^{\infty}}>> 
\{{\mathfrak T}_{{\cal M},n}\}_{n=1}^{\infty}\\ 
@V{\{h_{{\mathfrak T},n}\}_{n=1}^{\infty}}VV @VV\{h_n\}_{n=1}^{\infty}V\\
\{{\mathfrak T}_{U_N,n}
(T_N\times_S{\cal Y}_{\cal N})\}_{n=1}^{\infty} 
@>{\{q_n\}_{n=1}^{\infty}}>> 
\{{\mathfrak T}_{{\cal N},n}\}_{n=1}^{\infty}
\end{CD} 
\end{equation*} 
and 
\begin{equation*}
{\footnotesize{\begin{CD}
\{{\mathfrak T}_{U_{M}(1),n}(T_M\times_S{\cal Y}_{{\cal M},m}(1))\}_{n=1}^{\infty} 
=\{{\mathfrak T}_{U_{M,j}(1),n}(T_M\times_S{\cal Y}_{{\cal M},m}(1))\}_{n=1}^{\infty} 
@>{\{p_{j,n}(1)\}_{n=1}^{\infty}}>>
\{{\mathfrak T}_{U_M,n}
(T_M\times_S{\cal Y}_{\cal M})\}_{n=1}^{\infty} \\ 
@V{\{h_{{\mathfrak T}(1),n}\}_{n=1}^{\infty}}VV @VV{\{h_{{\mathfrak T},n}\}_{n=1}^{\infty}}V \\
\{{\mathfrak T}_{U_{N},n}(T_N\times_S{\cal Y}_m(1))\}_{n=1}^{\infty} 
=
\{{\mathfrak T}_{U_{N,j}(1),n}(T_N\times_S{\cal Y}_{{\cal N},m}(1))\}_{n=1}^{\infty} 
@>{\{q_{j,n}(1)\}_{n=1}^{\infty}}>>
\{{\mathfrak T}_{U_N,n}
(T_N\times_S{\cal Y}_{\cal N})\}_{n=1}^{\infty}.  
\end{CD}}} 
\end{equation*} 
Here $q_{j,n}(1)$ $(j=1,2)$ is the morphism $p_{j,n}(1)$ in the proof of (\ref{prop:lcl})
for $T_N\times_S{\cal Y}_m(1)$ and $T_N\times_S{\cal Y}_m$. 
Let ${\cal E}=\{{\cal E}_n\}_{n=1}^{\infty}$ be a coherent crystal of 
${\cal K}_{\os{\to}{\mathfrak T}_{{\cal N}}}$-modules. 
Set $h^*({\cal E})=\{h_n^*({\cal E}_n)\}_{n=1}^{\infty}$. 
Then, by (\ref{eqn:lueyset}), 
we have the following commutative diagram: 
\begin{equation*} 
\begin{CD} 
(\eps^{{\rm conv}*}_{(\os{\circ}{Y},M,N)/S}
L^{\rm UE}_{Y_N/S}({\cal E}))_{T_M} @>>> 
(L^{\rm UE}_{Y_M/S}(h^*({\cal E})))_{T_M} \\ 
@| @| \\ 
\vpl_nq'_{n*}q_n^*({\cal E}_n) @>>> 
\vpl_np'_{n*}p_n^*h^*_n({\cal E}_n). 
\end{CD}
\end{equation*} 
We also have the following commutative diagram: 
\begin{equation*} 
\begin{CD}
\vpl_n{\cal E}_{{\mathfrak T}_{U_M,n}(T_M\times_S{\cal Y}_{{\cal M},m}(1))}
@= 
\\
@AAA  \\
\vpl_n
({\cal K}_{{\mathfrak T}_{U_{M}(1),n}(T_M\times_S{\cal Y}_{{\cal M},m}(1))}
\otimes_{{\cal K}_{{\mathfrak T}_{U_{N}(1),n}(T_N\times_S{\cal Y}_{{\cal N},m}(1))}}
{\cal E}_{{\mathfrak T}_{U_N,n}(T_N\times_S{\cal Y}_{{\cal N},m}(1))})
@= 
\end{CD}
\end{equation*} 
\begin{equation*} 
\begin{CD}
\vpl_n
{\cal E}_{{\mathfrak T}_{U_{M,j}(1),n}(T_M\times_S{\cal Y}_{{\cal M},m}(1))} \\
@AAA \\
\vpl_n
({\cal K}_{{\mathfrak T}_{U_{M,j}(1),n}(T_M\times_S{\cal Y}_{{\cal M},m}(1))}
\otimes_{{\cal K}_{{\mathfrak T}_{U_{N,j}(1),n}(T_N\times_S{\cal Y}_{{\cal N},m}(1))}}
{\cal E}_{{\mathfrak T}_{U_{N,j}(1),n}(T_N\times_S{\cal Y}_{{\cal N},m}(1))}).  
\end{CD}
\end{equation*} 
Now the desired compatibility 
follows from the construction $\nabla^0$ 
in the proof of (\ref{prop:lcl}).  
\end{proof}

\begin{defi}\label{defi:forva}
For a coherent crystal $E$ of ${\cal K}_{Y_M/S}$-modules, 
we call $R\eps^{\rm conv}_{(\os{\circ}{Y},M,N)/S*}(E)$ 
the {\it vanishing cycle sheaf} of $E$ 
{\it along} $M\setminus N$. We call 
$R\eps^{\rm conv}_{(\os{\circ}{Y},M,N)/S*}
({\cal K}_{Y_M/S})$ the 
{\it vanishing cycle sheaf} of $Y_M/S$ 
{\it along} $M\setminus N$.
If $N$ is trivial, we omit the phrase  
``along $M\setminus N$''. 
\end{defi}

\par 
The following theorem is a main result in this section:

\begin{theo}[{\bf Poincar\'{e} lemma of a vanishing cycle sheaf 
(cf.~\cite[(2.3.10)]{nh2})}]
\label{theo:cpvcs} 
%Let the notations be as in {\rm (\ref{coro:epnr})}. 
Let the notations and the assumption be as in {\rm (\ref{lemm:itmne})}. 
Let $E$ be a coherent crystal of 
${\cal K}_{Y_N/S}$-modules and let 
$({\cal E}, \nabla)$ be the ${\cal K}_{\os{\to}{\mathfrak T}_{\cal M}}$-module 
with integrable connection {\rm (\ref{eqn:nbqe})} 
corresponding to 
$\eps^{{\rm conv}*}_{(\os{\circ}{Y},M,N)/S}(E):$ 
%$($cf.~{\rm \cite[(6.3)]{oc}}$):$ 
$\nabla \col {\cal E} \lo {\cal E}\otimes_{{\cal O}_{\cal Y}}
\Om^1_{{\cal Y}_{\cal M}/S}$. 
Assume that we are given 
the commutative diagrams {\rm (\ref{cd:cchgmn})}  
and {\rm (\ref{cd:zmnhg})} 
and that $\os{\circ}{h}_n$ is the identity for 
any $n\in {\mab Z}_{\geq 1}$. 
Assume that ${\cal Y}_{\cal M}$ and ${\cal Y}_{\cal N}$ are formally log smooth over $S$. 
Then there exists a canonical isomorphism 
\begin{equation*}
R\eps^{\rm conv}_{(\os{\circ}{Y},M,N)/S*}
(\eps^{{\rm conv}*}_{(\os{\circ}{Y},M,N)/S}(E)) \os{\sim}{\lo} 
L^{\rm UE}_{Y_N/S}({\cal E}\otimes_{{\cal O}_{\cal Y}}
\Om^{\bul}_{{\cal Y}_{\cal M}/S})
\tag{5.8.1}\label{eqn:ceslvc}
\end{equation*} 
in $D^+({\cal K}_{Y_N/S})$. 
\end{theo} 
\begin{proof} 
By (\ref{eqn:epdlym}) we have a canonical isomorphism 
$\eps^{{\rm conv}*}_{(\os{\circ}{Y},M,N)/S}(E) \os{\sim}{\lo} 
L^{\rm UE}_{Y_M/S}({\cal E}\otimes_{{\cal O}_{\cal Y}}
\Om^{\bul}_{{\cal Y}_{\cal M}/S})$.   
Applying $R\eps^{\rm conv}_{(\os{\circ}{Y},M,N)/S*}$ to 
both hand sides of the above and  
using (\ref{coro:epnr}) and (\ref{lemm:itmne}), 
we obtain the following: 
\begin{align*}
R\eps^{\rm conv}_{(\os{\circ}{Y},M,N)/S*}
(\eps^{{\rm conv}*}_{(\os{\circ}{Y},M,N)/S}(E)) 
& \os{\sim}{\lo}  
R\eps^{\rm conv}_{(\os{\circ}{Y},M,N)/S*}
L^{\rm UE}_{Y_M/S}({\cal E}
\otimes_{{\cal O}_{\cal Y}}
\Om^{\bul}_{{\cal Y}_{\cal M}/S})\\
{} & \os{\sim}{\longleftarrow}
\eps^{\rm conv}_{(\os{\circ}{Y},M,N)/S*}
L^{\rm UE}_{Y_M/S}
({\cal E}\otimes_{{\cal O}_{\cal Y}}
\Om^{\bul}_{{\cal Y}_{\cal M}/S})\\ 
{}&= 
L^{\rm UE}_{Y_N/S}({\cal E}
\otimes_{{\cal O}_{\cal Y}}\Om^{\bul}_{{\cal Y}_{\cal M}/S}).
\end{align*}
\end{proof}

\section{Log convergent linearization functors 
of smooth schemes with relative SNCD's}\label{sec:lcs}
In this section we show several properties of 
log convergent linearization functors 
of a smooth scheme with a relative SNCD. 
We recall and generalize 
several notions in \cite{nh2} and \cite{nh3}.  
\par 
Let $S$  be a $p$-adic formal ${\cal V}$-scheme 
in the sense of \cite{of}.  
(Recall that we always assume that $\os{\circ}{S}$ is flat over ${\cal V}$.) 
Set $S_1:=\ul{\rm Spec}_S({\cal O}_S/\pi{\cal O}_S)$. 
Let $f\col X\lo S_1$ be a smooth scheme 
with a relative SNCD $D$ on $X$ 
over $S_1$. 
Let $Z$ be a relative SNCD on 
$X$ over $S_1$ which 
intersects $D$ transversally over $S_1$. 
Assume that there exist immersions 
\begin{equation*} 
(X,D) \os{\sus}{\lo} `{\cal R}
\quad {\rm and} \quad 
(X,Z) \os{\sus}{\lo} {\cal R} 
\tag{6.0.1}\label{eqn:mdml} 
\end{equation*}  
into formally log smooth log noetherian $p$-adic formal schemes over $S$ 
such that $\os{\circ}{`{\cal R}}=\os{\circ}{\cal R}$. 
Set ${\cal P}:=(\os{\circ}{\cal R},
M_{`{\cal R}}\oplus_{{\cal O}^*_{\cal R}}M_{{\cal R}})$. 
Assume that $\os{\circ}{\cal R}$ is topologically of finite type 
over $S$ and that ${\cal P}$ is log smooth over $S$. 
Then we have a natural immersion 
$(X,D\cup Z) \os{\sus}{\lo} {\cal P}$ and   
the following commutative diagram:  
\begin{equation*} 
\begin{CD} 
(X,D\cup Z) @>{\subset}>> {\cal P} \\ 
@V{\eps_{(X,D\cup Z,Z)/S_1}}VV @VVV \\ 
(X,Z) @>{\subset}>> {\cal R}.  
\end{CD} 
\tag{6.0.2}\label{eqn:pvq} 
\end{equation*} 
%For the convenience of notation, 
%we denote $`{\cal R}$ by $(\os{\circ}{\cal P},{\cal M})$. 
Let $(X,D\cup Z)\os{\sus}{\lo} {\cal P}^{\rm ex}$ 
(resp.~$(X,Z)\os{\sus}{\lo} {\cal R}^{\rm ex}$)
be the exactification of the upper immersion 
(resp.~the lower immersion) in (\ref{eqn:pvq}). 
Let $(X,D)\os{\sus}{\lo} (`{\cal R})^{\rm ex}$ 
be the exactification of the first immersion in 
(\ref{eqn:mdml}). 
Let ${\cal M}^{\rm ex}$ (resp.~${\cal N}^{\rm ex}$) be 
the inverse image of the log structure of 
$(`{\cal R})^{\rm ex}$ 
(resp.~${\cal R}^{\rm ex}$) 
by the morphism 
${\cal P}^{\rm ex} \lo (`{\cal R})^{\rm ex}$ 
(resp.~${\cal P}^{\rm ex} \lo {\cal R}^{\rm ex}$). 
Then we have the following equality: 
\begin{equation*} 
M_{{\cal P}^{\rm ex}}={\cal M}^{\rm ex}
\oplus_{{\cal O}^*_{{\cal P}^{{\rm ex}}}}{\cal N}^{\rm ex}.  
\tag{6.0.3}\label{eqn:mpn} 
\end{equation*} 
Endow $\os{\circ}{\cal P}{}^{\rm ex}$ 
with the inverse image of the log structure of ${\cal R}^{\rm ex}$ 
and let ${\cal Q}^{\rm ex}$ be the resulting log formal scheme. 
Then we have the natural exact closed immersion 
$(X,Z)\os{\sus}{\lo}{\cal Q}^{\rm ex}$ 
and we see that 
${\cal Q}^{\rm ex}=(\os{\circ}{\cal P}{}^{\rm ex},{\cal N}^{\rm ex})$. 
We have the following commutative diagram 
\begin{equation*} 
\begin{CD} 
(X,D\cup Z) @>{\subset}>>{\cal P}^{\rm ex}\\ 
@V{\eps_{(X,D\cup Z,Z)/S}}VV @VVV \\ 
(X,Z) @>{\subset}>>{\cal Q}^{\rm ex}\\  
@| @VVV \\ 
(X,Z) @>{\subset}>>{\cal R}^{\rm ex}.  
\end{CD} 
\tag{6.0.4}\label{eqn:pevq} 
\end{equation*} 
If the former immersion (resp.~the latter immersion) 
in (\ref{eqn:mdml}) 
have a chart $P_1\lo Q_1$ (resp.~$P_2\lo Q_2$), 
we denote by $P^{\rm ex}_1$ (resp.~$P^{\rm ex}_2$) 
the inverse image of $Q_1$ (resp.~$Q_2$)
by the morphism $P^{\rm gp}_1\lo Q^{\rm gp}_1$ 
(resp.~$P^{\rm gp}_2\lo Q^{\rm gp}_2$). 
We set 
${\cal P}^{{\rm loc,ex}}:=
{\cal P}\times_{{\rm Spf}^{\log}
({\mab Z}_p\{P_1\oplus P_2\})}
{\rm Spf}^{\log}
({\mab Z}_p\{P^{\rm ex}_1\oplus P^{\rm ex}_2\})$ 
and 
${\cal R}^{{\rm loc,ex}}:=
{\cal R}\times_{{\rm Spf}^{\log}({\mab Z}_p\{P_2\})}
{\rm Spf}^{\log}({\mab Z}_p\{P^{\rm ex}_2\})$, 
respectively. 
Let ${\cal Q}^{\rm loc,ex}$ be 
the fine log formal ${\cal V}$-scheme 
whose underlying formal scheme is 
$\os{\circ}{\cal P}{}^{{\rm loc,ex}}$ and 
whose log structure is the inverse image of 
the log structure of ${\cal R}^{{\rm loc,ex}}$. 
Then ${\cal Q}^{\rm ex}$ (resp.~${\cal P}^{\rm ex}$) 
is the formal completion of 
${\cal Q}^{{\rm loc,ex}}$ (resp.~${\cal P}^{{\rm loc,ex}}$) along $X$. 
Note that ${\cal Q}^{\rm ex}$ and ${\cal P}^{\rm ex}$ are formally log smooth over $S$. 
\par 
%Let us recall the {\bf Notation} (3). 
For a nonnegative integer $i$ and an integer $k$, 
set 
\begin{equation*} 
P^{{\cal P}^{\rm ex}/{\cal Q}^{\rm ex}}_k
\Om^i_{{\cal P}^{\rm ex}/S} =
\begin{cases} 
0 & (k<0), \\
{\rm Im}(\Om^k_{{\cal P}^{\rm ex}/S}
{\otimes}_{{\cal O}_{{\cal P}^{\rm ex}}}
\Om^{i-k}_{{\cal Q}^{\rm ex}/S} \lo 
\Om^i_{{\cal P}^{\rm ex}/S}) & (0\leq k\leq i), \\
\Om^i_{{\cal P}^{\rm ex}/S} & (k > i).
\end{cases}
\tag{6.0.5}\label{eqn:lexfpw} 
\end{equation*}
Then we have a filtration 
$P^{{\cal P}^{\rm ex}/{\cal Q}^{\rm ex}}
:=\{P^{{\cal P}^{\rm ex}
/{\cal Q}^{\rm ex}}_k\}_{k\in {\mab Z}}$ 
on $\Om^i_{{\cal P}^{\rm ex}/S}$.

\begin{lemm}\label{lemm:exic}
$(1)$ The following formulas hold$:$
\begin{equation*} 
\Om^i_{{\cal P}/S}\otimes_{{\cal O}_{\cal P}}
{\cal O}_{{\cal P}^{\rm ex}}
=\Om^i_{{\cal P}^{\rm ex}/S}, \quad    
\Om^i_{{\cal R}/S}\otimes_{{\cal O}_{\cal P}}
{\cal O}_{{\cal P}^{\rm ex}}
=\Om^i_{{\cal R}^{\rm ex}/S} 
\quad (i\in {\mab N}).
\tag{6.1.1}\label{eqn:cext} 
\end{equation*} 
\par 
$(2)$ Assume that 
the two horizontal immersions in {\rm (\ref{eqn:pvq})}  
are exact.  
Set  
\begin{equation*} 
P^{{\cal P}/{\cal R}}_k\Om^i_{{\cal P}/S} = 
\begin{cases} 
0 & (k<0), \\
{\rm Im}(\Om^k_{{\cal P}/S}
{\otimes}_{{\cal O}_{\cal P}}
\Om^{i-k}_{{\cal R}/S} \lo 
\Om^i_{{\cal P}/S}) & (0\leq k\leq i), \\
\Om^i_{{\cal P}/S} & (k > i). 
\end{cases}
\tag{6.1.2}\label{eqn:ldfpw} 
\end{equation*} 
Then 
\begin{equation*} 
P^{{\cal P}^{\rm ex}/{\cal Q}^{\rm ex}}
=P^{{\cal P}/{\cal R}}\otimes_{{\cal O}_{\cal P}}
{\cal O}_{{\cal P}^{\rm ex}}. 
\tag{6.1.3}\label{eqn:tenexsfil}
\end{equation*} 
\end{lemm}
\begin{proof}
(1): To prove the first formula in (\ref{eqn:cext}), 
we have only to prove 
that $\Om^1_{{\cal P}/S}\otimes_{{\cal O}_{\cal P}}
{\cal O}_{{\cal P}^{\rm ex}}
=\Om^1_{{\cal P}^{\rm ex}/S}$. 
Because we have a natural morphism 
$\Om^1_{{\cal P}/S}\otimes_{{\cal O}_{\cal P}}
{\cal O}_{{\cal P}^{\rm ex}}
\lo \Om^1_{{\cal P}^{\rm ex}/S}$,  
the question is local. 
We may assume that the immersions 
$(X,D)\os{\sus}{\lo} `{\cal R}$ 
and 
$(X,Z)\os{\sus}{\lo} {\cal R}$ 
have charts $P_1\lo Q_1$ and $P_2\lo Q_2$, respectively.  
Let the notations be after (\ref{eqn:pevq}).  
Because the natural morphism 
${\cal P}^{\rm loc,ex}\lo {\cal P}$ is formally log \'{e}tale, 
we have the following by \cite[(3.12)]{klog1}: 
\begin{equation*} 
\Om^i_{{\cal P}^{\rm loc,ex}/S}=
\Om^i_{{\cal P}/S}\otimes_{{\cal O}_{\cal P}}
{\cal O}_{{\cal P}^{\rm loc,ex}}.
\tag{6.1.4}\label{eqn:led}
\end{equation*}  
Furthermore, we may assume that 
there exists a lift $({\cal X},{\cal D}\cup {\cal Z})$ of 
$(X,D\cup Z)$ over $S$ such that 
the exact closed immersion 
$(X,D\cup Z)\os{\sus}{\lo}{\cal P}^{\rm loc,ex}$ 
factors through 
$(X,D\cup Z) \os{\sus}{\lo} ({\cal X},{\cal D}\cup {\cal Z})$. 
By the proof of \cite[(2.1.4)]{nh2}, 
we may assume that 
${\cal P}^{\rm loc,ex}=({\cal X},{\cal D}\cup {\cal Z})
\wh{\times}_S\wh{{\mab A}}{}^c$ for some $c\in {\mab Z}_{\geq 1}$ 
\'{e}tale locally.
In this case we obtain the following equality:  
$\Om^i_{{\cal P}^{\rm loc,ex}/S}
\otimes_{{\cal O}_{{\cal P}^{\rm loc,ex}}}
{\cal O}_{{\cal P}^{\rm ex}}
=\Om^i_{{\cal P}^{\rm ex}/S}$. 
Hence we obtain the first formula in 
(\ref{eqn:cext}). 
Similarly we obtain the second formula 
in (\ref{eqn:cext}) by the argument above. 
%using the facts that the morphism 
%${\cal Q}^{\rm loc,ex} \lo {\cal Q}$ is log etale. 
%and that the morphism 
%${\cal Q}_{{\cal P}^{\rm loc,ex}} \lo 
%{\cal Q}^{\rm loc,ex}$ 
%is strict. 
%\begin{equation*} 
%\Om^i_{{\cal P}^{\rm ex}/S}=
%\Om^i_{{\cal P}/S}\otimes_{{\cal O}_{\cal P}}
%{\cal O}_{{\cal P}^{\rm ex}}.
%\tag{6.1.4}\label{eqn:lexd}
%\end{equation*} 
%\begin{equation*} 
%\Om^j_{{\cal Q}/S}\otimes_{{\cal O}_{\cal P}}
%{\cal O}_{{\cal P}^{\rm ex}}
%=\Om^j_{{\cal Q}^{\rm ex}/S} \quad (j\in {\mab N})
%\tag{6.1.6}\label{eqn:tptis}
%\end{equation*} 
%By the proof of \cite[(2.1.4)]{nh2}, 
%we may have 
%$\Om^i_{{\cal P}/S}\otimes_{{\cal O}_{\cal P}}
%{\cal O}_{{\cal P}^{\rm ex}}
%=\Om^i_{{\cal P}^{\rm ex}/S}$ $(i\in {\mab N})$.  
\par 
(2): 
Because we have a natural morphism 
$P^{{\cal P}/{\cal R}}
\otimes_{{\cal O}_{\cal P}}{\cal O}_{{\cal P}^{\rm ex}}
\lo P^{{\cal P}^{\rm ex}/{\cal Q}^{\rm ex}}$, 
the question is local. 
By the assumption, ${\cal P}^{\rm ex}$ is 
%${\cal Q}^{\rm ex}$ are 
the formal completion of 
${\cal P}$ 
%and ${\cal Q}$ 
along $X$ with the inverse image of 
the log structure of ${\cal P}$.  
%and ${\cal Q}$, respectively. 
%In particular, $\os{\circ}{{\cal P}^{\rm ex}}=
%\os{\circ}{{\cal Q}^{\rm ex}}$.  
By (1) we see that  
$\Om^i_{{\cal P}/S}\otimes_{{\cal O}_{\cal P}}
{\cal O}_{{\cal P}^{\rm ex}}
=\Om^i_{{\cal P}^{\rm ex}/S}$ $(i\in {\mab N})$
and 
$\Om^j_{{\cal R}/S}\otimes_{{\cal O}_{\cal P}}
{\cal O}_{{\cal P}^{\rm ex}}
=\Om^j_{{\cal Q}^{\rm ex}/S}$ $(j\in {\mab N})$. 
%By the proof of \cite[(2.1.4)]{nh2} again, 
Since ${\cal P}^{\rm ex}$ is the formal completion of 
${\cal P}$ along $X$, 
$\otimes_{{\cal O}_{\cal P}}
{\cal O}_{{\cal P}^{\rm ex}}$ is an exact functor.   
Hence, for $0\leq k\leq i$, we obtain the following equality:  
\begin{equation*}
{\rm Im}(\Om^k_{{\cal P}/S}
{\otimes}_{{\cal O}_{\cal P}}
\Om^{i-k}_{{\cal R}/S} \lo 
\Om^i_{{\cal P}/S})\otimes_{{\cal O}_{\cal P}}
{\cal O}_{{\cal P}^{\rm ex}}
={\rm Im}(\Om^k_{{\cal P}^{\rm ex}/S}
{\otimes}_{{\cal O}_{{\cal P}^{\rm ex}}}
\Om^{i-k}_{{\cal Q}^{\rm ex}/S} \lo 
\Om^i_{{\cal P}^{\rm ex}/S}).  
\tag{6.1.5}\label{eqn:pqop} 
\end{equation*} 
This is nothing but the equality (\ref{eqn:tenexsfil}). 
\end{proof}

\begin{rema}\label{rema:unusefil} 
Consider the case where the horizontal immersions in 
(\ref{eqn:pvq}) are not exact. Let  
$P^{{\cal P}/{\cal R}}$ be the filtration  
on $\Om^i_{{\cal P}/S}$ defined by 
the formula (\ref{eqn:ldfpw}). 
%Then we have a filtration 
%$P^{{\cal P}/{\cal R}}
%:=\{P^{{\cal P}/{\cal R}}_k\}_{k\in {\mab Z}}$ 
%on $\Om^i_{{\cal P}/S}$. 
It is natural to ask whether 
the following equalities 
\begin{equation*} 
P^{{\cal P}^{\rm ex}/{\cal Q}^{\rm ex}}
=P^{{\cal P}/{\cal R}}\otimes_{{\cal O}_{\cal P}}
{\cal O}_{{\cal P}^{\rm ex}}, \quad 
P^{{\cal P}^{\rm ex}/{\cal R}^{\rm ex}}
=P^{{\cal P}/{\cal R}}\otimes_{{\cal O}_{\cal P}}
{\cal O}_{{\cal P}^{\rm ex}}
\tag{6.2.1}\label{eqn:tensfil}
\end{equation*} 
hold. 
Here $P^{{\cal P}^{\rm ex}/{\cal R}^{\rm ex}}$ is the 
obvious analogue of $P^{{\cal P}^{\rm ex}/{\cal Q}^{\rm ex}}$. 
However neither of these does not hold in general. 
Indeed, consider the case 
$X = {\rm Spec}(\kap[t])$, $Z = \emptyset$, $D=\{t=0\}$, 
${\cal P} = ({\rm Spf}({\cal V}\{x,y\}), \{x=0\} \cup \{y=0\})$ 
and consider the immersion 
$(X,D) \os{\sus}{\lo} {\cal P}$ 
defined by $x \mapsto t, y \mapsto t$ 
and set
${\cal R}=(\overset{\circ}{{\cal P}},{\cal O}_{\cal P}^*)$. 
Then ${\cal P}^{\rm ex}=
({\rm Spf}({\cal V}\{x\}[[u-1]]),\{x=0\})$ 
and ${\cal Q}={\cal Q}^{\rm ex}=
(\overset{\circ}{{\cal P}^{\rm ex}},
{\cal O}_{{\cal P}^{\rm ex}}^*)$. 
By the definition of 
$P^{{\cal P}/{\cal R}}$, we see that 
$$P^{{\cal P}/{\cal R}}_1\Om^2_{{\cal P}/S} = 
{\rm Im}(\Om^1_{{\cal P}/S}
{\otimes}_{{\cal O}_{\cal P}}
\Om^{1}_{{\cal R}/S} \lo \Om^2_{{\cal P}/S})$$ 
is generated by 
$x d\log x \wedge d\log y$ and $y d\log x \wedge d\log y$. 
Hence the image of the following natural morphism 
\begin{equation*}
P^{{\cal P}/{\cal R}}_1
\Om^2_{{\cal P}/S} \otimes_{{\cal O}_{{\cal P}}} 
{\cal O}_{{\cal P}^{\rm ex}} \lo \Om^2_{{\cal P}^{\rm ex}/S}
\label{eqn:q1_1}
\end{equation*}
is generated by $xd\log x\we d\log u$. 
On the other hand, 
set ${\cal U}={\cal Q}$ or ${\cal R}$. 
Then 
\begin{equation*}
P^{{\cal P}^{\rm ex}/{\cal U}^{\rm ex}}_1
\Om^2_{{\cal P}^{\rm ex}/S} = 
{\rm Im}(\Om^1_{{\cal P}^{\rm ex}/S}
{\otimes}_{{\cal O}_{{\cal P}^{\rm ex}}}
\Om^1_{{\cal U}^{\rm ex}/S} \lo 
\Om^2_{{\cal P}^{\rm ex}/S})=
\Om^2_{{\cal P}^{\rm ex}/S}. 
\end{equation*}
Since $\langle xd\log x\wedge d\log u\rangle 
\not=\Om^2_{{\cal P}^{\rm ex}/S}$, 
neither of the formulas in (\ref{eqn:tensfil}) 
does not hold. 
\par 
In this way, we think that 
the filtration 
$P^{{\cal P}/{\cal R}}$
is not good when 
$(X,D) \os{\sus}{\lo} (\overset{\circ}{{\cal P}},{\cal M})$ 
is not exact 
and we do not consider 
$P^{{\cal P}/{\cal R}}$ any more in this paper.  
\end{rema}

%Let ${\cal X}$ be a formally smooth formal scheme over $S$ 
%and let ${\cal D}$ and ${\cal Z}$ be transversal SNCD's on ${\cal X}/S$. 
%When ${\cal P}=({\cal X},{\cal D}\cup {\cal Z})$, 
%we denote $P^{{\cal P}^{\rm ex}/{\cal Q}^{\rm ex}}$ 
%by $P^{{\cal D}^{\rm ex}}$.   
\par 
For a finitely generated monoid $Q$ 
and a set $\{q_1,\ldots,q_n\}$ 
$(n\in {\mab Z}_{\geq 1})$ of generators of $Q$, 
we say that $\{q_1,\ldots,q_n\}$ is {\it minimal} 
if there does not exist a set of generators of $Q$ whose cardinality is
strictly less than $n$. 
Let $Y=(\os{\circ}{Y},M_Y)$ be 
an fs(=fine and saturated) log (formal) scheme. 
Let $y$ be a point of $\os{\circ}{Y}$. 
Let $m_{1,y},\ldots, m_{r,y}$ be local sections of $M_Y$ around $y$ 
whose images in $M_{Y,y}/{\cal O}_{Y,y}^*$ 
form a minimal set of generators of $M_{Y,y}/{\cal O}_{Y,y}^*$. 
Let $D(M_Y)_i$ $(1\leq i \leq r)$ 
be the local closed subscheme of 
$\os{\circ}{Y}$ defined by the ideal sheaf generated by 
the image of $m_{i,y}$ in ${\cal O}_Y$. 
For a positive integer $k$, 
let $D^{(k)}(M_Y)$ be the disjoint union of 
the $k$-fold intersections of different $D(M_Y)_i$'s. 
Assume that 
\begin{equation*} 
M_{Y,y}/{\cal O}^*_{Y,y}\simeq {\mab N}^r
\tag{6.2.2}\label{eqn:epynr}
\end{equation*}
for any point $y$ of $\os{\circ}{Y}$ 
and for some $r\in {\mab N}$ depending on $y$. 
It is easy to show that 
${\rm Aut}({\mab N}^r)\simeq {\mathfrak S}_r$ 
(see \cite[\S4]{nh3} for the proof of this).  
Consequently 
${\rm Aut}(M_{Y,y}/{\cal O}^*_{Y,y})\simeq {\mathfrak S}_r$.  
Hence the scheme $D^{(k)}(M_Y)$ is 
independent of the choice of $m_{1,y},\ldots, m_{r,y}$ 
and it is globalized. 
We denote this globalized scheme 
by the same symbol $D^{(k)}(M_Y)$. 
Set $D^{(0)}(M_Y):=\os{\circ}{Y}$. 
Let $c^{(k)} \col D^{(k)}(M_Y) \lo \os{\circ}{Y}$ 
be the natural morphism. 
\par 
As in \cite[(3.1.4)]{dh2} and \cite[(2.2.18)]{nh2}, 
we have an orientation sheaf $\vp_{\rm zar}^{(k)}(D(M_Y))$ 
$(k\in {\mab N})$ associated to the set of $D(M_Y)_i$'s. 
We have the following equality 
\begin{equation*} 
c^{(k)}_*\vp_{\rm zar}^{(k)}(D(M_Y)) 
=\bigwedge^k(M_Y^{\rm gp}/{\cal O}_Y^*) 
\tag{6.2.3}\label{eqn:bkezps}
\end{equation*} 
as sheaves of abelian groups on $\os{\circ}{Y}$.   

\par 
The following has been proved in \cite{nh3}: 

\begin{prop}[{\bf \cite[(4.3)]{nh3}}]\label{prop:mmoo} 
Let $g\col Y \lo Y'$ be a morphism of fs log schemes 
satisfying the condition {\rm (\ref{eqn:epynr})}. 
Assume that, for each point $y\in \os{\circ}{Y}$ 
and for each member $m$ of the minimal generators 
of $M_{Y,y}/{\cal O}^*_{Y,y}$, there exists 
a unique member $m'$ of the minimal generators of 
$M_{Y',\os{\circ}{g}(y)}/{\cal O}^*_{Y',\os{\circ}{g}(y)}$ 
such that $g^*(m')\in m^{{\mab Z}_{>0}}$. 
Then there exists a canonical morphism 
$g^{(k)}\col D^{(k)}(M_Y)\lo D^{(k)}(M_{Y'})$ fitting 
into the following commutative diagram of schemes$:$
\begin{equation*} 
\begin{CD} 
D^{(k)}(M_Y) @>{g^{(k)}}>> D^{(k)}(M_{Y'}) \\ 
@VVV @VVV \\ 
\os{\circ}{Y} @>{\os{\circ}{g}}>> \os{\circ}{Y}{}'. 
\end{CD} 
\tag{6.3.1}\label{cd:dmgdm}
\end{equation*} 
\end{prop} 
%\begin{proof}(cf.~\cite[p.~614]{stwsl}) 
%First we consider the local case.  
%Let $\{m_{i,y}\}_{i=1}^r$ and  
%$\{D(M_Y)_i\}_{i=1}^r$ be as above. 
%Set $y':=\os{\circ}{g}(y)$. 
%Let 
%$\{m_{i',y'}\}_{i'=1}^{r'}$ and  
%$\{D(M_{Y'})_{i'}\}_{i'=1}^{r'}$  
%be the similar objects 
%for $Y'$ around $\os{\circ}{g}(y)$. 
%We may assume that 
%$g^*(m'_{i,y'}) \in m_{i,y}^{{\mab Z}_{>0}}$.  
%This means that $g$ induces a natural morphism 
%$D(M_Y)_{i} \lo D(M_{Y'})_{i}$ $(1\leq i\leq r)$. 
%Let $D^{(k)}(M_Y;g)$ be the disjoint union of 
%$k$-fold intersections of 
%$D(M_{Y'})_1,\ldots,D(M_{Y'})_r$. 
%Then we have a natural morphism 
%$D^{(k)}(M_Y)\lo D^{(k)}(M_Y;g)$. 
%By composing this with the natural inclusion morphism 
%$D^{(k)}(M_Y;g)\lo D^{(k)}(M_{Y'})$, we have a morphism 
%$D^{(k)}(M_Y)\lo D^{(k)}(M_{Y'})$. 
%This is a desired local morphism. 
%This local morphism is compatible with localization. 
%Hence we have a desired global morphism 
%$g^{(k)} \col D^{(k)}(M_Y)\lo D^{(k)}(M_{Y'})$. 
%\end{proof}

\par 
When $Y=(X,D)$, we denote $D^{(k)}(M_Y)$ by $D^{(k)}$ and 
$\vp_{\rm zar}^{(k)}(D(M_Y))$ by 
$\vp^{(k)}_{\rm zar}(D/S_1)$. 
The sheaf $\vp^{(k)}_{\rm zar}(D/S_1)$ is 
extended to an abelian sheaf  
$\vp^{(k)}_{{\rm conv}}(D/S;Z)$ 
in the log convergent topos 
$((D^{(k)},Z\vert_{D^{(k)})/S})_{\rm conv}$ 
since, for an object
$(U, T, \iota,u) 
\in {\rm Conv}({(D^{(k)},Z\vert_{D^{(k)}})/S})$, 
we have the pull-back 
$u^{-1}(\vp^{(k)}_{\rm zar}(D/S_1))$
of the sheaf 
$\vp^{(k)}_{\rm zar}(D/S_1)$ to $\os{\circ}{U}$ 
by $u$ 
and since the closed immersion 
$\iota \col \os{\circ}{U} \os{\subset}{\lo} \os{\circ}{T}$ 
is a homeomorphism of topological spaces. 
If $Z=\emptyset$,  then we denote 
$\vp^{(k)}_{{\rm conv}}(D/S;Z)$  
by 
$\vp^{(k)}_{\rm conv}(D/S)$. 

\begin{defi}\label{defi:wct}
We call 
$$\vp^{(k)}_{\rm zar}(D/S_1) 
\text{ (resp. }\vp^{(k)}_{\rm conv}(D/S), \text{ } 
\vp^{(k)}_{{\rm conv}}(D/S;Z))$$ 
the {\it zariskian orientation sheaf} 
(resp.~the {\it convergent orientation sheaf}, 
the {\it log convergent orientation sheaf}) of 
$D^{(k)}/S_1$ (resp.~$D^{(k)}/S$, 
$(D^{(k)},Z\vert_{D^{(k)}})/S$).
\end{defi}

%\begin{rema}
%The sheaves $\vp^{(k)}_{\rm zar}(D/S_1)$, 
%$\vp^{(k)}_{\rm conv}(D/S)$ and 
%$\vp^{(k){\rm log}}_{{\rm conv}}(D/S;Z))$  
%are independent of the choice of the 
%decomposition by smooth components of $D$ by 
%\cite[(28.1)]{nh2}.
%\end{rema}

\par 
Let $M_1$ and $M_2$ be two fine log structures 
on $\os{\circ}{Y}$ such that 
$M_Y=M_1\oplus_{{\cal O}^*_Y}M_2$. 
Assume that the condition (\ref{eqn:epynr}) 
is satisfied for $M_1$. 
Let $m_{1,y},\ldots, m_{r,y}$ be 
local sections of $M_1$ around $y$ 
whose images in $M_{1,y}/{\cal O}_{Y,y}^*$ 
form a minimal set of generators of $M_{1,y}/{\cal O}_{Y,y}^*$. 
Set $Y_i:=(\os{\circ}{Y},M_i)$ $(i=1,2)$. 
Endow $D^{(k)}(M_1)$ with the inverse image of 
the log structure of $M_2$ and denote 
the resulting log scheme by $(D^{(k)}(M_1),M_2)$ 
by abuse of notation. 
Let $b^{(k)} \col (D^{(k)}(M_1),M_2) \lo (\os{\circ}{Y},M_2)$ 
be the natural morphism. 
By abuse of notation, we denote by the same symbol 
$b^{(k)}$ the underlying morphism 
$D^{(k)}(M_1) \lo \os{\circ}{Y}$. 
\par 
Now we come back to the situation 
(\ref{eqn:mdml}) (and (\ref{eqn:pvq})).  

\begin{lemm}\label{lemm:peme}
The log scheme 
$(\os{\circ}{\cal P}{}^{\rm ex},{\cal M}^{\rm ex})$ 
satisfies the condition {\rm (\ref{eqn:epynr})}. 
In particular $D^{(k)}({\cal M}^{\rm ex})$ 
$(k\in {\mab N})$ 
is well-defined. 
\end{lemm}
\begin{proof} 
Identify the points of $X$ with the image of those 
in $\os{\circ}{\cal P}{}^{\rm ex}$. 
Because 
$(X,D)\os{\sus}{\lo} 
(\os{\circ}{\cal P}{}^{\rm ex},{\cal M}^{\rm ex})$ 
is exact and closed, 
we have an equality  
${\cal M}^{\rm ex}_x
/{\cal O}^*_{{\cal P}^{\rm ex},x}=
M(D)_x/{\cal O}^*_{D,x}$ 
for a point $x\in X$. Hence (\ref{lemm:peme}) is clear. 
\end{proof} 
By abuse of notation, denote 
by $(D^{(k)}({\cal M}^{\rm ex}),{\cal N}^{\rm ex})$ 
the log formal scheme whose underlying formal scheme 
is $D^{(k)}({\cal M}^{\rm ex})$ and 
whose log structure is the inverse image of ${\cal N}^{\rm ex}$ 
by the morphism 
$D^{(k)}({\cal M}^{\rm ex})\lo \os{\circ}{\cal P}{}^{\rm ex}$. 
%Assume that the condition (\ref{eqn:epynr}) is satisfied for 
%$(\os{\circ}{\cal P},{\cal M})$. 
\par
Set $Z\vert_{D^{(k)}}:=Z\times_XD^{(k)}$. 
%$Z\vert_{D^{(k)}({\cal M}^{\rm ex})}:=
%Z\times_{\cal P}D^{(k)}({\cal M}^{\rm ex})$. 
The scheme 
$Z\vert_{D^{(k)}}$ is a relative SNCD on $D^{(k)}/S_1$. 
Let 
$a^{(k)} \col (D^{(k)}, Z\vert_{D^{(k)}}) 
\lo (X,Z)$ and 
$b^{(k)}: 
(D^{(k)}({\cal M}^{\rm ex}),{\cal N}^{\rm ex})
\lo {\cal Q}^{\rm ex}$ 
be the morphisms induced by the natural closed immersions. 
\par 
Let the middle objects in the following table 
be the prewidenings obtained by
the left immersions over $S$ and 
let the right objects be the system 
of the universal enlargements of 
these prewidenings:

\begin{equation*}
\begin{tabular}{|l|l|l|} \hline 
$(X,D\cup Z) \os{\subset}{\lo} {\cal P}$  
& $T$ & 
$\os{\to}{T}:=\{T_n\}_{n=1}^{\infty}$\\ \hline    
$(X,Z) \os{\subset}{\lo} {\cal Q}^{\rm ex}$ 
& $T_{{\cal Q}^{\rm ex}}$ & 
$\os{\to}{T}_{{\cal Q}^{\rm ex}}
:=\{T_{{\cal Q}^{\rm ex},n}\}_{n=1}^{\infty}$ \\ \hline 
$(D^{(k)}, Z\vert_{D^{(k)}}) \os{\subset}{\lo} 
(D^{(k)}({\cal M}^{\rm ex}), {\cal N}^{\rm ex})$ &
$T^{(k)}$ & 
$\os{\to}{T}{}^{(k)}:=
\{T^{(k)}_n\}_{n=1}^{\infty}$\\ \hline 
\end{tabular}
\end{equation*}
\bigskip
\parno
Let $g_n \col T_n \lo {\cal P}$, 
$g^{\rm ex}_n \col T_n \lo {\cal P}^{\rm ex}$,  
$g_{{\cal Q}^{\rm ex},n} \col T_{{\cal Q}^{\rm ex},n}\lo {\cal Q}^{\rm ex}$ 
%$g^{\rm ex}_{{\cal Q}^{\rm ex},n} \col 
%T_{{\cal Q}^{\rm ex},n}\lo {\cal Q}^{\rm ex}$ 
and  
$g^{(k)}_n \col T^{(k)}_n \lo 
(D^{(k)}({\cal M}^{\rm ex}),{\cal N}^{\rm ex})$ 
$(n\in {\mab Z}_{\geq 1})$ 
be the natural morphisms.  
%Since $\os{\circ}{\cal P}=\os{\circ}{\cal Q}{}'$, 
%$\os{\circ}{T}_{{\cal P},n}=\os{\circ}{T}_n$. 
%The underlying formal schemes of the log formal schemes
%$T_{{\cal P},n}$ and $T_n$ are the same. 
Let  $c^{(k)}_n \col T^{(k)}_n \lo T_{{\cal Q}^{\rm ex},n}$
be the morphism  induced by 
$b^{(k)} \col 
(D^{(k)}({\cal M}^{\rm ex}),{\cal N}^{\rm ex})
\lo {\cal Q}^{\rm ex}$.

\begin{lemm}[{\rm {\bf cf.~\cite[(2.2.16) (1)]{nh2}}}]
\label{lemm:dpinc}
%The natural morphism $(D^{(k)},Z\vert_{D^{(k)}}) \lo 
%(D^{(k)}({\cal M}^{\rm ex}),M_{\cal Q}) \times_{\cal Q}(X,Z)$ 
%is an isomorphism. 
%$(1)$ The natural morphism $(D^{(k)},Z\vert_{D^{(k)}}) 
%\lo ({\cal D}^{(k)},{\cal Z}\vert_{{\cal D}^{(k)}})
%\allowbreak \times_{({\cal X},{\cal Z})}(X,Z)$ 
%is an isomorphism. 
%\par
%$(2)$ 
The natural morphism 
$T^{(k)}_n \lo 
T_{{\cal Q}^{\rm ex},n}\times_{{\cal Q}^{\rm ex}}
(D^{(k)}({\cal M}^{\rm ex}),{\cal N}^{\rm ex})$ 
is an isomorphism.
%\par
%$(3)$ Let $\ol{\cal J}$ $($resp.~$\ol{\cal J}^{(k)})$
%be the PD-ideal sheaf of ${\cal O}_{\mathfrak D}$
%$($resp.~${\cal O}_{{\mathfrak D}^{(k)}})$.
%Then the natural morphism
%$c^{(k)*} \col c^{(k)*}(\ol{\cal J}) \lo \ol{\cal J}^{(k)}$ 
%is surjective. 
\end{lemm}
\begin{proof}
%Apply (\ref{lemm:dxx}) to the SNCD 
%${\cal D}\cup {\cal Z}$ and assume 
%that ${\cal D}$ (resp.~${\cal Z}$) 
%is defined by an equation
%$x_1=\cdots =x_t=0$ 
%(resp.~$x_{t+1}=\cdots =x_s=0$) 
%$(1\leq t\leq s)$.
%\par 
%(1): (1) is obvious. 
%\par 
%(2):  
By the universality of the universal enlargement, 
this is a local question. 
Let $x$ be a point of $X$. 
By taking a local exactification 
$(X,D\cup Z)\os{\sus}{\lo}{\cal P}^{\rm loc,ex}$
of the immersion $(X,D\cup Z)\os{\sus}{\lo}{\cal P}$, 
we may assume that the immersion 
$(X,D\cup Z)\os{\sus}{\lo}{\cal P}$ is exact. 
Hence, by \cite[(2.1.5)]{nh2},  
we may assume that 
there exist a formally smooth formal scheme ${\cal X}$ 
and transversal relative SNCD's 
${\cal D}$ and ${\cal Z}$ on ${\cal X}$ such that 
${\cal P}=({\cal X},{\cal D}\cup {\cal Z})$ 
and such that there exist exact immersions 
$(X,D)\os{\sus}{\lo} ({\cal X},{\cal D})$  
and $(X,Z)\os{\sus}{\lo} ({\cal X},{\cal Z})$ over $S$. 
We may assume that ${\cal X}$ is affine 
(say ${\cal X}={\rm Spf}(A)$) and 
may have the following cartesian diagrams 
\begin{equation*}
\begin{CD}
{\cal D}\cup {\cal Z}@>{\subset}>> {\cal X}\\ 
@V{}VV  @VV{g}V \\
\ul{\rm Spf}_S({\cal O}_S\{x_1, \ldots, x_{d'}\}
/(x_1\cdots x_s)) 
@>{\subset}>> \ul{\rm Spf}_S
({\cal O}_S\{x_1, \ldots, x_{d'}\}),
\end{CD}
\tag{6.6.1}\label{cd:dxs}
\end{equation*}
\begin{equation*}
\begin{CD}
X @>{\subset}>> {\cal X}\\ 
@V{}VV  @VV{g}V \\
\ul{\rm Spf}_{S_1}
({\cal O}_{S_1}\{x_1, \ldots, x_{d'}\}
/(x_{d+1}, \ldots, x_{d'}))
@>{\subset}>> \ul{\rm Spf}_S
({\cal O}_S\{x_1, \ldots, x_{d'}\})
\end{CD}
\tag{6.6.2}\label{cd:xxd}
\end{equation*}
for some $s\leq d'$, where $g$ is an \'{e}tale morphism.   
Let ${\cal D}_{1\cdots r} \subset {\cal D}$ be a 
closed subscheme defined by an equation 
$x_1=\cdots =x_r=0$. 
Set $e:=d'-d$. 
For $\ul{m}:=(m_1,\ldots,m_e)\in {\mab N}^e$, set 
$x^{\ul{m}}:=x^{m_1}_{d+1}\cdots x^{m_e}_{d'}$. 
For $n\in {\mab Z}_{\geq 1}$ and 
$\ul{m}\in {\mab N}^e$ 
with $\vert \ul{m}\vert=n$, 
let $t_{\ul{m}}$ be independent variables.  
Then, by (\ref{eqn:ldra}),   
\begin{equation*} 
{\cal O}_{T_{{\cal Q}^{\rm ex},n}}\otimes_{{\cal O}_{\cal X}}
{\cal O}_{{\cal D}_{1\cdots r}}
=
(A[t_{\ul{m}}~\vert~\ul{m} 
\in {\mab N}^r,\vert\ul{m}\vert =n]
/((x^{\ul{m}}-\pi t_{\ul{m}})+
(p{\textrm -}{\rm torsion})))^{\wh{}}/(x_1,\ldots,x_r).  
\end{equation*} 
Set $D_{1\cdots r}:={\cal D}_{1\cdots r}\times_{\cal X}X$. 
On the other hand, 
the structure sheaf of the $n$-th universal log enlargement 
of the closed immersion $D_{1\cdots r} 
\os{\subset}{\lo} {\cal D}_{1\cdots r}$ is 
\begin{align*} 
&
({\cal O}_{{\cal D}_{1\cdots r}}
[t_{\ul{m}}~\vert~\ul{m} 
\in {\mab N}^k,\vert\ul{m}\vert =n]
/((x^{\ul{m}}-\pi t_{\ul{m}})+
(p{\textrm -}{\rm torsion})))^{\wh{}} \\
&=(A[t_{\ul{m}}~\vert~\ul{m} 
\in {\mab N}^k,\vert\ul{m}\vert =n]
/((x^{\ul{m}}-\pi t_{\ul{m}})+
(p{\textrm -}{\rm torsion})))^{\wh{}}
/(x_1,\ldots, x_r)\\
&={\cal O}_{T_{{\cal Q}^{\rm ex},n}}\otimes_{{\cal O}_{\cal X}}
{\cal O}_{{\cal D}_{1\cdots r}}. 
\end{align*} 
Furthermore it is immediate to see 
that there exists a natural isomorphism 
$T_n^{(k)} \simeq 
T_{{\cal Q}^{\rm ex},n}\times_{({\cal X},{\cal Z})}
({\cal D}^{(k)}({\cal M}^{\rm ex}),{\cal N}^{\rm ex})$
as log formal schemes. Hence we can complete the proof.  
%Because the question is local,  
%we may  assume, by \cite[(8.6)]{nh2}, that 
%there exist the following two cartesian diagrams$:$
%\begin{equation*}
%\begin{CD}
%{\cal D}\cup {\cal Z}@>{\subset}>> {\cal X}\\ 
%@V{}VV  @VV{g}V \\
%\ul{\rm Spf}_S({\cal O}_S\{x_1, \ldots, x_{d'}\}
%/(x_1\cdots x_s)) 
%@>{\subset}>> \ul{\rm Spf}_S
%({\cal O}_S\{x_1, \ldots, x_{d'}\}),
%\end{CD}
%\tag{6.6.1}\label{cd:dxs}
%\end{equation*}
%\begin{equation*}
%\begin{CD}
%X @>{\subset}>> {\cal X}\\ 
%@V{}VV  @VV{g}V \\
%\ul{\rm Spf}_{S_1}
%({\cal O}_{S_1}\{x_1, \ldots, x_{d'}\}
%/(x_{d+1}, \ldots, x_{d'}))
%@>{\subset}>> \ul{\rm Spf}_S
%({\cal O}_S\{x_1, \ldots, x_{d'}\}),
%\end{CD}
%\tag{6.6.2}\label{cd:xxd}
%\end{equation*}
%where $g$ is an etale morphism,  
%$\{{\cal D}_{\lam}\}_{\lam}
%=\{\{x_i=0\}\}_{i=1}^t$ and 
%$\{{\cal Z}_{\mu}\}_{\mu}
%=\{\{x_i=0\}\}_{i=t+1}^s$ $(0\leq t \leq s \leq d)$ 
%in the diagram {\rm (\ref{cd:dxs})}. 
%Now (\ref{lemm:dpinc}) is obvious. 
\end{proof}

\par
Set 
$(X_n,D_n\cup Z_n):=(X,D\cup Z)\times_{\cal P}T_n$, 
$(X_n,Z_n):=(X,Z)\times_{{\cal Q}^{\rm ex}}T_{{\cal Q}^{\rm ex},n}$ 
and 
$(D^{(k)}_n,Z_n\vert_{D^{(k)}_n})
:=(D^{(k)},Z\vert_{D^{(k)}})
\times_{(D^{(k)}({\cal M}^{\rm ex}),{\cal N}^{\rm ex})}
T^{(k)}_n$ $(n\in {\mab N})$.  
As usual, we denote the left representable objects 
in the following table 
by the right ones for simplicity of notation:

\bigskip
\begin{tabular}{|l|l|} \hline 
$((X_n,D_n\cup Z_n) \os{\subset}{\lo}T_n)\in 
((X,D\cup Z)/S)_{\rm conv}$
& $T_n$ \\  \hline
$((X_n,Z_n) \os{\subset}{\lo}T_{{\cal Q}^{\rm ex},n})
\in ((X,Z)/S)_{\rm conv}$
& $T_{{\cal Q}^{\rm ex},n}$ \\ \hline
$((D^{(k)}_n,Z_n\vert_{D^{(k)}_n}) 
\os{\subset}{\lo} T^{(k)}_n)\in 
((D^{(k)},Z\vert_{D^{(k)})/S})_{\rm conv}$
& $T^{(k)}_n$  \\ \hline 
\end{tabular}
\bigskip
\parno
Let $a^{(k)}_{\rm conv}: 
((D^{(k)},Z\vert_{D^{(k)})/S})_{\rm conv} 
\lo 
((X,Z)/S)_{\rm conv}$
be the morphism of topoi induced by the 
morphism 
$a^{(k)}$. 
%By the log version of \cite[6.2 Proposition]{bob},  
%the functor $a^{(k){\log}}_{{\rm conv}*}$ is exact. 
\par

\par 
Let

\bigskip
\begin{tabular}{|l|} \hline 
$j_T: ((X,D\cup Z)/S)_{\rm conv}\vert_{T} \lo 
((X,D\cup Z)/S)_{\rm conv}$ 
\\ \hline
$j_{T_{{\cal Q}^{\rm ex}}}: 
((X,Z)/S)_{\rm conv}\vert_{T_{{\cal Q}^{\rm ex}}} \lo 
((X,Z)/S)_{\rm conv}$ 
\\ \hline
$j_{T^{(k)}}: 
((D^{(k)},
Z\vert_{D^{(k)}})/S)_{\rm conv} 
\vert_{T^{(k)}} \lo 
((D^{(k)},Z\vert_{D^{(k)}})/S)_{\rm conv}$  
\\ \hline 
\end{tabular}
\bigskip  
\parno
be localization functors and let

\bigskip
\begin{tabular}{|l} \hline 
$\varphi^*: 
\{{\rm coherent~crystals~of~}{\cal K}_{\os{\to}{T}}{\rm -modules}\}
\lo $
\\ \hline
${\varphi}^*_{{\cal Q}^{\rm ex}}: 
\{{\rm coherent~crystals~of~}
{\cal K}_{\os{\to}{T}_{{\cal Q}^{\rm ex}}}{\rm -modules}\}
\lo $ 
\\ \hline
$\varphi^{(k)*}: 
\{{\rm coherent~crystals~of~}
{\cal K}_{\os{\to}{T}{}^{(k)}}{\rm -modules}\}
\lo $  
\\ \hline 
\end{tabular}

\bigskip
\begin{tabular}{l|} 
\hline 
$\{{\rm coherent~crystals~of~}
{\cal K}_{(X,D\cup Z)/S}\vert_T
{\rm -modules}\}$ \\ \hline
$\{{\rm coherent~crystals~of~}
{\cal K}_{(X,Z)/S}\vert_{T_{{\cal Q}^{\rm ex}}}{\rm -modules}\}$ 
\\ \hline
$\{{\rm coherent~crystals~of~}{\cal K}_{(D^{(k)},
Z\vert_{D^{(k)}})/S}\vert_{T^{(k)}}
{\rm -modules}\}$  
\\ \hline 
\end{tabular}
\bigskip
\parno 
be the morphisms of pull-backs 
defined in (\ref{eqn:phyt}).  
Let $h_n$ be $g_n$ or $g^{\rm ex}_n$.  
%and let $h_{{\cal Q}^{\rm ex},n}$ be 
%$g_{{\cal Q}^{\rm ex},n}$ or $g^{\rm ex}_{{\cal Q}^{\rm ex},n}$. 
Let ${\cal U}$ be ${\cal P}$ or ${\cal P}^{\rm ex}$.  
%and let ${\cal V}$ be ${\cal Q}$ or ${\cal Q}^{\rm ex}$. 
Let

\bigskip
\begin{tabular}{|l|} \hline 
$h_n: (\os{\circ}{T}_{n,{\rm zar}},{\cal K}_{T_n}) \lo 
(\os{\circ}{\cal U}_{\rm zar},{\cal K}_{\cal U})$ \\ \hline 
$g_{{\cal Q}^{\rm ex},n}: 
(\os{\circ}{T}_{{\cal Q}^{\rm ex},n,{\rm zar}},
{\cal K}_{T_{{\cal Q}^{\rm ex},n}}) \lo 
(\os{\circ}{\cal Q}{}^{\rm ex}_{\rm zar},{\cal K}_{{\cal Q}^{\rm ex}})$ 
\\ \hline 
$g^{(k)}_n: 
(\os{\circ}{T}{}^{(k)}_{n{\rm zar}},{\cal K}_{T_n}) \lo 
(D^{(k)}({\cal M}^{\rm ex})_{\rm zar},
{\cal K}_{D^{(k)}({\cal M}^{\rm ex})})$  \\ \hline 
\end{tabular}
\bigskip
\parno 
be natural morphisms 
of ringed topoi $(n\in {\mab Z}_{\geq 1})$.

\par 
%For a coherent ${\cal K}_{\cal U}$-module ${\cal E}$, set 
%$$L^{\rm conv}_{(X,D\cup Z)/S}({\cal E}):= 
%j_{T*} \varphi^* \{h^*_n({\cal E})\}_{n=1}^{\infty} 
%\in ((X,D\cup Z)/S)_{\rm conv}.$$  
%For a coherent ${\cal K}_{\cal V}$-module ${\cal E}$, set 
%$$L^{\rm conv}_{(X,Z)/S}({\cal E}) := 
%j_{T_{\cal V}*} \varphi^*_{\cal V}
%\{h^*_{{\cal V},n}({\cal E})\}_{n=1}^{\infty}  
%\in ((X,Z)/S)_{\rm conv}. $$
For a coherent 
${\cal K}_{D^{(k)}({\cal M}^{\rm ex})}$-module ${\cal E}$, 
set also
$$L^{(k){\rm conv}}({\cal E}) := 
j_{T^{(k)}*} \varphi^{(k)*}\{g^{(k)*}_n({\cal E})\}_{n=1}^{\infty} 
\in 
((D^{(k)},Z\vert_{D^{(k)}})/S)_{\rm conv}.$$ 
\par 
By (\ref{prop:lcl}) we have a complex
$L^{\rm conv}_{(X, D\cup Z)/S}
(\Om^{\bul}_{{\cal P}^{\rm ex}/S}
\otimes_{\mab Z}{\mab Q})$ of 
${\cal K}_{(X,D\cup Z)/S}$-modules.
By the log convergent Poincar\'{e} lemma ((\ref{theo:pl})), 
we have a natural quasi-isomorphism
\begin{equation*}
{\cal K}_{(X,D\cup Z)/S} \os{\sim}{\lo} 
L^{\rm conv}_{(X,D\cup Z)/S}
(\Om^{\bul}_{{\cal P}^{\rm ex}/S}\otimes_{\mab Z}{\mab Q}).
\tag{6.6.3}\label{eqn:oxdz}
\end{equation*}
Similarly we have the following two quasi-isomorphisms:
\begin{equation*}
{\cal K}_{(X,Z)/S} \os{\sim}{\lo} 
L^{\rm conv}_{(X,Z)/S}(\Om^{\bul}_{{\cal Q}^{\rm ex}/S}
\otimes_{\mab Z}{\mab Q}), 
\tag{6.6.4}\label{eqn:oxz}
\end{equation*}
\begin{equation*}
{\cal K}_{(D^{(k)},Z\vert_{D^{(k)}})/S} 
\os{\sim}{\lo}
L^{(k){\rm conv}}
(\Om^{\bul}_{D^{(k)}({\cal M}^{\rm ex})/S}
\otimes_{\mab Z}{\mab Q}). 
\tag{6.6.5}\label{eqn:odkzs}
\end{equation*}
\par
The family  
$\{P^{{\cal P}^{\rm ex}/{\cal Q}^{\rm ex}}_k
\Om^i_{{\cal P}^{\rm ex}/S}\}_{i\in {\mab N}}$
gives a subcomplex 
$P^{{\cal P}^{\rm ex}/{\cal Q}^{\rm ex}}_k
\Om^{\bul}_{{\cal P}^{\rm ex}/S}$ 
of $\Om^{\bul}_{{\cal P}^{\rm ex}/S}$. 
By using the natural morphism 
${\cal P}^{\rm ex}\lo {\cal Q}^{\rm ex}$, 
we can consider 
$P^{{\cal P}^{\rm ex}/{\cal Q}^{\rm ex}}_k
\Om^i_{{\cal P}^{\rm ex}/S}$ $(i\in {\mab N})$ 
as an ${\cal O}_{{\cal Q}^{\rm ex}}$-module. 
Set 
\begin{equation*}
P^D_kL^{\rm conv}_{(X,Z)/S}
(\Om^i_{{\cal P}^{\rm ex}/S}\otimes_{\mab Z}{\mab Q})
:= L^{\rm conv}_{(X,Z)/S}
(P^{{\cal P}^{\rm ex}/{\cal Q}^{\rm ex}}_k
\Om^i_{{\cal P}^{\rm ex}/S}
\otimes_{\mab Z}{\mab Q}) 
\quad (k\in {\mab Z}). 
\tag{6.6.6}\label{eqn:plc} 
\end{equation*}
Since the derivative of $L^{\rm conv}_{(X,Z)/S}(\Om^{\bul}_{{\cal P}^{\rm ex}/S}
\otimes_{\mab Z}{\mab Q})$ produces only log poles coming from 
$\Om^1_{{\cal Q}^{\rm ex}/S}$, 
$P^D_kL^{\rm conv}_{(X,Z)/S}
(\Om^{\bul}_{{\cal P}^{\rm ex}/S}\otimes_{\mab Z}{\mab Q})$ 
is stable under the derivative of 
$L^{\rm conv}_{(X,Z)/S}(\Om^{\bul}_{{\cal P}^{\rm ex}/S}
\otimes_{\mab Z}{\mab Q})$.

\par 

\par 
The following theorem (2) is 
one of key theorems in this paper. 

\begin{theo}\label{theo:injf}
$(1)$ 
For each $n\in {\mab Z}_{\geq 1}$, 
the natural morphism 
\begin{equation*} 
{\cal K}_{T_n}{\otimes}_{{\cal O}_{{\cal P}^{\rm ex}}}
P^{{\cal P}^{\rm ex}/{\cal Q}^{\rm ex}}_k
\Om^{\bul}_{{\cal P}^{\rm ex}/S} 
\lo 
{\cal K}_{T_n}{\otimes}_{{\cal O}_{{\cal P}^{\rm ex}}}
\Om^{\bul}_{{\cal P}^{\rm ex}/S}
\tag{6.7.1}\label{eqn:yxpd}
\end{equation*}
is injective. 
\par 
%$(2)$ 
%The natural morphism 
%\begin{equation*} 
%Q^{{\rm conv}*}_{(X,Z)/S}P^D_k
%L^{\rm conv}_{(X,Z)/S}
%(\Om^{\bul}_{{\cal P}^{\rm ex}/S}\otimes_{\mab Z}{\mab Q}) 
%\lo 
%Q^{{\rm conv}*}_{(X,Z)/S}
%L^{\rm conv}_{(X,Z)/S}
%(\Om^{\bul}_{{\cal P}^{\rm ex}/S}\otimes_{\mab Z}{\mab Q})  
%\tag{7.7.2}\label{eqn:yxrpdz}
%\end{equation*}
%is injective. 
\par 
$(2)$ 
%Assume that $Z=\emptyset$. 
%Then the natural morphism 
The natural morphism 
%\begin{equation*} 
%P^D_kL^{\rm conv}_{(X,Z)/S}
%(\Om^{\bul}_{{\cal P}^{\rm ex}/S}
%\otimes_{\mab Z}{\mab Q}) \lo 
%L^{\rm conv}_{(X,Z)/S}(\Om^{\bul}_{{\cal P}/S}
%\otimes_{\mab Z}{\mab Q})  
%\tag{7.7.3}\label{eqn:yxpdz}
%\end{equation*}
\begin{equation*} 
P^D_k
L^{\rm conv}_{(X,Z)/S}
(\Om^{\bul}_{{\cal P}^{\rm ex}/S}\otimes_{\mab Z}{\mab Q}) 
\lo 
L^{\rm conv}_{(X,Z)/S}(\Om^{\bul}_{{\cal P}^{\rm ex}/S}
\otimes_{\mab Z}{\mab Q})  
\tag{6.7.2}\label{eqn:yxpdz}
\end{equation*}
is injective. 
\end{theo}
\begin{proof} 
(1): (1) immediately follows from (\ref{lemm:rex}). 
\par 
(2):  Let $(U',T',\iota',u')$ be an object of 
${\rm Conv}((X,Z)/S)$. We prove that 
the morphism (\ref{eqn:yxpdz}) evaluated at $(U',T',\iota',u')$ is injective. 
The problem is local. 
We may assume that there exist the cartesian diagrams 
(\ref{cd:dxs}) and (\ref{cd:xxd}). 
As in the proof of (\ref{lemm:dpinc}), we may assume that 
${\cal P}=({\cal X},{\cal D}\cup {\cal Z})$,  
${\cal R}=({\cal X},{\cal Z})$,  
${\cal Q}^{\rm ex}=(\wh{\cal X},\wh{\cal Z})$ 
and that (\ref{cd:dxs}) and (\ref{cd:xxd}) exist and 
that $\wh{\cal Z}$ and $\wt{\cal D}$ are defined by 
$x_1\cdots x_a=0$ and $x_{a+1}\cdots x_s=0$. 
In fact, we may assume that $g$ in (\ref{cd:dxs}) and (\ref{cd:xxd}) 
is the identity morphism.  
We may also assume that $\os{\circ}{T}{}'$ is affine. 
Then we have a morphism $e\col T'\lo {\cal Q}^{\rm ex}$ over $S$ 
such that the composite morphism 
$U'\os{\sus}{\lo} T'\lo{\cal Q}^{\rm ex}$ 
is equal to the composite morphism 
$U'\lo (X,Z)\os{\sus}{\lo} {\cal Q}^{\rm ex}$. 
This composite immersion
induces an immersion 
$U'\os{\sus}{\lo} T'\times_S{\cal Q}^{\rm ex}$
and  
let $(T'\times_S{\cal Q}^{\rm ex})^{\rm ex}$ be the exactification 
of the immersion $U'\os{\sus}{\lo} T'\times_S{\cal Q}^{\rm ex}$. 
Let $P$ be a monoid generated by 
$(e_i,-e_i)\in {\mab Z}^{2a}$ $(1\leq i\leq a)$ and $(-e_i,e_i)\in {\mab Z}^{2a}$, where  
$e_i:=(0, \ldots, 0, \us{i}{1}, 0, \ldots, 0)\in {\mab N}^a$. 
%Let ${\cal Q}'$ be the formal log scheme whose underlying formal scheme 
%$\ul{\rm Spf}_S({\cal O}_S\{x_1,\ldots, x_a\})$ and whose log structure 
%${\mab N}^a\owns e_i\lom x_i\in {\cal O}_S\{x_1,\ldots, x_a\}$. 
%Set ${\cal Q}'':=\ul{\rm Spf}_S({\cal O}_S\{x_{a+1},\ldots, x_{d'}\})$.
%Then ${\cal Q}^{\rm ex}={\cal Q}'\times_S{\cal Q}''$.  
By the definition of the exactfication, 
${\cal O}_{(T'\times_S{\cal Q}^{\rm ex})^{\rm ex}}$ 
is locally isomorphic to the following sheaf of commutative rings: 
\begin{align*} 
&({\cal O}_{T'\times_S{\cal Q}^{\rm ex}}
\wh{\otimes}_{{\cal O}_S\{{\mab N}^a\oplus {\mab N}^a\}}
{\cal O}_S\{({\mab N}^a\oplus {\mab N}^a)^{\rm ex}\})^{\wh{}}\tag{6.7.3}\label{ali:utesq}\\
&\simeq ({\cal O}_{T'\times_S{\cal Q}^{\rm ex}}
\wh{\otimes}_{{\cal O}_S\{{\mab N}^a\oplus {\mab N}^a\}}
{\cal O}_S\{P\oplus {\mab N}^a\})^{\wh{}} \\
&\simeq ({\cal O}_{T'}\wh{\otimes}_{{\cal O}_S\{{\mab N}^a\}}
({{\cal O}_S\{P\}}\wh{\otimes}_{{\cal O}_S}
{\cal O}_{{\cal Q}^{\rm ex}}))^{\wh{}}\\
&=(\{{\cal O}_{T'}\wh{\otimes}_{{\cal O}_S\{{\mab N}^a\}}
({{\cal O}_S\{P\}}\wh{\otimes}_{{\cal O}_S}
{\cal O}_S[[x_1,\ldots, x_a]])\}[[x_{a+1},\ldots, x_d]]\{x_{d+1},\ldots,x_{d'}\})^{\wh{}}\\
&={\cal O}_{T'}[[(e_i,-e_i)-1~\vert~(1\leq i\leq a)]]
[[x_{a+1}-\wt{\ol{x}}_{a+1},\ldots,x_d-\wt{\ol{x}}_d]][[x_{d+1},\ldots,x_{d'}]]
%({{\cal O}_S\{P\}}\wh{\otimes}_{{\cal O}_S}
%{\cal O}_S[[x_1,\ldots, x_a]]))\wh{\otimes}_{{\cal O}_S}
%{\cal O}_S[[x_{a+1},\ldots, x_d]]\{x_{d+1},\ldots,x_{d'}\}\\
%&= 
%({\cal O}_{T'}\wh{\otimes}_{{\cal O}_S\{{\mab N}^a\}}
%({{\cal O}_S\{P\}}\wh{\otimes}_{{\cal O}_S}
%{\cal O}_S[[x_{a+1},\ldots, x_d]]\{x_{d+1},\ldots,x_{d'}\}
\end{align*} 
(cf.~the proof of (\ref{prop:lef})). 
Here $\ol{x}_j$ $(a+1\leq j\leq d)$ is the image of $x_j$ in ${\cal O}_U$ and 
$\wt{\ol{x}}_j\in {\cal O}_{T'}$  is a lift of $\ol{x}_j$. 
%Note that $e_i\in {\mab N}^a\subset {\cal O}_S[{\mab N}^a]$ 
%in the last sheaf in (\ref{ali:utesq}) 
%is identified with  
%$$1\otimes (e_i,-e_i)\otimes x_i\in \{{\cal O}_{T'}\wh{\otimes}_{{\cal O}_S\{{\mab N}^a\}}
%({{\cal O}_S\{P\}}\wh{\otimes}_{{\cal O}_S}
%{\cal O}_S[[x_1,\ldots, x_a]])\}[[x_{a+1},\ldots, x_d]]\{x_{d+1},\ldots,x_{d'}\}.$$  
Let $r\col (T'\times_S{\cal Q}^{\rm ex})^{\rm ex}\lo {\cal Q}^{\rm ex}$ 
be the induced morphism by the second projection 
$T'\times_S{\cal Q}^{\rm ex}\lo {\cal Q}^{\rm ex}$. 
Then we have the following diagram (cf.~the proof of (\ref{prop:afex})):  
\begin{equation*}
\begin{CD}
\{{\mathfrak T}_{U',n}(T'\times_S{\cal Q}^{\rm ex})\}_{n=1}^{\infty} 
@>{\{q_n\}_{n=1}^{\infty}}>> (T'\times_S{\cal Q}^{\rm ex})^{\rm ex}
@>{r}>> {\cal Q}^{\rm ex}\\ 
@V{\{p'_n\}_{n=1}^{\infty}}VV \\
T' @. @.. \\
\end{CD}
\tag{6.7.4}\label{cd:utsq} 
\end{equation*} 
Set 
${\cal M}^i_k:=
P^{{\cal P}^{\rm ex}/{\cal Q}^{\rm ex}}_k
\Om^i_{{\cal P}^{\rm ex}/S}
\otimes_{\mab Z}{\mab Q}$ 
$(i,k \in {\mab N})$. 
Then, by (\ref{eqn:lueyset}), 
\begin{equation*} 
\{P^D_kL^{\rm conv}_{X/S}
(\Om^i_{{\cal P}^{\rm ex}/S}
\otimes_{\mab Z}{\mab Q})\}_{T'}
= \vpl_np'_{n*}q_n^*r^*({\cal M}^i_k).  
\tag{6.7.5}\label{eqn:utpaq} 
\end{equation*}  
Because $\vpl_n$ and $p'_{n*}$ are left exact 
for ${\mathcal{K}}_{{\mathfrak T}_{U',n}(T'\times_S{\cal Q}^{\rm ex})}$-modules 
and because $q_n^*$ is left exact 
for 
${\cal K}_{(T'\times_S {\mathcal{Q}}^{\rm ex})^{\rm ex}}$-modules
by (\ref{lemm:rex}), 
we have only to prove that the morphism 
\begin{align*} 
r^*({\cal M}^i_k)\lo r^*({\cal M}^i_{\infty}) \quad (0\leq k\leq i)
\tag{6.7.6}\label{eqn:utpq} 
\end{align*} 
is injective. 
%$$p_2^*({\cal M}^i_k)=
%{\rm Im}{\cal O}_{({\cal Q}^{\rm ex}\otimes_S{\cal Q}^{\rm ex})^{\rm ex}}
%\otimes_{{\cal O}_{{\cal Q}^{\rm ex}}}(\Om^{k-i}_{{\cal Q}^{\rm ex}/S}
%\otimes_{{\cal O}_{{\cal Q}^{\rm ex}}}\Om^{k}_{{\cal P}^{\rm ex}/S}
%\otimes_{\mab Z}{\mab Q})\lo {\cal O}_{({\cal Q}^{\rm ex}\otimes_S{\cal Q}^{\rm ex})^{\rm ex}}
%\otimes_{{\cal O}_{{\cal Q}^{\rm ex}}}(\Om^i_{{\cal P}^{\rm ex}/S}\otimes_{\mab Z}{\mab Q})
%.$$ 
The sheaf $\Om^1_{{\cal P}^{\rm ex}/S}
\otimes_{\mab Z}{\mab Q}$ has a basis 
$\{d\log x_1,\ldots, d\log x_s, dx_{s+1},\ldots, dx_{d'}\}$ 
and 
$P^{{\cal P}^{\rm ex}/{\cal Q}^{\rm ex}}_0\Om^1_{{\cal P}^{\rm ex}/S}
\otimes_{\mab Z}{\mab Q}$
has a basis 
$\{d\log x_1,\ldots, d\log x_a, dx_{a+1},\ldots,dx_s,dx_{s+1},\ldots, dx_{d'}\}$. 
Consider a local section 
\begin{align*} 
\om& =f_{j_1,\ldots,j_m,k_1,\ldots,k_n,\lam_1,\ldots, \lam_o}
\om_{j_1}\wedge \cdots \wedge \om_{j_m}\wedge 
dx_{k_1}\wedge \cdots \wedge dx_{k_n}\wedge 
dx_{l_1}\wedge \cdots \wedge dx_{l_o} \\
& (f_{j_1,\ldots,j_m,k_1,\ldots,k_n,\lam_1,\ldots, \lam_o}\in 
{\cal O}_{(T'\times_S{\cal Q}^{\rm ex})^{\rm ex}}, m+n+o=i, \\
&a+1\leq k_1<\cdots <k_n\leq s,s+1\leq l_1<\cdots<l_o\leq d')
\end{align*}  
of $r^*({\cal M}^i_k)$, 
where $\om_{j_1}, \ldots, \om_{j_m}$ are $d\log x_{j_1},\ldots, d\log x_{j_m}$ or 
$dx_{j_1}, \ldots, dx_{j_m}$ $(1\leq j_1<\cdots<j_m\leq a)$, 
respectively. 
Since $dx_i=x_id\log x_i$ for $a+1\leq i\leq s$ in $\Om^1_{{\cal P}^{\rm ex}/S}$, 
$$\om=
x_{k_1}\cdots x_{k_n}f_{j_1,\ldots,j_m,k_1,\ldots,x_{k_n},\lam_1,\ldots, \lam_o}
\om_{j_1}\wedge \cdots \wedge \om_{j_m}\wedge 
d\log x_{k_1}\wedge \cdots \wedge d\log x_{k_n}\wedge 
\lam_{l_1}\wedge \cdots \wedge \lam_{l_o}.$$ 
%Let us take a local section $\om$ of ${\cal M}^i_k$. 
%Then it is equal to $\om=\om_1+\om_2$, where 
%$\om_1$ does not have forms $dx_{a+1}, \ldots, dx_s$. 
%The local section $\om_2$ is equal to the following form
%$$\om_2=\sum f_{i_1}d\log x_{i_1}\wedge \cdots \wedge d\log x_{i_k}$$
Hence it suffices to prove that $x_i$ $(a+1\leq i\leq s)$ is a nonzero divisor 
in the last sheaf in (\ref{ali:utesq}). However this is very clear. 
%For a log scheme ${\cal U}$ over ${\cal Q}^{\rm ex}$, denote 
%${\cal U}\times_{{\cal Q}^{\rm ex}}(T'\times_S {\mathcal{Q}}^{\rm ex})^{\rm ex}$ by 
%${\cal U}_{\os{\circ}{T}{}''}$. 
%Now the desired injectivity is clear because 
%$r^*({\cal M}^i_k)=
%P^{{\cal P}^{\rm ex}_{\os{\circ}{T}{}''}/{\cal Q}^{\rm ex}_{\os{\circ}{T}{}''}}_k
%\Om^i_{{\cal P}^{\rm ex}_{\os{\circ}{T}{}''/\os{\circ}{T}{}''}} 
%\otimes_{\mab Z}{\mab Q}$. 
\end{proof}

\begin{rema}\label{rema:neagain} 
%(1) 
%Let the notations be as in (\ref{theo:injf}) (2). 
(1) 
The theorem (2) in (\ref{theo:injf}) is different from the corresponding theorem 
in the log crystalline case (\cite[(2.2.17) (2)]{nh2}). 
This is a crucial theorem for the 
construction of the filtered complex $(C_{\rm conv}({\cal O}_{(X,D\cup Z)/S}),P^D)$ 
defined in (\ref{ali:dfg}) below which will be used in the proof of 
the $p$-adic purity (\ref{theo:calvc}) below. 
\par 
(2) Let $(U,T,\iota,u)$ be an object of ${\rm Conv}((X,D\cup Z)/S)$. 
We claim that the following pull-back morphism 
\begin{equation*} 
\{{\rm coherent}~{\cal K}_{{\cal P}}
{\textrm -}{\rm modules}\} 
\lo 
\{{\rm coherent }~
{\cal K}_{(T\times_S{{\cal P}})^{\rm ex}}
{\textrm -}{\rm modules}\} 
\tag{6.8.1}\label{eqn:ktsq}
\end{equation*} 
is not exact in general. 
Indeed, let the notations be as in (\ref{rema:utne}). 
Consider the case where $U=X$, $T={\cal P}$ with the diagonal immersion 
$U \os{\sus}{\lo} T\times_S{\cal P}$.  
Then, as in  (\ref{rema:utne}), 
we can see that the functor (\ref{eqn:ktsq}) is not exact in general. 
\par 
(3) If $Z=\emptyset$, the proof of (\ref{theo:injf}) (2) becomes simpler than 
the proof above.  
Indeed, because $(T'\times_S {\cal Q}^{\rm ex})^{\rm ex}=
T'\times_S {\cal Q}$ and 
$r^*({\cal M}^i_k)=
P^{{\cal P}^{\rm ex}_{T'}/{\cal Q}^{\rm ex}_{T'}}_k
\Om^i_{{\cal P}^{\rm ex}_{T'}/T'} 
\otimes_{\mab Z}{\mab Q}$, the injectivity of 
the morphism (\ref{eqn:utpq}) is clear. 
Here the subscript ${}_{T'}$ means the base change $\times_ST'$. 
%\par 
%(2) Consider the case $Z=\emptyset$ 
%and $M_{\cal R}={\cal O}^*_{\cal R}$ in (\ref{theo:injf}).   
%In this case, we can prove  
%(\ref{theo:injf}) (2) by using (\ref{lemm:rex}). 
%(Hence we can avoid the concrete but long calculation in 
%the proof of (\ref{theo:injf}) (2).)  
%Indeed, let $(U,T,\iota,u)$ be an object of $(X/S)_{\rm conv}$. 
%Then it is trivial that $(U,T,\iota,u)$ is exact. 
%Let the notations be as in (\ref{cd:utsq}).  
%Set 
%${\cal M}^i_k:=
%P^{{\cal P}^{\rm ex}/{\cal Q}^{\rm ex}}_k\Om^i_{{\cal P}^{\rm ex}/S}
%\otimes_{\mab Z}{\mab Q}$ $(i,k \in {\mab N})$. 
%Then, by (\ref{eqn:lueyset}), 
%\begin{equation*} 
%\{P^D_kL^{\rm conv}_{X/S}
%(\Om^i_{{\cal P}^{\rm ex}/S}
%\otimes_{\mab Z}{\mab Q})\}_T
%= \vpl_np'_{n*}q_n^*r^*({\cal M}^i_k).  
%\end{equation*}  
%Because $\vpl_n$ and $p'_{n*}$ are left exact 
%and because $q_n^*$ is left exact by (\ref{lemm:rex}), 
%we have only to note that the morphism 
%$r^*({\cal M}^i_k)  \lo r^*({\cal M}^i_{k+1})$ is injective. 
%For a scheme ${\cal U}$ over $S$, denote 
%${\cal U}\times_ST$ by ${\cal U}_T$. 
%Then the desired injectivity is clear because 
%$r^*({\cal M}^i_k)=
%P^{{\cal P}^{\rm ex}_T/{\cal Q}^{\rm ex}_T}_k
%\Om^i_{{\cal P}^{\rm ex}_T/T} 
%\otimes_{\mab Z}{\mab Q}$. 
\end{rema}

By (\ref{theo:injf}) (2), the family 
$\{P^D_kL^{\rm conv}_{(X,Z)/S}
(\Om^{\bul}_{{\cal P}^{\rm ex}/S}
\otimes_{\mab Z}{\mab Q})\}_{k \in {\mab Z}}$ 
of ${\cal K}_{(X,Z)/S}$-modules 
defines a filtration on the complex
$L^{\rm conv}_{(X,Z)/S}
(\Om^{\bul}_{{\cal P}^{\rm ex}/S}
\otimes_{\mab Z}{\mab Q})$. 
Hence we obtain an object 
$$(L^{\rm conv}_{(X,Z)/S}(\Om^{\bul}_{{\cal P}^{\rm ex}/S}
\otimes_{\mab Z}{\mab Q}), 
\{P^D_k
L^{\rm conv}_{(X,Z)/S}(\Om^{\bul}_{{\cal P}^{\rm ex}/S}
\otimes_{\mab Z}{\mab Q})\}_{k\in {\mab Z}})$$  
in ${\rm C}^+{\rm F}({\cal K}_{(X,Z)/S})$. 
For simplicity of notation, we denote it by 
$$(L^{\rm conv}_{(X,Z)/S}(\Om^{\bul}_{{\cal P}^{\rm ex}/S}
\otimes_{\mab Z}{\mab Q}),P^D).$$

\par
Because it seems to us that the following proposition cannot be obtained immediately 
by the Poincar\'{e} residue isomorphism for the Zariski topos of ${\cal P}^{\rm ex}$, 
the proof of the following is rather long: 

\begin{prop}\label{prop:grla}
Assume that $\os{\circ}{\cal P}$ is affine over $S$. 
Then there exist the following quasi-isomorphisms 
\begin{align*} 
{\rm gr}_k^{P^D}
& L^{\rm conv}_{(X,Z)/S}(\Om^{\bul}_{{\cal P}^{\rm ex}/S}
\otimes_{\mab Z}{\mab Q}) 
\tag{6.9.1}\label{eqn:grpdl}\\ 
{} & \os{\sim}{\lo} 
a^{(k)}_{{\rm conv}*}
(L^{(k){\rm conv}}(\Om^{\bul}_{D^{(k)}({\cal M}^{\rm ex})/S}
\otimes_{\mab Z}{\mab Q})
\otimes_{\mab Z}\vp^{(k)}_{{\rm conv}}(D/S;Z))[-k]\\
{} & \os{\sim}{\longleftarrow}  
a^{(k)}_{{\rm conv}*}
({\cal K}_{(D^{(k)},Z\vert_{D^{(k)}})/S}
\otimes_{\mab Z}\vp^{(k)}_{{\rm conv}}(D/S;Z))[-k]
\end{align*}
in $D^+({\cal K}_{(X,Z)/S})$. 
\end{prop} 
\begin{proof} 
First we define the following morphism 
\begin{align*} 
P^D_kL^{\rm conv}_{(X,Z)/S}(\Om^{\bul}_{{\cal P}^{\rm ex}/S}
\otimes_{\mab Z}{\mab Q}) \lo 
a^{(k)}_{{\rm conv}*}
(L^{(k){\rm conv}}(\Om^{\bul}_{D^{(k)}({\cal M}^{\rm ex})/S}
\otimes_{\mab Z}{\mab Q})
\otimes_{\mab Z}\vp^{(k)}_{{\rm conv}}(D/S;Z))[-k]. 
\tag{6.9.2}\label{ali:pdlkl}
\end{align*} 
Let $r$ be a nonnegative integer such that 
$M_{X,x}/{\cal O}_{X,x}^*\simeq {\mab N}^r$.
By abuse of notation,  
for $1\leq i\leq r$ and 
for any different $i_0,\ldots, i_k$ $(1\leq i_0,\ldots,i_k \leq r)$, 
denote $D(M_{{\cal P}^{\rm ex}})_i$  
and $\os{\circ}{\cal P}{}^{\rm ex}_{i_0}
\cap \cdots \cap \os{\circ}{\cal P}{}^{\rm ex}_{i_k}$ 
by ${\cal P}^{\rm ex}_i$ and 
$\os{\circ}{\cal P}{}^{\rm ex}_{i_0 \cdots i_k}$, respectively.  
Denote $\os{\circ}{D}{}^{(k)}(M_{{\cal P}^{{\rm ex}}})$  
and 
$\vp^{(k)}_{\rm zar}(\os{\circ}{D}(M_{{\cal P}^{\rm ex}}))$ 
by 
$\os{\circ}{\cal P}{}^{{\rm ex},(k)}$ and 
$\vp^{(k)}_{\rm zar}(\os{\circ}{\cal P}{}^{\rm ex}/\os{\circ}{S})$, respectively. 
Denote the natural local exact closed immersion 
$\os{\circ}{\cal P}{}^{\rm ex}_{i_0 \cdots i_k}\os{\sus}{\lo} 
\os{\circ}{\cal P}{}^{\rm ex}$ by $b_{i_0\cdots i_k}$. 
Let 
$b^{(k)} \col \os{\circ}{\cal P}{}^{{\rm ex},(k)}\lo 
\os{\circ}{\cal P}{}^{\rm ex}$ $(k\in {\mab N})$ 
be the natural morphism. 
We have a natural immersion 
$\os{\circ}{X}{}^{(k)} \os{\sus}{\lo} 
\os{\circ}{\cal P}{}^{{\rm ex},(k)}$ 
of (formal) schemes over $\os{\circ}{S}$. 
Because 
$\os{\circ}{X}{}^{(k)}= \os{\circ}{\cal P}{}^{{\rm ex},(k)}$ 
as topological spaces, 
we can identify $\vp^{(k)}_{\rm zar}(\os{\circ}{\cal P}{}^{\rm ex}/\os{\circ}{S})$ 
with $\vp^{(k)}_{\rm zar}(\os{\circ}{X}/\os{\circ}{S}_1)$. 
\par 
Identify the points of $\os{\circ}{T}_n$ 
with those of $\os{\circ}{X}$. 
Identify also the images of the points of $\os{\circ}{X}$ 
in $\os{\circ}{\cal P}{}^{\rm ex}$ with the points of $\os{\circ}{X}$. 
Let $x$ be a point of $\os{\circ}{T}_n$. 
Let 
$P^{{\cal P}^{\rm ex}/{\cal Q}^{\rm ex}}$ be the 
filtrration on ${\cal K}_{T_n}\otimes_{{\cal O}_{{\cal P}^{\rm ex}}}
\Om^{\bul}_{{\cal P}^{\rm ex}/S}$ induced by the filtration 
$P^{{\cal P}^{\rm ex}/{\cal Q}^{\rm ex}}$ on 
$\Om^{\bul}_{{\cal P}^{\rm ex}/S}$.  
Then, for a nonnegative integer $k$, 
there exists the following Poincar\'{e} residue morphism 
\begin{align*} 
{\rm Res}^{{\cal P}^{\rm ex}/{\cal Q}^{\rm ex}}: &
P^{{\cal P}^{\rm ex}/{\cal Q}^{\rm ex}}_k
\Om^{\bul}_{{\cal P}^{\rm ex}/S} \owns 
d\log m_{i_0}\wedge \cdots \wedge d\log m_{i_{k-1}} 
\wedge \omega 
\lom 
\tag{6.9.3}\label{eqn:regrp} \\
&b^*_{i_0\cdots i_{k-1}}(\omega)
\otimes({\rm orientation}~(i_0\cdots i_{k-1})) \\
&\in 
b^{(k)}_*(\Om^{\bul}_{\os{\circ}{\cal P}{}^{{\rm ex},(k)}/S}
\otimes_{\mab Z}
\vp^{(k)}_{\rm zar}(\os{\circ}{\cal P}{}^{{\rm ex},(k)}))[-k]\\
& (\om \in \Om^{\bul -k}_{{\cal P}^{\rm ex}/S})
\end{align*} 
(c.f.~\cite[(4.4.1)]{nh3}). 
\par 
Let $(U',T',\iota',u')$ be an object of 
${\rm Conv}((X,Z)/S)$. 
Set $U'{}^{(k)}:=U'\times_{{\cal Q}^{\rm ex}}(\os{\circ}{\cal P}{}^{{\rm ex},(k)},{\cal N}^{\rm ex})$. 
Consider the following commutative diagram: 
\begin{equation*}
\begin{CD}
\{{\mathfrak T}_{U'{}^{(k)},n}(T'\times_S
(\os{\circ}{\cal P}{}^{{\rm ex},(k)},{\cal N}^{\rm ex}))\}_{n=1}^{\infty} 
@>{\{q_n^{(k)}\}_{n=1}^{\infty}}>> \{T^{(k)}_n\}_{n=1}^{\infty} 
@>{r^{(k)}}>> (\os{\circ}{\cal P}{}^{{\rm ex},(k)},{\cal N}^{\rm ex})\\
@V{\{c^{(k)}_{nT'}\}_{n=1}^{\infty}}VV @V{\{c^{(k)}_n\}_{n=1}^{\infty}}VV @V{b^{(k)}}VV\\
\{{\mathfrak T}_{U',n}(T'\times_S{\cal Q}^{\rm ex})\}_{n=1}^{\infty} 
@>{\{q_n\}_{n=1}^{\infty}}>> \{T_{{\cal Q}^{\rm ex},n}\}_{n=1}^{\infty} 
@>{r}>> {\cal Q}^{\rm ex}\\ 
@V{\{p'_n\}_{n=1}^{\infty}}VV \\
T' @. @.. \\
\end{CD}
\tag{6.9.4}\label{cd:utqsq} 
\end{equation*} 
Here the morphism $r$ is different from the $r$ in the proof of (\ref{theo:injf}). 
The morphism (\ref{eqn:regrp}) induces the following morphism: 
\begin{align*} 
\vpl_np'_{n*}q_n^*r^*({\cal M}^{\bul}_k)\lo 
\vpl_np'_{n*}\{q_n^*r^*b^{(k)}_{*}
(\Om^{\bul}_{\os{\circ}{\cal P}{}^{{\rm ex},(k)}/S})
\otimes_{\mab Z}{\mab Q}\otimes_{\mab Z}
q_n^{-1}r^{-1}b^{(k)}_{*}(\vp^{(k)}_{\rm zar}(\os{\circ}{\cal P}{}^{{\rm ex},(k)}))[-k])\}. 
\tag{6.9.5}\label{ali:qnqsq} 
\end{align*} 
By (\ref{lemm:bcue}) we see that 
\begin{align*}
{\mathfrak T}_{U'{}^{(k)},n}(T'\times_S(\os{\circ}{\cal P}{}^{{\rm ex},(k)},{\cal N}^{\rm ex}))
=
\wt{{\mathfrak T}_{U',n}(T'\times_S{\cal Q}^{\rm ex})\times_{{\cal Q}^{\rm ex}}
(\os{\circ}{\cal P}{}^{{\rm ex},(k)},{\cal N}^{\rm ex})}.
\tag{6.9.6}\label{ali:utnqsq} 
\end{align*} 
Because $b^{(k)}$ is an affine morphism and the right square is cartesian by 
(\ref{lemm:dpinc}), 
$r^*b^{(k)}_{*}=c^{(k)}_{n*}r^{(k)*}$ 
for quasi-coherent ${\cal O}_{\os{\circ}{\cal P}{}^{{\rm ex},(k)}}$-modules
and $c^{(k)}_n$ is an affine morphism; 
because $c^{(k)}_n$ is an affine morphism, 
$q^*_nc^{(k)}_{n*}=c^{(k)}_{nT'*}q_n^{(k)*}$ for 
quasi-coherent ${\cal K}_{T^{(k)}_n}$-modules by (\ref{ali:utnqsq}). 
Moreover, because $b^{(k)}$ is a direct sum of a closed immersion, 
$q_n^{-1}r^{-1}b^{(k)}_{*}=c^{(k)}_{nT'*}q_n^{(k)-1}r^{(k)-1}$. 
Hence the right hand side of (\ref{ali:qnqsq}) is equal to 
\begin{align*} 
&\vpl_n\{c^{(k)}_{nT'*}q_n^{(k)*}r^{(k)*} (\Om^{\bul}_{\os{\circ}{\cal P}{}^{{\rm ex},(k)}/S})
\otimes_{\mab Z}{\mab Q}\otimes_{\mab Z}
c^{(k)}_{nT'*}q_n^{(k)-1}r^{(k)-1}(\vp^{(k)}_{\rm zar}(\os{\circ}{\cal P}{}^{{\rm ex},(k)}))\},
\end{align*} 
which is equal to 
$\{a^{(k)}_{{\rm conv}*}
(L^{(k){\rm conv}}(\Om^{\bul}_{\os{\circ}{\cal P}{}^{{\rm ex},(k)}/S}
\otimes_{\mab Z}{\mab Q})
\otimes_{\mab Z}\vp^{(k)}_{{\rm conv}}(D/S;Z))[-k]\}_{T'}$. 
Here we have used (\ref{linc}). 
Because  
$P^{{\cal P}^{\rm ex}/{\cal Q}^{\rm ex}}_{k-1}
\Om^{\bul}_{{\cal P}^{\rm ex}/S}$ is killed by the morphism (\ref{eqn:regrp}), 
the morphism  (\ref{ali:qnqsq}) gives us the following morphism 
\begin{align*} 
{\rm gr}_k^{P^D}
& L^{\rm conv}_{(X,Z)/S}(\Om^{\bul}_{{\cal P}^{\rm ex}/S}
\otimes_{\mab Z}{\mab Q}) 
\tag{6.9.7}\label{eqn:grepdl}\\ 
{} & {\lo} 
a^{(k)}_{{\rm conv}*}
(L^{(k){\rm conv}}(\Om^{\bul}_{\os{\circ}{\cal P}{}^{{\rm ex},(k)}/S}
\otimes_{\mab Z}{\mab Q})
\otimes_{\mab Z}\vp^{(k)}_{{\rm conv}}(D/S;Z))[-k].
\end{align*} 
\par 
We claim that this is an isomorphism.  Indeed, this is a local problem. 
To prove this, we fix a total order on 
the smooth components of $\os{\circ}{\cal P}{}^{{\rm ex},(k)}/S$. 
Then we can find an isomorphism 
$\vp^{(k)}_{{\rm conv}}(D/S;Z)\simeq {\mab Z}$. 
By replacing ${\cal P}$ with ${\cal P}^{\rm loc,ex}$ 
and ${\cal R}$ with ${\cal R}^{\rm loc,ex}$, respectively,  
we may assume that 
the immersions $(X,D\cup Z)\os{\sus}{\lo}{\cal P}$ 
and $(X,Z)\os{\sus}{\lo}{\cal R}$ are exact. 
Then, as in the existence of 
the Poincar\'{e} residue isomorphism with respect to 
${\cal P}/{\cal R}$ (\cite[(4.4)]{nh3}, cf.~\cite[(2.2.21)]{nh2}), 
we have the following isomorphism 
\begin{align*} 
{\rm Res}^{{\cal P}^{\rm ex}/{\cal Q}^{\rm ex}}: \quad
&  
{\rm gr}_k^{P^{{\cal P}^{\rm ex}/{\cal Q}^{\rm ex}}} 
({\cal K}_{T_n}\otimes_{{\cal O}_{{\cal P}^{\rm ex}}}
\Om^{\bul}_{{\cal P}^{\rm ex}/S}) 
=
{\rm gr}_k^{P^{{\cal P}/{\cal R}}\otimes_{{\cal O}_{\cal P}}
{\cal O}_{{\cal P}^{\rm ex}}} 
({\cal K}_{T_n}\otimes_{{\cal O}_{\cal P}}
\Om^{\bul}_{{\cal P}/S})
\tag{6.9.8}\label{eqn:regarp} \\
& \os{\sim}{\lo} c^{(k)}_{n*}({\cal K}_{T_n^{(k)}}
\otimes_{{\cal O}_{\os{\circ}{\cal P}{}^{{\rm ex},(k)}}}
\Om^{\bul}_{\os{\circ}{\cal P}{}^{{\rm ex},(k)}/S}
\otimes_{\mab Z}
\vp^{(k)}_{\rm zar}(\os{\circ}{\cal P}{}^{{\rm ex},(k)})[-k])\\
&\os{\sim}{\lo} 
c^{(k)}_{n*}({\cal K}_{T_n^{(k)}}
\otimes_{{\cal O}_{\os{\circ}{\cal P}{}^{{\rm ex},(k)}}}
\Om^{\bul}_{\os{\circ}{\cal P}{}^{{\rm ex},(k)}/S})[-k]. 
\end{align*} 
Here we have used (\ref{eqn:tenexsfil}) and (\ref{lemm:dpinc}) for 
the equality in (\ref{eqn:regarp}) and 
the first isomorphism in (\ref{eqn:regarp}), respectively. 
Hence, by (\ref{theo:injf}) (2) and (\ref{prop:afex}), 
we have the following equalities: 
\begin{align*}
& {\rm gr}_k^{P^D}
L^{\rm conv}_{(X,Z)/S}(\Om^{\bul}_{{\cal P}^{\rm ex}/S}
\otimes_{\mab Z}{\mab Q})   =
L^{\rm conv}_{(X,Z)/S}
({\rm gr}_k^{P^{{\cal P}^{\rm ex}/{\cal Q}^{\rm ex}}}
\Om^{\bul}_{{\cal P}^{\rm ex}/S}
\otimes_{\mab Z}{\mab Q}) \\ 
{} & 
\os{\sim}{=}
L^{\rm conv}_{(X,Z)/S}(b_*^{(k)}
(\Om^{\bul}_{\os{\circ}{\cal P}{}^{{\rm ex},(k)}/S}
\otimes_{\mab Z}{\mab Q}))[-k]. 
\end{align*}
By (\ref{linc})  
this complex is equal to 
$$a^{(k)}_{{\rm conv}*}
L^{(k){\rm conv}}(\Om^{\bul}_{\os{\circ}{\cal P}{}^{{\rm ex},(k)}/S}
\otimes_{\mab Z}{\mab Q})[-k].$$    
In conclusion, we obtain the first isomorphism in (\ref{eqn:grpdl}). 
\par 
By (\ref{theo:pl}) we obtain
the second quasi-isomorphism in (\ref{eqn:grpdl}).
\end{proof}

\begin{rema}\label{rem:gtraff}
By using the (filtered) cohomological descent 
(\cite[Lemma 1.5.1]{nh2}),  the commutativity of 
the diagram in [loc.~cit., Lemma 1.3.4.1] and (\ref{prop:rescos}) below, 
we can easily eliminate the assumption of the affineness 
in (\ref{prop:grla}). 
%Let $\bigcup_{i\in I}{\cal P}_i={\cal P}$ be 
%an affine open covering of ${\cal P}$ such that 
%the image of ${\cal P}_i$ in $S$ is contained in 
%an affine open subscheme of $S$. 
%Set $(X_i,D_i\cup Z_i)
%:=(X,D\cup Z)\times_{\cal P}{\cal P}_i$, 
%$(X_0,D_0\cup Z_0)
%:=\coprod_{i\in I}(X_i,D_i\cup Z_i)$ and 
%${\cal P}_0:=\coprod_{i\in I}{\cal P}_i$.   
%Set 
%$(X_{\bul},D_{\bul}\cup Z_{\bul})
%:={\rm cosk}^{(X,D\cup Z)}_0(X_0,D_0\cup Z_0)$ 
%and ${\cal P}_{\bul}:={\rm cosk}^S_0({\cal P}_0)$.  
\end{rema}

For simplicity of notation,  set 
$$({\cal K}_{\os{\to}{T}}\otimes_{{\cal O}_{{\cal P}^{\rm ex}}}
\Om^{\bul}_{{\cal P}^{\rm ex}/S},P^D):=
({\cal K}_{\os{\to}{T}}\otimes_{{\cal O}_{{\cal P}^{\rm ex}}}
\Om^{\bul}_{{\cal P}^{\rm ex}/S}, 
\{{\cal K}_{\os{\to}{T}}\otimes_{{\cal O}_{{\cal P}^{\rm ex}}}
P^{{\cal P}^{\rm ex}/{\cal Q}^{\rm ex}}_k
\Om^{\bul}_{{\cal P}^{\rm ex}/S}\}_{k\in {\mab Z}}).$$

\begin{prop}\label{prop:ulcrz}
Let $\bet_n \col X_n \lo X$ be a natural morphism 
and 
identify $(T_n)_{{\rm zar}}$ with $(X_n)_{\rm zar}$. 
Then
\begin{equation*}
Ru^{\rm conv}_{(X,Z)/S*}
(L^{\rm conv}_{(X,Z)/S}
(\Om^{\bul}_{{\cal P}^{\rm ex}/S}
\otimes_{\mab Z}{\mab Q}),P^D)
=(\vpl_n \bet_n)({\cal K}_{\os{\to}{T}}
\otimes_{{\cal O}_{{\cal P}^{\rm ex}}}
\Om^{\bul}_{{\cal P}^{\rm ex}/S},P^D)
\tag{6.11.1}\label{eqn:uxzl}
\end{equation*}
in ${\rm D}^+{\rm F}(f^{-1}({\cal K}_S))$.
\end{prop}
\begin{proof}
By the obvious relative version of 
\cite[Corollary 2.3.4]{s2} which is the log version of 
\cite[Corollary 4.4]{oc},  
the left hand side of (\ref{eqn:uxzl})
is equal to 
\begin{align*} 
& (u^{\rm conv}_{(X,Z)/S*}L^{\rm conv}_{(X,Z)/S}
(\Om^{\bul}_{{\cal P}^{\rm ex}/S}\otimes_{\mab Z}{\mab Q}), 
\{u^{\rm conv}_{(X,Z)/S*}L^{\rm conv}_{(X,Z)/S}
(P^{{\cal P}^{\rm ex}/{\cal Q}^{\rm ex}}_k
\Om^{\bul}_{{\cal P}^{\rm ex}/S} 
\otimes_{\mab Z}{\mab Q}))\}_{k\in {\mab Z}}). 
\end{align*}  
The rest of the proof is the same as that of (\ref{prop:cdfza}). 
%\par
%For an ${\cal K}_{\cal P}$-module ${\cal F}$, we have
%$$u^{\rm conv}_{(X,Z)/S*}L^{\rm conv}_{(X,Z)/S}({\cal F})
%= u^{\rm conv}_{(X,Z)/S*}j_{T*}{\varphi}^*g^* ({\cal F})
%= {\varphi}_{\os{\to}{T}*}{\varphi}^*_{\os{\to}{T}}g^* ({\cal F}) 
%= \vpl_n\bet_{n*}({\cal K}_{T_n}
%\otimes_{{\cal K}_{\cal P}}{\cal F})$$
%by the commutativity of the diagram of topoi 
%in \cite[p.~91]{s2}, which is a log version of that 
%of topoi in \cite[p.~146]{oc}.
%Hence
%$$u^{\rm conv}_{(X,Z)/S*}(L^{\rm conv}_{(X,Z)/S}
%(P^{\cal D}_k\Om^{\bul}_{{\cal P}/S}
%(\log({\cal D}\cup {\cal Z})))\otimes_{\mab Z}{\mab Q})= 
%{\cal O}_T\otimes_{{\cal O}_{{\cal P}}} 
%P^{\cal D}_k\Om^{\bul}_{{\cal X}/S}
%(\log({\cal D}\cup {\cal Z}))\otimes_{\mab Z}{\mab Q}. $$ 
%Thus (\ref{prop:ulcrz}) follows. 
\end{proof} 

\parno 
The following is a log convergent analogue of 
\cite[(4.8)]{nh3}: 

\begin{prop}\label{prop:rescos}
Let ${\cal V}'$, $\pi'$, $\kap'$ and $S'$ 
be similar objects to ${\cal V}$, $\pi$ $\kap$ and $S$
as in {\rm \S\ref{sec:logcd}}, respectively. 
Let 
\begin{equation*} 
(X',D') \os{\sus}{\lo} `{\cal R}'
\quad {\rm and} \quad 
(X',Z') \os{\sus}{\lo} {\cal R}' 
\end{equation*}  
be similar closed immersions to 
\begin{equation*} 
(X,D) \os{\sus}{\lo} `{\cal R}
\quad {\rm and} \quad 
(X,Z) \os{\sus}{\lo} {\cal R},
\end{equation*}  
Let ${\cal P}'$ and ${\cal Q}'$
be similar log formal schemes to  
${\cal P}$ and ${\cal Q}$, respectively.  
Let ${\cal M}'{}^{\rm ex}$ be the similar log structure 
to ${\cal M}^{\rm ex}$. 
Let 
\begin{equation*} 
\begin{CD} 
(X',D'\cup Z') @>{\subset}>> {\cal P}' \\ 
@V{\eps_{(X',D'\cup Z',Z')/S}}VV @VVV \\ 
(X',Z') @>{\subset}>> {\cal R}' 
\end{CD} 
\tag{6.12.1}\label{eqn:pvpq} 
\end{equation*}
be a similar diagram to $(\ref{eqn:pvq})$ over 
a $p$-adic formal ${\cal V}'$-scheme $S'$.   
%Assume that there exists an fs sub log structure 
%${\cal M}'$ of $M_{{\cal P}'}$ such that 
%\begin{equation*} 
%M_{{\cal P}'}=
%{\cal M}'\oplus_{{\cal O}_{{\cal P}'}^*}M_{{\cal Q}'},  
%\tag{6.14.2}\label{eqn:mppql} 
%\end{equation*} 
%such that the pull-back of ${\cal M}'$ to 
%$X'$ is equal to $M(D')$ and such that 
%the condition {\rm (\ref{eqn:epynr})} is satisfied for 
%$(\os{\circ}{{\cal P}'},{\cal M}')$. 
Let $a'{}^{(k)}\col (D'^{(k)},Z\vert_{D'^{(k)}})
\lo (X',Z')$ $(k\in {\mab N})$
be the natural morphism of log schemes over $S'_1$.  
Let $L'^{(k)}$ be the log linearization functor for 
coherent ${\cal K}_{D^{(k)}({\cal M}'{}^{\rm ex})}$-modules. 
Assume that there exists a morphism 
from the commutative diagram {\rm (\ref{eqn:pvq})} to 
the commutative diagram {\rm (\ref{eqn:pvpq})} over $S\lo S'$. 
%Assume also that the morphism 
%$g\col (\os{\circ}{\cal P},{\cal M})\lo 
%(\os{\circ}{\cal P}{}',{\cal M}')$ 
%satisfies the condition in {\rm (\ref{prop:mmoo})}.  
Let 
${\mathfrak g}\col 
(\os{\circ}{\cal P}{}^{\rm ex},{\cal M}^{\rm ex})\lo 
(\os{\circ}{\cal P}{}'^{\rm ex},{\cal M}{}'^{\rm ex})$ 
and 
${\mathfrak g}_{(X,Z)}\col (X,Z)\lo (X',Z')$ 
be the induced morphisms. 
Assume that, for each point $x\in \os{\circ}{\cal P}{}^{\rm ex}$ 
and for each member $m$ of the minimal generators 
of ${\cal M}^{\rm ex}_{x}/{\cal O}^*_{{\cal P}^{\rm ex},x}$, 
there exists a unique member $m'$ of the minimal generators of 
${\cal M}'{}^{\rm ex}_{\os{\circ}{g}(x)}
/{\cal O}^*_{{\cal P}'{}^{\rm ex},\os{\circ}{g}(x)}$ 
such that ${\mathfrak g}^*(m')= m$ and such that the image of 
the other minimal generators of 
${\cal M}'{}^{\rm ex}_{\os{\circ}{\mathfrak g}(x)}
/{\cal O}^*_{{\cal P}'{}^{\rm ex},\os{\circ}{\mathfrak g}(x)}$ by ${\mathfrak g}^*$ 
are the trivial element of 
${\cal M}^{\rm ex}_{x}/{\cal O}^*_{{\cal P}^{\rm ex},x}$. 
%Assume that, for each point $x\in \os{\circ}{\cal P}$ 
%and for each member $m$ of the minimal generators 
%of ${\cal M}_x^{\rm ex}/{\cal O}^*_{{\cal P}^{\rm ex},x}$, 
%there exists 
%a unique member $m'$ of the minimal generators of 
%${\cal M}{}'^{\rm ex}_{\os{\circ}{g}(x)}
%/{\cal O}^*_{{\cal P}{}'^{\rm ex},\os{\circ}{g}(x)}$ 
%such that $g^*(m')= m$ and such that the image of 
%the other minimal generators of 
%${\cal M}{}'^{\rm ex}_{\os{\circ}{g}(x)}
%/{\cal O}^*_{{\cal P}{}'^{\rm ex},\os{\circ}{g}(x)}$ by $g^*$ 
%are the trivial element of 
%${\cal M}^{\rm ex}_{x}/{\cal O}^*_{{\cal P}^{\rm ex},x}$. 
Let $b'{}^{(k)} \col D^{(k)}({\cal M}{}'^{\rm ex})
\lo \os{\circ}{\cal P}{}{}'^{\rm ex}$ 
be the similar morphism to $b^{(k)}$. 
Let $\{T'_n\}_{n=1}^{\infty}$ 
$($resp.~$\{T'{}^{(k)}_n\}_{n=1}^{\infty})$ 
be the system of the log universal enlargements of 
the immersion $(X',D'\cup Z')\os{\sus}{\lo}{\cal P}'$ 
$($resp.~$(D'{}^{(k)}, Z'\vert_{D'{}^{(k)}}) \os{\subset}{\lo} 
(D^{(k)}({\cal M}'{}^{\rm ex}), {\cal N}'{}^{\rm ex}))$
and 
let $c'{}^{(k)}_n \col T'{}^{(k)}_n \lo T'_n$
be the morphism  induced by $b'{}^{(k)}$. 
Then the morphisms 
${\cal P}\lo {\cal P}'$ and  
${\cal R}\lo {\cal R}'$ give us 
the following commutative diagram$:$
\begin{equation*} 
\begin{CD}  
{\rm Res}: 
{\mathfrak g}_{(X,Z)_{\rm conv}*}L^{\rm conv}_{(X,Z)/S}
({\rm gr}_k^{P^{{\cal P}^{\rm ex}/{\cal Q}^{\rm ex}}} 
\Om^{\bul}_{{\cal P}^{\rm ex}/S}) 
@>{\sim}>> \\
@AAA  \\
{\rm Res}: 
L^{\rm conv}_{(X',Z')/S'}
({\rm gr}_k^{P^{{\cal P}'{}^{\rm ex}/{\cal Q}'{}^{\rm ex}}} 
\Om^{\bul}_{{\cal P}'{}^{\rm ex}/S'}) 
@>{\sim}>>
\end{CD}
\tag{6.12.2}\label{eqn:relcmrp} 
\end{equation*} 
\begin{equation*} 
\begin{CD}  
{\mathfrak g}_{(X,Z)_{\rm conv}*}
a^{(k)}_{{\rm conv}*}
(L^{(k){\rm conv}}(\Om^{\bul}_{D^{(k)}({\cal M}^{\rm ex})/S}
\otimes_{\mab Z}{\mab Q})
\otimes_{\mab Z}\vp^{(k)}_{\rm conv}(D/S;Z))[-k]\\
@AAA \\
a'{}^{(k)}_{{\rm conv}*}
(L'{}^{(k){\rm conv}}
(\Om^{\bul}_{D^{(k)}({\cal M}{}'^{\rm ex})/S'}
\otimes_{\mab Z}{\mab Q})
\otimes_{\mab Z}\vp^{(k)}_{\rm conv}
(D'/S';Z'))[-k].  
\end{CD}
\end{equation*} 
\end{prop}
\begin{proof} 
We may assume that the immersions  
$(X,D\cup Z)\os{\sus}{\lo} {\cal P}$ 
and 
$(X',D'\cup Z')\os{\sus}{\lo} {\cal P}'$ 
are exact and closed. 
Let ${\mathfrak g}_n\col T_n\lo T'_n$ be 
the induced morphism by the morphism ${\cal P}\lo {\cal P}'$. 
Then, as in \cite[(4.8)]{nh3}, 
we obtain the following isomorphism 
\begin{equation*} 
\begin{CD}  
{\mathfrak g}_{n*}({\rm Res}^{{\cal P}/{\cal Q}}): 
{\mathfrak g}_{n*}({\rm gr}_k^{P^{{\cal P}/{\cal Q}}}
({\cal K}_{T_n}\otimes_{{\cal O}_{\cal P}}
\Om^{\bul}_{{\cal P}/S}))
@>{\sim}>> \\
@AAA \\
{\rm Res}^{{\cal P}'/{\cal Q}'}: 
{\rm gr}_k^{P^{{\cal P}'/{\cal Q}'}} 
({\cal K}_{T'_n}
\otimes_{{\cal O}_{{\cal P}'}}\Om^{\bul}_{{\cal P}'/S'})
@>{\sim}>> 
\end{CD}
\tag{6.12.3}\label{eqn:regdmrp} 
\end{equation*} 
\begin{equation*} 
\begin{CD}  
{\mathfrak g}_{n*}c^{(k)}_{n*}
({\cal K}_{T_n^{(k)}}
\otimes_{{\cal O}_{{\cal P}^{\rm ex}}}
\Om^{\bul}_{D^{(k)}({\cal M}^{\rm ex})/S}
\otimes_{\mab Z}
\vp^{(k)}_{\rm zar}(D^{(k)}({\cal M}^{\rm ex}))[-k])\\ 
@AAA\\
c'{}^{(k)}_{n*}(
{\cal K}_{T'{}^{(k)}_n}
\otimes_{{\cal O}_{{\cal P}{}'^{{\rm ex}}}}
\Om^{\bul}_{D^{(k)}({\cal M}{}'^{\rm ex})/S'}
\otimes_{\mab Z}
\vp^{(k)}_{\rm zar}(D^{(k)}({\cal M}'{}^{\rm ex}))[-k]). 
\end{CD}
\end{equation*} 
Hence we obtain (\ref{prop:rescos}) 
by the proof of (\ref{prop:grla}). 
\end{proof}

\section{(Bi)simplicial log convergent topoi}\label{sec:rlct}
In this section we give the notations of 
(bi)simplicial log convergent topoi and 
(bi)simplicial Zariski topoi for later sections. 
\par 
Let the notations be as in the beginning of 
\S\ref{sec:logcd}.  
Let $Y=\us{i\in I}{\bigcup}Y_i$ 
be an open covering of $Y$, where 
$I$ is a (not necessarily finite) set. 
%We fix a total order on $I_0$. 
%Let $I$ be a category whose objects are 
%$i=(i_0,\ldots, i_r)$'s 
%$(i_0<i_1<\cdots < i_r, r \in {\mab Z}_{\geq 0})$.
%Set $\{i\}:=\{i_0,\ldots, i_r\}$. 
%For two objects $i'$, $i\in {I}$, a 
%morphism from $i'$ to $i$ is, by definition, 
%the inclusion $\{i'\}\os{\subset}{\lo} \{i\}$. 
%For $i=(i_0, \ldots, i_r)\in {I}$ $(i_j \in I_0)$, set 
%$Y_i:=Y_{i_{0}}\cap \cdots \cap Y_{i_{r}}$.
%If there is a morphism $\al \col i' \lo i$, the open immersion
%$Y_i \lo Y_{i'}$ induces a morphism 
%$$\ul{\al} \col 
%(\wt{(Y_{i}/S})_{\rm conv},{\cal K}_{Y_{i}/S}) \lo 
%(\wt{(Y_{i'}/S})_{\rm conv},{\cal K}_{Y_{i'}/S})$$
%of ringed topoi.
Set $Y_0:=\coprod_{i\in I}Y_i$ and 
$Y_n:= {\rm cosk}_0^Y(Y_0)_n:= 
\underset{n~{\rm pieces}}{\underbrace{
Y_0\times_Y \cdots \times_YY_0}}$.  
Then we have the simplicial log scheme $Y_{\bul}=
Y_{\bul \in {\mab N}}$ over $Y$. 
We also have the ringed topos 
$(({Y_{\bul}/S})_{\rm conv},{\cal K}_{Y_{\bul}/S})$. 
Let 
${\pi}_{\rm conv} \col  
({Y_{\bul}/S})_{\rm conv} \lo 
({Y/S})_{\rm conv}$ 
be the natural morphism of topoi.
The morphism
${\pi}_{\rm conv}$ induces the following morphism 
\begin{equation*} 
\pi_{\rm conv} \col
(({Y_{\bul}/S})_{\rm conv},{\cal K}_{Y_{\bul}/S})
\lo 
(({Y/S})_{\rm conv},{\cal K}_{Y/S}) 
\tag{7.0.1}\label{eqn:pidf}
\end{equation*}
of ringed topoi and a morphism of filtered derived categories:
\begin{equation*}
R{\pi}_{{\rm conv}*} \col
{\rm D}^+{\rm F}({\cal K}_{Y_{\bul}/S}) \lo 
{\rm D}^+{\rm F}({\cal K}_{Y/S}).
\tag{7.0.2}\label{eqn:rtalc}
\end{equation*}
\par
Assume that we are given another open covering 
$\{Y_j\}_{j\in J}$ of $Y$. 
Set $Y'_0:=\coprod_{j\in J}Y_j$ 
and 
$Y_{mn}:={\rm cosk}_0^Y(Y_0)_m\times_Y
{\rm cosk}_0^Y(Y'_0)_n$. 
Then we have a bisimplicial log scheme 
$Y_{\bul \bul}$, a bisimplicial ringed topos 
$(({Y_{\bul \bul}/S})_{\rm conv},
{\cal K}_{Y_{\bul \bul}/S})$    
and the following natural  morphism of ringed topoi 
\begin{equation*} 
\eta_{\rm conv} \col 
(({Y_{\bul \bul}/S})_{\rm conv},
{\cal K}_{Y_{\bul \bul}/S}) \lo 
(({Y_{\bul}/S})_{\rm conv},
{\cal K}_{Y_{\bul}/S}).  
\tag{7.0.3}\label{eqn:edf}
\end{equation*}
We have the following two morphisms of 
filtered derived categories:
\begin{equation*}
R\eta_{{\rm conv}*} \col
{\rm D}^+{\rm F}({\cal K}_{Y_{\bul \bul}/S}) \lo 
{\rm D}^+{\rm F}({\cal K}_{Y_{\bul}/S}),
\tag{7.0.4}\label{eqn:retlcds}
\end{equation*}
\begin{equation*}
R\eta_{m,{\rm conv}*} \col
{\rm D}^+{\rm F}({\cal K}_{Y_{m \bul}/S}) \lo 
{\rm D}^+{\rm F}({\cal K}_{Y_m/S}) \quad (m\in {\mab N}).
\tag{7.0.5}\label{eqn:siretlc}
\end{equation*}
\par
By using the composite morphisms $f_i \col 
\os{\circ}{Y}_i \os{\subset}{\lo} \os{\circ}{Y} 
\os{\os{\circ}{f}}{\lo} \os{\circ}{S}$ $(i\in I)$, we 
also have an analogous simplicial ringed topos 
$({\os{\circ}{Y}}_{\bul{\rm zar}},
f^{-1}_{\bul}({\cal K}_S))$ and 
an analogous morphism 
\begin{equation*} 
{\pi}_{\rm zar} \col ({\os{\circ}{Y}}_{\bul {\rm zar}}, 
f^{-1}_{\bul}({\cal K}_S)) \lo 
({\os{\circ}{Y}}_{\rm zar}, f^{-1}({\cal K}_S)) 
\tag{7.0.6}\label{eqn:pizs}
\end{equation*} 
of ringed topoi.
We have the following 
three morphisms of filtered derived categories:
\begin{equation*}
R{\pi}_{{\rm zar}*} \col
{\rm D}^+{\rm F}(f^{-1}_{\bul}({\cal K}_S)) \lo 
{\rm D}^+{\rm F}(f^{-1}({\cal K}_S)), 
\tag{7.0.7}\label{eqn:rtazr}
\end{equation*}
\begin{equation*}
R\eta_{{\rm zar}*} \col
{\rm D}^+{\rm F}(f^{-1}_{\bul \bul}({\cal K}_S)) \lo 
{\rm D}^+{\rm F}(f^{-1}_{\bul}({\cal K}_S)),
\tag{7.0.8}\label{eqn:retz}
\end{equation*}
\begin{equation*}
R\eta_{m,{\rm zar}*} \col
{\rm D}^+{\rm F}(f^{-1}_{m\bul}({\cal K}_S)) \lo 
{\rm D}^+{\rm F}(f^{-1}_m({\cal K}_S)).
\tag{7.0.9}\label{eqn:retizr}
\end{equation*}
Let $u^{\rm conv}_{Y/S} \col 
((Y/S)_{\rm conv}, {\cal K}_{Y/S})
\lo ({\os{\circ}{Y}}_{{\rm zar}},f^{-1}({\cal K}_S))$ 
be the projection in (\ref{eqn:uwksr}). 
We also have the following natural projections
\begin{equation*}
u^{\rm conv}_{Y_{\bul}/S} 
\col 
(({Y_{\bul}/S})_{\rm conv}, {\cal K}_{Y_{\bul}/S})
\lo ({\os{\circ}{Y}}_{\bul{\rm zar}},f^{-1}_{\bul}({\cal K}_S)),
\tag{7.0.10}\label{eqn:uysr}
\end{equation*}
\begin{equation*}
u^{\rm conv}_{Y_{\bul \bul}/S} \col 
(({Y_{\bul \bul}/S})_{\rm conv},
{\cal K}_{Y_{\bul \bul}/S})
\lo ({\os{\circ}{Y}}_{\bul \bul{\rm zar}},
f^{-1}_{\bul \bul}({\cal K}_S)).
\tag{7.0.11}\label{eqn:uysrdb}
\end{equation*}
We have the following commutative diagrams: 
\begin{equation*}
\begin{CD}
{\rm D}^+{\rm F}({\cal K}_{Y_{\bul}/S}) 
@>{R\pi_{\rm conv*}}>> 
{\rm D}^+{\rm F}({\cal K}_{Y/S})  \\ 
@V{Ru^{\rm conv}_{Y_{\bul}/S*}}VV 
@VV{Ru^{\rm conv}_{Y/S*}}V \\
{\rm D}^+{\rm F}(f^{-1}_{\bul}({\cal K}_S))  
@>{R\pi_{\rm zar*}}>> 
{\rm D}^+{\rm F}(f^{-1}({\cal K}_S))
\end{CD}
\tag{7.0.12}\label{eqn:ubys}
\end{equation*}
and 
\begin{equation*}
\begin{CD}
{\rm D}^+{\rm F}({\cal K}_{Y_{\bul \bul}/S}) 
@>{R\eta_{\rm conv*}}>> 
{\rm D}^+{\rm F}({\cal K}_{Y_{\bul}/S})  \\ 
@V{Ru^{\rm conv}_{Y_{\bul \bul}/S*}}VV 
@VV{Ru^{\rm conv}_{Y_{\bul}/S*}}V \\
{\rm D}^+{\rm F}(f^{-1}_{\bul \bul}({\cal K}_S))  
@>{R\eta_{\rm zar*}}>>  
{\rm D}^+{\rm F}(f^{-1}_{\bul}({\cal K}_S)). 
\end{CD}
\tag{7.0.13}\label{eqn:retuys}
\end{equation*}
\par
Assume that we are given the commutative diagram 
(\ref{cd:yypssp}) and open coverings 
$\{Y_i\}_{i \in I}$ of $Y$ and 
$\{Y'_i\}_{i \in I}$ of $Y'$, respectively, 
such that $g$ induces a morphism 
$g_i \col Y_i \lo Y_i'$ of log schemes. 
Let
$\pi_{Y{\rm conv}}$ and 
$\pi_{Y'{\rm conv}}$ 
be morphisms of (ringed) topoi defined in (\ref{eqn:pidf})  
for $Y/S$ and $Y'/S'$, respectively. 
Let ${\pi}_{Y{\rm zar}}$ and ${\pi}_{Y'{\rm zar}}$ 
be morphisms of (ringed) topoi defined in 
(\ref{eqn:pizs}) for $Y/S$ and $Y'/S'$, respectively. 
Then the family $\{g_i\}_{i\in I}$ 
induces the natural morphism 
$g_{\bul}\col Y_{\bul} \lo Y'_{\bul}$
of simplicial log schemes over the morphism $Y\lo Y'$. 
Hence we obtain the following commutative diagram:
\begin{equation*}
\begin{CD}
{\rm D}^+{\rm F}({\cal K}_{Y_{\bul}/S}) 
@>{Rg_{\bul {\rm conv}*}}>> 
{\rm D}^+{\rm F}({\cal K}_{Y_{\bul}'/S'}) \\ 
@V{R{\pi}_{Y{\rm conv}*}}VV 
@VV{R{\pi}_{Y'{\rm conv}*}}V \\
{\rm D}^+{\rm F}({\cal K}_{Y/S}) 
@>{Rg_{{\rm conv}*}}>>  
{\rm D}^+{\rm F}({\cal K}_{Y'/S'}).
\end{CD}
\tag{7.0.14}\label{cd:rgrtc}
\end{equation*} 
We also have the following commutative diagram 
\begin{equation*}
\begin{CD}
{\rm D}^+{\rm F}(f^{-1}_{\bul}({\cal K}_S)) 
@>{Rg_{\bul{\rm zar}*}}>> 
{\rm D}^+{\rm F}(f'{}^{-1}_{\bul}({\cal K}_{S'})) \\ 
@V{R{\pi}_{Y{\rm zar}*}}VV @VV{R{\pi}_{Y'{\rm zar}*}}V \\
{\rm D}^+{\rm F}(f^{-1}({\cal K}_S)) @>{Rg_{{\rm zar}*}}>>  
{\rm D}^+{\rm F}(f'{}^{-1}({\cal K}_{S'})).
\end{CD}
\tag{7.0.15}\label{cd:rgztac}
\end{equation*}
%We leave the analogous formulations of the 
%left derived functors to the reader. 

%\par 
%For later sections we need 
%the diagrams $\eps_{{\rm conv}\bul}$ and 
%$\eps_{{\rm conv}\bul \bul}$ of $\eps_{\rm conv}$. 
\par 
Let $M$ be the log structure of $Y$. 
Let $N$ be the sub log structure of $M$. 
Let $M_i$ (resp.~$N_i$) be the pull-back of 
$M$ (resp.~$N$) to $Y_i$ $(i\in I)$.  
The morphism 
$\eps \col Y=(\os{\circ}{Y},M) 
\lo (\os{\circ}{Y},N)$ 
forgetting the structure $M\setminus N$ 
induces the following morphisms of ringed topoi: 
\begin{equation*}
\eps_{{\rm conv}{\bul}} \col 
(({(Y_{\bul},M_{\bul})/S})_{\rm conv},
{\cal K}_{(Y_{\bul},M_{\bul})/S}) 
\lo 
(({(Y_{\bul},N_{\bul})/S})_{\rm conv},
{\cal K}_{(Y_{\bul},N_{\bul})/S}),  
\tag{7.0.16}\label{eqn:ymnb}
\end{equation*} 
\begin{equation*}
\eps_{{\rm conv}{\bul \bul}} \col 
(({(Y_{\bul \bul},M_{\bul \bul})/S})_{\rm conv},
{\cal K}_{(Y_{\bul \bul},M_{\bul \bul})/S}) 
\lo 
(({(Y_{\bul \bul},N_{\bul \bul})/S})_{\rm conv},
{\cal K}_{(Y_{\bul \bul},N_{\bul \bul})/S}).  
\tag{7.0.17}\label{eqn:yddb}
\end{equation*} 
We have the following commutative diagram
\begin{equation*}
\begin{CD} 
(({(Y_{\bul},M_{\bul})/S})_{\rm conv},
{\cal K}_{(Y_{\bul},M_{\bul})/S}) 
@>{\eps_{{\rm conv}\bul}}>> 
(({(Y_{\bul},N_{\bul})/S})_{\rm conv},
{\cal K}_{(Y_{\bul},N_{\bul})/S})\\ 
@V{{\pi}_{M{\rm conv}}}VV  
@VV{{\pi}_{N{\rm conv}}}V \\
(({(Y,M)/S})_{\rm conv},{\cal K}_{(Y,M)/S}) 
@>{\eps_{\rm conv}}>> 
(({(Y,N)/S})_{\rm conv},{\cal K}_{(Y,M)/S}).
\end{CD}
\tag{7.0.18}\label{cd:epep}
\end{equation*}
Here 
${\pi}_{M{\rm conv}}$ and ${\pi}_{N{\rm conv}}$ 
are the morphisms of ringed topoi defined in (\ref{eqn:pidf});
we have written the symbols $M$ and $N$ in 
subscripts for clarity. 
\par 
We also have the following commutative diagram:
\begin{equation*}
\begin{CD} 
(({(Y_{\bul \bul},M_{\bul \bul})/S})_{\rm conv}, 
{\cal K}_{(Y_{\bul \bul},M_{\bul \bul})/S}) 
@>{\eps_{{\rm conv} \bul \bul}}>> 
(({(Y_{\bul \bul},N_{\bul \bul})/S})_{\rm conv},
{\cal K}_{(Y_{\bul \bul},N_{\bul \bul})/S})\\ 
@V{{\eta}_{M{\rm conv}}}VV  
@VV{{\eta}_{N{\rm conv}}}V \\
(({(Y_{\bul},M_{\bul})/S})_{\rm conv}, 
{\cal K}_{(Y_{\bul},M_{\bul})/S}) @>{\eps_{{\rm conv}\bul}}>> 
(({(Y_{\bul},N_{\bul})/S})_{\rm conv},
{\cal K}_{(Y_{\bul},N_{\bul})/S})).
\end{CD}
\tag{7.0.19}\label{cd:etmnep}
\end{equation*}
Here ${\eta}_{M{\rm conv}}$ 
and ${\eta}_{N{\rm conv}}$ 
are the morphisms of ringed topoi defined in (\ref{eqn:edf}). 
\par 
Let 
$u^{\rm conv}_{(Y_{\bul},L_{\bul})/S}$ and 
$u^{\rm conv}_{(Y_{\bul \bul},L_{\bul \bul})/S}$ $(L:=M,N)$ 
be the projections in (\ref{eqn:uysr}) and 
(\ref{eqn:uysrdb}) for $(Y,L)$, respectively. 
We have the following equalities: 
\begin{equation*}
u^{\rm conv}_{(Y_{\bul},N_{\bul})/S}\circ \eps_{{\rm conv}\bul}
= u^{\rm conv}_{(Y_{\bul},M_{\bul})/S}, \quad 
u^{\rm conv}_{(Y_{\bul \bul},N_{\bul \bul})/S}
\circ \eps_{{\rm conv}\bul \bul}= 
u^{\rm conv}_{(Y_{\bul \bul},M_{\bul \bul})/S}
\tag{7.0.20}\label{eqn:uepu}
\end{equation*}
as morphisms of ringed topoi. 
\par   

\section{Weight-filtered convergent complexes and $p$-adic purity}\label{sec:wfcipp}
Let ${\cal V}$, $\pi$, $S$ and $S_1$ 
be as in \S\ref{sec:logcd}. 
Let $f\col (X,D\cup Z)\lo S_1$ 
be a smooth scheme with transversal relative SNCD's 
over $S_1$. 
In this section we construct a filtered complex 
$(C_{\rm conv}({\cal K}_{(X,D\cup Z)/S}),P^D)$ 
in ${\rm D}^+{\rm F}({\cal K}_{(X,Z)/S})$ and prove 
the $p$-adic purity (\ref{theo:pap}). 
\par 
%We say that a log scheme 
%${\cal Y}$ over $S$ is obtained from 
%a formally smooth scheme with a relative SNCD over $S$ 
%if ${\cal Y}\simeq ({\cal X},M({\cal D}))$, 
%where ${\cal X}$ is a smooth scheme and ${\cal D}$ 
%is a relative SNCD on ${\cal X}/S$. 
%In this case we say that $M_{\cal Y}$ is associated to  
%the relative SNCD $D$. 

\par 
Let $X=\bigcup_iX_i$ be an affine open covering of $X$.  
%such that $\Gam(X_i,M(D\cup Z))/\Gam(X_i,{\cal O}^*_{X_i})$ 
%is isomorphic to a finite direct sum of ${\mab N}$. 
Set $D_i:=D\vert_{X_i}$, $Z_i:=Z\vert_{X_i}$ 
and $(X_0,D_0\cup Z_0)
:=(\coprod_iX_i,\coprod_i(D_i\cup Z_i))$. 
Set also 
$(X_n,D_n\cup Z_n):={\rm cosk}_0^X(X_0,D_0\cup Z_0)_n$ 
$(n\in {\mab N})$. 
Then we have the simplicial log scheme 
$(X_{\bul},D_{\bul}\cup Z_{\bul})$ over $S_1$. 
Let 
\begin{equation*} 
(X_i,D_i) \os{\sus}{\lo} `{\cal R}_i \quad {\rm and} \quad 
(X_i,Z_i) \os{\sus}{\lo} {\cal R}_i 
\tag{8.0.1}\label{eqn:xdap} 
\end{equation*}  
be closed immersions into 
formally log smooth log noetherian $p$-adic formal schemes 
such that 
$\os{\circ}{`{\cal R}}_i=\os{\circ}{\cal R}_i$.   
Assume that $\os{\circ}{\cal R}_i$ 
is topologically of finite type over $S$.    
Set ${\cal P}_i:=(\os{\circ}{\cal R}_i,
M_{`{\cal R}_i}\oplus_{{\cal O}^*_{{\cal R}_i}}M_{{\cal R}_i})$.  
Assume that ${\cal P}_i$ is formally log smooth over $S$. 
Then we have a natural closed immersion 
$(X_i,D_i\cup Z_i) \os{\sus}{\lo} {\cal P}_i$.  
Set  
$`{\cal R}_0:=\coprod_i{`{\cal R}}_i$, 
${\cal R}_0:=\coprod_i{\cal R}_i$, 
${\cal P}_0:=\coprod_i{\cal P}_i$, 
$`{\cal R}_{\bul}:={\rm cosk}_0^S(`{\cal R}_0)$,  
${\cal R}_{\bul}:={\rm cosk}_0^S({\cal R}_0)$ 
and 
${\cal P}_{\bul}:={\rm cosk}_0^S({\cal P}_0)$.  
%${\cal P}_n:=\underset{n~{\rm pieces}}
%{\underbrace{{\cal P}_0\times_S
%\cdots \times_S{\cal P}_0}}$ 
%(resp.~${\cal R}_n:=\underset{n~{\rm pieces}}
%{\underbrace{{\cal R}_0\times_S
%\cdots \times_S{\cal R}_0}})$.  
Then we have the following simplicial immersions  
\begin{equation*} 
(X_{\bul},D_{\bul}) \os{\sus}{\lo} `{\cal R}_{\bul} 
\quad {\rm and} \quad 
(X_{\bul},Z_{\bul}) \os{\sus}{\lo} {\cal R}_{\bul} 
\tag{8.0.2}\label{eqn:xdr}
\end{equation*} 
and the following commutative diagram:  
\begin{equation*} 
\begin{CD} 
(X_{\bul},D_{\bul}\cup Z_{\bul}) @>{\sus}>> {\cal P}_{\bul} \\ 
@VVV @VVV \\ 
(X_{\bul},Z_{\bul}) @>{\sus}>> {\cal R}_{\bul}. 
\end{CD}
\tag{8.0.3}\label{cd:pfcd}
\end{equation*} 
Note that 
$\os{\circ}{\cal P}_{\bul}=\os{\circ}{\cal R}_{\bul}$. 
Let $(X_{\bul},D_{\bul}\cup Z_{\bul})
\os{\sus}{\lo} {\cal P}^{\rm ex}_{\bul}$ 
(resp.~
$(X_{\bul},Z_{\bul})\os{\sus}{\lo} 
{\cal R}^{\rm ex}_{\bul}$)
be the exactification of 
the upper immersion (resp.~the lower immersion) 
in (\ref{cd:pfcd}). 
(Indeed, by the universality of the exactification, 
we have the simplicial log schemes 
${\cal P}^{\rm ex}_{\bul}$ and ${\cal R}^{\rm ex}_{\bul}$.) 
Let $(X_{\bul},D_{\bul})\os{\sus}{\lo} 
(`{\cal R}_{\bul})^{\rm ex}$ 
be the exactification of the former immersion in (\ref{eqn:xdr}).  
Let ${\cal M}^{\rm ex}_{\bul}$ 
(resp.~${\cal N}^{\rm ex}_{\bul}$) be 
the pull-back of the log structure of 
$(`{\cal R}_{\bul})^{\rm ex}$ 
(resp.~${\cal R}_{\bul}^{\rm ex}$) 
by the morphism ${\cal P}^{\rm ex}_{\bul} \lo 
(`{\cal R}_{\bul})^{\rm ex}$ 
(resp.~${\cal P}^{\rm ex}_{\bul} \lo {\cal R}_{\bul}^{\rm ex}$). 
Then we have the following equality: 
\begin{equation*} 
M_{{\cal P}^{\rm ex}_{\bul}}={\cal M}^{\rm ex}_{\bul}
\oplus_{{\cal O}^*_{{\cal P}^{\rm ex}_{\bul}}}
{\cal N}^{\rm ex}_{\bul}.  
\tag{8.0.4}\label{eqn:mdpn} 
\end{equation*} 
Endow $\os{\circ}{\cal P}{}^{\rm ex}_{\bul}$ 
with the pull-back of the log structure of ${\cal R}^{\rm ex}_{\bul}$ 
and let ${\cal Q}_{\bul}$ be the resulting log scheme. 
Then we have the natural simplicial exact closed immersion 
$(X_{\bul},Z_{\bul})\os{\sus}{\lo}{\cal Q}_{\bul}$ 
%Let ${\cal Q}_{\bul}={\cal Q}^{\rm ex}_{\bul}$ 
%be the exactification of this immersion. 
and we see that 
${\cal Q}_{\bul}={\cal Q}^{\rm ex}_{\bul}=
(\os{\circ}{\cal P}{}^{\rm ex}_{\bul},{\cal N}^{\rm ex}_{\bul})$. 
We have the following commutative diagram 
\begin{equation*} 
\begin{CD} 
(X_{\bul},D_{\bul}\cup Z_{\bul}) 
@>{\subset}>>{\cal P}^{\rm ex}_{\bul}\\ 
@V{\eps_{(X_{\bul},D_{\bul}\cup Z_{\bul},Z_{\bul})/S}}VV 
@VVV \\ 
(X_{\bul},Z_{\bul}) @>{\subset}>>{\cal Q}^{\rm ex}_{\bul}\\  
@| @VVV \\ 
(X_{\bul},Z_{\bul}) @>{\subset}>>{\cal R}^{\rm ex}_{\bul}.  
\end{CD} 
\tag{8.0.5}\label{eqn:pedvq} 
\end{equation*} 
\par 
Let $\{(U_{\bul n},T_{\bul n})\}_{n=1}^{\infty}$ be 
the system of 
the universal enlargements of the simplicial immersion 
$(X_{\bul},Z_{\bul})  \os{\sus}{\lo} {\cal Q}^{\rm ex}_{\bul}$ 
over $S$. 
Let $\bet_{\bul n} \col U_{\bul n} \lo X_{\bul}$ 
be the natural morphism.  
Identify $({T_{\bul n}})_{\rm zar}$ 
with $({U_{\bul n}})_{\rm zar}$. 
We have the following filtered complexes by (\ref{theo:injf}):  
\begin{align*} 
(L^{\rm conv}_{(X_{\bul},Z_{\bul})/S}
(\Om^{\bul}_{{\cal P}^{\rm ex}_{\bul}/S}),P^{D_{\bul}}):= 
(L^{\rm conv}_{(X_{\bul},Z_{\bul})/S}
(\Om^{\bul}_{{\cal P}^{\rm ex}_{\bul}/S}),
\{L^{\rm conv}_{(X_{\bul},Z_{\bul})/S}
(P^{{\cal P}^{\rm ex}_{\bul}/{\cal Q}^{\rm ex}_{\bul}}_k
\Om^{\bul}_{{\cal P}^{\rm ex}_{\bul}/S})\}_{k\in {\mab Z}})  
\end{align*} 
and 
\begin{equation*} 
\vpl_n\bet_{\bul n*}({\cal K}_{T_{\bul n}}
\otimes_{{\cal O}_{{\cal P}^{\rm ex}_{\bul}}}
\Om^{\bul}_{{\cal P}^{\rm ex}_{\bul}/S},P^{D_{\bul}})
:=
\vpl_n\bet_{\bul n*}({\cal K}_{T_{\bul n}}
\otimes_{{\cal O}_{{\cal P}^{\rm ex}_{\bul}}}
\Om^{\bul}_{{\cal P}^{\rm ex}_{\bul}/S},
\{{\cal K}_{T_{\bul n}}\otimes_{{\cal O}_{{\cal P}^{\rm ex}_{\bul}}}
P^{{\cal P}^{\rm ex}_{\bul}/{\cal Q}^{\rm ex}_{\bul}}_k
\Om^{\bul}_{{\cal P}^{\rm ex}_{\bul}/S}\}_{k\in {\mab Z}}). 
\end{equation*} 
Let $\pi_{(X,Z)/S{\rm conv}}$ 
and $\pi_{\rm zar}$ 
be the morphisms of ringed topoi in 
(\ref{eqn:pidf}) and (\ref{eqn:pizs}), respectively, 
for $Y=(X,Z)$.  
We set 
\begin{equation*} 
(C_{\rm conv}
({\cal O}_{(X,D\cup Z)/S}),P^D)
:=R\pi_{(X,Z)/S {\rm conv}*}
((L^{\rm conv}_{(X_{\bul},Z_{\bul})/S}
(\Om^{\bul}_{{\cal P}^{\rm ex}_{\bul}/S}),P^{D_{\bul}})) 
\tag{8.0.6}\label{ali:dfg}
\end{equation*} 
and 
\begin{align*} 
(C_{\rm isozar}
({\cal K}_{(X,D\cup Z)/S}),P^D)
& := Ru^{\rm conv}_{(X,Z)/S*}
(C_{\rm conv}({\cal K}_{(X,D\cup Z)/S}),P^D) 
\tag{8.0.7}\label{ali:dfztg}\\ 
{} & = R\pi_{{\rm zar}*}
(\vpl_{n*}\bet_{\bul n*}({\cal K}_{\os{\to}{T}_{\bul n}}
\otimes_{{\cal O}_{{\cal P}^{\rm ex}_{\bul}}}
\Om^{\bul}_{{\cal P}^{\rm ex}_{\bul}/S},P^{D_{\bul}})). 
\end{align*} 
%(In \cite{nh2} we have used the notation 
%$(C^{{\log},Z}_{\rm Rcrys}
%({\cal O}_{(X,D\cup Z)/S}),P^D)$; in this paper 
%we omit to write the ``$\log$'' in $(C_{\rm conv}
%({\cal O}_{(X,D\cup Z)/S}),P^D)$ for  
%the left hand side in (\ref{ali:dfg}).) 
Here we have used (\ref{eqn:ubys}) and (\ref{eqn:uxzl}). 
Let 
$a^{(k)}_{\bul}\col 
(D^{(k)}_{\bul},Z\vert_{D^{(k)}_{\bul}}) 
\lo 
(X_{\bul},Z_{\bul})$ $(k\in {\mab N})$ 
be the natural morphism of ringed topoi. 
\par 
%If each member of the open covering of $X$ is a small affine 
%scheme, then there exists a simplicial admissible immersions 
%(\ref{eqn:exadm}) (\cite[(12.7)]{nh2}).  

The following lemma has been proved in \cite[(1.3.4)]{nh2}:

\begin{lemm}[{\rm {\bf \cite[(1.3.4)]{nh2}}}]\label{lemm:grrdld}
For a morphism $f \col ({\cal T},{\cal A}) \lo 
({\cal T}',{\cal A}')$ of ringed topoi, 
the following diagram is commutative$:$
\begin{equation*}
\begin{CD}
{\rm D}^+{\rm F}({\cal A}) 
@>{{\rm gr}_k}>> D^+({\cal A})\\ 
@V{Rf_*}VV  @VV{Rf_*}V \\
{\rm D}^+{\rm F}({\cal A}')  
@>{{\rm gr}_k}>> D^+({\cal A}'). 
\end{CD}
\tag{8.1.1}\label{cd:grfcty}
\end{equation*} 
\end{lemm}

\begin{lemm}[{\rm {\bf cf.~\cite[(4.14)]{nh3}}}]
\label{lemm:knit}
Let $k$ be a nonnegative integer. Then the following hold$:$ 
\par 
$(1)$ For the morphism 
$g \col  \os{\circ}{\cal P}{}^{\rm ex}_n\lo 
\os{\circ}{\cal P}{}^{\rm ex}_{n'}$ 
corresponding to a morphism $[n']\lo [n]$ in $\Del$, 
the assumption in {\rm (\ref{prop:mmoo})} is satisfied  
for $Y=(\os{\circ}{\cal P}{}^{\rm ex}_n,{\cal M}^{\rm ex}_n)$ 
and 
$Y'=(\os{\circ}{\cal P}{}^{\rm ex}_{n'},{\cal M}^{\rm ex}_{n'})$. 
Consequently there exists a morphism 
$g^{(k)}\col D^{(k)}({\cal M}^{\rm ex}_n)
\lo D^{(k)}({\cal M}^{\rm ex}_{n'})$ 
$(k\in {\mab N})$ fitting into 
the commutative diagram {\rm (\ref{cd:dmgdm})} 
for $Y=(\os{\circ}{\cal P}{}^{\rm ex}_n,
{\cal M}^{\rm ex}_n)$ and 
$Y'=(\os{\circ}{\cal P}{}^{\rm ex}_{n'},
{\cal M}{}^{\rm ex}_{n'})$.
\par 
$(2)$ 
The family 
$\{D^{(k)}({\cal M}^{\rm ex}_n)\}_{n\in {\mab N}}$ 
gives a simplicial formal scheme $D^{(k)}({\cal M}^{\rm ex}_{\bul})$ 
with natural morphism 
$b^{(k)}_{\bul} \col 
D^{(k)}({\cal M}^{\rm ex}_{\bul})
\lo \os{\circ}{\cal P}{}^{\rm ex}_{\bul}$ 
of simplicial formal schemes. 
\end{lemm} 
\begin{proof}
The proof is the same as that of \cite[(4.14)]{nh3}. 
\end{proof} 

The following is a key lemma for 
(\ref{theo:ifc}) below:  

\begin{lemm}\label{lemm:grkpd}
\begin{equation*}
{\rm gr}_k^{P^D}
C_{\rm conv}({\cal K}_{(X,D\cup Z)/S})
= a^{(k)}_{{\rm conv}*}
({\cal K}_{(D^{(k)},Z\vert_{D^{(k)}})/S}\otimes_{\mab Z}
\vp^{(k){\rm log}}_{{\rm conv}}(D/S;Z))[-k]. 
\tag{8.3.1}\label{eqn:grpwwd}
\end{equation*}
\end{lemm}
\begin{proof}  
Let $a_{\bul} \col (D^{(k)}_{\bul},Z_{\bul}\vert_{D^{(k)}_{\bul}}) 
\lo (X_{\bul},Z_{\bul})$ be the natural morphism of 
simplicial log schemes. 
Let $L^{(k){\rm conv}}_{\bul}$ $(k\in {\mab N})$ 
be the log convergent linearization functor for 
coherent 
${\cal K}_{D^{(k)}({\cal M}^{\rm ex}_{\bul})}$-modules. 
Then, by (\ref{eqn:grpdl}), (\ref{prop:rescos}) and (\ref{lemm:knit}) (2), 
we have the following isomorphism 
in $D^+({\cal K}_{(X_{\bul},Z_{\bul})/S})$: 
%of cosimplicial complexes: 
\begin{equation*} 
{\rm Res}: 
{\rm gr}_k^{P^{D_{\bul}}}
L^{\rm conv}_{(X_{\bul},Z_{\bul})/S}
(\Om^{\bul}_{{\cal P}^{\rm ex}_{\bul}/S}) 
\os{\sim}{\lo} 
a^{(k)\log}_{\bul{\rm conv}*}
(L^{(k){\rm conv}}_{\bul}(\Om^{\bul}_{D^{(k)}
({\cal M}^{\rm ex}_{\bul})/S}
\otimes_{\mab Z}{\mab Q})
\otimes_{\mab Z}
\vp^{(k)}_{\rm conv}(D_{\bul}/S;Z_{\bul}))[-k].  
\tag{8.3.2}\label{eqn:relcdbzb} 
\end{equation*} 
\begin{equation*} 
\end{equation*} 
Hence, by (\ref{lemm:grrdld}) and the cohomological descent, 
we obtain (\ref{eqn:grpwwd}) as follows:  
%as in the proof of (\ref{prop:czc}):  
%(cf.~(\ref{eqn:akbc})): 
\begin{align*} 
{} & {\rm gr}_k^{P^D}
C_{\rm conv}({\cal K}_{(X,D\cup Z)/S})
\tag{8.3.3}\label{eqn:pdkcz}\\
{} & = R\pi^{\rm log}_{(X,Z)/S{\rm conv}*}
{\rm gr}_k^{P^{D_{\bul}}}
(L^{\rm conv}_{(X_{\bul},Z_{\bul})/S}
(\Om^{\bul}_{{\cal P}^{\rm ex}_{\bul}/S})
\otimes_{\mab Z}{\mab Q}) \\ 
{} & = 
R\pi^{\rm log}_{(X,Z)/S{\rm conv}*}
a^{(k)\log}_{\bul{\rm conv}*}
(L^{(k){\rm conv}}_{\bul}(\Om^{\bul}_{D^{(k)}({\cal M}^{\rm ex}_{\bul})/S}
\otimes_{\mab Z}{\mab Q})
\otimes_{\mab Z}
\vp^{(k)}_{\rm conv}(D_{\bul}/S;Z_{\bul}))[-k] 
\\ 
{} & = R\pi^{\rm log}_{(X,Z)/S{\rm conv}*}
a^{(k)\log}_{{\bul}{\rm conv}*}
({\cal K}_{(D^{(k)}_{\bul},Z_{\bul}\vert_{{D^{(k)}_{\bul}}})/S}\otimes_{\mab Z}
\vp^{(k)}_{{\rm conv}}(D_{\bul}/S;Z_{\bul}))[-k] \\
{} & =a^{(k)\log}_{{\rm conv}*}
({\cal K}_{(D^{(k)},Z\vert_{{D^{(k)}}})/S}\otimes_{\mab Z}
\vp^{(k){\rm log}}_{{\rm conv}}(D/S;Z))[-k]. 
\end{align*}
\end{proof}

\begin{theo}\label{theo:ifc} 
The filtered complexes  
$(C_{\rm conv}({\cal K}_{(X,D\cup Z)/S}),P^D)\in 
{\rm D}^+{\rm F}({\cal K}_{(X,Z)/S})$ 
and  
$(C_{\rm isozar}({\cal K}_{(X,D\cup Z)/S}),P^D)
\in {\rm D}^+{\rm F}(f^{-1}({\cal K}_S))$
are independent of the choice of the open covering of $X$ 
and the immersions in {\rm (\ref{eqn:xdr})}. 
\end{theo}
\begin{proof}  
(Though we can give the same proof of 
(\ref{theo:ifc}) as that of \cite[(2.5.3), (2.5.4)]{nh2}, 
we give a shorter proof of it here.) 
%The proof is different from the proof of \cite[(2.5.3)]{nh2} 
%for 
%$(C_{\rm Rcrys}({\cal O}_{(X,D\cup Z)/S}),P^D)$.) 
\par 
Assume that we are given the immersions in (\ref{eqn:xdr}). 
Let $X'_0$ be the disjoint union of 
another affine open covering of $X$.  
Set $(X'_{\bul},D'_{\bul})
:={\rm cosk}_0^{(X,D)}(X'_0,D'_0)$ 
and $(X'_{\bul},Z'_{\bul})
:={\rm cosk}_0^{(X,Z)}(X'_0,Z'_0)$.  
Assume that we are given other immersions 
\begin{equation*} 
(X'_{\bul},D'_{\bul}) \os{\sus}{\lo} `{\cal R}'_{\bul} 
\quad {\rm and} \quad 
(X'_{\bul},Z'_{\bul}) \os{\sus}{\lo} {\cal R}'_{\bul} 
\tag{8.4.1}\label{eqn:xadr}
\end{equation*} 
which are analogous to (\ref{eqn:xdr}). 
Let  
\begin{equation*} 
\begin{CD} 
(X'_{\bul},D'_{\bul}\cup Z'_{\bul}) 
@>{\subset}>>{\cal P}'_{\bul} \\ 
@VVV @VVV \\ 
(X'_{\bul},Z'_{\bul})  @>{\subset}>> {\cal R}'_{\bul}   
\end{CD} 
\tag{8.4.2}\label{cd:fcxadz}
\end{equation*} 
be a commutative diagram of simplicial log schemes 
which is analogous to (\ref{cd:pfcd}). 
Set $(X_{mn},D_{mn}\cup Z_{mn}):=
{\rm cosk}_0^X(X_0,D_0\cup Z_0)_m
\times_X
{\rm cosk}_0^X(X'_0,D'_0\cup Z'_0)_n$ 
$(m,n\in {\mab N})$. 
Then we have the natural bisimplicial log schemes 
$(X_{\bul \bul},D_{\bul \bul}\cup Z_{\bul \bul})$
with natural morphisms 
$(X_{\bul \bul},D_{\bul \bul}\cup Z_{\bul \bul}) \lo 
(X_{\bul},D_{\bul}\cup Z_{\bul})$ and 
$(X_{\bul \bul},D_{\bul \bul}\cup Z_{\bul \bul}) 
\lo (X'_{\bul},D'_{\bul}\cup Z'_{\bul})$. 
Set 
${\cal P}_{\bul \bul}:={\cal P}_{\bul}\times_S{\cal P}'_{\bul}$ 
and ${\cal R}_{\bul \bul}
:={\cal R}_{\bul}\times_S{\cal R}'_{\bul}$.  
%by using (\ref{lemm:ysuls}) and the universality of 
%quasi-exactifications (\cite[(.8) (8)]{nh3}),  
Then we have the following diagram 
\begin{equation*} 
\begin{CD} 
(X_{\bul \bul},D_{\bul \bul}\cup Z_{\bul \bul})  
@>{\subset}>> {\cal P}_{\bul \bul} \\ 
@VVV @VVV \\ 
(X_{\bul \bul},Z_{\bul \bul})  
@>{\subset}>> {\cal R}_{\bul \bul}  
\end{CD} 
\tag{8.4.3}\label{cd:qcxpdz}
\end{equation*} 
such that there exist two natural morphisms 
from this commutative diagram to 
(\ref{cd:pfcd}) and (\ref{cd:fcxadz}). 
Let ${\cal Q}_{\bul \bul}={\cal Q}^{\rm ex}_{\bul \bul}$ 
be the bisimplicial version of 
${\cal Q}_{\bul}={\cal Q}^{\rm ex}_{\bul}$. 
Let ${\cal P}^{\rm ex}_{\bul \bul}$ 
be the exactification of the immersion 
$(X_{\bul \bul},D_{\bul \bul}\cup Z_{\bul \bul})  
\os{\sus}{\lo} {\cal P}_{\bul \bul}$.  
Then there exists an fs bicosimplicial sub log structure 
${\cal M}^{\rm ex}_{\bul \bul}$ of 
$M_{{\cal P}^{\rm ex}_{\bul \bul}}$ 
such that 
\begin{equation*} 
M_{{\cal P}^{\rm ex}_{\bul \bul}}={\cal M}^{\rm ex}_{\bul \bul}
\oplus_{{\cal O}_{{\cal P}^{\rm ex}_{\bul \bul}}^*}
M_{{\cal Q}^{\rm ex}_{\bul \bul}},  
\tag{8.4.4}\label{eqn:pbq}
\end{equation*} 
such that the immersion 
$(X_{\bul \bul},D_{\bul \bul}\cup Z_{\bul \bul})  
\os{\subset}{\lo} {\cal P}_{\bul \bul}$ 
induces an immersion 
$(X_{\bul \bul},D_{\bul \bul})
\os{\sus}{\lo} 
(\os{\circ}{\cal P}{}^{\rm ex}_{\bul \bul},
{\cal M}^{\rm ex}_{\bul \bul})$. 
%Then, for a nonnegative integer $m$, 
%$X_m$ has an affine open 
%covering $\bigcup_{j\in J}(X_m)_j$. 
%Set $X_{m0}:=\coprod_{j\in J}(X_m)_j$, 
%$D_{m0}:=\coprod_{j\in J}(D_m\times_{X_m}(X_m)_j)$ 
%and 
%$Z_{m0}:=\coprod_{j\in J}(Z_m\times_{X_m}(X_m)_j)$, 
%and  
%set 
%\begin{equation*} 
%(X_{mn}, D_{mn}\cup Z_{mn})
%:=({\rm cosk}^{X_m}_0(X_{m0})_n,
%{\rm cosk}^{D_m}_0(D_{m0})_n
%\cup {\rm cosk}^{Z_m}_0(Z_{m0})_n). 
%\end{equation*}
%Then, in \cite[\S12]{nh2}, 
%we have constructed an admissible immersion 
%\begin{equation*} 
%(X_{\bul \bul}, D_{\bul \bul}\cup Z_{\bul \bul})  \os{\sus}{\lo}
%({\cal X}_{\bul \bul}, D_{\bul \bul}\cup Z_{\bul \bul}) 
%\tag{8.1.3}\label{eqn:xdzbs} 
%\end{equation*}
%of bisimplicial log schemes fitting into 
%the following commutative diagram 
%\begin{equation*} 
%\begin{CD} 
%(X_{\bul}, D_{\bul}\cup Z_{\bul}) @<<< 
%(X_{\bul \bul}, D_{\bul \bul}\cup Z_{\bul \bul}) @>>> 
%(X_{\bul'}, D_{\bul'}\cup Z_{\bul'}) \\ 
%@V{\bigcap}VV @V{\bigcap}VV @V{\bigcap}VV \\ 
%({\cal X}_{\bul},{\cal D}_{\bul}\cup {\cal Z}_{\bul}) @<<< 
%({\cal X}_{\bul \bul},{\cal D}_{\bul \bul}\cup 
%{\cal Z}_{\bul \bul}) @>>> 
%({\cal X}'_{\bul},{\cal D}'_{\bul'}\cup {\cal Z}'_{\bul}).  
%\end{CD}
%\end{equation*}
Let 
\begin{equation*}
R\eta_{{\rm conv}*} \col
{\rm D}^+{\rm F}({\cal K}_{(X_{\bul \bul},Z_{\bul \bul})/S}) \lo 
{\rm D}^+{\rm F}({\cal K}_{(X_{\bul},Z_{\bul})/S}) 
\end{equation*}
be the morphism (\ref{eqn:retlcds}). 
Then we have only to prove that 
\begin{align*}
(L^{\rm conv}_{(X_{\bul},Z_{\bul})/S}
(\Om^{\bul}_{{\cal P}^{\rm ex}_{\bul}/S}
\otimes_{\mab Z}{\mab Q}),P^{D_{\bul}})
\os{\sim}{\lo} R\eta_{{\rm conv}*}
((L^{\rm conv}_{(X_{\bul \bul},Z_{\bul \bul})/S}
(\Om^{\bul}_{{\cal P}^{\rm ex}_{\bul \bul}/S}
\otimes_{\mab Z}{\mab Q}),P^{D_{\bul \bul}})). 
\tag{8.4.5}\label{eqn:lxzso}
\end{align*} 
Let 
\begin{equation*}
R\eta_{m,{\rm conv}*} \col
{\rm D}^+{\rm F}({\cal K}_{(X_{m\bul},Z_{m\bul})/S}) \lo 
{\rm D}^+{\rm F}({\cal K}_{(X_m,Z_m)/S}) 
\quad (m\in {\mab N})
\end{equation*}
be the induced morphism by 
$R\eta_{{\rm conv}*}$. 
It is well-known that there exists 
a filtered flasque resolution of 
$(I^{\bul \bul \bul},P)$ of 
$(L^{\rm conv}_{(X_{\bul \bul},Z_{\bul \bul})/S}
(\Om^{\bul}_{{\cal P}^{\rm ex}_{\bul \bul}/S}
\otimes_{\mab Z}{\mab Q}),P^{D_{\bul \bul}})$ 
such that $(I^{m\bul \bul},P)$ is a 
filtered flasque resolution of 
$(L^{\rm conv}_{(X_{m \bul},Z_{m\bul})/S}
(\Om^{\bul}_{{\cal P}^{\rm ex}_{m\bul}/S}
\otimes_{\mab Z}{\mab Q}),P^{D_{m\bul}})$ 
for each $m\in {\mab N}$ 
(cf.~\cite[(1.5.0.4)]{nh2}).      
Here the third degree of $(I^{\bul \bul \bul},P)$ 
is the complex degree. 
Hence, to prove (\ref{eqn:lxzso}), 
we have only to prove that the morphism 
\begin{align*}
(L^{\rm conv}_{(X_m,Z_m)/S}
(\Om^{\bul}_{{\cal P}^{\rm ex}_m/S}
\otimes_{\mab Z}{\mab Q}),P^{D_m})
{\lo} 
R\eta_{m,{\rm conv}*}
((L^{\rm conv}_{(X_{m \bul},Z_{m\bul})/S}
(\Om^{\bul}_{{\cal P}^{\rm ex}_{m\bul}/S}
\otimes_{\mab Z}{\mab Q}),P^{D_{m\bul}}))  
\tag{8.4.6}\label{eqn:fmzso} 
\end{align*} 
is a filtered isomorphism. 
\par 
Let $a^{(k)}_{\bul \bul} \col 
(D^{(k)}_{\bul \bul},Z_{\bul \bul}
\vert_{D^{(k)}_{\bul \bul}}) \lo 
(X_{\bul \bul},Z_{\bul \bul})$ be the natural morphism 
of log schemes.  
By (\ref{lemm:grkpd}) we have the following isomorphism 
\begin{align*} 
{\rm gr}^{P^{D_m}}_k
L^{\rm conv}_{(X_m,Z_m)/S}
(\Om^{\bul}_{{\cal P}^{\rm ex}_m/S}) 
& \os{\sim}{\lo} 
a^{(k)\log}_{m,{\rm conv}*} 
({\cal K}_{(D^{(k)}_m,Z\vert_{D^{(k)}_m})/S}
\otimes_{\mab Z}
\vp^{(k)}_{{\rm conv}}(D_m/S;Z_m))[-k].  
\tag{8.4.7}\label{ali:pdlg} 
\end{align*} 
On the other hand, 
by the proof of (\ref{lemm:knit}) (1), 
we have the simplicial scheme 
$D^{(k)}({\cal M}^{\rm ex}_{m\bul})$ 
$(m\in {\mab N})$.  
As in (\ref{eqn:pdkcz}), we have the following formula  
\begin{align*} 
& {\rm gr}^{P^{D_{m}}}_k
R\eta_{m,{\rm conv}*} 
(L^{\rm conv}_{(X_{m \bul},Z_{m\bul})/S}
(\Om^{\bul}_{{\cal P}^{\rm ex}_{m\bul}/S}\otimes_{\mab Z}{\mab Q}))  
\tag{8.4.8}\label{ali:pglg} \\ 
%& =R\eta_{m,{\rm conv}*} 
%{\rm gr}^{P^{D_{m\bul}}}_k
%(L^{\rm conv}_{(X_{m \bul},Z_{m\bul})/S}
%(\Om^{\bul}_{{\cal P}^{\rm ex}_{m\bul}/S}
%\otimes_{\mab Z}{\mab Q}))  \\
%& \os{\sim}{\lo} 
%R\eta_{m,{\rm conv}*} 
%a^{(k)\log}_{m\bul{\rm conv}*} 
%(L^{\rm conv}_{(D^{(k)}_{m \bul},
%Z_{m\bul}\vert_{D^{(k)}_{m \bul}})/S}
%(\Om^{\bul}_{D^{(k)}({\cal M}^{\rm ex}_{m\bul})/S}
%\otimes_{\mab Z}{\mab Q}))\\
%&\os{\sim}{\longleftarrow} 
%R\eta_{m,{\rm conv}*} 
%a^{(k)\log}_{m\bul{\rm conv}*} 
%({\cal K}_{(D^{(k)}_{m\bul},Z\vert_{D^{(k)}_{m\bul}})/S}
%\otimes_{\mab Z}
%\vp^{(k)}_{{\rm conv}}(D_{m\bul}/S;Z_{m\bul}))[-k] \\  
& =a^{(k)\log}_{m,{\rm conv}*} 
({\cal K}_{(D^{(k)}_m,Z\vert_{D^{(k)}_m})/S}
\otimes_{\mab Z}\vp^{(k)}_{{\rm conv}}(D_m/S;Z_m))[-k]  
\end{align*} 
by replacing $R\pi_{(X,Z)/S{\rm conv}*}$ 
and $\bul$ with $R\eta_{m,{\rm conv}*}$ 
and $m\bul$, respectively. 
Hence the morphism (\ref{eqn:fmzso}) is 
a filtered isomorphism.    
\end{proof}

\begin{defi}\label{defi:ccx}  
We call 
$(C_{\rm conv}({\cal K}_{(X,D\cup Z)/S}),P^D)$ 
(resp.~
$(C_{\rm isozar}({\cal K}_{(X,D\cup Z)/S}),P^D)$)
the {\it weight-filtered convergent complex} 
(resp.~ the {\it weight-filtered isozariskian complex}) 
of ${\cal K}_{(X,D\cup Z)/S}$ {\it with respect to} $D$. 
\end{defi}

\begin{prop}\label{prop:czc} 
There exists a canonical isomorphism 
\begin{equation*}
C_{\rm conv}
({\cal K}_{(X,D\cup Z)/S})
\os{\sim}{\longleftarrow} 
R\eps^{\rm conv}_{(X,D\cup Z, Z)/S*}
({\cal K}_{(X,D\cup Z)/S}).
\tag{8.6.1}\label{eqn:czxdz}
\end{equation*}
\end{prop}
\begin{proof} 
Let 
\begin{equation*} 
\pi_{{\rm conv}} \col 
(({(X_{\bul},D_{\bul}\cup Z_{\bul})/S})_{\rm conv},
{\cal K}_{(X_{\bul},D_{\bul}\cup Z_{\bul})/S})
\lo 
(({(X,D\cup Z)/S})_{\rm conv},
{\cal K}_{(X,D\cup Z)/S}) 
\tag{8.6.2}\label{eqn:pidxdzf}
\end{equation*}
be the natural morphism in (\ref{eqn:pidf}). 
We have the following equalities: 
\begin{align*}
 C_{\rm conv}
({\cal K}_{(X,D\cup Z)/S}) 
& = 
R\pi_{(X,Z)/S{\rm conv}*}
(L^{\rm conv}_{(X_{\bul},Z_{\bul})/S}
(\Om^{\bul}_{{\cal P}^{\rm ex}_{\bul}/S}\otimes_{\mab Z}{\mab Q})) \\ 
{} & = R\pi_{(X,Z)/S{\rm conv}*}
R\eps^{\rm conv}_{(X_{\bul},D_{\bul}\cup Z_{\bul},Z_{\bul})/S*}
(L^{\rm conv}_{(X_{\bul},D_{\bul}\cup Z_{\bul})/S}
(\Om^{\bul}_{{\cal P}^{\rm ex}_{\bul}/S}\otimes_{\mab Z}{\mab Q})) \\ 
{} & = R\pi_{(X,Z)/S{\rm conv}*}
R\eps^{\rm conv}_{(X_{\bul},D_{\bul}\cup Z_{\bul},Z_{\bul})/S*}
({\cal K}_{(X_{\bul},D_{\bul}\cup Z_{\bul})/S}) \\ 
{} & = R\eps^{\rm conv}_{(X,D\cup Z,Z)/S*}
R\pi_{{\rm conv}*}
({\cal K}_{(X_{\bul},D_{\bul}\cup Z_{\bul})/S}) \\
{} & = R\eps^{\rm conv}_{(X,D\cup Z,Z)/S*}
({\cal K}_{(X,D\cup Z)/S}).
\end{align*} 
Here the first equality (resp.~the second one, 
the third one, the fourth one, the fifth one) follows 
from the definition of 
$C_{\rm conv}({\cal K}_{(X,D\cup Z)/S})$ 
(resp.~(\ref{eqn:ceslvc}), 
%the log convergent Poincar\'{e} lemma 
(\ref{theo:pl}), (\ref{cd:epep}), 
the cohomological descent).  
\end{proof}

\begin{coro}\label{coro:specdz}
There exists the following spectral sequence
\begin{equation*}
E^{-k,h+k}_1 := 
R^{h-k}f^{\rm conv}_{(D^{(k)},Z\vert_{D^{(k)}})/S*}
({\cal K}_{(D^{(k)},Z\vert_{D^{(k)}})/S}
\otimes_{\mab Z}
\vp^{(k)}_{{\rm conv}}(D/S;Z))
\Longrightarrow 
\tag{8.7.1}\label{eqn:wtspkdz}
\end{equation*} 
$$R^hf^{\rm conv}_{(X,D\cup Z)/S*}({\cal K}_{(X,D\cup Z)/S}).$$  
\end{coro}
\begin{proof} 
(\ref{coro:specdz}) immediately follows from 
(\ref{prop:czc}) and (\ref{eqn:grpwwd}). 
\end{proof}

\begin{theo}[{\bf $p$-adic purity, {cf.~\cite[(2.7.1)]{nh2}}}]\label{theo:calvc}
Let $k$ be a nonnegative integer. Then
\begin{equation*}
R^k\eps^{\rm conv}_{(X,D\cup Z, Z)/S*}
({\cal K}_{(X, D\cup Z)/S}) 
=a^{(k)}_{{\rm conv}*}
({\cal K}_{(D^{(k)}, Z\vert_{D^{(k)}})/S}\otimes_{\mab Z}
\vp^{(k)}_{{\rm conv}}(D/S;Z)). 
\tag{8.8.1}\label{eqn:ppuri}
\end{equation*}
\end{theo}
\begin{proof}
The ``increasing filtration''
$\{P^D_kC_{\rm conv}({\cal K}_{(X,D\cup Z)/S})\}_{k\in {\mab Z}}$  
on $C_{\rm conv}({\cal K}_{(X,D\cup Z)/S})$ gives 
the following spectral sequence
\begin{equation*}
E_1^{-k,h+k} = {\cal H}^h({\rm gr}_k^{P^D}
C_{\rm conv}
({\cal K}_{(X,D\cup Z)/S})) \Lo
{\cal H}^h(C_{\rm conv}
({\cal K}_{(X,D\cup Z)/S})). 
\tag{8.8.2}\label{eqn:ttc}
\end{equation*}
By (\ref{eqn:czxdz}),  
$ {\cal H}^h(C_{\rm conv}
({\cal K}_{(X,D\cup Z)/S}))
= 
R^h\eps^{\rm conv}_{(X,D\cup Z, Z)/S*}({\cal K}_{(X,D\cup Z)/S})$ 
and by (\ref{eqn:grpwwd}) we have the following equality: 
\begin{equation*} 
{\cal H}^h({\rm gr}_k^{P^D}
C_{\rm conv}({\cal K}_{(X,D\cup Z)/S}))
= {\cal H}^{h-k}(a^{(k)}_{{\rm conv}*}
({\cal K}_{(D^{(k)},Z\vert_{D^{(k)}})/S}\otimes_{\mab Z}
\vp^{(k)}_{{\rm conv}}(D/S;Z))).
\tag{8.8.3}\label{eqn:ckdzs}
\end{equation*}
Hence $E_1^{-k,h+k}=0$ for $h\not= k$. 
Now we obtain (\ref{eqn:ppuri}) by the following equalities: 
\begin{align*} 
R^k\eps^{\rm conv}_{(X,D\cup Z, Z)/S*}({\cal K}_{(X, D\cup Z)/S}) 
& = {\cal H}^k(C_{\rm conv}({\cal K}_{(X,D\cup Z)/S})) 
=E_1^{-k,2k} \\
{} & = 
a^{(k)}_{{\rm conv}*}
({\cal K}_{(D^{(k)},Z\vert_{D^{(k)}})/S}\otimes_{\mab Z}
\vp^{(k)}_{{\rm conv}}(D/S;Z)). 
\end{align*}
\end{proof} 

\begin{rema}\label{rema:wnjp} 
In this remark, we show (\ref{eqn:jnvp}) in general 
as in \cite[VI Lemme 1.2.2]{bb}. 
\par 
Let $Y$ be a scheme over $S_1$ 
and let $a$ be a global section 
of $\Gam(Y,{\cal O}_Y)$. Let $U$ be the largest open subscheme of 
$Y$ such that $a$ is invertible on $U$ and let 
$j \col U \os{\sus}{\lo} Y$ be the open immersion. 
Then we would like to prove that, 
for a coherent ${\cal K}_{U/S}$-module $E$, 
\begin{equation*} 
R^kj_{{\rm conv}*}(E) =0 \quad (k\in {\mab Z}_{\geq 1}). 
\tag{8.9.1}\label{eqn:ekj0}
\end{equation*}
Indeed, let $(V,T,\iota,u)$ be an object of ${\rm Conv}(Y/S)$. 
Assume that $T$ is affine. 
%Set $a_V:=u^*(a)\vert_V\in \Gam(V,{\cal O}_V)$. 
Let $\wt{u^*(a)}\in \Gam(T,{\cal O}_T)$ be a lift of $u^*(a)$.  
Set $V_U:=V\times_YU$ and consider the open subscheme 
$T\vert_{V_U}$ of $T$. 
Then $T\vert_{V_U}$ is the largest open subscheme of 
$T$ such that $\wt{u^*(a)}$ 
is invertible on $T\vert_{V_U}$. 
Let $j_T \col T\vert_{V_U} \lo T$ be the open immersion. 
Then  
$R^kj_{T*}(E_{(V_U,T\vert_{V_U},\iota \vert_{V_U},u\vert_{V_U})})=0$ 
for $k\in {\mab Z}_{\geq 1}$. 
As in [loc.~cit., VI Lemme 1.2.1], 
$R^kj_{{\rm conv}*}(E)_{(V,T,\iota,u)} 
=R^kj_{T*}(E_{(V_U,T\vert_{V_U},\iota \vert_{V_U},u\vert_{V_U})})$. 
Consequently we obtain (\ref{eqn:ekj0}). 
\end{rema}

\par
By the Leray spectral sequence for the functor 
$\eps^{\rm conv}_{(X,D\cup Z, Z)/S} \col 
({(X,D\cup Z)/S})_{\rm conv} \lo 
({(X,Z)/S})_{\rm conv}$ and
$f^{\rm conv}_{(X,Z)/S} \col 
({(X,Z)/S})_{\rm conv} \lo {S}_{\rm zar}$,
we obtain the following spectral sequence
\begin{equation*}
E^{st}_2 := R^sf^{\rm conv}_{(X,Z)/S*}
R^t\eps^{\rm conv}_{(X,D\cup Z, Z)/S*}
({\cal K}_{(X,D\cup Z)/S})
\Longrightarrow 
\tag{8.9.2}\label{eqn:levcs}
\end{equation*} 
$$R^{s+t}f^{\rm conv}_{(X,D\cup Z)/S*}
({\cal K}_{(X,D\cup Z)/S}).$$  
By (\ref{theo:calvc}), 
(\ref{eqn:levcs}) is equal to 
the following spectral sequence
\begin{equation*}
E^{st}_2 = R^sf^{\rm conv}_{(D^{(t)},Z\vert_{D^{(t)}})/S*}
({\cal K}_{(D^{(t)},Z\vert_{D^{(t)}})/S}\otimes_{\mab Z}
\vp^{(t)}_{{\rm conv}}(D/S;Z)) 
\Longrightarrow \tag{8.9.3}\label{eqn:lerpw}
\end{equation*} 
$$R^{s+t}f^{\rm conv}_{(X,D\cup Z)/S*}({\cal K}_{(X,D\cup Z)/S}),$$
which is equal to (\ref{eqn:wtspkdz}) by renumbering 
the $E_2$-term $E^{k,h-k}_2$ by $E_1^{-k,h+k}$.

\par
Now we give another description of 
$(C_{\rm conv}({\cal K}_{(X,D\cup Z)/S}),P^D)$. 
We need the following lemma whose proof is easy:

\begin{lemm}[{\bf {\cite[(2.7.2)]{nh2}}}]\label{lemm:canhigh}
Let $\tau$ be the canonical filtration on a complex.  
Let 
$f \col ({\cal T}, {\cal A}) 
\lo ({\cal T}', {\cal A}')$ 
be a morphism of ringed topoi. 
Then, for an object $E^{\bul}$ in 
$D^+({\cal A})$, 
there exists a canonical morphism 
\begin{equation*}
(Rf_*(E^{\bul}),\tau) \lo 
Rf_*((E^{\bul},\tau))
\tag{8.10.1}\label{eqn:taure}
\end{equation*} 
in ${\rm D}^+{\rm F}({\cal A}').$
\end{lemm}

\begin{theo}[{\bf Comparison theorem}]\label{theo:wtvsca} 
Set 
\begin{equation*}
(E_{\rm conv}({\cal K}_{(X,D\cup Z)/S}),P^D)
:=(R\eps^{\rm conv}_{(X,D\cup Z, Z)/S*}
({\cal K}_{(X,D\cup Z)/S}),\tau).  
\end{equation*} 
Then there exists a 
canonical isomorphism 
\begin{equation*}
(E_{\rm conv}({\cal K}_{(X,D\cup Z)/S}),P^D)
\os{\sim}{\lo} (C_{\rm conv}
({\cal K}_{(X,D\cup Z)/S}), P^D).
\tag{8.11.1}\label{eqn:rtupdk}
\end{equation*} 
In particular, 
\begin{equation*}(C_{\rm conv}
({\cal K}_{(X,D\cup Z)/S}), \tau)=
(C_{\rm conv}
({\cal K}_{(X,D\cup Z)/S}),P^D). 
\tag{8.11.2}\label{eqn:tpccri}
\end{equation*}
\end{theo} 
\begin{proof}  
(The proof is the same as that of \cite[(2.7.3)]{nh2}.) 
Consider the commutative diagram (\ref{cd:pfcd}). 
Then there exists the following natural morphism 
of filtered complexes of 
${\cal K}_{(X_{\bul},Z_{\bul})/S}$-modules:
\begin{equation*}
(L^{\rm conv}_{(X_{\bul},Z_{\bul})/S}
(\Om^{\bul}_{{\cal P}^{\rm ex}_{\bul}/S}
\otimes_{\mab Z}{\mab Q}),\tau) 
\lo (L^{\rm conv}_{(X_{\bul},Z_{\bul})/S}
(\Om^{\bul}_{{\cal P}^{\rm ex}_{\bul}/S}
\otimes_{\mab Z}{\mab Q}),P^{D_{\bul}}).
\tag{8.11.3}\label{eqn:ltpd}
\end{equation*} 
By (\ref{lemm:canhigh}) there exists a canonical morphism
\begin{equation*}
(R{\pi}_{(X,Z)/S{\rm conv}*}
L^{\rm conv}_{(X_{\bul},Z_{\bul})/S}
(\Om^{\bul}_{{\cal P}^{\rm ex}_{\bul}/S}
\otimes_{\mab Z}{\mab Q}),\tau) 
\lo 
\tag{8.11.4}\label{eqn:rtlodz}
\end{equation*} 
$$R{\pi}_{(X,Z)/S{\rm conv}*}
((L^{\rm conv}_{(X_{\bul},Z_{\bul})/S}
(\Om^{\bul}_{{\cal P}^{\rm ex}_{\bul}/S}
\otimes_{\mab Z}{\mab Q}),\tau)).$$ 
By composing (\ref{eqn:rtlodz}) with the morphism 
$R{\pi}_{(X,Z)/S{\rm conv}*}
((\ref{eqn:ltpd}))$, 
we obtain the following morphism 
\begin{equation*}
(R{\pi}_{(X,Z)/S{\rm conv}*}
L^{\rm conv}_{(X_{\bul},Z_{\bul})/S}
(\Om^{\bul}_{{\cal P}^{\rm ex}_{\bul}/S}
\otimes_{\mab Z}{\mab Q}),\tau) 
\lo 
\tag{8.11.5}
\end{equation*}
$$R{\pi}_{(X,Z)/S{\rm conv}*}
((L^{\rm conv}_{(X_{\bul},Z_{\bul})/S}
(\Om^{\bul}_{{\cal P}^{\rm ex}_{\bul}/S}
\otimes_{\mab Z}{\mab Q}),P^{D_{\bul}}))$$ 
which is nothing but a morphism 
\begin{equation*}
(R\eps_{(X,D\cup Z, Z)/S*}
({\cal K}_{(X,D\cup Z)/S}),\tau) 
\lo 
(C_{\rm conv}({\cal K}_{(X,D\cup Z)/S}),P^D)
\tag{8.11.6}\label{eqn:repto}
\end{equation*}
by (\ref{theo:cpvcs}). 
To prove that the morphism (\ref{eqn:repto}) is a 
filtered isomorphism, 
it suffices to prove that the induced morphism
\begin{equation*}
{\rm gr}_k^{\tau}
R\eps_{(X,D\cup Z, Z)/S*}
({\cal K}_{(X,D\cup Z)/S}) 
\lo {\rm gr}_k^{P^D}
C_{\rm conv}
({\cal K}_{(X,D\cup Z)/S}) \quad 
(k\in {\mab Z})
\tag{8.11.7}
\end{equation*}
is an isomorphism. 
By (\ref{theo:calvc}) we have the following equalities: 
\begin{align*} 
{} & {\cal H}^h
({\rm gr}_k^{\tau}R\eps_{(X,D\cup Z, Z)/S*}
({\cal K}_{(X,D\cup Z)/S})) 
\tag{8.11.8}\label{caseali:grtkr}\\
{} & = 
\begin{cases}
R^k\eps_{(X,D\cup Z, Z)/S*}
({\cal K}_{(X,D\cup Z)/S}) & (h=k),  \\
0 & (h \not= k) 
\end{cases} \\
{} & = 
\begin{cases} 
a^{(k){\log}}_{{\rm conv}*}
({\cal K}_{(D^{(k)}, Z\vert_{D^{(k)}})/S}\otimes_{\mab Z}
\vp^{(k){\rm log}}_{{\rm conv}}(D/S;Z)) & (h=k), \\
0 & (h\not=k).  
\end{cases} 
\end{align*} 
By (\ref{eqn:ckdzs}),  
${\cal H}^h({\rm gr}_k^{P^D}
C_{\rm conv}({\cal K}_{(X,D\cup Z)/S}))$ 
is also equal to the last formulas in (\ref{caseali:grtkr}). 
Hence the morphism (\ref{eqn:repto}) 
is an isomorphism. 
\par
The rest we have to show is that 
the morphism (\ref{eqn:repto}) is independent of 
the choice of the commutative diagram (\ref{cd:pfcd}). 
Let the notations be as in the proof of (\ref{theo:ifc}).  
%Let us first consider the following commutative diagram 
%\begin{equation*} 
%\begin{CD} 
%(X_{\bul},D_{\bul}\cup Z_{\bul}) 
%@>{\subset}>> {\cal P}'_{\bul} \\
%@VVV @VVV \\
%(X_{\bul},Z_{\bul}) @>{\subset}>> `{\cal Q}_{\bul}, 
%\end{CD}
%\tag{8.11.9}\label{cd:xzpq}
%\end{equation*} 
%where the two horizontal morphisms are exact immersions 
%into formally log smooth log $p$-adic formal schemes 
%over $S$ 
%which are obtained from smooth $p$-adic formal schemes 
%over $S$ with relative SNCD's over $S$. 
%Let $\{X_{j_0}\}_{j_0\in J_0}$ 
%be another open covering of $X$. 
%Then we have a natural ringed topos 
%$$(({(X_{\bul \bul},D_{\bul \bul}\cup 
%Z_{\bul \bul})/S})_{\rm conv},
%{\cal K}_{(X_{\bul \bul},D_{\bul \bul}\cup Z_{\bul \bul})
%/S})_{\bul \bul \in {\mab N}^2}$$ 
%with a natural morphism 
%\begin{align*}
%\eta_{{\rm conv}*} \col & 
%(({(X_{\bul \bul},
%D_{\bul \bul}\cup Z_{\bul \bul})/S})_{\rm conv},
%{\cal K}_{(X_{\bul \bul},
%D_{\bul \bul}\cup 
%Z_{\bul \bul})/S})_{\bul \bul \in {\mab N}^2} \\
%{} & \lo 
%(({(X_{\bul},D_{\bul}\cup Z_{\bul})/S})_{\rm conv},
%{\cal K}_{(X_{\bul},
%D_{\bul}\cup Z_{\bul})/S})_{\bul \in {\mab N}} 
%\end{align*}
%of ringed topoi (\S\ref{sec:rlct}). 
%Here $\bul$ of the target is the first $\bul$ or 
%the second $\bul$. 
%We also have the natural bisimplicial log schemes 
%${\cal P}_{\bul \bul}$ and ${\cal Q}_{\bul \bul}$ over $S$. 
By the cohomological descent we have the following commutative diagram:
\begin{equation*}
\begin{CD}
(R{\pi}_{(X,Z)/S{\rm conv}*}
L^{\rm conv}_{(X_{\bul},Z_{\bul})/S}
(\Om^{\bul}_{{\cal P}^{\rm ex}_{\bul}/S}
\otimes_{\mab Z}{\mab Q}),\tau)
@>>> \\ 
@VVV  \\
(R{\pi}_{(X,Z)/S{\rm conv}*}
R\eta_{{\rm conv}*}
L^{\rm conv}_{(X_{\bul \bul},Z_{\bul \bul})/S}
(\Om^{\bul}_{{\cal P}^{\rm ex}_{\bul \bul}/S}
\otimes_{\mab Z}{\mab Q}),\tau)
@>{}>> 
\end{CD}
\tag{8.11.9}
\end{equation*}
\begin{equation*}
\begin{CD}
(R{\pi}_{(X,Z)/S{\rm conv}*}
L^{\rm conv}_{(X_{\bul},Z_{\bul})/S}
(\Om^{\bul}_{{\cal P}^{\rm ex}_{\bul}/S}
\otimes_{\mab Z}{\mab Q}),P^{D_{\bul}})\\ 
@| \\
R{\pi}_{(X,Z)/S{\rm conv}*}
R\eta_{{\rm conv}*}
(L^{\rm conv}_{(X_{\bul \bul},Z_{\bul \bul})/S}
(\Om^{\bul}_{{\cal P}^{\rm ex}_{\bul \bul}/S}
\otimes_{\mab Z}{\mab Q}),
P^{D_{\bul \bul}}).
\end{CD}
\end{equation*}
Here we obtain the equality above by (\ref{eqn:lxzso}). 
We see that, by the proof above, 
the horizontal morphisms are isomorphisms. 
Hence the vertical morphism above is an isomorphism. 
We have proved the desired independence. 
\end{proof}

\begin{defi}\label{defi:vccc} 
We call 
$(E_{\rm conv}({\cal K}_{(X,D\cup Z)/S}),P^D)$  
the {\it weight-filtered vanishing cycle convergent complex} 
of $(X,D\cup Z)/S$ {\it with respect to} $D$. 
\end{defi}

%\begin{coro}
%The filtered complex $(C_{\rm conv}
%({\cal K}_{(X,D\cup Z)/S}), P^D)$ 
%is independent of the choice of the decompositions of 
%$D$ and $Z$ by their smooth components. 
%\end{coro}

\par 
Let $\pi'$ be another non-zero element of 
the maximal ideal of ${\cal V}$. 
Assume that $\pi{\cal V} \subset \pi' {\cal V}$. 
Set  $S'_1:=\ul{\rm Spec}_S({\cal O}_S/\pi')$. 
Then we have a natural closed immersion 
$S'_1 \os{\subset}{\lo} S_1$.  
Set $(X',D'\cup Z'):=(X,D\cup Z)\times_{S_1}S'_1$. 
Let $i \col (X',D'\cup Z') \os{\subset}{\lo} (X,D\cup Z)$ 
and $i^Z \col (X',Z') \os{\subset}{\lo} (X,Z)$ 
be the natural closed immersions. 
\par 
Because $i^Z_{{\rm conv}*}$ is exact ((\ref{prop:toi}) (1)), 
we have the following functor:  
\begin{equation*} 
i^Z_{{\rm conv}*} \col 
{\rm D}^+{\rm F}({\cal K}_{(X',Z')/S}) 
\lo 
{\rm D}^+{\rm F}({\cal K}_{(X,Z)/S}). 
\tag{8.12.1}\label{eqn:dfkxs}
\end{equation*} 

\begin{prop}\label{prop:iekp}  
\begin{equation*} 
i^Z_{{\rm conv}*}
((E_{\rm conv}({\cal K}_{(X',D'\cup Z')/S}),P^{D'}))
= 
(E_{\rm conv}({\cal K}_{(X,D\cup Z)/S}),P^D). 
\tag{8.13.1}\label{eqn:ikcp}
\end{equation*} 
\end{prop} 
\begin{proof}   
We have the following equalities: 
\begin{align*}  
i^Z_{{\rm conv}*}
((E_{\rm conv}({\cal K}_{(X',D'\cup Z')/S}),P^{D'}))
& = i^Z_{{\rm conv}*}
(R\eps^{\rm conv}_{(X',D'\cup Z',Z')/S*}
({\cal K}_{(X',D'\cup Z')/S}),\tau) \\ 
{} & = (R\eps^{\rm conv}_{(X,D\cup Z,Z)/S*}
i_{{\rm conv}*}({\cal K}_{(X',D'\cup Z')/S}),\tau) \\
{} & = (R\eps^{\rm conv}_{(X,D\cup Z,Z)/S*}
({\cal K}_{(X,D\cup Z)/S}),\tau) \\ 
{} & =(E_{\rm conv}({\cal K}_{(X,D\cup Z)/S}),P^D). 
\end{align*}  
Here the third equality follows from (\ref{eqn:iuoe}).  
\end{proof} 

\begin{coro}\label{coro:cc} 
$(1)$ 
\begin{equation*} 
i^Z_{{\rm conv}*}
((C_{\rm conv}({\cal K}_{(X',D'\cup Z')/S}),P^{D'}))
= 
(C_{\rm conv}({\cal K}_{(X,D\cup Z)/S}),P^D). 
\tag{8.14.1}\label{eqn:iccp}
\end{equation*} 
\par 
$(2)$ 
\begin{equation*} 
\os{\circ}{i}{}^Z_{*}
((C_{\rm isozar}({\cal K}_{(X',D'\cup Z')/S}),P^{D'}))
= 
(C_{\rm isozar}({\cal K}_{(X,D\cup Z)/S}),P^D). 
\tag{8.14.2}\label{eqn:ziccp}
\end{equation*} 
\end{coro}
\begin{proof}   
(\ref{eqn:iccp}) follows from (\ref{eqn:rtupdk}) 
and (\ref{eqn:ikcp}). 
(\ref{eqn:ziccp}) follows from (\ref{eqn:iccp}). 
\end{proof}

The following has an application for the unipotent $\pi_1$ 
as in \cite{kiha}:

\begin{theo}\label{theo:afkz} 
Let ${\rm A}^{\geq 0}{\rm F}({\cal K}_{(X,Z)/S})$ 
be the category of 
filtered positively-graded graded commutative dga's 
of ${\cal K}_{(X,Z)/S}$-algebras 
and let ${\rm D}({\rm A}^{\geq 0}{\rm F}({\cal K}_{(X,Z)/S}))$ 
be the localized category of  
${\rm A}^{\geq 0}{\rm F}({\cal K}_{(X,Z)/S})$  
inverting the weakly equivalent morphisms 
in ${\rm A}^{\geq 0}{\rm F}({\cal K}_{(X,Z)/S})$ 
{\rm (\cite[p.~145, p.~334]{gelma})}.   
Let 
$U_{\rm conv} \col 
{\rm D}({\rm A}^{\geq 0}{\rm F}({\cal K}_{(X,Z)/S})) 
\lo {\rm D}^+{\rm F}({\cal K}_{(X,Z)/S})$ be 
the natural forgetful functor. 
Then there exists  
a functorial filtered object 
$(\wt{C},\wt{P})\in 
{\rm D}({\rm A}^{\geq 0}{\rm F}({\cal K}_{(X,Z)/S}))$ 
such that 
$U_{\rm conv}((\wt{C},\wt{P})) 
\simeq (C_{\rm conv}({\cal K}_{(X,D\cup Z)/S}),P^D)$. 
\end{theo} 
\begin{proof} 
Let the notations be as in the beginning of this section. 
Let $(X_{\bul'},Z_{\bul'})$ be the constant simplicial 
scheme defined by $(X,Z)$. 
Let 
\begin{equation*} 
R{\pi}_{(X,Z)/S{\rm conv,TW}\bul'*} 
\col 
{\rm D}({\rm A}^{\geq 0}{\rm F}
({\cal K}_{(X_{\bul},Z_{\bul})/S})) \lo 
{\rm D}({\rm A}^{\geq 0}{\rm F}
({\cal K}_{(X_{\bul'},Z_{\bul'})/S}))
\end{equation*} 
be the derived direct image of Thom-Whitney in 
\cite[\S4]{nav}. 
Let   
\begin{equation*} 
{\bf s}_{\rm TW} \col 
{\rm D}({\rm A}^{\geq 0}{\rm F}
({\cal K}_{(X_{\bul'},Z_{\bul'})/S})) \lo 
{\rm D}({\rm A}^{\geq 0}{\rm F}
({\cal K}_{(X,Z)/S}))
\end{equation*} 
be the single complex of Thom-Whitney in [loc.~cit., \S3]. 
Then, by [loc.~cit., (2.6)], we have only to set 
\begin{equation*} 
(\wt{C},\wt{P}):=
{\bf s}_{\rm TW}
R{\pi}_{(X,Z)/S{\rm conv,TW}\bul'*} 
((L^{\rm conv}_{(X_{\bul},Z_{\bul})/S}
(\Om^{\bul}_{{\cal P}^{\rm ex}_{\bul}/S}\otimes_{\mab Z}{\mab Q}),
P^{{\cal P}^{\rm ex}_{\bul}/{\cal Q}^{\rm ex}_{\bul}})).  
\end{equation*} 
\end{proof}

\section{The functoriality of the weight-filtered 
convergent and isozariskian complexes}\label{sec:fpw}
Let ${\cal V}$, $K$ and $\pi$ be as in \S\ref{sec:logcd} 
and let $S$, $S_1$ and $(X,D\cup Z)$ 
be as in \S\ref{sec:lcs}. 
In this section we prove the contravariant functoriality of  
$(C_{\rm conv}({\cal K}_{(X,D\cup Z)/S}),P^D)$ 
and $(C_{\rm isozar}({\cal K}_{(X, D\cup Z)/S}),P^D)$ 
by using (\ref{theo:wtvsca}). 
\par 
Let ${\cal V}'$, $K'$, $\pi'$, $S'$ and $S'_1$ 
be analogous objects to 
${\cal V}$, $K$, $\pi$, $S$ and $S_1$, respectively.  
%Let ${\cal V}'$ be a complete discrete valuation ring of 
%mixed characteristics $(0,p)$ with an element $\pi'$ 
%of the maximal ideal of ${\cal V}'$. 
%Let ${\cal K}'$ be the fraction field of ${\cal V}'$. 
%Let $S'$ be a fine log formal scheme 
%whose underlying scheme is a 
%$p$-adic formal ${\cal V}'$-scheme 
%in the sense of \cite{of}.  
%We assume that $\os{\circ}{S}$ is flat over 
%${\rm Spf}({\cal V})$. 
%Let $\pi$ be a non-zero element of the maximal ideal 
%of ${\cal V}$.  As in \cite{oc},  $S_1$ 
%denotes the closed sub log scheme defined by 
%the ideal sheaf $\pi{\cal O}_S$ and 
%$S_0$ denotes $(S_1)_{\rm red}$. 
%Let $Y$ be a fine log scheme over $S_1$ 
%whose underlying scheme is of finite type over $S_1$. 
Let  
\begin{equation*}
\begin{CD} 
S @>{u}>> S' \\
@VVV @VVV \\
{\rm Spf}({\cal V})  @>>> {\rm Spf}({\cal V}')
\end{CD}
\tag{9.0.1}\label{cd:bpsmlgs}
\end{equation*}
be  a commutative diagram of $p$-adic 
${\cal V}$-formal schemes and $p$-adic ${\cal V}'$-formal schemes.
We assume that $u$  
induces a morphism $u_1 \col S_1 \lo S'_1$ 
of schemes.
Let $(X',D'\cup Z')$ 
be a  smooth scheme 
with transversal relative SNCD's 
over $S'_1$. Let  
\begin{equation*}
\begin{CD} 
(X,D\cup Z) @>{g}>> (X',D'\cup Z') \\
@VVV @VVV \\
S_1  @>{u_1}>> S'_1
\end{CD}
\tag{9.0.2}\label{cd:mpsmlgs}
\end{equation*}
be  a commutative diagram of log schemes.  
Assume that the morphism $g$ induces 
$g_{(X,D)} \col (X,D) \lo (X',D')$ and 
$g_{(X,Z)} \col (X,Z) \lo (X',Z')$ over
$u_1 \col S_1 \lo S'_1$.
Let 
$$\eps_{\rm conv} \col ({(X,D\cup Z)/S})_{\rm conv} \lo 
({(X,Z)/S})_{\rm conv}$$ 
and
$$\eps'_{\rm conv} \col ({(X',D'\cup Z')/S'})_{\rm conv} \lo 
({(X',Z')/S'})_{\rm conv}$$ 
be the 
morphism of topoi 
forgetting the log structure along $D$ and $D'$, 
respectively. 
Let 
$$ g_{{\rm conv}} \col 
(({(X, D\cup Z)/S})_{\rm conv},{\cal K}_{(X,D\cup Z)/S}) 
\lo 
(({(X',D'\cup Z')/S'})_{\rm conv},
{\cal K}_{(X',D'\cup Z')/S'})$$ 
be the morphism of log convergent ringed topoi induced by $g$.

\begin{theo}[{\bf Functoriality}]\label{theo:fcpwczo}
Let the notations be as above. 
Then the following hold$:$
\par
$(1)$ There exists  a canonical 
morphism 
\begin{equation*}
g^{*}_{(X,Z){\rm conv}}\col 
(C_{\rm conv}
({\cal K}_{(X',D'\cup Z')/S'}),P^{D'}) 
\lo 
Rg_{(X,Z){\rm conv}*}
(C_{\rm conv}
({\cal K}_{(X,D\cup Z)/S}),P^D). 
\tag{9.1.1}\label{eqn:fdfcc}
\end{equation*} 
\par
$(2)$ There exists  a canonical 
morphism 
\begin{equation*}
g_{\rm zar}^* \col  
(C_{{\rm isozar}}({\cal K}_{(X',D'\cup Z')/S'}),P^{D'}) 
\lo 
Rg_{{\rm zar}*}(C_{\rm isozar}({\cal K}_{(X,D\cup Z)/S}),P^D). 
\tag{9.1.2}\label{eqn:fczpwf}
\end{equation*} 
\end{theo}
\begin{proof} 
%(The proof is the same as that of \cite[(16.1) (2)]{nh2}.) 
(1): (1) immediately follows from (\ref{theo:wtvsca}). 
\par 
(2): (2) immediately follows from (1) and the first equality in (\ref{ali:dfztg}). 
%\begin{align*}
%(C_{\rm isozar}
%({\cal K}_{(X',D'\cup Z')/S'}),P^{D'}) &
%=Ru^{\rm conv}_{(X',Z')/S'*}
%(E_{\rm conv}({\cal K}_{(X',D'\cup Z')/S'}),P^{D'}) \\ 
%{} & = Ru^{\rm conv}_{(X',Z')/S'*}
%(R\eps'_{{\rm conv}*}({\cal K}_{(X',D'\cup Z')/S'}),\tau) \\
%{} & \lo Ru^{\rm conv}_{(X',Z')/S'*}(R\eps'_{{\rm conv}*}
%Rg_{{\rm conv}*}
%({\cal K}_{(X,D\cup Z)/S}),\tau)  \\
%{} &  = Ru^{\rm conv}_{(X',Z')/S'*}(Rg_{(X,Z){\rm conv}*}
%R\eps_{{\rm conv}*}({\cal K}_{(X,D\cup Z)/S}),\tau) \\ 
%{} & \lo Ru^{\rm conv}_{(X',Z')/S'*}Rg_{(X,Z){\rm conv}*}
%(R\eps_{{\rm conv}*}({\cal K}_{(X,D\cup Z)/S}),\tau) \\ 
%{} & = Ru^{\rm conv}_{(X',Z')/S'*}Rg_{(X,Z){\rm conv}*}
%(E_{\rm conv}({\cal K}_{(X,D\cup Z)/S}),P^D) \\ 
%{} & =Rg_{{\rm zar}*}Ru^{\rm conv}_{(X,Z)/S'*}
%(E_{\rm conv}({\cal K}_{(X,D\cup Z)/S}),P^D) \\ 
%{} & =Rg_{{\rm zar}*}(C_{\rm isozar}
%({\cal K}_{(X,D\cup Z)/S}),P^D). 
%\end{align*}
%Here the first and the last 
%equalities follow from 
%(\ref{theo:wtvsca}); the first 
%arrow is induced by $g^{\log *}_{\rm conv}$ 
%and the second arrow is obtained from  
%(\ref{eqn:taure}). 
\end{proof}

\begin{coro}\label{coro:spfcz} 
Let $E_{\rm ss}((X,D\cup Z)/S)$ 
$($resp.~$E_{\rm ss}((X',D'\cup Z')/S'))$ 
be the  spectral sequence {\rm (\ref{eqn:wtspkdz})} 
$($resp.~{\rm (\ref{eqn:wtspkdz})} for 
$(X',D'\cup Z')/S')$. 
Then the morphism $g^*_{\rm conv}$
induces a morphism 
\begin{equation*}
g^*_{\rm conv} \col 
u^{-1}E_{\rm ss}((X',D'\cup Z')/S') \lo E_{\rm ss}((X,D\cup Z)/S) 
\tag{9.2.1}
\end{equation*} 
of spectral sequences. 
\end{coro}
\begin{proof}
The proof is straightforward.
\end{proof}

\par 
Let 
$a'{}^{(k)} \col 
(D'{}^{(k)},Z'\vert_{D'{}^{(k)}}) \lo (X',Z')$ 
be the natural morphism.  
Let $a'{}^{(k)}_{\rm conv}: 
((D'{}^{(k)},Z'\vert_{D'{}^{(k)})/S})_{\rm conv} 
\lo 
((X',Z')/S)_{\rm conv}$
be the morphism of topoi induced by the 
morphism 
$a'{}^{(k)}$. 
By (\ref{theo:fcpwczo}) and (\ref{eqn:grpwwd}), 
the morphism 
$g^*_{(X,Z){\rm conv}}$ 
induces the following morphism 
\begin{align*}
{\rm gr}^{P^D}_k(Ru_{(X',Z')/S'*}(g^*_{(X,Z){\rm conv}})) 
\col  Ru_{(X',Z')/S'*}a'{}^{(k)}_{{\rm conv}*}
({\cal O}_{(D'{}^{(k)},Z'\vert_{D'{}^{(k)}})/S'}
\otimes_{\mab Z}
\vp^{(k)}_{\rm conv}(D'/S';Z'))[-k] 
\tag{9.2.2}\label{ali:gdkcl}\\ 
\lo 
Ru_{(X',Z')/S'*}a'{}^{(k)}_{{\rm conv}*}
Rg_{D^{(k)}{\rm conv}*}
({\cal O}_{(D^{(k)},Z\vert_{D^{(k)}})/S}
\otimes_{\mab Z}
\vp^{(k)}_{\rm conv}(D/S;Z))[-k]. 
\end{align*} 
Assume that  $g$ induces a morphism 
$g_{D^{(k)}} \col (D^{(k)},Z\vert_{D^{(k)}}) \lo 
(D'{}^{(k)},Z'\vert_{D'{}^{(k)}})$ for any $k\in {\mab N}$. 
In the following, we make the morphism 
${\rm gr}^{P^D}_k(g^*_{(X,Z){\rm conv}})$ in 
(\ref{ali:gdkcl}) explicit in a certain case.
\par
Assume that the following two conditions hold:
\bigskip
\parno  
(9.2.3): there exists the same cardinality of 
smooth components of $D$ and $D'$ over 
$S_1$ and $S'_1$, respectively:
$D=\bigcup_{\lam \in \Lam}D_{\lam}$, 
$D'=\bigcup_{\lam \in \Lam}D'_{\lam}$, 
where $D_{\lam}$ and $D'_{\lam}$ are 
smooth Cartier divisors on 
$X/S_1$ and $X'/S'_1$, respectively. 
\medskip
\parno 
(9.2.4): 
there exist positive integers 
$e_{\lam}$ 
$(\lam \in \Lam)$ such that
$g^*(D'_{\lam})=e_{\lam}D_{\lam}$. 
\bigskip 
\par
Set $\ul{\lam}:=\{\lam_1, \ldots, \lam_k\}$
$(\lam_j \in \Lam,~(\lam_i \not= \lam_j~(i\not= j)))$ 
and 
$D_{\ul{\lam}}:=D_{\lam_1}\cap \cdots \cap D_{\lam_k}$, 
$D'_{\ul{\lam}}:=D'_{\lam_1}\cap \cdots \cap D'_{\lam_k}$.  
Let $a_{\ul{\lam}} \col 
(D_{\ul{\lam}}, Z\vert_{D_{\ul{\lam}}}) 
\lo (X,Z)$ and 
$a'_{\ul{\lam}} \col
(D'_{\ul{\lam}},Z'\vert_{D'_{\ul{\lam}}}) 
\lo (X',Z')$
be the natural exact closed immersions.
Consider the following 
direct factor of the morphism (\ref{ali:gdkcl}): 
\begin{align*}
Ru_{(X',Z')/S'*}(g^*_{\ul{\lam}{\rm conv}}) \col 
Ru_{(X',Z')/S'*}a'{}_{\ul{\lam}{\rm conv}*}
({\cal O}_{(D'_{\ul{\lam}},
Z'\vert_{D'_{\ul{\lam}}})/S'}
\otimes_{\mab Z}
\vp_{\ul{\lam}{\rm conv}}(D'/S';Z'))
[-k]  \tag{9.2.5}\label{ali:grgm}\\
\lo 
Ru_{(X',Z')/S'*}a'{}_{\ul{\lam}{\rm conv}*}
Rg_{\ul{\lam}{\rm conv}*}
({\cal O}_{(D_{\ul{\lam}},
Z\vert_{D_{\ul{\lam}}})/S}
\otimes_{\mab Z}
\vp_{\ul{\lam}{\rm conv}}(D/S;Z))
[-k],  
\end{align*}
where $\vp_{\ul{\lam}{\rm conv}}(D'/S';Z')$ and 
$\vp_{\ul{\lam}{\rm conv}}(D/S;Z)$ are 
the log convergent orientation sheaves  
associated to the sets $\{D'_{\lam_j}\}_{j=1}^k$ 
and $\{D_{\lam_j}\}_{j=1}^k$, respectively,  
as defined in \cite[p.~81]{nh2} 
and 
$Ru_{(X',Z')/S'*}a'{}_{\ul{\lam}{\rm conv}*}
Rg_{\ul{\lam}{\rm conv}*}
({\cal O}_{(D_{\ul{\lam}},
Z\vert_{D_{\ul{\lam}}})/S}
\otimes_{\mab Z}
\vp_{\ul{\lam}{\rm conv}}(D/S;Z))$ 
is the ``$\ul{\lam}$-th direct factor'' of 
$Ru_{(X',Z')/S'*}a'{}^{(k)}_{{\rm conv}*}
Rg_{D^{(k)}{\rm conv}*}
({\cal O}_{(D^{(k)},Z\vert_{D^{(k)}})/S}
\otimes_{\mab Z}
\vp^{(k)}_{\rm conv}(D/S;Z))$.

\begin{prop}\label{prop:grloc}
Let the notations 
and the assumptions be as above.
Let 
$$g_{(D_{\ul{\lam}},Z\vert_{D_{\ul{\lam}}})} 
\col (D_{\ul{\lam}}, Z\vert_{D_{\ul{\lam}}}) 
\lo
(D'_{\ul{\lam}}, 
Z'\vert_{D'_{\ul{\lam}}})$$ 
be the induced morphism by $g$.
Then the morphism 
$Ru_{(X',Z')/S'*}(g^*_{\ul{\lam}{\rm conv}})$ in 
{\rm (\ref{ali:grgm})}  is equal to $(\prod_{j=1}^ke_{\lam_j})
Ru_{(X',Z')/S'*}(a'{}_{\ul{\lam}{\rm conv}*}
g^*_{(D_{\ul{\lam}},Z\vert_{D_{\ul{\lam}}})
{\rm conv}})$ for 
$k \geq 0$. Here we define 
$\prod_{j=1}^ke_{\lam_j}$ as $1$ for $k=0$.
\end{prop}
\begin{proof}
One has only to imitate the proof of 
\cite[(2.9.3)]{nh2}. 
\end{proof}

\begin{defi}\label{defi:wcel}
(1) We call 
$\{e_{\lam}\}_{\lam \in \Lam}\in 
({\mab Z}_{>0})^{\Lam}$ 
the {\it multi-degree} of $g$ 
with respect to the decompositions of 
$\{D_{\lam}\}_{\lam}$ 
and $\{D'_{\lam}\}_{\lam}$ of $D$ and $D'$ 
(\cite[(2.1.8)]{nh2}), respectively.   
We  denote it by 
${\rm deg}_{D,D'}(g) \in ({\mab Z}_{>0})^{\Lam}$. 
If $e_\lam$'s for all 
$\lam$'s 
are equal, 
we also denote 
$e_{\lam}\in {\mab Z}_{>0}$ 
by 
${\rm deg}_{D,D'}(g) \in 
{\mab Z}_{>0}$.
\par
(2) Assume that 
$e_\lam$'s for all $\lam$'s are equal. 
Let $v \col {\cal E} \lo {\cal F}$ 
be a morphism of  
${\cal K}_S$-modules. 
Let $k$ be a nonnegative integer. 
The $k$-{\it twist} 
$$v(-k) \col {\cal E}(-k;g;D,D') 
\lo {\cal F}(-k;g;D,D')$$  
of $v$ with respect to $g$ 
is, by definition, the morphism  
${\rm deg}_{D,D'}(g)^kv 
\col {\cal E} \lo {\cal F}$.
\end{defi}

\begin{coro}\label{coro:dts}
Assume that $e_\lam$'s for 
all $\lam$'s are equal. 
Let $E_{\rm ss}((X,D\cup Z)/S)$ be 
the following spectral sequence  
\begin{equation*} 
E_1^{-k,h+k}((X,D\cup Z)/S)=
R^{h-k}f^{\rm conv}_{(D^{(k)},Z\vert_{D^{(k)}})/S*}
({\cal K}_{(D^{(k)},Z\vert_{D^{(k)}})/S}\otimes_{\mab Z}
\vp^{(k)}_{{\rm conv}}(D/S;Z))(-k;g;D,D') 
\end{equation*}
$$\Lo 
R^hf^{\rm conv}_{(X,D\cup Z)/S*}({\cal K}_{(X,D\cup Z)/S})$$ 
and let $E_{\rm ss}((X',D'\cup Z')/S')$ be the obvious 
analogue of the above for $(X',D'\cup Z')/S'$.  
Then there exists a morphism  
\begin{equation*} 
g^*_{\rm conv}
\col u^{-1}E_{\rm ss}((X',D'\cup Z')/S') 
\lo E_{\rm ss}((X,D\cup Z)/S)
\tag{9.5.1}\label{eqn:dtw}
\end{equation*} 
of spectral sequences which induces the morphism 
$g^*_{\rm conv}\col 
R^hf^{\rm conv}_{(X',D'\cup Z')/S'*}
({\cal K}_{(X',D'\cup Z')/S'})
\lo 
R^hf^{\rm conv}_{(X,D\cup Z)/S*}({\cal K}_{(X,D\cup Z)/S})$. 
\end{coro}
\begin{proof}
(\ref{coro:dts}) immediately follows 
from (\ref{prop:grloc}).
\end{proof}

Let $F_{S_1} \col S_1 \lo S_1$ be 
the absolute Frobenius endomorphism of $S_1$.
Set 
$(X',D'\cup Z'):=(X,D\cup Z)\times_{S_1,F_{S_1}}S_1$.
The relative Frobenius morphism 
$F \col (X, D\cup Z) \lo (X', D'\cup Z')$ 
over $S_1$ induces the relative 
Frobenius morphisms $F_{(X,Z)} \col (X,Z) \lo (X',Z')$ 
and  
$F^{(k)} \col 
(D^{(k)}, Z\vert_{D^{(k)}}) \lo
(D^{(k)'}, Z'\vert_{D^{(k)'}})$. 
Let $a^{(k)} \col (D^{(k)},Z\vert_{D^{(k)}}) \lo (X,D\cup Z)$ 
and 
$a^{(k)'} \col (D^{(k)'},Z'\vert_{D^{(k)'}}) \lo (X',D'\cup Z')$ 
be the natural morphisms. 
We define the 
relative Frobenius action
$$\Phi_{(D^{(k)},Z\vert_{D^{(k)}})/S} 
\col  a^{(k)'}_{{\rm conv}*}
\vp^{(k)}_{\rm conv}(D'/S;Z') 
\lo F_{{\rm conv}*}
a^{(k)}_{{\rm conv}*}
\vp^{(k)}_{{\rm conv}}(D/S;Z)$$
as the identity 
under the natural identification
$$\vp^{(k)}_{\rm conv}(D'/S';Z')
\os{\sim}{\lo}
F^{(k)}_{{\rm conv}*}
\vp^{(k)}_{{\rm conv}}(D/S;Z).$$
When $g$ in (\ref{cd:mpsmlgs}) 
is equal to the relative Frobenius 
$F \col 
(X,D\cup Z) \lo 
(X',D'\cup Z')$,
we denote (\ref{eqn:wtspkdz})+(the compatibility with Frobenius) 
by the following:  
\begin{equation*}
\begin{split}
E_1^{-k,h+k}((X,D\cup Z)/S)= &
R^{h-k}f^{\rm conv}_{(D^{(k)},Z\vert_{D^{(k)}})/S*}
({\cal K}_{(D^{(k)},Z\vert_{D^{(k)}})/S} \\ 
{} & 
\phantom{R^{h-k}f^{\rm conv}_{(D^{(k)},Z\vert_{D^{(k)}})/S*}(}
\otimes_{\mab Z}
\vp^{(k)}_{{\rm conv}}(D/S;Z)))(-k) \\
\Lo 
{} & R^hf^{\rm conv}_{(X,D\cup Z)/S*}({\cal K}_{(X,D\cup Z)/S}).
\end{split}
\tag{9.5.2}\label{ali:wtfapwt}
\end{equation*}

\begin{defi}
We call the sequence (\ref{ali:wtfapwt}) 
the {\it weight spectral sequence of} 
$(X, D\cup Z)/S$ (more precisely over $(S_1,S)$) 
{\it with respect to} $D$. If $Z=\emptyset$, then 
we call it 
the {\it weight spectral sequence of} $(X, D)/S$. 
\end{defi}

%\par 
%Let $k$ be a nonnegative integer.
%By (\ref{theo:calvc}) and by the proof of (\ref{prop:grloc}), 
%we also have the following  
%commutative diagram which is 
%compatible with the Frobenius
%\begin{equation*}
%\begin{CD}
%R^k\eps_{(X',D'\cup Z',Z')/S*}
%({\cal K}_{(X',D'\cup Z')/S}) @=  \\ 
%@V{F^{\log *}_{(X,Z){\rm conv}}}VV  \\
%F_{(X,Z){\rm conv}*}
%R^k\eps_{(X,D\cup Z,Z)/S*}({\cal K}_{(X,D\cup Z)/S}) @= 
%\end{CD}
%\tag{9.9.2}\label{cd:fvcs}
%\end{equation*}
%\begin{equation*}
%\begin{CD}
%a^{(k)'{\log}}_{{\rm conv}*}
%({\cal K}_{(D^{(k)'},Z'\vert_{D^{(k)'}})/S}
%\otimes_{\mab Z}
%\vp^{(k)}_{\rm conv}(D'/S;Z')) \\
%@VV{p^kF^{\log *}_{(X,Z){\rm conv}}}V \\
%F_{(X,Z){\rm conv}*}
%a^{(k){\log}}_{{\rm conv}*}
%({\cal K}_{(D^{(k)},Z\vert_{D^{(k)}})/S}\otimes_{\mab Z}
%\vp^{(k){\rm log}}_{{\rm conv}}(D/S;Z))).
%\end{CD}
%\end{equation*}
%As usual, we denote 
%the commutative diagram (\ref{cd:fvcs}) 
%by the following formula:
%\begin{equation*}
%R^k\eps_{(X,D\cup Z,Z)/S*}
%({\cal K}_{(X,D\cup Z,Z)/S})=
%a^{(k){\log}}_{{\rm conv}*}
%({\cal K}_{(D^{(k)},Z\vert_{D^{(k)}})/S}
%\otimes_{\mab Z}\vp^{(k)\log}_{\rm conv}(D/S;Z))(-k). 
%\tag{9.6.3}
%\end{equation*}
By (\ref{eqn:levcs}) we also 
have the following Leray spectral sequence
\begin{equation*}
E^{st}_2 := R^sf^{\rm conv}_{(D^{(t)},Z\vert_{D^{(t)}})/S*}
({\cal K}_{(D^{(t)},Z\vert_{D^{(t)}})/S}\otimes_{\mab Z}
\vp^{(t)}_{{\rm conv}}(D/S;Z))(-t)
\tag{9.6.1}\label{eqn:sp}
\end{equation*} 
$$\Longrightarrow 
R^{s+t}f^{\rm conv}_{(X,D\cup Z)/S*}({\cal K}_{(X,D\cup Z)/S}).$$

\begin{prop}\label{prop:iwep}  
Let the notations be as in {\rm (\ref{prop:iekp})}. 
Then the spectral sequence 
{\rm (\ref{ali:wtfapwt})} for 
$(X',D'\cup Z')/(S'_1,S)$ is equal to that for 
$(X,D\cup Z)/(S_1,S)$. 
\end{prop} 
\begin{proof}   
(\ref{prop:iwep}) follows from (\ref{prop:iekp}). 
\end{proof}

\section{Boundary morphisms}\label{sec:bd}
In this section we define the log cycle class 
in a log convergent cohomology for a smooth divisor which 
intersects a log locus transversally 
(cf.~\cite[\S2]{fao}, \cite[(2.8)]{nh2}).
Then, as in \cite[(2.8)]{nh2}, we give a 
description of the boundary morphism between 
the $E_1$-terms of the spectral sequence 
(\ref{ali:wtfapwt}). Because the proof is the same as that of \cite[(2.8)]{nh2}, 
we omit the proofs in this section. 
\par
Let $S$ and $S_1$ be as in \S\ref{sec:lcs}. 
Let $f \col (X,Z) \lo S_1$ 
be a smooth scheme with a relative SNCD over $S_1$. 
By abuse of notation, 
we also denote by $f$ the composite 
morphism $(X,Z) \lo S_1 \os{\subset}{\lo} S$. 
Let $D$ be 
a smooth divisor on $X$ which intersects  
$Z$ transversally over $S_1$.  
%for a decomposition 
%$\Del=\{Z_{\mu}\}_{\mu}$ of $Z$ by smooth 
%components of $Z$, 
%$\Del(D):=\{D, Z_{\mu}\}_{\mu}$ is 
%a decomposition of $D\cup Z$ by smooth 
%components of $D\cup Z$. 
%The closed scheme  
%$Z\vert_D:=Z\cap D$ in $D$ 
%is a relative SNCD on $D/S_1$;
%$\Del\vert_D:=\{Z_{\mu}\vert_D\}_{\mu}$ 
%is a decomposition of $Z\vert_D$ by 
%smooth components of $Z\vert_D$.
Let 
$a \col (D,Z\vert_D) 
\os{\subset}{\lo} (X,Z)$ be 
the natural 
closed immersion over $S_1$.  
Let 
$a_{\rm zar} \col D_{\rm zar}\lo X_{\rm zar}$ 
be the induced morphism of Zariski topoi.
Let $a_{\rm conv} \col
(({(D, Z\vert_D)/S})_{\rm conv},
{\cal K}_{(D, Z\vert_D)/S}) 
\lo 
(({(X, Z)/S})_{\rm conv},
{\cal K}_{(X, Z)/S})$ 
be also the induced morphism of 
log convergent ringed topoi.
%Let 
%\begin{equation*}
%{\rm Res}^D \col 
%\Om^{\bul}_{X/S_1}(\log (D\cup Z))  \lo  
%a_{{\rm zar}*}(\Om^{\bul}_{D/S_1}(\log Z\vert_D)
%\otimes_{\mab Z}
%\vp^{(1)}_{\rm zar}(D/S_1))[-1] 
%\tag{11.0.1}
%\end{equation*}
%be the Poincar\'{e} residue morphism 
%with respect to $D/S_1$.
%Then we have the following exact sequence: 
%\begin{equation*}
%0 \lo \Om_{X/S_1}^{\bul}(\log Z) 
%\lo \Om_{X/S_1}^{\bul}(\log(D\cup Z)) 
%\tag{11.0.2}
%\end{equation*}
%$$\os{{\rm Res}^D}{\lo}  
%a_{{\rm zar}*}(\Om_{D/S_1}^{\bul}(\log Z\vert_D)
%\otimes_{\mab Z}\vp^{(1)}_{\rm zar}(D/S_1))[-1] \lo 0.$$
\par
Let the notations be as 
in the beginning of \S\ref{sec:wfcipp} for the log scheme 
$(X,D\cup Z)$ above.
%Let ${\cal M}_{\bul}$ be a cosimplicial fs sub log structure 
%${\cal M}_{\bul}$ of $M_{{\cal P}_{\bul}}$ such that 
%\begin{equation*} 
%M_{{\cal P}_{\bul}}
%={\cal M}_{\bul}\oplus_{{\cal O}^*_{{\cal P}_{\bul}}}
%M_{{\cal Q}_{\bul}},  
%\tag{11.0.1}\label{eqn:mbpql} 
%\end{equation*} 
%such that the pull-back of ${\cal M}_{\bul}$ to 
%$X_{\bul}$ is equal to $M(D_{\bul})$ and such that 
%the condition (\ref{eqn:epynr}) is satisfied for 
%$(\os{\circ}{\cal P}_{\bul},{\cal M}_{\bul})$. 
Let $b_{\bul} \col 
(D^{(1)}({\cal M}^{\rm ex}_{\bul}),
{\cal N}^{\rm ex}_{\bul})
\os{\sus}{\lo} {\cal Q}_{\bul}$ 
be the natural simplicial immersion. 
By using the Poincar\'{e} residue isomorphism 
with respect to $D^{(1)}({\cal M}^{\rm ex}_{\bul})$, 
we have the following exact sequence:
\begin{equation*}
0 \lo 
\Om_{{\cal Q}^{\rm ex}_{\bul}/S}^{\bul} 
\lo \Om_{{\cal P}^{\rm ex}_{\bul}/S}^{\bul} 
\os{\rm Res}{\lo} 
b_{{\bul}{\rm zar}*}
(\Om_{D^{(1)}({\cal M}^{\rm ex}_{\bul})/S}^{\bul}
\otimes_{\mab Z}
\vp^{(1)}_{\rm zar}(D^{(1)}({\cal M}^{\rm ex}_{\bul})/S))
[-1] \lo 0.
\tag{10.0.1}\label{eqn:omresd} 
\end{equation*} 
Let $L^{\rm conv}_{(X_{\bul},Z_{\bul})/S}$ 
(resp.~$L^{\rm conv}_{(D_{\bul},Z_{\bul}\vert_{D_{\bul}})/S}$)
be the log convergent linearization functor 
with respect to the simplicial immersion  
$(X_{\bul},Z_{\bul}) \os{\subset}{\lo} {\cal Q}_{\bul}$
(resp.~$(D_{\bul},Z_{\bul}\vert_{D_{\bul}}) 
\os{\subset}{\lo} 
(D^{(1)}({\cal M}^{\rm ex}_{\bul}),
{\cal N}^{\rm ex}_{\bul})$). 
By (\ref{linc}), 
$L^{\rm conv}_{(X_{\bul},Z_{\bul})/S}b_{\bul{\rm zar}*}= 
a_{{\rm conv}\bul*}
L^{\rm conv}_{(D_{\bul},Z_{\bul}\vert_{D_{\bul}})/S}$. 
Hence we have the following exact 
sequence by (\ref{prop:grla}) and (\ref{theo:injf}) (2): 
\begin{equation*}
0 \lo 
L^{\rm conv}_{(X_{\bul},Z_{\bul})/S}
(\Om_{{\cal Q}^{\rm ex}_{\bul}/S}^{\bul}) \lo 
L^{\rm conv}_{(X_{\bul},Z_{\bul})/S}
(\Om_{{\cal P}^{\rm ex}_{\bul}/S}^{\bul})
\tag{10.0.2}\label{eqn:xzl}
\end{equation*} 
$$\lo 
a_{{\rm conv}\bul*}
(L^{\rm conv}_{(D_{\bul},Z_{\bul}\vert_{D_{\bul}})/S}
(\Om_{D^{(1)}({\cal M}^{\rm ex}_{\bul})/S}^{\bul})
\otimes_{\mab Z}\vp^{(1)}_{{\rm conv}}(D_{\bul}/S;Z_{\bul}))
[-1]) 
\lo 0.$$
Recall the  morphism 
${\pi}_{(X,Z)/S{\rm conv}}$ 
of ringed topoi in \S\ref{sec:wfcipp}
and let 
${\pi}_{(D,Z\vert_D)/S{\rm conv}}$ 
be an analogously defined morphism. 
By applying $R{\pi}_{(X,Z)/S{\rm conv}*}$ to (\ref{eqn:xzl}), 
we obtain the following triangle:  
\begin{equation*}
\lo R{\pi}_{(X,Z)/S{\rm conv}*}
L^{\rm conv}_{(X_{\bul},Z_{\bul})/S}
(\Om_{{\cal Q}^{\rm ex}_{\bul}/S}^{\bul}) \lo 
R{\pi}_{(X,Z)/S{\rm conv}*}
L^{\rm conv}_{(X_{\bul},Z_{\bul})/S}
(\Om_{{\cal P}^{\rm ex}_{\bul}/S}^{\bul}) \lo 
\tag{10.0.3} 
\end{equation*}
$$
a_{{\rm conv}*}
R{\pi}_{(D,Z\vert_D)/S{\rm conv}*}
(L^{\rm conv}_{(D_{\bul},Z_{\bul}\vert_{D_{\bul}})/S}
(\Om^{\bul}_{D^{(1)}({\cal M}^{\rm ex}_{\bul})/S})  
\otimes_{\mab Z}
\vp^{(1)}_{{\rm conv}}(D_{\bul}/S;Z_{\bul}))[-1] 
\os{+1}{\lo} 
\cdots. $$
By (\ref{theo:pl}), (\ref{theo:cpvcs}) and 
by the cohomological descent,  
we have the following triangle:
\begin{equation*}
\lo {\cal K}_{(X,Z)/S} \lo 
R\eps^{\rm conv}_{(X,D\cup Z,Z)/S*}
({\cal K}_{(X,D\cup Z)/S}) \lo 
\tag{10.0.4}\label{eqn:oealog}
\end{equation*}
$$
a_{{\rm conv}*}
({\cal K}_{(D,Z\vert_D)/S}
\otimes_{\mab Z}
\vp^{(1)}_{{\rm conv}}(D/S;Z))
[-1] 
\os{+1}{\lo} \cdots.$$ 
Hence we have the boundary morphism 
\begin{equation*}
d \col a_{{\rm conv}*}
({\cal K}_{(D,Z\vert_D)/S}\otimes_{\mab Z}
\vp^{(1)}_{{\rm conv}}(D/S;Z))[-1] \lo 
{\cal K}_{(X,Z)/S}[1]
\tag{10.0.5}
\end{equation*}
in $D^+({\cal K}_{(X,Z)/S})$.   
Hence we have the following morphism
\begin{equation*}
d \col 
a_{{\rm conv}*}({\cal K}_{(D,Z\vert_D)/S}
\otimes_{\mab Z}\vp^{(1)}_{{\rm conv}}(D/S;Z)) 
\lo {\cal K}_{(X,Z)/S}[2]
\tag{10.0.6} 
\end{equation*}
in $D^+({\cal K}_{(X,Z)/S})$. 
Set 
\begin{equation*}
G^{\rm conv}_{D/(X,Z)}:=-d.  
\tag{10.0.7}\label{eqn:gdxz}
\end{equation*}
and call $G^{\rm conv}_{D/(X,Z)}$ the 
{\it log convergent Gysin morphism} of $D$.
Then we have the following cohomology class 
\begin{align*}
c^{\rm conv}_{(X,Z)/S}(D):=G^{\rm conv}_{D/(X,Z)} \in 
{\cal E}{\it xt}^0_{{\cal K}_{(X,Z)/S}} 
(& a_{{\rm conv}*}
({\cal K}_{(D,Z\vert_D)/S}\otimes_{\mab Z}
\vp^{(1)}_{{\rm conv}}(D/S;Z)),  
\tag{10.0.8}\label{eqn:lcycls} \\
{} & 
{\cal K}_{(X,Z)/S}[2]).
\end{align*}
It is a routine work to prove that $G^{\rm conv}_{D/(X,Z)}$ 
(and $c^{\rm conv}_{(X,Z)/S}(D)$) depends only  
on $(X,D\cup Z)/S$. 
Since $\vp^{(1)}_{{\rm conv}}(D/S;Z)$ 
is canonically isomorphic to ${\mab Z}$ and since 
there exists a natural morphism 
${\cal K}_{(X,Z)/S} \lo 
a_{{\rm conv}*}
({\cal K}_{(D,Z\vert_D)/S})$,
we have a cohomology class 
\begin{align*}
c^{\rm conv}_{(X,Z)/S}(D) & \in  
{\cal E}{\it xt}^0_{{\cal K}_{(X,Z)/S}}
({\cal K}_{(X,Z)/S},
{\cal K}_{(X,Z)/S}[2]) \\
{} & =:
\ul{\cal H}_{{\log}{\textrm -}{\rm conv}}^2((X,Z)/S).
\end{align*}
If $Z=\emptyset$, 
we denote $G^{\rm conv}_{D/(X,Z)}$ 
and $c^{\rm conv}_{(X,Z)/S}(D)$ simply by 
$G^{\rm conv}_{D/X}$ and $c^{\rm conv}_X(D)$, 
respectively.

\begin{prop}\label{prop:gysbcg}
Let $\pi'$ be a nonzero element of $\pi{\cal V}$. 
Let $u \col S'  \lo S$ be a morphism of 
$p$-adic formal ${\cal V}$-schemes. 
Set $S'_1:=\ul{\rm Spec}_{S'}({\cal O}_{S'}/\pi'{\cal O}_{S'})$. 
Let $g \col Y \lo S'_1$ be a smooth morphism of schemes 
fitting into the following commutative diagram 
\begin{equation*} 
\begin{CD} 
Y @>>> X \\ 
@VVV @VVV \\ 
S'_1 @>>> S_1. 
\end{CD} 
\end{equation*}
Set $E:=D\times_XY$ and $W:=Z\times_XY$. 
Assume that $E$ and $W$ are transversal relative SNCD's on $Y$ 
over $S_1$. 
Let $b\col (E,W\vert_E) \os{\sus}{\lo} (Y,W)$ 
be the natural closed immersion. 
Then the image of $c^{\rm conv}_{(X,Z)/S}(D)$ in
$(\ref{eqn:lcycls})$ by the natural morphism
$${\cal E}{\it xt}^0_{{\cal K}_{(X,Z)/S}}
(a_{{\rm conv}*}({\cal K}_{(D,Z\vert_D)/S}
\otimes_{\mab Z}\vp^{(1)}_{{\rm conv}}(D/S;Z)), 
{\cal K}_{(X,Z)/S}[2]) 
\lo $$
$${\cal E}{\it xt}^0_{{\cal K}_{(Y,W)/S'}}
(b_{{\rm conv}*}
({\cal K}_{(E,W\vert_E)/S'}
\otimes_{\mab Z}\vp^{(1)}_{{\rm conv}}(E/S';W)), 
{\cal K}_{(Y,W)/S'}[2])$$ 
is equal to $c^{\rm conv}_{(Y,W)/S'}(E)$.
\end{prop}

\par
Finally we express the edge morphism $d_1^{\bul \bul}$ of
(\ref{ali:wtfapwt}) by the \v{C}ech-Gysin morphism.
\par
Henceforth $D$ denotes a (not necessarily smooth) 
relative SNCD on $X$ over $S_1$ 
which meets $Z$ transversally.
First, fix a decomposition 
$\{D_{\lam}\}_{\lam \in \Lam}$ of $D$ 
by smooth components of $D$ over $S_1$ 
\cite[(2.1.8)]{nh2} and fix a total order on $\Lam$. 
Set $\ul{\lam}:=\{\lam_0,  \ldots, \lam_{k-1}\}$ 
for $\lam_0<\cdots<\lam_{k-1}$ 
and 
$\ul{\lam}_j:=\{\lam_0,\ldots,  \wh{\lam}_j,  
\ldots, \lam_{k-1}\}$ $(k\in {\mab Z}_{\geq 1})$. 
Here $~\wh{}~$ means the elimination. 
Set 
$D_{\ul{\lam}}
:=D_{\lam_0}\cap \cdots \cap D_{\lam_{k-1}}$ 
and $D_{\ul{\lam}_j}:=D_{\lam_0}\cap \cdots \cap 
\wh{D}_{{\lam}_j}\cap \cdots \cap D_{\lam_{k-1}}$ for $k\geq 2$ 
and $D_{\ul{\lam}_0}:=X$ for $k=1$. 
Assume that
$D_{\ul{\lam}}\not= \emptyset$. 
Then $D_{\ul{\lam}}$
is a smooth divisor on $D_{\ul{\lam}_j}$ over $S_1$.
Let 
$\iota^{\ul{\lam}_j}_{\ul{\lam}} 
\col (D_{\ul{\lam}},Z\vert_{D_{\ul{\lam}}}) 
\os{\subset}{\lo} 
(D_{\ul{\lam}_j},Z\vert_{D_{\ul{\lam}_j}})$ 
be the closed immersion.
As in \cite[p.~81]{nh2}, 
we have the following orientation sheaves  
$$\vp_{\ul{\lam}{\rm conv}}(D/S;Z):=
\vp_{\lam_0 \cdots \lam_{k-1}{\rm conv}}(D/S;Z)$$ 
and
$$\vp_{\ul{\lam}_j{\rm conv}}(D/S;Z):=
\vp_{\lam_0 \cdots \wh{\lam}_j 
\cdots \lam_{k-1}{\rm conv}}(D/S;Z)$$
associated to the sets $\{D_{\lam_i}\}_{i=0}^k$ 
and $\{D_{\lam_i}\}_{i=0}^k\setminus 
\{D_{\lam_j}\}$, respectively.   
We denote by $(\lam_0\cdots \lam_{k-1})$ 
the local section giving a basis of 
$\vp_{\lam_0\cdots \lam_{k-1}{\rm conv}}(D/S;Z)$; 
we also denote by $(\lam_0\cdots \lam_{k-1})$ 
the local section giving a basis of 
$\vp_{\lam_0\cdots \lam_{k-1}{\rm zar}}(D/S)$. 
By (\ref{eqn:gdxz}) we have a morphism
\begin{align*}
G^{{\ul{\lam}_j}{\rm conv}}_{\ul{\lam}}:=
G^{\rm conv}_{D_{\ul{\lam}}/(D_{\ul{\lam}_j},Z\vert_{D_{\ul{\lam}_j}})} 
\col & \iota^{\ul{\lam}_j{\log}}_{\ul{\lam}{\rm conv}*}
({\cal K}_{(D_{\ul{\lam}},
Z\vert_{D_{\ul{\lam}}})/S}
\otimes_{\mab Z}\vp_{\lam_j{\rm conv}}(D/S;Z))
\tag{10.1.1}\label{eqn:glmoo} \\
{} & \lo 
{\cal K}_{(D_{\ul{\lam}_j},
Z\vert_{D_{\ul{\lam}_j}})/S}[2].
\end{align*}
We fix an isomorphism 
\begin{equation*}
\vp_{\lam_j{\rm conv}}(D/S;Z)
\otimes_{\mab Z}\vp_{\ul{\lam}_j{\rm conv}}(D/S;Z) 
\os{\sim}{\lo}\vp_{\ul{\lam}{\rm conv}}(D/S;Z)
\tag{10.1.2}\label{eqn:coxts}
\end{equation*}
by the following morphism
$$(\lam_j)\otimes (\lam_0\cdots  
\wh{\lam}_j \cdots \lam_{k-1})
\lom (-1)^j(\lam_0\cdots \lam_{k-1}).$$
We identify 
$\vp_{\lam_j{\rm conv}}(D/S;Z)\otimes_{\mab Z} 
\vp_{\ul{\lam}_j{\rm conv}}(D/S;Z)$ with 
$\vp_{\ul{\lam}{\rm conv}}(D/S;Z)$ 
by this isomorphism.
We also have the following composite morphism
\begin{equation*}
(-1)^jG^{{\ul{\lam}_j}{\rm conv}}_{\ul{\lam}}
\col \iota^{\ul{\lam}_j{\log}}_{\ul{\lam}{\rm conv}*}
({\cal K}_{(D_{\ul{\lam}},Z\vert_{D_{\ul{\lam}}})/S}
\otimes_{\mab Z}\vp_{\ul{\lam}{\rm conv}}(D/S;Z)) 
\os{\sim}{\lo} \tag{10.1.3}\label{eqn:m1gom}
\end{equation*} 
$$\iota^{\ul{\lam}_j{\log}}_{\ul{\lam}{\rm conv}*}
({\cal K}_{(D_{\ul{\lam}},Z\vert_{D_{\ul{\lam}}})/S}
\otimes_{\mab Z}
\vp_{\lam_j{\rm conv}}(D/S;Z)\otimes_{\mab Z} 
\vp_{\ul{\lam}_j{\rm conv}}(D/S;Z)) 
\os{G^{{\ul{\lam}_j}{\rm conv}}_{\ul{\lam}}\otimes 1}{\lo} $$
$${\cal K}_{(D_{\ul{\lam}_j},Z\vert_{D_{\ul{\lam}_j}})/S}
\otimes_{\mab Z}\vp_{\ul{\lam}_j{\rm conv}}(D/S;Z)[2]$$ 
defined by 
\begin{equation*}{\rm ``} x\otimes (\lam_0 \cdots  \lam_{k-1}) \lom 
(-1)^jG^{{\ul{\lam}_j}{\rm conv}}_{\ul{\lam}}(x)\otimes 
(\lam_0\cdots \wh{\lam}_j \cdots
\lam_{k-1}){\rm "}.
\tag{10.1.4}\label{eqn:ooolll}
\end{equation*}
The morphism (\ref{eqn:ooolll}) induces 
the following morphism 
of log convergent cohomologies:
\begin{equation*}
(-1)^j G^{{\ul{\lam}_j}{\rm conv}}_{\ul{\lam}} \col 
R^{h-k}
f^{\rm conv}_{(D_{\ul{\lam}},Z\vert_{D_{\ul{\lam}}})/S*}
({\cal K}_{(D_{\ul{\lam}},Z\vert_{D_{\ul{\lam}}})/S}
\otimes_{\mab Z}
\vp_{\ul{\lam}{\rm conv}}(D/S;Z))\lo 
\tag{10.1.5}\label{eqn:cohgxdz}
\end{equation*}
$$R^{h-k+2}
f^{\rm conv}_{(D_{\ul{\lam}_j},Z\vert_{D_{\ul{\lam}_j}})/S*}
({\cal K}_{(D_{\ul{\lam}_j},Z\vert_{D_{\ul{\lam}_j}})/S} 
\otimes_{\mab Z}\vp_{\ul{\lam}_j{\rm conv}}(D/S;Z)).$$ 
If $D_{\{\lam_0,  \ldots, \lam_{k-1}\}}= \emptyset$, set
$(-1)^jG^{{\ul{\lam}_j}{\rm conv}}_{\ul{\lam}}:=0$.
\par
Denote by $a_{\ul{\lam}}$ (resp.~$a_{\ul{\lam}_j}$)
the natural exact closed immersion 
$(D_{\ul{\lam}},Z\vert_{D_{\ul{\lam}}}) \os{\subset}{\lo} (X,Z)$
(resp.~$(D_{\ul{\lam}_j},Z\vert_{D_{\ul{\lam}_j}}) 
\os{\subset}{\lo} (X,Z)$).

\begin{prop}\label{prop:dbd} 
Let 
$d_1^{-k, h+k} \col 
E_1^{-k, h+k} \lo E_1^{-k+1, h+k}$ 
be the boundary
morphism of {\rm (\ref{ali:wtfapwt})}. 
Set 
$G^{\rm conv}:=\sum_{\{\lam_0,\ldots, \lam_{k-1}~\vert~\lam_i 
\not= \lam_j~(i\not= j)\}}
\sum_{j=0}^{k-1}(-1)^j
G^{{\ul{\lam}_j}{\rm conv}}_{\ul{\lam}}$.  
Then
$d_1^{-k, h+k} =-G^{\rm conv}$. 
\end{prop}

\section{Comparison theorem}\label{sec:ct}
In this section we compare 
the weight-filtered isozariskian complex 
of a certain $N$-truncated simplicial smooth scheme 
with a certain $N$-truncated simplicial relative SNCD
and the weight-filtered zariskian complex defined in 
\cite{nh2} and \cite{nh3}. 
First we consider the non-filtered case.  
The following has been proved in \cite[(6.1)]{nh3}:

\begin{lemm}[{\cite[(6.1)]{nh3}}]\label{lemm:disj}
Let $Y_{\bul}$ be 
a split simplicial fine log $($formal$)$ scheme 
over a fine log $($formal$)$ scheme $Y$ over 
a fine log $($formal$)$ scheme $T$.
Then there exists 
a fine split simplicial log $($formal$)$ scheme 
$Y'_{\bul}$ with a morphism $Y'_{\bul} \lo Y_{\bul}$ 
of simplicial log $($formal$)$ schemes over $Y$ 
satisfying the following conditions$:$
\medskip 
\parno
$(1)$ $Y'_m$ $(m\in {\mab N})$ is the disjoint 
union of affine log $($formal$)$ open subschemes 
which cover $Y_m$ and whose images in 
$Y$ are contained in affine log $($formal$)$ open subschemes of $Y$.
\medskip 
\parno  
$(2)$ If $\os{\circ}{Y}_m$ $(m\in {\mab N})$ 
is quasi-compact, then the 
number of the open subschemes in $(1)$ 
can be assumed to be finite.
\medskip
\parno
In particular, there exists a natural morphism
$Y_{\bul \bul} \lo Y_{\bul}$ over $T$  
by setting 
$Y_{mn}:=
{\rm cosk}_0^{Y_m}(Y_m')_n$.
Moreover, for each $n\in {\mab N}$, 
$Y_{\bul n}$ is split. 
\end{lemm}

\parno 
In [loc.~cit.] we have called 
the simplicial log (formal) scheme 
$Y'_{\bul}$ satisfying $(1)$ and $(2)$ in 
(\ref{lemm:disj}) the {\it disjoint union of the members of 
an affine simplicial open covering} of $Y_{\bul}/Y$; 
we have called the bisimplicial scheme 
$Y_{\bul \bul}$ in  (\ref{lemm:disj}) 
the {\it \v{C}ech diagram} of $Y'_{\bul}$
over $Y_{\bul}/Y$. 
%(3) In (1) and (2), if $Y=T$, 
%then we say ``over $T$'' instead of ``over $Y$''. 

\par 
The following has been proved in \cite[(6.3)]{nh3}: 
\begin{prop}[{\cite[(6.3)]{nh3}}]\label{prop:dissmp}
$(1)$ Let $Y_{\bul} \lo Z_{\bul}$ 
be a morphism of 
split simplicial fine log $($formal$)$ schemes 
over a morphism $Y\lo Z$ of  
fine log $($formal$)$ schemes 
over a fine log $($formal$)$ scheme $T$. 
Then there exist the disjoint unions of 
the members of affine simplicial open coverings
$Y'_{\bul}$ and $Z'_{\bul}$
of $Y_{\bul}/Y$ and $Z_{\bul}/Z$, 
respectively, which fit into 
the following commutative 
diagram$:$
\begin{equation*}
\begin{CD}
Y'_{\bul} @>>> Z'_{\bul}  \\
@VVV @VVV  \\
Y_{\bul} @>>> Z_{\bul}.
\end{CD}
\tag{11.2.1}\label{cd:cechbs} 
\end{equation*}  
\par
$(2)$  Let $Y_{\bul} \lo Y\lo T$ be as in 
{\rm (\ref{lemm:disj})}. 
Let $Y'_{\bul}$ and $Y''_{\bul}$ 
be two disjoint unions of the members of 
affine simplicial open coverings of $Y_{\bul}/Y$. 
Then there exists a disjoint union of 
the members of 
an affine simplicial open covering $Y'''_{\bul}$
of $Y_{\bul}/Y$ 
fitting into the following 
commutative diagram$:$
\begin{equation*}
\begin{CD}
Y'''_{\bul} @>>> Y''_{\bul}\\
@VVV @VVV  \\
Y'_{\bul} @>>> Y_{\bul}.
\end{CD}
\tag{11.2.2}\label{cd:ceccov}
\end{equation*} 
\end{prop}

\begin{rema}\label{rema:trhd} 
Let $N$ be a nonnegative integer. 
The $N$-truncated versions of (\ref{lemm:disj})
and (\ref{prop:dissmp}) also hold. 
\end{rema}

\par 
Let ${\cal V}$, $\pi$, $S$ and $S_1$ 
be as in \S\ref{sec:logcd}. 
Assume that $\pi{\cal V}=p{\cal V}$. 
Let $f \col Y_{\bul \leq N}\lo S_1$ be a fine 
$N$-truncated simplicial log scheme over $S_1$.  
Then we have the log crystalline topos 
$({Y_{\bul \leq N}/S})_{\rm crys}$ of $Y_{\bul \leq N}/(S,p{\cal O}_S,[~])$, 
where $[~]$ is the canonical PD-structure on 
$p{\cal O}_S$. 
Let 
$$u^{\rm crys}_{Y_{\bul \leq N}/S} \col 
(({Y_{\bul \leq N}/S})_{\rm crys},{\cal O}_{Y_{\bul \leq N}/S}) 
\lo  
((Y_{\bul \leq N})_{\rm zar},f^{-1}({\cal O}_S))$$ 
be the canonical projection.

\begin{theo}[{\bf Comparison theorem}]\label{theo:hct} 
Assume that $\pi{\cal V}=p{\cal V}$. 
Let $N$ be a nonnegative integer. Let $Y_{\bul \leq N}$ be 
a log smooth 
$N$-truncated simplicial log scheme over $S_1$ 
which has the disjoint union $Y'_{\bul \leq N}$ 
of the members of an affine $N$-truncated simplicial 
open covering of $Y_{\bul \leq N}$. 
Denote by 
\begin{equation*}
\Xi_{\bul \leq N} \col {\rm Isoc}_{\rm conv}(Y_{\bul \leq N}/S)
(:={\rm Isoc}_{\rm conv, zar}(Y_{\bul \leq N}/S)) 
\lo {\rm Isoc}_{\rm crys}(Y_{\bul \leq N}/S)
\tag{11.4.1}\label{eqn:icym}
\end{equation*} 
the relative simplicial version of 
the functor defined in {\rm \cite[Theorem 5.3.1]{s1}}
and denoted by $\Phi$ in $[${\rm loc.~cit.}$]$. 
Let $E^{\bul \leq N}$ be an object of 
${\rm Isoc}_{\rm conv}(Y_{\bul \leq N}/S)$. 
Then there exists a functorial isomorphism 
\begin{equation*}
Ru^{\rm conv}_{Y_{\bul \leq N}/S*}(E^{\bul \leq N})
\os{\sim}{\lo}
Ru^{\rm crys}_{Y_{\bul \leq N}/S*}(\Xi_{\bul \leq N}(E^{\bul \leq N})).
\tag{11.4.2}\label{eqn:kymc}  
\end{equation*} 
\end{theo}
\begin{proof}  
%Let $Y'_{\bul \leq N}$ be the disjoint union of 
%the members of 
%an affine simplicial open covering of $Y_{\bul \leq N}/S_1$. 
Since $Y'_N$ is affine, 
we have a closed immersion 
$Y'_N \os{\sus}{\lo} {\cal Y}'(N)$ into 
a formally log smooth log noetherian formal scheme over $S$. 
To construct a natural morphism 
\begin{equation*}
Ru^{\rm conv}_{Y_{\bul \leq N}/S*}(E^{\bul \leq N}) \lo
Ru^{\rm crys}_{Y_{\bul \leq N}/S*}(\Xi_{\bul \leq N}(E^{\bul \leq N})),
\tag{11.4.3}\label{eqn:ntkymc}  
\end{equation*} 
we need Tsuzuki's functor 
(\cite[\S11]{ctze}, \cite[(7.3.1)]{tzcp}) 
as follows. 
\par
Let $\Del$ be the standard simplicial 
category; an object of $\Del$ is denoted by 
$[m]:=\{0, \ldots, m\}$ $(m\in {\mab N})$. 
Let ${\cal C}$ be a category which has 
finite inverse limits.   
Recall a simplicial object 
$\Gam:=\Gam^{\cal C}_{n}(X):=
\Gam^{\cal C}_{n}(X)^{\leq n}$ in 
${\cal C}$
for an object $X\in {\cal C}$ and 
$n\in {\mab N}$
(\cite[(7.3.1)]{tzcp}):
for an object $[m]$ of $\Del$, set 
$$\Gam_{{}{m}}:=\prod_{\del \in {\rm Hom}_{\Del}([n],[m])}X_{\del}$$ 
with $X_{\del}=X$; 
for a morphism $\al \col [m'] \lo [m]$ 
in $\Del$, $\al_{\Gam} \col \Gam_{m} \lo 
\Gam_{m'}$ is defined to be the following:
``$(c_\gam) \lom (d_{\bet})$'' with $d_{\bet}=c_{\al \bet}$. 
In fact, we have a functor
\begin{equation*}
\Gam^{\cal C}_N(?) \col {\cal C} \lo 
{\cal C}^{\Del}:=
\{\text{simplicial objects in ${\cal C}$}\}.
\tag{11.4.4}
\end{equation*} 
\par 
Let $X_{\bul \leq N}$ be 
an $N$-truncated simplicial object of ${\cal C}$ and 
let $f \col X_N \lo Y$ be a morphism in ${\cal C}$.
Then we have a morphism
\begin{equation*}
X_{\bul \leq N} \lo 
\Gam^{\cal C}_{{}{N}}(Y)_{\bul \leq N}
\tag{11.4.5}\label{eqn:cgam}
\end{equation*}
defined by the following commutative diagram:
\begin{equation*}
\begin{CD}
X_m @>>> \prod_{\del \in {\rm Hom}_{\Del}([N],[m])}Y_{\del}\\ 
@V{X(\gam)}VV @VV{\rm proj.}V \\
X_N @>{f}>> Y_{\gam}=Y
\end{CD}
\tag{11.4.6}\label{cd:embdf}
\end{equation*} 
for any morphism $\gam \col [N]\lo [m]$ in $\Del$. 
\par 
Let ${\cal C}$ be the category of fine log formal schemes over $S$. 
Set 
$\Gam{\cal Y}'(N)_{\bul}:=
\Gam^{\cal C}_N({\cal Y}'(N))$. 
Then we have an immersion 
$Y'_{\bul \leq N} \os{\sus}{\lo} \Gam{\cal Y}'(N)_{\bul \leq N}$. 
Set  $Y_{lm}:={\rm cosk}_0^{Y_l}(Y'_l)_m$ 
$(l,m\in{\mab N})$ 
and   
$\Gam{\cal Y}'(N)_{lm}:=
{\rm cosk}_0^S(\Gam{\cal Y}'(N)_l)_m$ 
$(l, m\in {\mab N})$. 
Then we have an immersion 
$Y_{\bul \leq N,\bul} \os{\sus}{\lo} 
\Gam{\cal Y}'(N)_{\bul \leq N,\bul}$. 
Let 
${\mathfrak T}(N)_{\bul \leq N,\bul}=\{({\mathfrak T}(N)_{\bul \leq N,\bul})_k\}_{k=1}^{\infty}$ 
(resp.~${\mathfrak D}(N)_{\bul \leq N,\bul}$) 
be the system of the universal enlargements 
(resp.~the log PD-envelopes) of this immersion. 
Because the ideal of definition of 
$Y_{lm} \os{\sus}{\lo} \Gam{\cal Y}'(N)_{lm}$ for 
each $0\leq l \leq N$ and  each $m\in {\mab N}$ 
is locally finitely generated and because 
$a^p=p! a^{[p]}$ for a local section $a$ of the PD-ideal of 
${\mathfrak D}(N)_{l m}$, we have a natural morphism 
\begin{equation*}
(g_{lm})_k \col 
{\mathfrak D}(N)_{l m} \lo ({\mathfrak T}(N)_{lm})_k \quad 
%(0\leq \forall l \leq N,\forall m,\exists k)
\tag{11.4.7}\label{eqn:dyt}
\end{equation*}
for some nonnegative integer $k$ by the universality of ${\mathfrak T}(N)_{lm}$
(cf.~\cite[(7.6)]{oc}, the proof of \cite[Theorem 5.3.1]{s1}). 
For any nonnegativer integer $k'\geq k$, the morphisms 
$(g_{lm})_{k'}$ and $(g_{lm})_k$ are compatible in the obvious sense.  
%Let ${\cal M}_{\bul \leq N,\bul}$ be the log structure of 
%$\Gam{\cal Y}'(N)_{\bul \leq N,\bul}$ and 
Let $f(N)_{\bul \leq N,\bul}$ be 
the structural morphism $\Gam{\cal Y}'(N)_{\bul \leq N,\bul}\lo S$. 
Let $({\cal E}_{\os{\to}{{\mathfrak T}(N)}_{\bul \leq N,\bul}},\nabla_T)$ 
(resp.~
$({\cal E}_{{\mathfrak D(N)}_{\bul \leq N,\bul}},\nabla_D)
\otimes_{\mab Z}{\mab Q}$)
be the coherent crystal of 
${\cal K}_{\os{\to}{{\mathfrak T}(N)}_{\bul \leq N,\bul}}$-modules 
(resp.~the coherent 
${\cal O}_{{\mathfrak D(N)}_{\bul \leq N,\bul}}
\otimes_{\mab Z}{\mab Q}$-modules) 
with integrable connection corresponding to 
$E^{\bul \leq N,\bul}$ 
(resp.~$\Xi_{\bul \leq N,\bul}(E^{\bul \leq N,\bul})$). 
The morphism (\ref{eqn:dyt}) induces the following morphism 
\begin{equation*}
{\cal E}_{({\mathfrak T}(N)_{lm})_k}
\otimes_{{\cal O}_{\Gam{\cal Y}'(N)_{lm}}}
\Om^{\bul}_{\Gam{\cal Y}'(N)_{lm}/S}
\lo \tag{11.4.8}\label{eqn:toyuml}
\end{equation*} 
$$(g_{lm})_{k*}
({\cal E}_{{\mathfrak D}(N)_{lm}})
\otimes_{{\cal O}_{\Gam{\cal Y}'(N)_{lm}}}
\Om^{\bul}_{\Gam{\cal Y}'(N)_{lm}/S}$$  
(See also \cite[Theorem 7.7]{oc}.) 
Let 
$\eta \col 
(({Y}_{\bul \leq N, \bul})_{\rm zar},f^{-1}_{\bul \leq N,\bul}({\cal O}_S)) 
\lo 
(({Y}_{\bul \leq N})_{\rm zar},f^{-1}_{\bul \leq N}({\cal O}_S))$ 
be the natural morphism of ringed topoi, 
where $f_{\bul \leq N,\bul} \col Y_{\bul \leq N,\bul} \lo S$ 
and $f_{\bul \leq N} \col Y_{\bul \leq N} \lo S$
are the structural morphisms. 
Set $(U(N)_{l\bul})_k:=
({\mathfrak T}(N)_{l\bul})_k
\times_{\Gam{\cal Y}'(N)_{l\bul}}Y_{l\bul}$ 
with natural morphism 
$(\bet_{l\bul})_k \col (U(N)_{l\bul})_k\lo Y_{l\bul}$. 
Identify ${{\mathfrak T}(N)_{l\bul}}_{\rm zar}$ with 
${U(N)_{l\bul}}_{\rm zar}$. 
The morphism (\ref{eqn:toyuml}) induces the following morphism 
\begin{equation*}
R\eta_*
\vpl_kR(\bet_{\bul \leq N,\bul})_{k*}
({\cal E}_{({\mathfrak T}(N)_{\bul \leq N,\bul})_k}
\otimes_{{\cal O}_{\Gam{\cal Y}'(N)_{\bul \leq N,\bul}}}
\Om^{\bul}_{\Gam{\cal Y}'(N)_{\bul \leq N,\bul}/S}
\otimes_{\mab Z}{\mab Q})
\lo 
\tag{11.4.9}\label{eqn:bbyuml}
\end{equation*}
$$R\eta_{*}
({\cal E}_{{\mathfrak D}(N)_{\bul \leq N,\bul}}
\otimes_{{\cal O}_{\Gam{\cal Y}'(N)_{\bul \leq N,\bul}}}
\Om^{\bul}_{\Gam{\cal Y}'(N)_{\bul \leq N,\bul}/S}
\otimes_{\mab Z}{\mab Q})  
$$
in $D^+(f^{-1}_{\bul \leq N}({\cal K}_S))$. 
By (\ref{eqn:uybo}), the crystalline analogue of 
(\ref{eqn:uybo}) (\cite[(6.4)]{klog1}) 
and the cohomological descent, 
the morphism (\ref{eqn:bbyuml}) 
turns out to be the following morphism 
\begin{equation*}
Ru^{\rm conv}_{Y_{\bul \leq N}/S*}(E^{\bul \leq N}) 
\lo
Ru^{\rm crys}_{Y_{\bul \leq N}/S*}
(\Xi_{\bul \leq N}(E^{\bul \leq N})).
\tag{11.4.10}\label{eqn:uyms}
\end{equation*} 
\par 
We claim that the morphism (\ref{eqn:uyms}) 
is independent of the choice of 
an affine simplicial open covering of $Y_{\bul \leq N}/S$. 
Indeed, let $Y''_{\bul \leq N}$ be another affine simplicial open covering of $Y_{\bul \leq N}/S$. 
Then there exists an affine simplicial open covering 
$Y'''_{\bul \leq N}$ of $Y_{\bul \leq N}/S$ in (\ref{cd:ceccov}). 
Therefore we may assume that there exists a morphism 
$Y'_{\bul \leq N}\lo Y''_{\bul \leq N}$ over $Y_{\bul \leq N}$. 
Then the rest of the proof of 
the independence is a routine work. 
\par 
By the proof of \cite[Theorem (3.1.1)]{s2} 
(see also the proof of \cite[Proposition 1.9]{bfi}), 
we see that the following morphism 
\begin{equation*}
Ru^{\rm conv}_{Y_m/S*}(E^m) \lo
Ru^{\rm crys}_{Y_m/S*}(\Xi_m(E^m)) \quad (m\leq N)
\tag{11.4.11}\label{eqn:umxs}
\end{equation*} 
is an isomorphism. Consequently 
the morphism (\ref{eqn:uyms}) is an isomorphism. 
\par 
Finally we prove the functoriality of the morphism 
(\ref{eqn:kymc}). 
Let 
\begin{equation*} 
\begin{CD} 
Y_{\bul \leq N} @>{g}>> Z_{\bul \leq N} \\ 
@VVV @VVV \\ 
S_1 @>>> S'_1 \\ 
@V{\bigcap}VV @VV{\bigcap}V \\ 
S @>>> S' 
\end{CD}
\end{equation*}
be a commutative diagram, where 
$S'$ is a fine log flat $p$-adic formal ${\cal V}'$-scheme, 
where ${\cal V}'$ is a complete discrete valuation ring 
of mixed characteristics and $S'_1$ 
is a closed log subscheme 
defined by the ideal sheaf $p{\cal O}_{S'}$.  
Then there exists the obvious $N$-truncated version of 
the commutative diagram (\ref{cd:cechbs}). 
For any nonnegative integer $n$, 
we also have the following commutative diagram 
\begin{equation*} 
\begin{CD} 
Y'_N @>{\subset}>> {\cal Y}'(N) \\
@VVV @VVV \\ 
Z'_N @>{\subset}>> {\cal Z}'(N), 
\end{CD} 
\end{equation*}
where the horizontal morphisms are closed immersions 
into formally log smooth log noetherian formal schemes over $S$ and $S'$. 
Let $F^{\bul \leq N}$ be an object of  
${\rm Isoc}_{\rm conv}(Z_{\bul \leq N}/S)$ with a morphism 
$g^*(F^{\bul \leq N}) \lo E^{\bul \leq N}$ in 
${\rm Isoc}_{\rm conv}(Y_{\bul \leq N}/S)$. 
Using Tsuzuki's functor, 
we have the following commutative diagram 
\begin{equation*}
\begin{CD} 
Rg_*Ru^{\rm conv}_{Y_{\bul \leq N}/S*}(E^{\bul \leq N}) @>{\sim}>> 
Rg_*
Ru^{\rm crys}_{Y_{\bul \leq N}/S*}(\Xi_{\bul \leq N}(E^{\bul \leq N})) \\ 
@AAA @AAA \\ 
Ru^{\rm conv}_{Z_{\bul \leq N}/S'*}(F^{\bul \leq N}) @>{\sim}>> 
Ru^{\rm crys}_{Z_{\bul \leq N}/S'*}(\Xi_{\bul \leq N}(F^{\bul \leq N})).   
\end{CD} 
\tag{11.4.12}\label{eqn:funuyz}
\end{equation*} 
%which can be shown to be independent of the choice of $n$ 
%as above. 
We complete the proof of (\ref{theo:hct}). 
\end{proof}

\begin{rema}\label{rema:nth} 
(\ref{theo:hct}) is an $N$-truncated simplicial 
and relative sheafied version of Shiho's comparison theorem 
(\cite[Theorem 3.1.1]{s2}).  
See also \cite[Theorem 2.3.6]{s3} 
for the relative version of [loc.~cit.].  
%and proves the former part of the conjecture 
%in \cite[(16.5)]{nh3}. 
%(In \cite{s3} the second named author has also proved 
%the latter part of the conjecture in \cite[(16.5)]{nh3}.) 
\end{rema}

%\begin{theo}\label{theo:comcc}
%There exists a functorial isomorphism 
%\begin{equation*}
%Rf_{(Y_{\bul},M_{\bul})/S_{\bul}*}
%({\cal K}_{(Y_{\bul},M_{\bul})/S})
%\os{\sim}{\lo}
%Rf_{(Y_{\bul},M_{\bul})/S_{\bul}*}
%({\cal O}^{\rm crys}_{(Y_{\bul},M_{\bul})/S})_K.
%\tag{3.9.4}\label{eqn:kymc}  
%\end{equation*}
%\end{theo}

The following is a generalization of (\cite[(7.9)]{oc}) and 
(\cite[Corollary 3.1.1]{s2}).  
\begin{coro}\label{coro:pf} 
$($We do not assume that $\pi{\cal V}=p{\cal V}$ in this corollary.$)$ 
Let $r$ be a positive integer. 
Set $\ul{\bul}_r:=(\bul \cdots \bul)$ $(r$-points$)$.  
Let $f_{\ul{\bul}_r}\col Y_{\ul{\bul}_r} \lo S_1$ 
be a log smooth $r$-simplicial fine scheme over $S_1$.   
Assume that 
$\os{\circ}{f}_{\ul{\bul}_r} \col \os{\circ}{Y}_{\ul{\bul}_r} \lo 
\os{\circ}{S}_1$ 
is proper. 
Let $\ul{N}_r:=(N_1, \ldots, N_r)$ be an element of 
${\mab N}^r$. Endow ${\mab N}^r$ with the following order 
induced by the order of ${\mab N}:$ 
$(m_1,\ldots,m_r) \leq (n_1,\ldots,n_r) 
\us{\rm def.}{\Longleftrightarrow} 
m_i \leq n_i$ $(1\leq \forall i \leq r)$.  
Let $E^{\ul{\bul}_r}$ be an object of 
${\rm Isoc}_{\rm conv}(Y_{\ul{\bul}_r}/S)$. 
Set 
$f^{\rm conv}_{Y_{\ul{\bul}_r \leq \ul{N}_r}/S}:=
f_{\ul{\bul}_r \leq \ul{N}_r}\circ 
u^{\rm conv}_{Y_{\ul{\bul}_r \leq \ul{N}_r}/S}$. 
Then 
$Rf^{\rm conv}_{Y_{\ul{\bul}_r \leq \ul{N}_r}/S*}
(E^{\ul{\bul}_r\leq \ul{N}_r})$ 
is quasi-isomorphic to a strictly perfect complex of 
${\cal K}_S$-modules. 
\end{coro}
\begin{proof}
As in \cite[(7.9)]{oc}, we may assume that ${\cal V}={\cal W}$ 
and that $\pi=p$.   
By the comparison theorem (\ref{theo:hct}) 
for the constant simplicial case and 
by the finite tor-dimension and the finite generation
of log crystalline cohomologies, 
$Rf^{\rm conv}_{(Y_{i_1\cdots i_r},
M_{i_1\cdots i_r})/S*}(E^{i_1\cdots i_r})$ 
has finite tor-dimension and coherent cohomology. 
Set 
$$E_1^{ij}:=
R^jf^{\rm conv}_{Y_{\ul{\bul}_{r-1} \leq \ul{N}_{r-1},i}/S*}
(E^{\ul{\bul}_r\leq \ul{N}_{r-1},i})$$ 
for $0\leq i\leq N_r$ and $E_1^{ij}=0$ for $i> N_r$. 
Then we have the following spectral sequence 
$$E_1^{ij} \Lo R^{i+j}
f^{\rm conv}_{Y_{\ul{\bul} \leq \ul{N}_r}/S*}
(E^{\ul{\bul}_r\leq \ul{N}_r}).$$
Truncating the cosimplicial degrees similarly, 
we see that 
$$Rf^{\rm conv}_{Y_{\ul{\bul}_{r-1} \leq \ul{N}_{r-1},i}/S*}
(E^{\ul{\bul}_{r-1} \leq \ul{N}_{r-1},i})$$  
has finite tor-dimension and finitely generated cohomology. 
Hence (\ref{coro:pf}) follows from (\ref{rema:nth}) and \cite[(7.15)]{bob}.  
\end{proof}

\par 
%In the following we assume that $p{\cal V}\subset \pi{\cal V}$. 
Next we give the comparison theorem stated 
in the beginning of this section. 
When we consider the log crystalline setting, we consider the canonical PD-structure 
on $p{\cal O}_S$. 
\par 
Let $N$ be a nonnegative integer. 
Let $(X_{\bul \leq N},D_{\bul \leq N}\cup Z_{\bul \leq N})$ 
be a smooth  $N$-truncated simplicial scheme over $S~{\rm mod}~p$ 
with transversal $N$-truncated simplicial relative SNCD's 
$D_{\bul \leq N}$ and $Z_{\bul \leq N}$ on 
$X_{\bul \leq N}/S~{\rm mod}~p$. 
Let 
\begin{align*} 
\eps^{\rm crys}_{(X_{\bul \leq N},
D_{\bul \leq N}\cup Z_{\bul \leq N},Z_{\bul \leq N})/S} 
\col & 
(({(X_{\bul \leq N},
D_{\bul \leq N}\cup Z_{\bul \leq N})/S})_{\rm crys},
{\cal O}_{(X_{\bul \leq N},
D_{\bul \leq N}\cup Z_{\bul \leq N})/S}) \\
{} & \lo 
(({(X_{\bul \leq N},Z_{\bul \leq N})/S})_{\rm crys},
{\cal O}_{(X_{\bul \leq N},Z_{\bul \leq N})/S}) 
\end{align*} 
be the morphism of ringed topoi forgetting the log structure 
along $D_{\bul \leq N}$. 
Let $f_{\bul \leq N}\col X_{\bul \leq N}\lo S$ be the structural morphism. 
%(note that the log crystalline topos 
%$({(X_{\bul},D_{\bul}\cup Z_{\bul})/S})_{\rm crys}$ 
%has enough points as in the classical case 
%(cf.~\cite[III Proposition 2.1.10]{bb}).  
Then, set 
\begin{align*} 
& (E_{\rm crys}({\cal O}_{(X_{\bul \leq N},
D_{\bul \leq N}\cup Z_{\bul \leq N})/S}), P^{D_{\bul \leq N}}):= 
\tag{11.6.1}\label{eqn:essdf}\\ 
&(R\eps^{\rm crys}_{(X_{\bul \leq N},
D_{\bul \leq N}\cup Z_{\bul \leq N},Z_{\bul \leq N})/S*}
({\cal O}_{(X_{\bul \leq N},
D_{\bul \leq N}\cup Z_{\bul \leq N})/S}),
\tau^{\bul \leq N})
\end{align*}
in 
${\rm D}^+{\rm F}
({\cal O}_{(X_{\bul \leq N},Z_{\bul \leq N})/S})$. 
%(We omit to write ``log'' in the previous notation 
%$(E^{\log,Z_{\bul \leq N}}_{\rm crys}
%({\cal O}_{(X_{\bul \leq N},
%D_{\bul \leq N}\cup Z_{\bul \leq N})/S}),P^{D_{\bul \leq N}})$ 
%in \cite{nh2} and \cite{nh3}.)
Let $I^{\bul  \leq N,\bul}$ be a flasque resolution of 
${\cal O}_{(X_{\bul \leq N},
D_{\bul \leq N}\cup Z_{\bul \leq N})/S}$. 
Here the left degree of $I^{\bul \leq N,\bul}$ 
is the $N$-truncated cosimplicial degree 
corresponding to the $N$-truncated simplicial degree of $X_{\bul \leq N}$ 
and the second degree is the complex degree. 
The filtered complex 
$(E_{\rm crys}({\cal O}_{(X_{\bul \leq N},
D_{\bul \leq N}\cup Z_{\bul \leq N})/S}),
P^{D_{\bul \leq N}})$ 
is represented by 
$$(\eps^{\rm crys}_{(X_{\bul \leq N},
D_{\bul \leq N}\cup Z_{\bul \leq N},Z_{\bul \leq N})/S*}
(I^{\bul  \leq N,\bul}),\tau^{\bul \leq N}).$$ 
Let 
\begin{equation*} 
u^{\rm crys}_{(X_{\bul \leq N},Z_{\bul \leq N})/S} 
\col 
(({(X_{\bul \leq N},Z_{\bul \leq N})/S})_{\rm crys},
{\cal O}_{(X_{\bul \leq N},Z_{\bul \leq N})/S}) 
\lo 
({X_{\bul \leq N}}_{\rm zar},f^{-1}_{\bul \leq N}({\cal O}_S)) 
\end{equation*} 
be the natural morphism of ringed topoi. 
Set  
\begin{align*} 
& (E_{\rm zar}
({\cal O}_{(X_{\bul \leq N},
D_{\bul \leq N}\cup Z_{\bul \leq N})/S}),
P^{D_{\bul \leq N}}):= \tag{11.6.2}\label{eqn:cssdf}\\
& Ru^{{\rm crys}}_{(X_{\bul \leq N},Z_{\bul \leq N})/S*}
(E_{\rm crys}({\cal O}_{(X_{\bul \leq N},D_{\bul \leq N}
\cup Z_{\bul \leq N})/S}),P^{D_{\bul \leq N}})
\end{align*} 
in
${\rm D}^+{\rm F}(f^{-1}_{\bul \leq N}({\cal O}_S))$. 
%Let $(J^{\bul  \leq N,\bul},\{J^{\bul  \leq N,\bul}_k\}_{k\in {\mab Z}})$ 
%be a filtered flasque resolution of 
%$(\eps^{{\rm crys}}_{(X_{\bul \leq N},
%D_{\bul \leq N}\cup Z_{\bul \leq N},Z_{\bul \leq N})/S*}
%(I^{\bul  \leq N,\bul}),\tau^{\bul \leq N})$. 
%The filtered complex 
%$$(E_{\rm zar}({\cal O}_{(X_{\bul \leq N},
%D_{\bul \leq N}\cup Z_{\bul \leq N})/S}),P^{D_{\bul \leq N}})$$ 
%is represented by 
%$u^{{\rm crys}}_{(X_{\bul \leq N},Z_{\bul \leq N})/S*}
%((J^{\bul \leq N,\bul},\{J^{\bul \leq N,\bul}_k\}_{k\in {\mab Z}}))$. 

\begin{defi}
We call 
$(E_{\rm crys}({\cal O}_{(X_{\bul \leq N},D_{\bul \leq N}\cup Z_{\bul \leq N})/S}),
P^{D_{\bul \leq N}})$
$$({\rm resp}.~(E_{\rm zar}
({\cal O}_{(X_{\bul \leq N},
D_{\bul \leq N}\cup Z_{\bul \leq N})/S}),P^{D_{\bul \leq N}}))$$ 
the $N${\it -truncated cosimplicial weight-filtered 
vanishing cycle crystalline complex} 
(resp.~$N${\it -truncated cosimplicial weight-filtered 
vanishing cycle zariskian complex})
of $(X_{\bul \leq N},D_{\bul \leq N}\cup Z_{\bul \leq N})/S$ 
{\it with respect to} $D_{\bul \leq N}$. 
\end{defi}

\par 
For the time being we consider the constant case: 
let $X$ be a smooth scheme 
over $S~{\rm mod}~p$ and let $D$ and $Z$ 
be transversal relative SNCD's on $X/S~{\rm mod}~p$. 
In \cite[(2.5.7), (2.7.5)]{nh2}  
we have essentially proved the following 
which is a special case of 
(\ref{theo:excp}) below:

\begin{theo}[{\bf \cite[(2.5.7), (2.7.5)]{nh2}}]\label{theo:cfi} 
Let $(X_{\bul},D_{\bul}\cup Z_{\bul})_{\bul \in {\mab N}}$ 
be the \v{C}ech diagram of an affine open covering of 
$(X,D\cup Z)$ with structural morphism 
$f_{\bul} \col (X_{\bul},D_{\bul}\cup Z_{\bul})\lo S$.    
Let 
\begin{equation*} 
(X_{\bul},D_{\bul}) \os{\sus}{\lo} `{\cal R}_{\bul}
\quad {\rm and} \quad 
(X_{\bul},Z_{\bul}) \os{\sus}{\lo} {\cal R}_{\bul} 
\tag{11.8.1}\label{cd:cdcpq}
\end{equation*}  
be immersions into  
formally log smooth simplicial 
log noetherian $p$-adic formal schemes over $S$ 
such that 
$\os{\circ}{`{\cal R}}_{\bul}=\os{\circ}{\cal R}_{\bul}$. 
Assume that $\os{\circ}{\cal R}_{\bul}$ is 
topologically of finite type over $S$.    
Let ${\cal P}_{\bul}$ and 
${\cal Q}_{\bul}={\cal Q}^{\rm ex}_{\bul}$ 
be the simplicial log schemes in the beginning of 
{\rm \S\ref{sec:wfcipp}}. 
Let ${\mathfrak D}_{\bul}$ 
be the log PD-envelope of the immersion 
$(X_{\bul},D_{\bul}\cup Z_{\bul})  
\os{\sus}{\lo} {\cal P}_{\bul}$ over $(S,p{\cal O}_S,[~])$. 
Let ${\cal P}^{\rm ex}_{\bul}$ 
%and ${\cal Q}^{\rm ex}_{\bul}$ 
be the exactifcation of the immersion 
$(X_{\bul},D_{\bul}\cup Z_{\bul}) 
\os{\sus}{\lo} {\cal P}_{\bul}$.  
%and $(X_{\bul},Z_{\bul}) \os{\sus}{\lo} {\cal Q}_{\bul}$, 
%respectively. 
Let 
\begin{equation*} 
\pi_{\rm zar} \col 
(({X}_{\bul})_{\rm zar},f^{-1}_{\bul}({\cal O}_S))
\lo 
({X}_{\rm zar},f^{-1}({\cal O}_S))
\end{equation*} 
be a natural morphism of ringed topoi. 
Then there exists the following functorial 
isomorphism 
\begin{align*} 
(E_{\rm zar}
({\cal O}_{(X,D\cup Z)/S}),P^D) \os{\sim}{\lo} 
R\pi_{{\rm zar}*}
({\cal O}_{{\mathfrak D}_{\bul}}
{\otimes}_{{\cal O}_{{\cal P}^{\rm ex}_{\bul}}}
\Om^{\bul}_{{\cal P}^{\rm ex}_{\bul}/S},
P^{{\cal P}^{\rm ex}_{\bul}/{\cal Q}^{\rm ex}_{\bul}})
\tag{11.8.2}\label{eqn:clzci}
\end{align*}
in ${\rm D}^+{\rm F}(f^{-1}({\cal O}_S))$. 
$($Here the right hand side 
is, by definition, equal to 
$(C_{\rm zar}({\cal O}_{(X,D\cup Z)/S}),P^D)$ 
{\rm (cf.~\cite[(2.5.7), (2.5.8)]{nh2})}.$)$ 
Here the functoriality means the functoriality for 
the pair of the following three commutative diagrams  
\begin{equation*} 
\begin{CD} 
(X,D\cup Z) @>>> (X',D'\cup Z') \\
@VVV @VVV \\
S_1 @>>> S'_1, 
\end{CD}
\end{equation*}
where $S$ and $S'$ are in {\rm (\ref{cd:bpsmlgs})} 
and $(X',D'\cup Z')$ is in {\rm (\ref{cd:mpsmlgs})}, 
\begin{equation*} 
\begin{CD} 
(X_{\bul},D_{\bul})@>{\sus}>> `{\cal R}_{\bul} \\
@VVV @VVV \\
(X'_{\bul},D'_{\bul}) @>{\sus}>> `{\cal R}'_{\bul} 
\end{CD}
\end{equation*}
and 
\begin{equation*} 
\begin{CD} 
(X_{\bul},Z_{\bul})@>{\sus}>> {\cal R}_{\bul} \\
@VVV @VVV \\
(X'_{\bul},Z'_{\bul}) @>{\sus}>> {\cal R}'_{\bul},  
\end{CD}
\end{equation*}
over $u\col S\lo S'$ in {\rm (\ref{cd:bpsmlgs})}, 
where the four horizontal morphisms above are immersions  
into formally log smooth simplicial 
log $p$-adic formal schemes over $S$ 
such that 
$`\os{\circ}{\cal R}_{\bul}=`\os{\circ}{\cal R}{}'_{\bul}$ 
and 
$\os{\circ}{\cal R}_{\bul}=\os{\circ}{\cal R}{} '_{\bul}$ . 
\end{theo}

The following is 
an $N$-truncated simplicial version of (\ref{theo:cfi}): 

\begin{theo}[{\bf cf.~\cite[(6.17)]{nh3}, 
Explicit description of cosimplicial 
weight-filtered vanishing cycle zariskian complex}]
\label{theo:excp}
Let $N$ be a nonnegative integer. 
Assume that 
$(X_{\bul \leq N},D_{\bul \leq N}\cup Z_{\bul \leq N})/S~{\rm mod}~p$ 
has the disjoint union of the member of 
an affine $N$-truncated 
simplicial open covering of 
$(X_{\bul \leq N},D_{\bul \leq N}\cup Z_{\bul \leq N})/S~{\rm mod}~p$. 
Let 
$(X_{\bul \leq N,\bul},
D_{\bul \leq N,\bul}\cup Z_{\bul \leq N\bul})$ 
be the \v{C}ech diagram of 
this affine $N$-truncated simplicial open covering.  
Let 
\begin{equation*} 
(X_{\bul \leq N,\bul},D_{\bul \leq N,\bul}) 
\os{\sus}{\lo} `{\cal R}_{\bul \leq N,\bul}
\quad {\rm and} \quad 
(X_{\bul \leq N,\bul},Z_{\bul \leq N,\bul}) 
\os{\sus}{\lo} {\cal R}_{\bul \leq N,\bul} 
\tag{11.9.1}\label{cd:xdzpd} 
\end{equation*}  
be an $(N,\infty)$-truncated version of 
{\rm (\ref{cd:cdcpq})}. 
Let 
${\cal P}_{\bul \leq N,\bul}$, 
${\cal Q}_{\bul \leq N,\bul}
={\cal Q}^{\rm ex}_{\bul \leq N,\bul}$,   
${\mathfrak D}_{\bul \leq N,\bul}$ and 
${\cal P}^{\rm ex}_{\bul \leq N,\bul}$ 
be the $(N,\infty)$-truncated versions of 
${\cal P}_{\bul}$, 
${\cal Q}_{\bul}={\cal Q}^{\rm ex}_{\bul}$,   
${\mathfrak D}_{\bul}$ and     
${\cal P}^{\rm ex}_{\bul}$ 
in {\rm (\ref{theo:cfi})}. 
Let 
\begin{equation*} 
\eta_{\rm zar} \col 
(({X}_{\bul \leq N, \bul})_{\rm zar},
f^{-1}_{\bul \leq N, \bul}({\cal O}_S))
\lo (({X}_{\bul \leq N})_{\rm zar},
f^{-1}_{\bul \leq N}({\cal O}_S))
\end{equation*} 
be a natural morphism of ringed topoi. 
Then 
\begin{equation*} 
(E_{\rm zar}({\cal O}_{(X_{\bul \leq N},D_{\bul \leq N} \cup 
Z_{\bul \leq N})/S}),P^{D_{\bul \leq N}})= 
\tag{11.9.2}\label{eqn:czzpd}
\end{equation*} 
$$R\eta_{{\rm zar}*}
({\cal O}_{{\mathfrak D}_{\bul \leq N,\bul}}
{\otimes}_{{\cal O}_{{\cal P}^{\rm ex}_{\bul \leq N,\bul}}}
\Om^{\bul}_{{\cal P}^{\rm ex}_{\bul \leq N,\bul}/S},
P^{{\cal P}^{\rm ex}_{\bul \leq N,\bul}/
{\cal Q}^{\rm ex}_{\bul \leq N,\bul}})$$
in ${\rm D}^+{\rm F}(f^{-1}_{\bul \leq N}({\cal O}_S))$. 
\end{theo}
\begin{proof}
The same proof as that of \cite[(6.17)]{nh3} works. 
\end{proof}

%\begin{coro}\label{coro:excop}
%Assume that $(X_{\bul},D_{\bul}\cup Z_{\bul})/S$ 
%is split.  Let $(X_{\bul \bul},D_{\bul \bul}\cup Z_{\bul \bul})$ 
%be the {\it affine \v{C}ech diagram} of 
%an affine simplicial open covering 
%of $(X_{\bul},D_{\bul}\cup Z_{\bul})/S$. 
%Let 
%$$(X_{\bul \bul},D_{\bul \bul}\cup Z_{\bul \bul}) 
%\os{\sus}{\lo}  
%({\cal X}_{\bul \bul},{\cal D}_{\bul \bul}\cup 
%{\cal Z}_{\bul \bul})$$ 
%be an admissible immersion  
%over $S$ with the induced decompositions of 
%$D_{\bul \bul}$ and $Z_{\bul \bul}$ 
%by their smooth components 
%by those  of $D_{\bul}$ and $Z_{\bul}$ 
%by their smooth components. 
%Let 
%\begin{equation*} 
%\eta_{\rm zar} \col 
%((({X}_{\bul \bul})_{\rm zar},f^{-1}_{\bul \bul}({\cal O}_S))
%\lo 
%((({X}_{\bul})_{\rm zar},f^{-1}_{\bul}({\cal O}_S))
%\end{equation*} 
%be a natural morphism of ringed topoi. 
%Then 
%\begin{equation*} 
%(C_{\rm zar}
%({\cal O}_{(X_{\bul},D_{\bul} \cup Z_{\bul})/S}),P^{D_{\bul}})= 
%\tag{11.11.1}\label{eqn:coxdnp}
%\end{equation*} 
%$$R\eta_{{\rm zar}*}
%({\cal O}_{{\mathfrak D}_{\bul \bul}}
%{\otimes}_{{\cal O}_{{\cal X}_{\bul \bul}}}
%\Om^{\bul}_{{\cal X}_{\bul \bul}/S}
%(\log ({\cal D}_{\bul \bul}\cup {\cal Z}_{\bul \bul})),
%P^{{\cal D}_{\bul \bul}})$$
%in ${\rm D}^+{\rm F}(f^{-1}_{\bul}({\cal O}_S))$. 
%\end{coro}

\par 
Let 
\begin{align*} 
\eps^{\rm conv}_{(X_{\bul \leq N},D_{\bul \leq N}\cup Z_{\bul \leq N},Z_{\bul \leq N})/S} 
\col & (({(X_{\bul \leq N},D_{\bul \leq N}\cup Z_{\bul \leq N})/S})_{\rm conv},
{\cal K}_{(X_{\bul \leq N},D_{\bul \leq N}\cup Z_{\bul \leq N})/S}) 
 \\
{} & \lo (({(X_{\bul \leq N},Z_{\bul \leq N})/S})_{\rm conv},
{\cal K}_{(X_{\bul \leq N},Z_{\bul \leq N})/S}) 
\end{align*} 
be the morphism of topoi forgetting the log structure 
along $D_{\bul \leq N}$. 
Then, set 
\begin{align*} 
& (E_{\rm conv}
({\cal K}_{(X_{\bul \leq N},D_{\bul \leq N}\cup Z_{\bul \leq N})/S}), P^{D_{\bul \leq N}})
:= \\
& (R\eps^{\rm conv}_{
(X_{\bul \leq N},D_{\bul \leq N}\cup Z_{\bul \leq N},Z_{\bul \leq N})/S*} 
({\cal K}_{(X_{\bul \leq N},D_{\bul \leq N}\cup Z_{\bul \leq N})/S}),\tau^{\bul \leq N})
\end{align*} 
in 
${\rm D}^+{\rm F}({\cal K}_{(X_{\bul \leq N},Z_{\bul \leq N})/S})$,  
and 
\begin{align*} 
&(E_{\rm isozar}({\cal K}_{(X_{\bul \leq N},D_{\bul \leq N} \cup Z_{\bul \leq N})/S}),P^{D_{\bul \leq N}})
:= \\ 
& 
Ru^{\rm conv}_{(X_{\bul \leq N},Z_{\bul \leq N})/S*}
(E_{\rm conv}
({\cal K}_{(X_{\bul \leq N},D_{\bul \leq N}\cup Z_{\bul \leq N})/S}), P^{D_{\bul \leq N}})
\end{align*} 
in ${\rm D}^+{\rm F}(f^{-1}_{\bul \leq N}({\cal K}_S))$.

\begin{lemm}\label{lemm:kevp}
Let ${\cal V}'$ and $\pi'$ be as in {\rm (\ref{prop:toi})}.  
Assume that ${\cal V}'={\cal V}$. Set $S'_1:=S~{\rm mod}~\pi'$.   
Set $(X',D'\cup Z'):=(X,D\cup Z)\times_{S_1}S'_1$. 
Then 
\begin{align*} 
(E_{\rm conv}
({\cal K}_{(X_{\bul \leq N},D_{\bul \leq N}\cup Z_{\bul \leq N})/S}), P^{D_{\bul \leq N}})
=
(E_{\rm conv}
({\cal K}_{(X'_{\bul \leq N},D'_{\bul \leq N}\cup Z'_{\bul \leq N})/S}), P^{D'_{\bul \leq N}}). 
\tag{11.10.1}\label{eqn:covzpd}
\end{align*} 
Consequently 
\begin{align*} 
(E_{\rm isozar}
({\cal K}_{(X_{\bul \leq N},D_{\bul \leq N}\cup Z_{\bul \leq N})/S}), P^{D_{\bul \leq N}})
=
(E_{\rm isozar}({\cal K}_{(X'_{\bul \leq N},D'_{\bul \leq N}\cup Z'_{\bul \leq N})/S}), P^{D'_{\bul \leq N}}). 
\tag{11.10.2}\label{eqn:coazzpd}
\end{align*} 
\end{lemm}
\begin{proof} 
The equality (\ref{eqn:covzpd}) immediately follows from (\ref{eqn:iuoe}). 
\end{proof} 

\par 
By (\ref{theo:wtvsca}), (\ref{prop:cdfza}) 
and the same proof as that of  
(\ref{theo:excp})(=the proof of \cite[(6.17)]{nh3}), 
we obtain the following:

\begin{theo}[{\bf Explicit description of 
cosimplicial weight-filtered isozariskian complex}]\label{theo:excop}
Let the notations be as in {\rm (\ref{theo:excp})}. 
Let 
$(\{(U_{\bul \leq N,\bul})_k\}_{k\in {\mab Z}_{>0}},
\{(T_{\bul \leq N,\bul})_k\}_{k\in {\mab Z}_{>0}})$ 
be the system of the universal enlargements of 
the immersion 
$(X_{\bul \leq N,\bul},
D_{\bul \leq N,\bul}\cup Z_{\bul \leq N,\bul})
\os{\sus}{\lo}
{\cal P}_{\bul \leq N,\bul}$. 
Let $\bet_k \col (U_{\bul \leq N,\bul})_k 
\lo X_{\bul \leq N, \bul}$ 
$(k\in {\mab Z}_{>0})$ 
be the natural morphism. 
Identify $((T_{\bul \leq N, \bul})_k)_{{\rm zar}}$ with 
$((U_{\bul \leq N,\bul})_k)_{{\rm zar}}$. 
Let $\eta_{{\rm zar}*} \col 
((U_{\bul \leq N, \bul})_k)_{{\rm zar}} 
\lo 
(U_{\bul \leq N})_{{\rm zar}}$ 
be the natural morphism of topoi. 
Then 
\begin{equation*} 
(E_{\rm isozar}
({\cal K}_{(X_{\bul \leq N},D_{\bul \leq N} \cup Z_{\bul \leq N})/S}), 
P^{D_{\bul \leq N}}) =
\tag{11.11.1}\label{eqn:cozzpd}
\end{equation*} 
$$R\eta_{{\rm zar}*}\vpl_kR\bet_{k*}
({\cal K}_{T_{\bul \leq N,\bul k}}
{\otimes}_{{\cal O}_{{\cal P}^{\rm ex}_{\bul \leq N,\bul}}}
\Om^{\bul}_{{\cal P}^{\rm ex}_{\bul \leq N,\bul}/S},
P^{{\cal P}^{\rm ex}_{\bul \leq N,\bul}
/{\cal Q}^{\rm ex}_{\bul \leq N,\bul}})$$
in ${\rm D}^+{\rm F}(f^{-1}_{\bul \leq N}({\cal K}_S))$. 
\end{theo}

\begin{coro}[{\bf Comparison theorem of weight-filtered 
vanishing cycle isozariskian and zariskian complexes}]\label{coro:ccco}
%Assume that $\pi{\cal V}\subset p{\cal V}$. 
There exists a canonical isomorphism  
\begin{equation*} 
(E_{\rm isozar}
({\cal K}_{(X_{\bul \leq N},
D_{\bul \leq N} \cup Z_{\bul \leq N})/S}), 
P^{D_{\bul \leq N}}) \os{\sim}{\lo}
(E_{\rm zar}
({\cal O}_{(X_{\bul \leq N},D_{\bul \leq N} \cup 
Z_{\bul \leq N})/S}), 
P^{D_{\bul \leq N}})\otimes_{\mab Z}{\mab Q}
\tag{11.12.1}\label{eqn:cccopd}
\end{equation*} 
in ${\rm D}^+{\rm F}(f^{-1}_{\bul \leq N}({\cal K}_S))$. 
Here the left complex is defined for 
a general nonzero element $\pi$ of the maximal ideal of ${\cal V}$ 
and the right complex is defined for the prime number $p$ of the maximal ideal of ${\cal V}$. 
\end{coro}
\begin{proof}
By (\ref{lemm:kevp}) we can replace $\pi$ by $p$. 
%Hence we may assume that $p{\cal V}=\pi{\cal V}$.  
By the explicit descriptions (\ref{eqn:czzpd}) and 
(\ref{eqn:cozzpd}), 
the morphism (\ref{eqn:dyt}) induces the morphism 
(\ref{eqn:cccopd}). 
Let $m \leq N$ be a nonnegative integer.  
For a nonnegative integer $k$, we have 
\begin{align*}
& {\rm gr}^{P^{D_m}}_kE_{\rm isozar}
({\cal K}_{(X_m,D_m \cup Z_m)/S})
= \tag{11.12.2}\label{ali:cgrmcopd}\\ 
& a^{(k)}_*Ru^{\rm conv}_{(D^{(k)}_m,
Z_m\vert_{D^{(k)}_m})/S*}
({\cal K}_{(D^{(k)}_m,
Z_m\vert_{D^{(k)}_m})/S}\otimes_{\mab Z}
\vp^{(k)}_{{\rm conv}}(D/S;Z))
\end{align*}
and 
\begin{align*}
& {\rm gr}^{P^{D_m}}_k
E_{\rm zar}
({\cal O}_{(X_m,D_m \cup Z_m)/S})
=\tag{11.12.3}\label{ali:cgrmcnvpd} \\
& a^{(k)}_*Ru^{\rm crys}_{(D^{(k)}_m,
Z_m\vert_{D^{(k)}_m})/S}
({\cal O}_{(D^{(k)}_m,
Z_m\vert_{D^{(k)}_m})/S}\otimes_{\mab Z}
\vp^{(k)}_{{\rm crys}}(D/S;Z))
\end{align*}
by (\ref{cd:grfcty}), (\ref{caseali:grtkr}) and \cite[(2.6.1.2), (2.7.3.2)]{nh2}. 
By the relative version of Shiho's comparison isomorphism 
\cite[Corollary 2.3.9]{s2} 
(or (\ref{eqn:kymc})), the right hand sides of (\ref{ali:cgrmcopd}) and 
(\ref{ali:cgrmcnvpd}) are isomorphic. 
Hence the morphism (\ref{eqn:cccopd}) 
is a filtered quasi-isomorphism. 
By using (\ref{prop:dissmp}) and the proof of (\ref{theo:hct}), 
we can prove, in a standard way, that 
the isomorphism (\ref{eqn:cccopd}) is independent 
of the choice of the affine simplicial open covering of 
$(X_{\bul \leq N},D_{\bul \leq N}\cup Z_{\bul \leq N})$ 
and the immersions in (\ref{cd:xdzpd}).
\end{proof}

\par 
Let $(X_{\bul},D_{\bul}\cup Z_{\bul})$ be 
a smooth simplicial scheme 
with transversal simplicial relative SNCD's 
over $S_1$. 
\par 
Set $f^{\rm conv}_{(X_{\bul},D_{\bul}\cup Z_{\bul})/S}:=
f \circ u^{\rm conv}_{(X_{\bul},D_{\bul}\cup Z_{\bul})/S}$.  
To construct a weight spectral sequence of 
$R^hf^{\rm conv}_{(X_{\bul},
D_{\bul}\cup Z_{\bul})/S*}({\cal K}_{(X_{\bul}, 
D_{\bul}\cup Z_{\bul})/S})$ $(h\in {\mab N})$,  
let us recall the cosimplicial diagonal filtration 
in \cite{dh3} and \cite{nh3}. 
\par 
Let $({\cal T}, {\cal A})$ be a ringed topos. 
Let $({\cal T}_t, {\cal A}^t)_{t\in {\mab N}}$ 
be a constant simplicial ringed topos 
defined by $({\cal T}, {\cal A})$: 
${\cal T}_t={\cal T}$, ${\cal A}^t={\cal A}$.
Let $M$ be an object of 
${\rm C}({\cal A}^{\bul})$. 
The object $M$ defines a double complex 
$M^{\bul \bul}
=(M^{t \bul})_{t\in {\mab N}}$ of ${\cal A}$-modules 
and the single complex 
${\bf s}(M)$ (${\bf s}(M)^n=\bigoplus_{t+s=n}M^{ts})$ 
with the following boundary morphism:  
\begin{align*}
d(x^{ts}) =&  
\sum_{i= 0}^{t+1}(-1)^{i}
\del^i(x^{ts})+(-1)^td_M(x^{ts}) \; (x^{ts}\in M^{ts}),
\tag{11.12.4}\label{eqn:bdsignss}
\end{align*} 
where 
$d_M \col M^{ts} 
\lo  M^{t,s+1}$ 
is the boundary morphism arising 
from the boundary 
morphism of the complex $M$ and 
$\del^i \col M^{ts} 
\lo M^{ts}$ 
$(0 \leq i \leq t+1)$ is a 
standard coface morphism.
Let $L$ be the stupid
filtration on ${\bf s}(M)$ 
with respect to the first index: 
\begin{equation*}
L^t({\bf s}(M))=\bigoplus_{t'\geq t}M^{t'\bul}. 
\tag{11.12.5}
\end{equation*}
\par
Next we consider the filtered version of the above. 
\par 
Let $(M, P):=(M,\{P_kM\}_{k\in {\mab Z}})$ 
be a complex of increasingly filtered 
${\cal A}^{\bul}$-modules.
Then $(M,P)$ defines a filtered double complex 
$(\bigoplus_{t\geq 0,s}M^{ts},
\{\bigoplus_{t\geq 0,s}P_kM^{ts}\}_{k\in {\mab Z}})$
of ${\cal A}$-modules. 
\par
Let 
$\del(L,P)$ be the diagonal filtration of 
$L$  and $P$ on ${\bf s}(M)$
(cf.~\cite[(7.1.6.1), (8.1.22)]{dh3}): 
\begin{align*}
\del(L,P)_k{\bf s}(M) & =  
\bigoplus_{t \geq 0,s}
P_{t+k}M^{t s} \tag{11.12.6}\label{ali:ddifl} \\
{} & =\sum_{t\geq 0}
L^t({\bf s}(M))\cap {\bf s}(P_{t+k}M). 
\end{align*} 
%(In the case $r=1$, the first formula in 
%\cite[(7.1.6.1)]{dh3} have to be replaced by 
%a formula $\del(W,L)_n({\bf s}(K))=
%\bigoplus_{p,q}W_{n+p}
%(K^{q,p})$ (cf.~\cite[(5.1.9) (IV)]{dh3})).
Then we have the following equality: 
\begin{equation*}
{\rm gr}_k^{\del(L,P)}{\bf s}(M)=
\bigoplus_{t\in {\mab N}^r}
{\rm gr}^P_{t+k}
M^{t \bul}[-t].
\tag{11.12.7}\label{eqn:grdif}
\end{equation*}
%\par
%We can obtain the filtered complex 
%$({\bf s}(M), \del(\ul{L}, P))$ 
%inductively by the following formula (\cite[(8.1.22)]{dh3}): 
%\begin{align*} 
%({\bf s}, \del(\ul{L}, P))& =
%\{{\bf s}_1, \del(L_1, \del(L_2, \ldots, \del(L_r, P))\cdots)\}
%\circ \cdots \tag{11.14.7}\label{ali:indfs}\\ 
%{} & \circ 
%\{{\bf s}_{r-1}, \del(L_{r-1}, \del(L_r, P))\}
%\circ 
%\{{\bf s}_r, \del(L_r, P)\}. 
%\end{align*} 

\par 
Set $f^{\rm crys}_{(X_{\bul},D_{\bul}\cup Z_{\bul})/S}:=
f \circ u^{\rm crys}_{(X_{\bul},D_{\bul}\cup Z_{\bul})/S}$. 
Next let us construct a spectral sequence of 
$R^hf^{\rm conv}_{(X_{\bul},
D_{\bul}\cup Z_{\bul})/S*}({\cal K}_{(X_{\bul}, 
D_{\bul}\cup Z_{\bul})/S})$ $(h\in {\mab N})$ 
and compare the spectral sequence of 
$R^hf^{\rm crys}_{(X_{\bul},
D_{\bul}\cup Z_{\bul})/S*}({\cal O}_{(X_{\bul}, 
D_{\bul}\cup Z_{\bul})/S})\otimes_{\mab Z}{\mab Q}$ 
constructed in \cite[\S4]{nh3}. 
\par 
Let $S_{\bul}$ be 
the constant simplicial formal scheme defined by $S$. 
Let 
$I^{\bul \bul}_{\rm conv}$ be an injective resolution of 
${\cal K}_{(X_{\bul},D_{\bul}\cup Z_{\bul})/S}$. 
Let $(J^{\bul \bul},P)$ be a filtered injective resolution of 
$(\eps^{\rm conv}_{(X_{\bul},D_{\bul}\cup Z_{\bul},Z_{\bul})
/S_{\bul}*}(I^{\bul \bul}_{\rm conv}),\tau)$ 
(\cite[(1.1.7)]{nh2}). 
Set 
$$(K^{\bul \bul}, P):=
u^{\rm conv}_{(X_{\bul},Z_{\bul})/S_{\bul}*}
(J^{\bul \bul},P).$$ 
\par 
Let $\del(L,P)$ be the diagonal filtration of $L$ and $P$ 
on ${\bf s}(K^{\bul \bul})$. 
Then, by (\ref{eqn:ppuri}), we have 
\begin{equation*}
\begin{split}
{\cal H}^h({\rm gr}_k^{\del(L,P)}{\bf s}(K^{\bul \bul})) &  
=  \bigoplus_{t \geq 0}{\cal H}^h
({\rm gr}^P_{t+k}J^{t \bul}[-t]) \\ 
& =\bigoplus_{t\geq 0}R^{h-2t-k}
f^{\rm conv}_{(D_{t}^{(t+k)}, 
Z\vert_{D^{(t+k)}_{t}})/S*}
({\cal K}_{(D_{t}^{(t+k)}, 
Z\vert_{D^{(t+k)}_t})/S} 
\otimes_{\mab Z}\\
&\phantom{=\bigoplus_{t\geq 0}R^{h-2t-k}
f^{\rm conv}_{(D_t^{(t+k)},Z\vert_{D^{(t+k)}}}}
\vp^{(t+k)}_{\rm conv}
(D_t/S;Z_t))(-(t+k)).
\end{split} 
\tag{11.12.8}\label{ali:genwt}
\end{equation*}
Hence we have the following 
spectral sequence as in \cite[(5.1.3)]{nh3}: 
\begin{equation*}
\begin{split}
E_{1,{\rm conv}}^{-k, h+k}((X_{\bul}, 
D_{\bul}\cup Z_{\bul})/S_K)
= & \bigoplus_{t\geq 0}
R^{h-2t-k}f^{\rm conv}_{(D_{t}^{(t+k)}, 
Z\vert_{D_{t}^{(t+k)}})/S*} \\ 
&\phantom{\bigoplus_{t\geq 0}R} 
({\cal K}_{(D_{t}^{(t+k)},
Z\vert_{D_{t}^{(t+k)}})/S}
\otimes_{\mab Z}\vp^{(t+k)}_{\rm conv}
(D_{t}/S;Z_{t}))\\ 
{} & \phantom{({\cal K}_{(D_{t}^{(t+k)},
Z\vert_{D_{t}^{(t+k)}})/S}
\otimes_{\mab Z}\vp^{(t+k)}}
(-(t+k)) \\ 
{} & \Lo 
R^hf^{\rm conv}_{(X_{\bul},D_{\bul}\cup Z_{\bul})/S*}
({\cal K}_{(X_{\bul},D_{\bul}\cup Z_{\bul})/S}). 
\end{split} 
\tag{11.12.9}
\label{ali:simppws}
\end{equation*}
In \cite{nh3} we have constructed 
the following analogous spectral sequence:  
\begin{equation*}
\begin{split}
E_{1,{\rm crys}}^{-k, h+k}((X_{\bul}, 
D_{\bul}\cup Z_{\bul})/S)
= & \bigoplus_{t\geq 0}
R^{h-2t-k}f^{\rm crys}_{(D_{t}^{(t+k)}, 
Z\vert_{D_{t}^{(t+k)}})/S*} \\ 
&\phantom{\bigoplus_{t\geq 0}
R} 
({\cal O}_{(D_t^{(t+k)},
Z\vert_{D_t^{(t+k)}})/S}
\otimes_{\mab Z}\vp^{(t+k)}_{\rm crys}(D_t/S;Z_t))\\ 
{} & \phantom{({\cal O}_{(D_t^{(t+k)},Z\vert_{D_t^{(t+k)}})/S}
\otimes_{\mab Z}\vp^{(t+k)}}
(-(t+k)) \\ 
{} & \Lo 
R^hf^{\rm crys}_{(X_{\bul},D_{\bul}\cup Z_{\bul})/S*}
({\cal O}_{(X_{\bul},D_{\bul}\cup Z_{\bul})/S}). 
\end{split} 
\tag{11.12.10}\label{ali:opws}
\end{equation*}

As a corollary of (\ref{coro:ccco}), we have the following: 

\begin{coro}\label{coro:compss}
%Assume that $\pi{\cal V}\subset p{\cal V}$. 
Assume that $(X_{\bul},D_{\bul}\cup Z_{\bul})~{\rm mod}~p$ 
has the disjoint union of an affine simplicial open covering 
over $S$. 
Then the spectral sequences 
$(\ref{ali:simppws})$ and 
$(\ref{ali:opws})\otimes_{\mab Z}{\mab Q}$ 
are canonically isomorphic. 
\end{coro}
\begin{proof} 
Fix $h\in {\mab N}$. Let $N$ be a sufficiently large integer. 
Then we have the spectral sequence
of 
$$R^hf^{\rm conv}_{(X_{\bul \leq N},D_{\bul \leq N}
\cup Z_{\bul \leq N})/S*}
({\cal K}_{(X_{\bul \leq N},
D_{\bul \leq N}\cup Z_{\bul \leq N})/S})$$ 
which is analogous to 
(\ref{ali:simppws}). 
We also have the following spectral sequence of 
$$R^hf^{\rm crys}_{(X_{\bul \leq N},
D_{\bul \leq N}\cup Z_{\bul \leq N})/S*}
({\cal O}_{(X_{\bul \leq N},
D_{\bul \leq N}\cup Z_{\bul \leq N})/S})$$ 
which is analogous to (\ref{ali:opws}). 
Let 
$$E_{r,{\rm conv}}^{-k, h+k}((X_{\bul \leq N}, 
D_{\bul \leq N}\cup Z_{\bul \leq N})/S_K) 
\quad (r\geq 1)$$ 
and 
$$E_{r,{\rm crys}}^{-k, h+k}((X_{\bul \leq N},
D_{\bul \leq N}\cup Z_{\bul \leq N})/S)\quad  (r\geq 1)$$ 
be the $E_r$-terms of these spectral sequences, respectively.  
Then 
$E_{r,{\rm conv}}^{-k, h+k}((X_{\bul \leq N}, 
D_{\bul \leq N}\cup Z_{\bul \leq N})/S_K) \simeq 
E_{r,{\rm crys}}^{-k, h+k}((X_{\bul \leq N},
D_{\bul \leq N}\cup Z_{\bul \leq N})/S)$
by (\ref{eqn:cccopd}). 
Because $N$ is sufficiently large, 
$E_{r,{\rm conv}}^{-k, h+k}((X_{\bul \leq N}, 
D_{\bul \leq N}\cup Z_{\bul \leq N})/S_K) \simeq 
E_{r,{\rm conv}}^{-k, h+k}((X_{\bul}, 
D_{\bul}\cup Z_{\bul})/S_K)$ and  
$E_{r,{\rm crys}}^{-k, h+k}((X_{\bul \leq N},
D_{\bul \leq N}\cup Z_{\bul \leq N})/S) 
\simeq 
E_{r,{\rm crys}}^{-k, h+k}
((X_{\bul},D_{\bul}\cup Z_{\bul})/S)$. 
Hence 
$E_{r,{\rm conv}}^{-k, h+k}((X_{\bul}, 
D_{\bul}\cup Z_{\bul})/S_K) \simeq 
E_{r,{\rm crys}}^{-k, h+k}((X_{\bul}, 
D_{\bul}\cup Z_{\bul})/S_K)$ (cf.~\cite[(2.2)]{nh3}). 
We complete the proof. 
\end{proof}

%\begin{theo}[{\bf Comparison theorem}]\label{theo:ctcc}
%There exists the following canonical isomorphism 
%\begin{equation*} 
%Rf_*({\cal K}_{X/S})=Rf_{\rm conv}({\cal K}_{X/S}). 
%\end{equation*}  
%\end{theo}

Using (\ref{coro:compss}) we obtain the following: 

\begin{theo}[{\bf $E_2$-degeneration}]\label{theo:ssd}
The spectral sequence 
$(\ref{ali:simppws})$ degenerates at $E_2$ if 
$X_{\bul}$ is proper over $S_1$ and if 
$Z_{\bul}=\emptyset$. 
\end{theo}
\begin{proof} 
%First assume that $\pi{\cal V}=p{\cal V}$. 
In \cite[(5.6)]{nh3}, by the same proof as that of 
\cite[(2.17.2)]{nh2}, 
we have proved the $E_2$-degeneration of 
$(\ref{ali:opws})\otimes_{\mab Z}{\mab Q}$ 
when $X_{\bul}$ is proper over $S_1$ and 
when $Z_{\bul}=\emptyset$. 
Hence (\ref{theo:ssd}) follows from 
(\ref{coro:compss}). 
%\par 
%Next let us consider the case where 
%$\pi{\cal V}$ is not necessarily equal to $p{\cal V}$. 
%By the same proof as that of \cite[(2.17.2)]{nh2}, 
%we may assume that $S={\rm Spf}({\cal V})$. 
%Consider the following diagram 
%\begin{equation*} 
%\begin{CD} 
%(X_{\bul},D_{\bul})\\
%@VVV \\
% {\rm Spec}(\kap)  @>{\subset}>>  {\rm Spec}({\cal V}/p) 
% @>{\subset}>> {\rm Spf}({\cal V})  \\
% @| @VVV @VVV \\
%{\rm Spec}(\kap) @=  {\rm Spec}(\kap) 
% @>{\subset}>> {\rm Spf}({\cal W}).  
%\end{CD}
%\end{equation*} 
%Because ${\rm Spf}({\cal V})\lo {\rm Spf}({\cal W})$ is finite and flat, 
%the base change formula 
%\begin{equation*}
%K\otimes_{K_0}
%(R^hf^{\rm conv}_{(X_{\bul},D_{\bul})/{\cal W}*}
%({\cal K}_{(X_{\bul},D_{\bul})/{\cal W}}),P)
%\os{\sim}{\lo} 
%(R^hf^{\rm conv}_{(X_{\bul},D_{\bul})/{\cal V}*}
%({\cal K}_{(X_{\bul},D_{\bul})/{\cal V}}),P)
%\tag{11.14.1}\label{eqn:simbcat}
%\end{equation*}
%is an isomorphism by (\ref{eqn:cozzpd}). 
%Hence (\ref{theo:ssd}) follows from the first argument in this proof. 
\end{proof}

%Finally, we state the 
%base change theorem of the 
%weight-filtered crystalline
%complex and the K\"{u}nneth formula of it.

\begin{prop}[{\bf Base change theorem}]\label{prop:bchange} 
Let ${\cal V}'$, $K'$, $\pi'$, $S'$ and $S_1'$ 
be as in {\rm \S\ref{sec:fpw}}. 
%Assume that 
%$\pi{\cal V}\subset p{\cal V}$ and $\pi'{\cal V}'\subset p{\cal V}'$. 
Let  
\begin{equation*}
\begin{CD} 
S' @>{u}>> S \\
@VVV @VVV \\
{\rm Spf}({\cal V}')  @>>> {\rm Spf}({\cal V})
\end{CD}
\tag{11.15.1}\label{cd:bbsmlgs}
\end{equation*}
be  a commutative diagram of $p$-adic formal schemes. 
Let $Y$ $($resp.~$Y')$ 
be a quasi-compact smooth scheme over 
$S_1$ $($resp.~$S'_1)$ $($with trivial log structure$)$. 
Let $N$ be a nonnegative integer. 
Let 
$f \col 
(X_{\bul \leq N},D_{\bul \leq N}\cup 
Z_{\bul \leq N}) \lo Y$ 
be a morphism from a smooth 
$N$-truncated simplicial scheme with 
transversal $N$-truncated simplicial relative SNCD's
$D_{\bul \leq N}$ and $Z_{\bul \leq N}$ 
such that the morphism 
$\os{\circ}{f} \col X_{\bul \leq N} \lo Y$ 
is smooth, quasi-compact and quasi-separated.   
Let
\begin{equation*}
\begin{CD}
(X'_{\bul \leq N},D'_{\bul \leq N}\cup Z'_{\bul \leq N}) 
@>{g_{\bul \leq N}}>> 
(X_{\bul \leq N},D_{\bul \leq N}\cup Z_{\bul \leq N})\\
 @V{f'}VV  @VV{f}V \\
Y' @>{h}>> Y \\
@V{}VV  @VV{}V \\
(S',p{\cal O}_{S'},[~]') @>{u}>> (S,p{\cal O}_S,[~])
\end{CD}
\end{equation*}
be a commutative diagram of 
$N$-truncated simplicial $($log$)$ schemes 
such that the upper rectangle is cartesian, 
where $[~]'$ and $[~]$ are the canonical PD-structures on 
$p{\cal O}_{S'}$ and $p{\cal O}_S$, respectively.  
Let 
$$f_{(X_{\bul \leq N},Z_{\bul \leq N})} 
\col (X_{\bul \leq N},Z_{\bul \leq N}) \lo Y \quad  
{\text and} \quad f_{(X'_{\bul \leq N},Z'_{\bul \leq N})} 
\col (X'_{\bul \leq N},Z'_{\bul \leq N}) \lo Y'$$ 
be the induced morphisms by $f$ and $f'$, respectively. 
Then the base change morphism 
\begin{equation*}
{\cal K}_{S'}\otimes^L_{{\cal K}_S}
(Rf^{\rm conv}_{(X_{\bul \leq N},D_{\bul \leq N}\cup Z_{\bul \leq N})/S*}
({\cal K}_{(X_{\bul \leq N},D_{\bul \leq N}\cup Z_{\bul \leq N})/S}),
\del(L,P^{D_{\bul \leq N}}))
\os{\sim}{\lo} 
\tag{11.15.2}\label{eqn:simbct}
\end{equation*}
$$
(Rf^{\rm conv}_{(X'_{\bul \leq N},D'_{\bul \leq N}\cup Z'_{\bul \leq N})/S*}
({\cal K}_{(X'_{\bul \leq N},D'_{\bul \leq N}\cup Z'_{\bul \leq N})/S}),
\del(L,P^{D_{\bul \leq N}}))$$
is an isomorphism in the filtered derived category  
${\rm DF}({\cal K}_{Y'/S'})$. 
\end{prop}
\begin{proof}
In \cite[(5.9)]{nh2} 
we have proved that the log crystalline version of 
the morphism (\ref{eqn:simbct}) is a filtered isomorphism.  
Thus (\ref{prop:bchange}) immediately follows 
from the comparison theorem (\ref{eqn:cccopd}). 
\end{proof}

\begin{theo}[{\bf K\"{u}nneth formula}]\label{theo:kcrsp}
%Assume that $\pi{\cal V}\subset p{\cal V}$. 
Let $h$ be a nonnegative integer and 
let $N$ be a nonnegative integer satisfying 
the following inequality in {\rm \cite[(2.2.1)]{nh3}}$:$ 
\begin{equation*}
N > \max \{i+2^{-1}(h-i+1)(h-i+2)~\vert~0\leq i \leq h\}
=2^{-1}(h+1)(h+2). 
\end{equation*} 
Let $Y$ and 
$f^i \col (X^i_{\bul \leq N}, 
D^i_{\bul \leq N}\cup Z^i_{\bul \leq N})\lo Y$ 
$(i=1,2)$ be as in {\rm (\ref{prop:bchange})}. 
Set $f_3:=f_1\times_Y f_2$, 
$X^3_{\bul \leq N}:=
X^1_{\bul \leq N}\times_Y X^2_{\bul \leq N}$, 
$D^3_{\bul \leq N}:=
(D^1_{\bul \leq N}\times_Y X^2_{\bul \leq N})
\cup(X^1_{\bul \leq N}\times_Y D^2_{\bul \leq N})$ and 
$Z^3_{\bul \leq N}:=(Z^1_{\bul \leq N}
\times_Y X^2_{\bul \leq N})\cup
(X^1_{\bul \leq N}\times_Y Z^2_{\bul \leq N})$. 
Set 
\begin{align*} 
(F^{\bul}_i,\del_i)  := &
\{Rf_{(X^i_{{\bul}\leq {N}}, 
Z^i_{{\bul}\leq {N}}){\rm conv}*}
(E^{\log,Z^i_{{\bul}\leq {N}}}_{\rm conv}
({\cal O}_{(X^i_{{\bul}\leq {N}},
D^i_{{\bul}\leq {N}}\cup 
Z^i_{{\bul}\leq {N}})/S})), \del({L},P^{D^i_{{\bul}\leq {N}}})\} \\
{} &  (i=1,2,3). 
\end{align*}
Then, for an integer $k$, 
there exists a canonical isomorphism
\begin{equation*}
{\cal H}^{h}[
(\del_1{\otimes}_{{\cal O}_S}^L\del_2)_k
\{(F^{\bul}_1,\del_1){\otimes}_{{\cal O}_S}^L
(F^{\bul}_2,\del_2)\}] 
\os{\sim}{\lo}
{\cal H}^h[(\del_3)_k(F^{\bul}_3,\del_3)].
\tag{11.16.1}\label{eqn:kcrsp}
\end{equation*} 
%The isomorphism $(\ref{eqn:kcrsp})$ is compatible with 
%the base change isomorphism $(\ref{eqn:simbct})$. 
\end{theo} 
\begin{proof} 
In \cite[(5.10)]{nh3} we have proved 
the analogous preweight-filtered K\"{u}nneth formula 
for the log crystalline case. 
Thus (\ref{eqn:kcrsp}) immediately follows from 
the comparison theorem (\ref{eqn:cccopd})  as in (\ref{theo:ssd}).  
\end{proof}

\begin{theo}[{\bf Strict compatibility}]\label{theo:stcom}  
Assume that ${\cal V}'={\cal V}$ and $S'=S$ 
in {\rm (\ref{prop:bchange})}.  
Let $f \col (X_{\bul},D_{\bul}) \lo S_1$ 
and $f' \col (X'_{\bul},D'_{\bul})\lo S_1$ 
be proper smooth simplicial schemes 
with relative SNCD's over $S_1$.
Let 
$g\col (X'_{\bul},D'_{\bul}) \lo (X_{\bul},D_{\bul})$ 
be a morphism of simplicial log 
schemes over $S_1$. 
Let $h$ be an integer. 
Then the induced morphism 
\begin{equation*} 
g^* \col 
R^hf^{\rm conv}_{(X_{\bul},D_{\bul})/S*}
({\cal K}_{(X_{\bul},D_{\bul})/S})
\lo R^hf^{\rm conv}_{(X'_{\bul},D'_{\bul})/S*}
({\cal K}_{(X'_{\bul},D'_{\bul})/S}) 
\tag{11.17.1}\label{eqn:hfcds}
\end{equation*}
is strictly compatible with the weight filtration.
\end{theo}
\begin{proof} 
In \cite[(2.18.2)]{nh2} we have proved 
the strict compatibility of $g^*$ 
with respect to the weight filtration in the case of 
log crystalline cohomologies for the constant case. 
In fact we obtain the strict compatibility 
in the case of log crystalline cohomologies 
for the simplicial case (cf.~\cite[(8.5)]{nh3}) by the same proof as that of [loc.~cit.]. 
Now (\ref{theo:stcom}) follows from the comparison theorem 
(\ref{eqn:cccopd}). 
\end{proof}

\section{Log convergent cohomologies  
with compact supports}\label{sec:cptsc}
In this section we give fundamental results for the log convergent cohomology sheaf
with compact support. 
Because the proofs of the results are almost 
the same as those in \cite[(2.11)]{nh2}, we omit several of them. 
\par  
Let $S$ be as in \S\ref{sec:logcd}. 
Let $Y$ be a log smooth scheme over $S_1$. 
Let $U$ be a log open subscheme of $Y$. 
Let $\iota \col U \os{\sus}{\lo} {\cal U}$ 
be an immersion  into a log smooth noetherian 
log $p$-adic formal scheme over $S$ 
which is topologically of finite type over $S$. 
Let $\{T_n\}_{n=1}^{\infty}$ be 
the system of the universal enlargements 
of $\iota$. 
Let $(Y/S)_{\rm Rconv}$ be 
a subtopos of ${\rm Conv}(Y/S)$ 
whose objects are $T_n$'s for all 
immersions $\iota$'s above.  
Let 
\begin{equation*} 
Q^{\rm conv}_{Y/S} \col ({Y/S})_{\rm Rconv} \lo (Y/S)_{\rm conv} 
\tag{12.0.1}\label{eqn:rccd}
\end{equation*} 
be a morphism of topoi such that 
$Q^{{\rm conv}*}_{Y/S}(E)$ for an object $E\in (Y/S)_{\rm conv}$ 
is the natural restriction of $E$. 
Set 
$$\ol{u}{}^{\rm conv}_{Y/S}
:=u^{\rm conv}_{Y/S}\circ Q^{\rm conv}_{Y/S}
\col ({Y/S})_{\rm Rconv} \lo Y_{\rm zar}.$$

\begin{lemm}\label{lemm:nex}
Let $(U,{\cal U},\iota)$ be as above.   
Let $\bet_n\col {\mathfrak T}_{U,n}({\cal U})\lo U$ and 
$\bet_n(1) \col {\mathfrak T}_{U,n}({\cal U}\times_S{\cal U})\lo U$ 
be the natural morphisms. 
Then the following sequence 
\begin{equation*} 
u^{\rm conv}_{Y/S*}(E)\vert_U \lo 
\vpl_n\bet_{n*}(E_{{\mathfrak T}_{U,n}({\cal U})})
\os{\lo}{\lo} \vpl_n\bet_n(1)_*
(E_{{\mathfrak T}_{U,n}({\cal U}\times_S{\cal U})}) 
\tag{12.1.1}\label{eqn:rcucd}
\end{equation*} 
is exact. 
\end{lemm}
\begin{proof} 
(The proof is the obvious log convergent analogue of the proof 
in \cite[IV Proposition 2.3.2]{bb}.) 
To give a section of $u^{\rm conv}_{Y/S*}(E)\vert_U$ is 
to give compatible sections $s_T$'s of $E_{(V,T,j,v)}$'s 
for all objects $(V,T,j,v)$ of $(Y/S)_{\rm conv}$ 
with $v\col V\lo U$:  for any morphism 
$g\col (V',T',j',v')\lo (V,T,j,v)$, $s_{T'}=g^*(s_T)$ in ${\rm Conv}(Y/S)$. 
Hence we have the following natural morphism 
\begin{equation*} 
u^{\rm conv}_{Y/S*}(E)\vert_U \lo 
{\rm Ker}(\vpl_n\bet_{n*}(E_{{\mathfrak T}_{U,n}({\cal U})})
\os{\lo}{\lo} \vpl_n\bet_n(1)_*
(E_{{\mathfrak T}_{U,n}({\cal U}\times_S{\cal U})})). 
\end{equation*} 
Conversely let $\{s_n\}_{n=1}^{\infty}$ be a section of  
$\vpl_n\bet_{n*}(E_{{\mathfrak T}_{U,n}({\cal U})})$ such that 
the pull-back of the two projections 
${\mathfrak T}_{U,n}({\cal U}\times_S{\cal U})
\os{\lo}{\lo}{\mathfrak T}_{U,n}({\cal U})$ 
are equal. 
Let $(V,T,j,v)$ of $(Y/S)_{\rm conv}$ 
with $v\col V\lo U$.  
Then there exists a positive integer 
$n$ such that $V\lo U$ lifts to a morphism 
$T\lo {\mathfrak T}_{U,n}({\cal U})$. 
Now the rest of the proof is a routine work. 
\end{proof} 

\begin{coro}\label{coro:rex} 
Let $F$ be a sheaf of sets in $(Y/S)_{\rm Rconv}$. 
Let $(U,{\cal U},\iota)$, $\bet_n$ and $\bet_n(1)$ be as in 
{\rm (\ref{lemm:nex})}.  
Then the following sequence 
\begin{equation*} 
\ol{u}{}^{\rm conv}_{Y/S*}(F)\vert_U \lo 
\vpl_n\bet_{n*}(F_{{\mathfrak T}_{U,n}({\cal U})})
\os{\lo}{\lo} \vpl_n\bet_n(1)_*
(F_{{\mathfrak T}_{U,n}({\cal U}\times_S{\cal U})}) 
\tag{12.2.1}\label{eqn:rclucd}
\end{equation*} 
is exact. 
\end{coro}
\begin{proof} 
The proof is the same as that of \cite[IV Proposition 2.3.3]{bb}: 
we have only to use the formulas:
$\ol{u}{}^{\rm conv}_{Y/S*}(F)=
u^{\rm conv}_{Y/S*}(Q^{\rm conv}_{Y/S*}(F))$,  
$Q^{\rm conv}_{Y/S*}(F)_{{\mathfrak T}_{U,n}({\cal U})}
=F_{{\mathfrak T}_{U,n}({\cal U})}$ and 
$Q^{\rm conv}_{Y/S*}(F)_{{\mathfrak T}_{U,n}({\cal U}\times_S{\cal U})}
=F_{{\mathfrak T}_{U,n}({\cal U}\times_S{\cal U})}$. 
\end{proof}

\begin{coro}\label{coro:reux} 
Let $E$ be a sheaf of sets in $(Y/S)_{\rm conv}$. 
Then there exists the following canonical isomorphism$:$
\begin{equation*} 
u^{\rm conv}_{Y/S*}(E) 
\os{\sim}{\lo}\ol{u}{}^{\rm conv}_{Y/S*}(Q^{{\rm conv}*}_{Y/S}(E)).  
\end{equation*} 
\end{coro}
\begin{proof} 
By using (\ref{eqn:rcucd}) and (\ref{eqn:rclucd}), 
the proof is the obvious log convergent analogue of the proof 
in \cite[IV Proposition 2.3.5]{bb}; we omit the proof. 
\end{proof}

\begin{prop}\label{prop:injq}
Any injective module in $(Y/S)_{\rm Rconv}$ 
is of the form $Q^{{\rm conv}*}_{Y/S}(I)$ for an 
injective module $I$ in $(Y/S)_{\rm conv}$.  
\end{prop}
\begin{proof}
The proof is the obvious log convergent analogue of the proof 
in \cite[IV Proposition 2.2.5]{bb}; we omit the proof. 
\end{proof}

\parno 
The morphism 
$Q^{\rm conv}_{Y/S}$ induces the following morphism of 
ringed topoi 
\begin{equation*} 
(({Y/S})_{\rm Rconv},Q^{{\rm conv}*}_{Y/S}({\cal K}_{Y/S}))
\lo 
((Y/S)_{\rm conv},{\cal K}_{Y/S}),   
%\tag{12.0.1}\label{eqn:rccd}
\end{equation*}  
which we denote by $Q^{\rm conv}_{Y/S}$ again. 
Consequently we have the following morphism 
of ringed topoi: 
$$\ol{u}{}^{\rm conv}_{Y/S}
\col 
(({Y/S})_{\rm Rconv},
Q^{{\rm conv}*}_{Y/S}({\cal K}_{Y/S})) 
\lo (Y_{\rm zar},f^{-1}({\cal K}_S)).$$  

Then the following hold$:$

\begin{coro}[{\bf cf.~\cite[V Corollaire
1.3.3]{bb}}]\label{coro:ubq}
$R\ol{u}_{Y/S*}Q^*_{Y/S}
=Ru_{Y/S*}$. 
\end{coro}

%\begin{lemm}\label{prop:risoc}
%The sheaf $Q^{{\rm conv}*}_{Y/S}
%({\cal I}^{D,{\rm conv}}_{(X,D\cup Z)/S})$ 
%is an isocrystal on $((X,D\cup Z)/S)_{\rm Rconv}$.
%\end{lemm}
%\begin{proof} 
%Let $\iota \col (X,D\cup Z) \os{\sus}{\lo} {\cal Y}_j$ 
%$(j=1,2)$ be two immersions 
%into log smooth schemes over $S$. 
%Let $\{T^{(j)}_n\}_{n=1}^{\infty}$ be the system of 
%the universal enlargements of $\iota_j$. 
%Let $T^{(1)}_n \lo T^{(2)}_{n'}$ $(n,n'\in {\mab N})$ 
%be a morphism of enlargements of $(X,D\cup Z)$ over $S$. 
%Because the problem is local, we may assume that 
%${\cal Y}_j$ is obtained from a smooth scheme 
%${\cal X}_j$ with a transversal SNCD 
%${\cal D}_j\cup {\cal Z}_j$ over $S$ by \cite[(2.1.5)]{nh2}.  
%Now, by [loc.~cit.] again and the description 
%(\ref{eqn:ldra}), it is clear that the morphism 
%\begin{equation*} 
%({\cal I}^{D,{\rm conv}}_{(X,D\cup Z)/S})_{T^{(2)}_{n'}}
%\otimes_{{\cal O}_{T^{(2)}_{n'}}}{\cal O}_{T^{(1)}_{n}} 
%\lo ({\cal I}^{D,{\rm conv}}_{(X,D\cup Z)/S})_{T^{(1)}_{n}}
%\tag{12.3.1}\label{eqn:idcv} 
%\end{equation*} 
%is injective (cf.~the proof of (\ref{lemm:injf})); 
%indeed this morphism is an isomorphism. 
%\end{proof} 

%By the obvious log convergent version of 
%\cite[(5.3)]{tsp}, we see that 
%${\cal I}^{D,{\rm conv}}_{(X,D\cup Z)/S}$ 
%is an isocrystal on $((X,D\cup Z)/S)_{\rm Rconv}$.

\par 
Let $I_Y$ be a quasi-coherent ideal sheaf of the log structure $M_Y$ of $Y$. 
Let $(U,T,\iota,u)$ be an object of the log convergent site ${\rm Conv}(Y/S)$. 
Let $M_U$ and $M_T$ be the log structures of $U$ and $T$, respectively. 
Because $\iota \col U \lo T$ 
is an exact closed immersion,
$M_T/{\cal O}_{T}^*=M_U/{\cal O}_{U}^*$ 
on $U_{\rm zar}=T_{\rm zar}$.
Hence local sections of   
$I_Y\vert_U$ on $U$ lift locally to local sections $m$'s of $M_T$.
We define an ideal sheaf 
${\cal I}^{{\rm conv}}_{Y/S}\subset {\cal K}_{(X,D\cup Z)/S}$ by the following: 
${\cal I}^{{\rm conv}}_{Y/S}(T)$=
the ideal generated by the image of the above $m$'s by  
the following composite morphism 
$M_T \lo {\cal O}_T \lo {\cal K}_T$. 
\par 

\begin{prop}\label{prop:risoc}
The sheaf $\mathcal{I}^{\rm conv}_{Y/S}$ defines an isocrystal on $(Y/S)_{\rm Rconv}$. 
\end{prop}
\begin{proof}
Because the problem is local on $Y$, we can assume 
that there exists a chart 
$I \longrightarrow I_Y$ of $I_Y$. Then, for an enlargement 
$(U,T,\iota,u)$ in $(Y/S)_{\rm Rconv}$, the associated sheaf 
of $\mathcal{I}^{\rm conv}_{Y/S}$
at $T_{\rm zar}$ is 
an the ideal sheaf of $\mathcal{K}_T$
generated by the images of local sections of 
$\mathcal{M}_T$ which lift local sections of ${\rm Im}(I \longrightarrow I_Y|_U)$. 
Denote this sheaf by $\langle I \rangle \mathcal{K}_{T}$. 
Let $g\col (U'_{n'}, T'_{n'}, \iota'_{n'}, u'_{n'}) \longrightarrow 
(U_n, T_n, \iota_n, u_n)$ be a morphism in 
$(Y/S)_{\rm Rconv}$, where 
$(U_n, T_n, \iota_n, u_n)$ (resp.\ $(U'_{n'}, T'_{n'}, \iota'_{n'}, u'_{n'})$) is 
the $n$-th (resp.\ $n'$-th) member of the system of the universal enlargements 
of the exact closed immersion $i\col U \overset{\subset}{\longrightarrow} \mathcal{U}$ (resp.\  $i'\col U' \overset{\subset}{\longrightarrow} \mathcal{U}'$), where 
$U$ (resp.~$U'$) is an affine log open subscheme of $Y$ and $\mathcal{U}$ 
(resp.~$\mathcal{U}')$ is an affine noetherian log 
$p$-adic formal scheme which is formally log smooth over $S$. 
We claim that the following morphism 
\begin{equation*}
g^*\col \langle I \rangle \mathcal{K}_{T_n} \otimes_{\mathcal{K}_{T_n}} 
\mathcal{K}_{T'_{n'}} \longrightarrow \langle I \rangle \mathcal{K}_{T'_{n'}} 
\tag{12.6.1} \label{eqn:12.6.1}
\end{equation*}
induced by $g$ is an isomorphism, which proves (\ref{prop:risoc}).  
\par 
First we prove that it suffices to prove that our claim holds 
only in the case $n'=n$. 
If $n>n'$, the morphism $g$ is factoralized as 
$$ (U'_{n'}, T'_{n'}, \iota'_{n'}, u'_{n'}) \longrightarrow 
(U'_n, T'_n, \iota'_n, u'_n) \longrightarrow 
(U_n, T_n, \iota_n, u_n) $$
and the morphism \eqref{eqn:12.6.1} 
is equal to the following composite morphism 
%correspondingly factors as 
\begin{equation*}
\langle I \rangle \mathcal{K}_{T_n} \otimes_{\mathcal{K}_{T_n}} \mathcal{K}_{T'_{n'}} 
= 
(\langle I \rangle \mathcal{K}_{T_n} \otimes_{\mathcal{K}_{T_n}} \mathcal{K}_{T'_n} ) \otimes_{\mathcal{K}_{T'_n}} \mathcal{K}_{T'_{n'}} 
\longrightarrow 
\langle I \rangle \mathcal{K}_{T'_n}  \otimes_{\mathcal{K}_{T'_n}} \mathcal{K}_{T'_{n'}} \longrightarrow 
\langle I \rangle \mathcal{K}_{T'_{n'}}. 
\tag{12.6.2} \label{eqn:12.6.2}
\end{equation*}
Since the second arrow in (\ref{eqn:12.6.2}) is an isomorphism by 
the exactness of the functor $\iota^*_{n,n'}$ in the proof of (3.3), 
it suffices to prove that (\ref{eqn:12.6.1}) is an isomorphism only 
in the case $n'=n$, as claimed. 
If $n<n'$, the morphism $g$ fits into the diagram 
$$ (U'_{n'}, T'_{n'}, \iota'_{n'}, u'_{n'}) 
\overset{g}{\longrightarrow} 
(U_n, T_n, \iota_n, u_n) \longrightarrow 
(U_{n'}, T_{n'},\iota_{n'},u_{n'}), $$
and this composite morphism induces the morphism 
$$ 
\langle I \rangle \mathcal{K}_{T_{n'}} 
\otimes_{\mathcal{K}_{T_{n'}}} \mathcal{K}_{T'_{n'}} 
\longrightarrow \langle I \rangle \mathcal{K}_{T'_{n'}}.$$
This morphism is equal to the following composite morphism 
\begin{equation*}
\langle I \rangle \mathcal{K}_{T_{n'}} \otimes_{\mathcal{K}_{T_{n'}}} \mathcal{K}_{T'_{n'}}
= 
(\langle I \rangle \mathcal{K}_{T_{n'}} \otimes_{\mathcal{K}_{T_{n'}}} \mathcal{K}_{T_n} ) \otimes_{\mathcal{K}_{T_n}} \mathcal{K}_{T'_{n'}} 
\longrightarrow 
\langle I \rangle \mathcal{K}_{T_n}  
\otimes_{\mathcal{K}_{T_n}} \mathcal{K}_{T'_{n'}} \longrightarrow 
\langle I \rangle \mathcal{K}_{T'_{n'}}. 
\tag{12.6.3} \label{eqn:12.6.3}
\end{equation*}
Since the first morphism in (\ref{eqn:12.6.3}) is an isomorphism again 
by the exactness of the functor $\iota^*_{n',n}$ in the proof of (3.3), 
it suffices to prove that (\ref{eqn:12.6.1}) is an isomorphism only in the case $n'=n$, 
as claimed.  
Hence we can and do assume that $n'=n$ in the rest of the proof. 
\par 
Let $\widehat{\mathcal{U}}$ and  $\widehat{\mathcal{U}}'$ be the completions of 
$\mathcal{U}$ and $\mathcal{U}'$ along $U$ and $U'$, respectively, and 
let $\widehat{i}\col  U \overset{\subset}{\longrightarrow} \widehat{\mathcal{U}}$ and 
$\widehat{i}'\col U' \overset{\subset}{\longrightarrow} \widehat{\mathcal{U}}'$ be 
the exact closed immersions induced by $i$ and $i'$, respectively. 
Let $\widehat{\mathcal{U}}'|_{U \cap U'}$ be the log formal open subscheme of 
$\widehat{\mathcal{U}}'$ whose underlying topological space is equal to 
that of $U \cap U'$. 
Because the morphism $u'_n\col U'_n \longrightarrow Y$ factors through 
the morphism $U'_n \overset{g}{\longrightarrow} U_n \longrightarrow U$ and 
the morphism $U'_n \longrightarrow U'$, 
it factors through $U \cap U'$ and hence the morphism 
$(U'_n, T'_n, \iota_n, u_n) \longrightarrow (U', \widehat{U}', \widehat{i}', j')$ factors through 
$(U \cap U', \widehat{U}'|_{U \cap U'}, \widehat{i}'|_{U \cap U'}, j'|_{U \cap U'})$. Then, 
by checking the universal property, 
one easily sees that 
$(U'_n, T'_n, \iota_n, u_n)$ is the $n$-th member of the system of 
the universal enlargements of $(U \cap U', \widehat{U}'|_{U \cap U'}, \widehat{i}'|_{U \cap U'}, j'|_{U \cap U'})$. Hence we may replace $(U', \widehat{U}', \widehat{i}', j')$ by $(U \cap U', \widehat{U}'|_{U \cap U'}, \widehat{i}'|_{U \cap U'}, j'|_{U \cap U'})$ and 
we may assume that $U'$ is contained in $U$ to prove that 
the morphism (\ref{eqn:12.6.1}) for the case $n'=n$ is an isomorphism. 
\par 
Next, let $\widehat{\mathcal{U}}|_{U'}$ be the log formal open subscheme of $\widehat{\mathcal{U}}$ 
whose underlying topological space is equal to $U'$, and let 
$(U_n|_{U'}, T_n|_{U'}, \iota_n|_{U'}, u_n|_{U'})$ be 
the $n$-th member of the system of the universal enlargements of 
$(U', \widehat{\mathcal{U}}|_{U'}, \widehat{i}|_{U'}, j|_{U'})$. 
Then $U_n|_{U'}, T_n|_{U'}$ are open in $U_n, T_n$ respectively and the morphism $g$ factors through $(U_n|_{U'}, T_n|_{U'}, \iota_n|_{U'}, u_n|_{U'})$. 
Hence the morphism \eqref{eqn:12.6.1} factors as 
\begin{equation*}
\langle I \rangle \mathcal{K}_{T_n} \otimes_{\mathcal{K}_{T_n}} \mathcal{K}_{T'_n} 
= 
(\langle I \rangle \mathcal{K}_{T_n} \otimes_{\mathcal{K}_{T_n}} \mathcal{K}_{T_n|_{U'}} ) \otimes_{\mathcal{K}_{T_n|_{U'}}} \mathcal{K}_{T'_n} 
\longrightarrow 
\langle I \rangle \mathcal{K}_{T_n|_{U'}}  \otimes_{\mathcal{K}_{T_n|_{U'}}} \mathcal{K}_{T'_n} \longrightarrow 
\langle I \rangle \mathcal{K}_{T'_n}. 
\tag{12.6.4} \label{eqn:12.6.4}
\end{equation*}
Since $T_n|_{U'}$ is open in $T_n$, the morphism 
$\mathcal{K}_{T_n} \longrightarrow \mathcal{K}_{T_n|_{U'}}$ 
is flat and hence the first morphism in (\ref{eqn:12.6.4}) is an isomorphism. 
Hence we are reduced to showing that the second morphism in (\ref{eqn:12.6.4}) 
is an isomorphism. Thus we may replace $(U_n,T_n,\iota_n,u_n)$ and $(U,\widehat{\mathcal{U}}, \widehat{i},j)$ by $(U_n|_{U'}, T_n|_{U'}, \iota_n|_{U'}, u_n|_{U'})$ and 
$(U', \widehat{U}|_{U'}, \widehat{i}|_{U'}, j|_{U'})$, respectively, namely, 
we may assume that $U = U'$ to prove the assertion. 
\par 
Now let $\widehat{\mathcal{U}}''$ 
be the exactification of the immersion 
$i''\col U \overset{\subset}{\longrightarrow} \mathcal{U} \times_S \mathcal{U}'$ 
induced by $i$ and $i'$, and let 
$\widehat{i}''\col U \overset{\subset}{\longrightarrow} \widehat{\mathcal{U}}''$ 
be the exact closed immersion induced by $i''$. 
Let $(U''_n, T''_n, \iota''_n, u''_n)$ be the $n$-th member of the system of 
the universal enlargements of $(U, \widehat{\mathcal{U}}'', \widehat{i}'', j)$. 
Then the projections 
$$ 
(U, \widehat{\mathcal{U}}'', \widehat{i}'', j) \longrightarrow 
(U, \widehat{\mathcal{U}}, \widehat{i}, j), \quad 
(U, \widehat{\mathcal{U}}'', \widehat{i}'', j) \longrightarrow 
(U, \widehat{\mathcal{U}}', \widehat{i}', j) $$
induce morphisms 
$$ 
p_1\col (U''_n, T''_n, \iota''_n, u''_n) \longrightarrow 
(U_n, T_n, \iota_n, u_n), \quad 
p_2\col (U''_n, T''_n, \iota''_n, u''_n) \longrightarrow 
(U'_n, T'_n, \iota'_n, u'_n)
$$
repectively, and the morphisms 
\begin{align*}
& (U'_n, T'_n, \iota'_n, u'_n) \overset{g}{\longrightarrow} 
(U_n, T_n, \iota_n, u_n) \longrightarrow 
(U, \widehat{\mathcal{U}}, \widehat{i},j), \\
& (U'_n, T'_n, \iota'_n, u'_n) \longrightarrow 
(U, \widehat{\mathcal{U}}', \widehat{i}', j)
\end{align*}
induce a morphism 
$$h\col (U'_n, T'_n, \iota'_n, u'_n) \longrightarrow 
(U''_n, T''_n, \iota''_n, u''_n). $$
We can check the equalities 
$p_1 \circ h = g, p_2 \circ h = {\rm id}_{(U'_n, T'_n, \iota'_n, u'_n)}$ 
by using the universal property. 
Thus the morphism 
(\ref{eqn:12.6.1}) for the case $n'=n$ is factorized as 
\begin{equation*}
g^*\col 
\langle I \rangle \mathcal{K}_{T_n} 
\otimes_{\mathcal{K}_{T_n}} \mathcal{K}_{T'_n} 
= 
(\langle I \rangle \mathcal{K}_{T_n} 
\otimes_{\mathcal{K}_{T_n},p_1^*} \mathcal{K}_{T''_n} ) 
\otimes_{\mathcal{K}_{T''_n},h^*} \mathcal{K}_{T'_n} 
\os{p_1^*\otimes {\rm id}_{\mathcal{K}_{T''_n}}}{\longrightarrow} 
\langle I \rangle \mathcal{K}_{T''_n}  
\otimes_{\mathcal{K}_{T''_n},h^*} \mathcal{K}_{T'_n} \os{h^*}{\longrightarrow} 
\langle I \rangle \mathcal{K}_{T'_n}. 
\tag{12.6.5} \label{eqn:12.6.5}
\end{equation*}
Since $\widehat{i}\col U 
\overset{\subset}{\longrightarrow} \widehat{\mathcal{U}}$, $\widehat{i}''\col U \overset{\subset}{\longrightarrow} \widehat{\mathcal{U}}''$ are closed immersions and $\overset{\circ}{\widehat{\mathcal{U}}''} \longrightarrow \overset{\circ}{\widehat{\mathcal{U}}}$ is formally smooth, 
$p_1^*\col \mathcal{K}_{T_n} \longrightarrow \mathcal{K}_{T''_n}$ is flat by 
(\ref{lemm:eisd}) (2) (cf.~(\ref{prop:lef})). 
Hence the first morphism in \eqref{eqn:12.6.5} is an isomorphism. 
Moreover, the second morphism in \eqref{eqn:12.6.5} fits into the following commutative diagram: 
\begin{equation*}
\begin{CD}
(\langle I \rangle \mathcal{K}_{T'_n} \otimes_{\mathcal{K}_{T'_n},p_2^*} \mathcal{K}_{T''_n} ) \otimes_{\mathcal{K}_{T''_n},h^*} \mathcal{K}_{T'_n} 
@>{p_2^*\otimes {\rm id}_{\mathcal{K}_{T'_n}}}>> 
\langle I \rangle \mathcal{K}_{T''_n}  \otimes_{\mathcal{K}_{T''_n},h^*} \mathcal{K}_{T'_n} \\
@| @VV{h^*}V \\ 
\langle I \rangle \mathcal{K}_{T'_n} @= \langle I \rangle \mathcal{K}_{T'_n}.
\end{CD}
\tag{12.6.6} \label{eqn:12.6.6}
\end{equation*}
Since $\widehat{i}'\col U \overset{\subset}{\longrightarrow} \widehat{\mathcal{U}}'$, $\widehat{i}''\col U \overset{\subset}{\longrightarrow} \widehat{\mathcal{U}}''$ are closed immersions and $\overset{\circ}{\widehat{\mathcal{U}}''} \longrightarrow \overset{\circ}{\widehat{\mathcal{U}'}}$ is formally smooth, 
$p_2^*\col \mathcal{K}_{T'_n} \longrightarrow \mathcal{K}_{T''_n}$ is flat by 
(\ref{lemm:eisd}) (2). Hence the top horizontal morphism in \eqref{eqn:12.6.6} is an isomorphism, 
and then the commutativity of the diagram \eqref{eqn:12.6.6} implies that the second
morphism $h^*$ in \eqref{eqn:12.6.5} is also an isomorphism. Consequently we see that 
the morphism \eqref{eqn:12.6.1} is an isomorphism. 
We can complete the proof of the lemma. 
\end{proof}

\begin{rema}
The proposition (\ref{prop:risoc}) includes another 
and more direct proof of \cite[Lemma 5.3]{tsp}. 
\end{rema}

\par 
Let the notations be as in \S\ref{sec:lcs}. 
Let $(U,T,\iota,u)$ be an object of the 
log convergent site 
${\rm Conv}((X,D\cup Z)/S)={\rm Conv}((X,M({D\cup Z}))/S)$. 
Because the morphism $u\col U\lo (X,D\cup Z)$ is strict, 
$M_U:=u^*(M({D\cup Z}))$. 
Let $M_T$ be the log structure of $T$. 
Because
$\iota \col U \lo T$ is an exact closed immersion,
$M_T/{\cal O}_{T}^*=M_U/{\cal O}_{U}^*$ 
on $U_{\rm zar}=T_{\rm zar}$.
Hence the local generator of $u^*(M({D\cup Z}))/{\cal O}_U^*=
u^{-1}(M({D\cup Z})/{\cal O}_X^*)$ 
lifts to a local section $m$ of $M_T$.
We define an ideal sheaf 
${\cal I}^{D,{\rm conv}}_{(X,D\cup Z)/S}
\subset {\cal K}_{(X,D\cup Z)/S}$ 
by the following: 
${\cal I}^{D,{\rm conv}}_{(X,D\cup Z)/S}(T)$=
the ideal generated by
the image of $m$ by  
the following composite morphism 
$M_{T} \lo {\cal O}_{T} \lo {\cal K}_T$.

\begin{defi}\label{defi:csc}
We call the higher direct image sheaf
$R^hf_{(X,D\cup Z)/S*}
({\cal I}^{D,{\rm conv}}_{(X,D\cup Z)/S})$ $(h\in {\mab N})$ 
in ${S}_{\rm zar}$ the 
{\it log convergent cohomology sheaf with compact support} 
{\it with respect to} $D$ and denote it by
$R^hf_{(X, D\cup Z)/S*, 
{\rm c}}({\cal K}_{(X,D\cup Z;Z)/S})$. 
\end{defi}

\par
Let $\{D_{\lam}\}_{\lam}$ be 
a decomposition of $D$ by smooth 
components of $D$.
Fix a total order on $\Lam$. 
Let the notations be as in \S\ref{sec:bd}.
The exact closed immersion 
$\iota^{\ul{\lam}_j}_{\ul{\lam}} \col 
(D_{\ul{\lam}}, Z\vert_{D_{\ul{\lam}}}) 
\os{\subset}{\lo} 
(D_{\ul{\lam}_j}, Z\vert_{D_{\ul{\lam}_j}}) $
induces the following morphism  
\begin{equation*}
(-1)^j\iota_{\ul{\lam}{\rm conv}}^{\ul{\lam}_j*}
\col 
{\cal K}_{(D_{\ul{\lam}_j},Z\vert_{D_{\ul{\lam}_j}})/S} 
\otimes_{\mab Z}
\vp_{\ul{\lam}_j{\rm conv}}(D/S;Z) 
\lo 
\tag{12.8.1}\label{eqn:defcbd}
\end{equation*}
$$\iota_{\ul{\lam}{\rm conv}*}^{\ul{\lam}_j}
({\cal K}_{(D_{\ul{\lam}}, Z\vert_{D_{\ul{\lam}}})/S})
\otimes_{\mab Z}\vp_{\ul{\lam}{\rm conv}}(D/S;Z)$$ 
defined by
$x\otimes (\lam_0 \cdots \wh{\lam}_j \cdots \lam_{k-1}) 
\lom (-1)^j\iota_{\ul{\lam}{\rm conv}}^{\ul{\lam}_j*}
(x)\otimes (\lam_0 \cdots  \lam_{k-1})$. 
%It is easy to check that the morphism 
%$(-1)^j\iota_{\ul{\lam}{\rm conv}}^{\ul{\lam}_j*}$ 
%is well-defined.  
Set 
\begin{equation*}
\iota^{(k-1)*}_{\rm conv}:=
\sum_{\{\lam_0, \lam_1, \cdots, \lam_{k-1}~\vert~\lam_i 
\not= \lam_l~(i\not=l)\}}
\sum_{j=0}^{k-1}
a_{\ul{\lam}_j{\rm conv}*} \circ ((-1)^j
\iota_{\ul{\lam}{\rm conv}}^{\ul{\lam}_j*}) 
\col
\tag{12.8.2}\label{eqn:descbd}
\end{equation*}
$$a_{{\rm conv}*}^{(k-1)}
({\cal K}_{(D^{(k-1)},Z\vert_{D^{(k-1)}})/S}
\otimes_{\mab Z}\vp^{(k-1)}_{\rm conv}(D/S;Z)) 
\lo $$
$$a_{{\rm conv}*}^{(k)}
({\cal K}_{(D^{(k)},Z\vert_{D^{(k)}})/S}
\otimes_{\mab Z}\vp^{(k)}_{\rm conv}(D/S;Z)).$$
The composite morphism
$\iota^{(k)*}_{\rm conv}\circ 
\iota^{(k-1)*}_{\rm conv}$ vanishes. 
The morphism $\iota^{(k-1)*}_{\rm conv}$ is 
independent of the choice of 
the decomposition of 
$\{D_{\lam}\}_{\lam}$ by smooth components of $D/S_1$ 
as in \cite[(2.11.2)]{nh2}. 

\begin{lemm}\label{lemm:exercr}
Let 
\begin{equation*}
\eps_{\rm conv}: ({(X,D \cup Z)/S})_{\rm Rconv} \lo 
({(X,Z)/S})_{\rm Rconv} 
\end{equation*}
be a morphism of topoi forgetting the log structure of $D$. 
Then there exists a morphism of topoi 
\begin{equation*}
\eps_{\rm Rconv}: ({(X,D \cup Z)/S})_{\rm Rconv} \lo 
({(X,Z)/S})_{\rm Rconv} 
\end{equation*}
fitting into the following commutative diagram of topoi$:$
\begin{equation*}
\begin{CD}
({(X,D\cup Z)/S})_{\rm Rconv} @>{\eps_{\rm Rconv}}>> 
({(X,Z)/S})_{\rm Rconv} \\
@V{Q_{(X,D\cup Z)/S}}VV @VV{Q_{(X,Z)/S}}V \\ 
({(X,D\cup Z)/S})_{\rm conv} @>{\eps_{\rm conv}}>> 
({(X,Z)/S})_{\rm conv}. 
\end{CD}
\tag{12.9.1}\label{diag:ercrys}
\end{equation*}
\end{lemm}
\begin{proof}
First we show the existence of $\eps_{\rm Rconv}$. 
To show this, it suffices to see that, for an object 
$T:=(U,T,M_T,\iota) \in ((X,D \cup Z)/S)_{\rm Rconv}$, 
the object $(U,T,N_T,\iota)$ constructed in 
\S\ref{sec:vflvc} belongs to $((X,Z)/S)_{\rm Rconv}$ 
Zariski locally on $T$ because we can define the exact functor 
$\eps_{\rm Rconv}^*$ in the same way as $\eps^*$ in \cite[p.~85]{nh2}. 
Assume that $T$ is an enlargement $T_n$ for some $n$, 
where $\{T_n\}_{n=1}^{\infty}$ is the universal enlargement of a closed immersion 
$i: (U,(D\cup Z)|_U) \os{\subset}{\lo} ({\cal U},M_{{\cal U}})$ 
into a log smooth formal scheme over $S$.  
Since the log structure $M_{{\cal U}}$ is defined on the 
Zariski site of ${\cal U}$, we have a factorization 
$$ (U,(D\cup Z)|_U) \os{\subset}{\lo} ({\cal U}',M_{{\cal U}'}) 
\lo ({\cal U},M_{{\cal U}}) $$ 
of $i$ Zariski locally on ${\cal U}$ 
such that the first morphism is an exact closed immersion and that 
the second morphism is log \'{e}tale. 
Then $T$ is a member of the universal enlargement of the first 
morphism. Hence we may assume that $i$ is an exact closed immersion. Then, by 
\cite[(2.1.5)]{nh2}, 
we may assume that $i$ is an admissible closed immersion 
$(U,(D\cup Z)|_U) \os{\subset}{\lo} ({\cal U},{\cal D}\cup {\cal Z})$. 
In this case, the log structure $N_T$ on $T$ is nothing but the 
pull-back of the log structure on ${\cal U}$ defined by ${\cal Z}$. 
Consequently $(U,T,N_T,\iota)$ is a member of 
the universal enlargement of the exact closed immersion 
$(U,Z|_U) \os{\subset}{\lo} ({\cal U},{\cal Z})$. 
Hence $(U,T,N_T,\iota)$ belongs to 
$((X,Z)/S)_{\rm Rconv}$ Zariski locally on $T$. 
Now it is clear that 
we have the morphism $\eps_{\rm Rconv}$ of topoi. It is easy to see 
that we have the commutative diagram (\ref{diag:ercrys}). 
%\qed
\end{proof} 

\begin{lemm}\label{lemm:notba} 
Let the notations be as in {\rm (\ref{lemm:exercr})}. 
Then the following natural morphism of functors 
$$ Q^*_{(X,Z)/S} R\eps_{{\rm conv}*}
\lo R\eps_{{\rm Rconv}*} Q^*_{(X,D\cup Z)/S} $$ 
for ${\cal O}_{(X,D\cup Z)/S}$-modules 
is an isomorphism. 
\end{lemm}
\begin{proof} 
We say that, for an ${\cal O}_{(X,D\cup Z)/S}$-module $F$, 
$F$ is parasitic if $F_{T_n}=0$ for all $n\in {\mab Z}_{\geq 1}$, 
where $\{T_n\}_{n=1}^{\infty}$ is the system of the universal enlargements 
of any immersion $\iota \col (U,D\cup Z) \os{\sus}{\lo} {\cal U}$ 
into a log smooth noetherian 
log $p$-adic formal scheme over $S$ 
which is topologically of finite type over $S$. 
As in \cite[(2.11.5)]{nh2}, we have to prove that 
for any parasitic ${\cal O}_{(X,D\cup Z)/S}$-module 
$F$ of $((X,D\cup Z)/S)_{\rm conv}$, 
$R^q\eps_*(F)$ is also parasitic for any $q\geq 0$.  
%To see this, it suffices to prove that, for any 
%object $T_n \in ((X,Z)/S)_{\rm Rconv}$ 
%with $T_n$ sufficiently small, the sheaf $(R^q\eps_*(F))_T$ on $(T_n)_{\rm zar}$ 
%induced by $R^q\eps_*(F)$ is equal to zero. Hence we may assume that 
\par 
We may assume that there exists a closed immersion 
$\iota: (U,Z) \os{\subset}{\lo} {\cal U}$ 
into an affine log smooth scheme over $S$, where $(U,Z)$ is a log open subscheme of 
$(X,Z)$. Let $\{T_n\}_{n=1}^{\infty}$ be the system of the universal enlargements 
of any immersion $\iota$.  
%such that 
%$(T,M_T)$ is the log PD-envelope of $i$. 
On the other hand, let 
us take a closed immersion $\iota':(U,D\cup Z) \os{\subset}{\lo} {\cal Y}$ 
into an affine log smooth formal scheme $S$. 
Then, for any $m \in {\mab Z}_{\geq 1}$, we have the closed immersion 
$\iota_m: (U,(D\cup Z)|_U) \os{\subset}{\lo} {\cal U} \times_S {\cal Y}^m$ 
induced by $\iota \circ (\eps \vert_U)$ and $\iota'$. Let $\{T(m)_n\}_{n=1}^{\infty}$ be the 
system of the closed immersion $\iota_m$ over $S$. 
%Then $T(m)_n$ is isomorphic to the log PD-envelope of the closed 
%immersion $(U,(D\cup Z)\vert_U) \os{\subset}{\lo} (T,M_T) \times_S {\cal Y}^n$ 
%(induced by the composite $\iota \circ (\eps |_U)$ and $i'$) 
%compatible with $\ol{\delta}$, where $\ol{\delta}$ is the PD-structure on 
%${\rm Ker}({\cal O}_T \lo {\cal O}_U) + {\cal I}{\cal O}_T$ extending 
%$\gamma$ and $\delta$. 
By the log convergent version of \cite[V 1.2.5]{bb}, 
we have the following equalities: 
$$ (R^q\eps_*(F))_{T_n} = R^qf_{(U,(D\cup Z)|_U)/T_n *}(F) 
= R^q(\iota \circ (\eps|_U))_*\text{{\rm \v{C}A}}(F)_n, $$ 
where $\text{{\rm \v{C}A}}(F)_n = F_{T(\bul)_n}$ is 
the \v{C}ech-Alexander complex of $F$ (cf.~\cite[V 1.2.3]{bb}). Since 
$F$ is parasitic, we have $F_{T(m)_n} = 0$ for any $m$. 
Hence $(R^q\eps_*(F))_{T_n}=0$. 
\end{proof}

\begin{theo}\label{theo:dcscy}
%Let $\eps_{(X,D\cup Z,Z)/S} \col ({(X,D\cup Z)/S})_{\rm conv} 
%\lo ({(X,Z)/S})_{\rm conv}$ be 
%the morphism of ringed topoi 
%forgetting the log structure along $D$ 
%{\rm ((\ref{defi:forg}))}. 
Let the notations be as in {\rm (\ref{lemm:exercr})}. 
Set 
\begin{align*}
E_{{\rm conv},c}({\cal K}_{(X,D\cup Z)/S})
:= ( & {\cal K}_{(X,Z)/S}\otimes_{\mab Z} 
\vp^{(0)}_{\rm conv}(D/S;Z) 
\os{\iota^{(0)*}_{\rm conv}}{\lo} 
\tag{12.11.1}\label{eqn:eoxdz}\\ 
{} & a^{(1)}_{{\rm conv}*}
({\cal K}_{(D^{(1)},Z\vert_{D^{(1)}})/S}
\otimes_{\mab Z} 
\vp^{(1)}_{\rm conv}(D/S;Z))
\os{\iota^{(1)*}_{\rm conv}}{\lo} \\
{} & a^{(2)}_{{\rm conv}*}
({\cal K}_{(D^{(2)},Z\vert_{D^{(2)}})/S}
\otimes_{\mab Z} 
\vp^{(2)}_{\rm conv}(D/S;Z))
\os{\iota^{(2)*}_{\rm conv}}{\lo} \cdots). 
\end{align*} 
Then there exists 
the following canonical isomorphism in
$D^+({\cal K}_{(X,Z)/S}):$
\begin{equation*}
Q^*_{(X,Z)/S}R\eps_{{\rm conv}*}
({\cal I}^{D,{\rm conv}}_{(X,D\cup Z)/S}) \os{\sim}{\lo}  
Q^*_{(X,Z)/S}E_{{\rm conv},c}({\cal K}_{(X,D\cup Z)/S}).
\tag{12.11.2}\label{ali:csct} 
\end{equation*} 
\end{theo}
\begin{proof} 
(By using (\ref{theo:cpvcs}), 
the proof is the same as that of 
\cite[(2.11.3)]{nh2}.)
Assume that we are given 
the commutative diagram (\ref{cd:pfcd}) for $(X,D\cup Z)$.  
Let 
$b^{(k)}_{\bul} \col {\cal D}^{(k)}_{\bul} \lo
{\cal X}_{\bul}$ be the natural morphism. 
Let $\{(U_{\bul n},T_{\bul n})\}_{n=1}^{\infty}$ be 
the system of 
the universal enlargements of the simplicial immersion 
$(X_{\bul},D_{\bul}\cup Z_{\bul})  \os{\sus}{\lo} {\cal P}_{\bul}$ 
over $S$. Note that the notation $\{(U_{\bul n},T_{\bul n})\}_{n=1}^{\infty}$ 
denotes the different objects in this section and \S\ref{sec:wfcipp}. 
By (\ref{prop:risoc}) we have the following connection 
\begin{align*} 
\{({\cal I}^D_{(X_{\bul},D_{\bul}\cup Z_{\bul})/S})_{{T}_{\bul n}}\}_{n=1}^{\infty} \lo 
\{({\cal I}^D_{(X_{\bul},D_{\bul}\cup Z_{\bul})/S})_{{T}_{\bul n}}
\otimes_{{\cal O}_{{\cal P}_{\bul}}}
\Omega^1_{{\cal P}_{\bul}/S}\}_{n=1}^{\infty}. 
\end{align*} 
Let ${\cal F}^{\bul}$ 
be the isocrystal on $(X_{\bul}, D_{\bul}\cup Z_{\bul})/S$ 
corresponding to the integrable log connection 
$\vpl_n({\cal I}^D_{(X,D\cup Z)/S})_{{T}_{\bul n}} \lo 
\vpl_n({\cal I}^D_{(X,D\cup Z)/S})_{{T}_{\bul n}}\otimes_{{\cal O}_{{\cal P}_{\bul}}}
\Omega^1_{{\cal P}_{\bul}/S}$. 
By taking the exactification of the immersion 
$(X_{\bul}, D_{\bul}\cup Z_{\bul})\os{\sus}{\lo} {\cal P}_{\bul}$, 
we may assume that ${\cal P}_{\bul}=({\cal X}_{\bul},{\cal D}_{\bul}\cup {\cal Z}_{\bul})$, 
where ${\cal D}_{\bul}$ and ${\cal Z}_{\bul}$ are transversal simplicial SNCD's on 
${\cal X}_{\bul}/S$ such that ${\cal Q}^{\rm ex}_{\bul}=
({\cal X}_{\bul},{\cal Z}_{\bul})$. 
Let 
\begin{equation*}
\eps_{\bul{\rm conv}}: ({(X_{\bul},D_{\bul} \cup Z_{\bul})/S})_{\rm conv} \lo 
({(X_{\bul},Z_{\bul})/S})_{\rm conv} 
\end{equation*}
and 
\begin{equation*}
\eps_{\bul{\rm Rconv}}: ({(X_{\bul},D_{\bul} \cup Z_{\bul})/S})_{\rm Rconv} \lo 
({(X_{\bul},Z_{\bul})/S})_{\rm Rconv} 
\end{equation*}
be the simplicial versions of $\eps_{\rm conv}$ and $\eps_{\rm Rconv}$, respectively. 
Then we obtain the following isomorphisms: 
{\allowdisplaybreaks{
\begin{align*}
& \phantom{\,\,\,\os{=}{\lo}} 
Q^*_{(X,Z)/S}R\eps_*({\cal I}^D_{(X,D\cup Z)/S}) 
\tag{12.11.3}\label{ali:omnx}\\ 
& \os{=}{\lo} 
Q^*_{(X,Z)/S}R\eps_*
R\pi_{(X,D\cup Z)/S{\rm conv}*} \pi^{-1}_{(X,D\cup Z)/S{\rm conv}}
({\cal I}^D_{(X,D\cup Z)/S}) \\
& \os{=}{\lo} 
R\eps_{{\rm Rconv}*}R\pi_{(X,D\cup Z)/S{\rm Rconv}*}
Q^*_{(X_{\bul},D_{\bul}\cup Z_{\bul})/S} 
\pi^{-1}_{(X,D\cup Z)/S{\rm conv}} ({\cal I}^D_{(X,D\cup Z)/S}) \\ 
& \os{=}{\longleftarrow}
R\eps_{{\rm Rconv}*}R\pi_{(X,D\cup Z)/S{\rm Rconv}*}
Q^*_{(X_{\bul},D_{\bul}\cup Z_{\bul})/S} ({\cal F}^{\bul}) \\ 
& \os{=}{\longrightarrow}
R\pi_{(X,Z)/S{\rm Rconv}*}R\eps_{\bul{\rm Rconv}*}
Q^*_{(X_{\bul},D_{\bul}\cup Z_{\bul})/S} ({\cal F}^{\bul}) \\ 
& \os{=}{\longleftarrow} 
R\pi_{(X,Z)/S{\rm Rconv}*}
Q^*_{(X_{\bul},Z_{\bul})/S}R\eps_{\bul *} ({\cal F}^{\bul}) \\ 
& \os{=}{\lo} 
R\pi_{(X,Z)/S{\rm Rconv}*}
Q^*_{(X_{\bul},Z_{\bul})/S}R\eps_{\bul *}
L^{\rm conv}_{(X_{\bul},D_{\bul}\cup Z_{\bul})/S}(\Omega^{\bul}_{{\cal X}_{\bul}/S}
(\log ({\cal Z}_{\bul}-{\cal D}_{\bul}))\otimes_{\mab Z}{\mab Q}) \\ 
& \os{=}{\longleftarrow} 
R\pi_{(X,Z)/S{\rm Rconv}*}
Q^*_{(X_{\bul},Z_{\bul})/S}
L^{\rm conv}_{(X_{\bul},Z_{\bul})/S}(\Omega^{\bul}_{{\cal X}_{\bul}/S}
(\log ({\cal Z}_{\bul}-{\cal D}_{\bul}))\otimes_{\mab Z}{\mab Q}). 
\end{align*}}}
Here we have used the simplicial version of (\ref{lemm:notba}). 
By the same argument as that in \cite[(4.2.2) (a), (c)]{di}, 
the following sequence
\begin{equation*} 
0 \lo  \Om^{\bul}_{{\cal X}_{\bul}/S}
(\log ({\cal Z}_{\bul}-{\cal D}_{\bul})) 
\lo  \Om^{\bul}_{{\cal X}_{\bul}/S}(\log {\cal Z}_{\bul})
\otimes_{\mab Z}
\vp_{\rm zar}^{(0)}({\cal D}_{\bul}/S)
\tag{12.11.4}
\label{eqn:ombcx}
\end{equation*} 
$$\os{\iota_{\bul, {\rm zar}}^{(0)*}}{\lo} 
b^{(1)}_{{\bul}*}(\Om^{\bul}_{{\cal D}_{\bul}^{(1)}/S}
(\log {\cal Z}_{\bul}\vert_{{\cal D}^{(1)}_{\bul}})
\otimes_{\mab Z}
\vp_{\rm zar}^{(1)}({\cal D}_{\bul}/S)) 
\os{\iota_{\bul, {\rm zar}}^{(1)*}}{\lo} \cdots $$
is exact. Here we define $\iota_{\bul, {\rm zar}}^{(k)*}$ 
similarly as for $\iota^{(k)*}_{\rm conv}$. 
Hence  
$\Om^{\bul}_{{\cal X}_{\bul}/S}
(\log ({\cal Z}_{\bul}-{\cal D}_{\bul}))$ 
is quasi-isomorphic 
to the single complex of the following double complex

\begin{equation*}
{\footnotesize{
\begin{CD}
\cdots @>>> \cdots \\
@A{d}AA  @A{-d}AA  \\
\Om^2_{{\cal X}_{\bul}/S}(\log {\cal Z}_{\bul})\otimes_{\mab Z}
\vp_{\rm zar}^{(0)}({\cal D}_{\bul}/S) 
@>{\iota_{\bul, {\rm zar}}^{(0)*}}>> 
b^{(1)}_{{\bul}*}(\Om^2_{{\cal D}_{\bul}^{(1)}/S}
(\log {\cal Z}_{\bul}\vert_{{\cal D}_{\bul}^{(1)}})
\otimes_{\mab Z}
\vp_{\rm zar}^{(1)}({\cal D}_{\bul}/S))  
\\ 
@A{d}AA  @A{-d}AA \\
\Om^1_{{\cal X}_{\bul}/S}
(\log {\cal Z}_{\bul})\otimes_{\mab Z}
\vp_{\rm zar}^{(0)}({\cal D}_{\bul}/S)   
@>{\iota_{\bul, {\rm zar}}^{(0)*}}>> 
b^{(1)}_{{\bul}*}(\Om^1_{{\cal D}_{\bul}^{(1)}/S}
(\log {\cal Z}_{\bul}\vert_{{\cal D}_{\bul}^{(1)}})
\otimes_{\mab Z}\vp_{\rm zar}^{(1)}({\cal D}_{\bul}/S)) \\ 
@A{d}AA  @A{-d}AA  \\
{\cal O}_{{\cal X}_{\bul}}
\otimes_{\mab Z}\vp_{\rm zar}^{(0)}({\cal D}_{\bul}/S)   
@>{\iota_{\bul, {\rm zar}}^{(0)*}}>>
b^{(1)}_{{\bul}*}({\cal O}_{{\cal D}_{\bul}^{(1)}}
\otimes_{\mab Z}\vp_{\rm zar}^{(1)}({\cal D}_{\bul}/S)) 
\end{CD}}}
\tag{12.11.5}
\label{cd:bgcsrl}
\end{equation*}
\begin{equation*}
{\footnotesize{
\begin{CD}
@>>> \cdots @>>> \cdots\\
@. @A{d}AA  \\
@>{\iota_{\bul, {\rm zar}}^{(1)*}}>> 
b^{(2)}_{{\bul}*}(\Om^2_{{\cal D}_{\bul}^{(2)}/S}
(\log {\cal Z}_{\bul}\vert_{{\cal D}_{\bul}^{(2)}})
\otimes_{\mab Z}\vp_{\rm zar}^{(2)}({\cal D}_{\bul}/S))  
@>{\iota_{\bul, {\rm zar}}^{(2)*}}>> \cdots \\
@. @A{d}AA   \\
@>{\iota_{\bul, {\rm zar}}^{(1)*}}>> 
b^{(2)}_{{\bul}*}(\Om^1_{{\cal D}_{\bul}^{(2)}/S}
(\log {\cal Z}_{\bul}\vert_{{\cal D}_{\bul}^{(2)}})
\otimes_{\mab Z}\vp_{\rm zar}^{(2)}({\cal D}_{\bul}/S)) 
@>{\iota_{\bul, {\rm zar}}^{(2)*}}>> \cdots \\
@. @A{d}AA \\
@>{\iota_{\bul, {\rm zar}}^{(1)*}}>> 
b^{(2)}_{{\bul}*}({\cal O}_{{\cal D}_{\bul}^{(2)}}
\otimes_{\mab Z}\vp_{\rm zar}^{(2)}({\cal D}_{\bul}/S)) 
@>{\iota_{\bul, {\rm zar}}^{(2)*}}>> \cdots.
\end{CD}}}
\end{equation*} 
We claim that the following sequence 
\begin{align*} 
&0 \lo Q^*_{(X_{\bul},Z_{\bul})/S}
L^{\rm conv}_{(X_{\bul},Z_{\bul})/S}
(\Om^{\bul}_{{\cal X}_{\bul}/S}
(\log({\cal Z}_{\bul}-{\cal D}_{\bul}))\otimes_{\mab Z}{\mab Q}) \lo \tag{12.11.6}
\label{eqn:lics} \\ 
& Q^*_{(X_{\bul},Z_{\bul})/S}\{L^{\rm conv}_{(X_{\bul},Z_{\bul})/S}
(\Om^{\bul}_{{\cal X}_{\bul}/S}
(\log {\cal Z}_{\bul})) \otimes_{\mab Z}\vp_{\rm conv}^{(0)}(D_{\bul}/S)\otimes_{\mab Z}{\mab Q}\} 
\os{Q^*_{(X_{\bul},Z_{\bul})/S}(\iota_{\bul,{\rm conv}}^{(0)*})}{\lo}\\ 
& Q^*_{(X_{\bul},Z_{\bul})/S}a^{(1)}_{{\bul}{\rm conv}*}
\{L^{\rm conv}_{(X_{\bul},Z_{\bul})/S}((\Om^{\bul}_{{\cal D}_{\bul}^{(1)}/S}
(\log {\cal Z}_{\bul}\vert_{{\cal D}^{(1)}_{\bul}}) 
\otimes_{\mab Z}\vp_{\rm conv}^{(1)}(D_{\bul}/S)\otimes_{\mab Z}{\mab Q})\}
\os{Q^*_{(X_{\bul},Z_{\bul})/S}(\iota_{\bul,{\rm conv}}^{(1)*})}{\lo} 
\cdots
\end{align*}
is exact.  Indeed, the question is local and 
we have only to prove that the sequence 
(\ref{eqn:lics}) for $\bul=m$ is exact for a fixed $m \in {\mab N}$. 
Since the question is local, we may ignore the orientation sheaves 
$\vp_{\rm conv}^{(k)}(D_{\bul}/S)$ and $\vp_{\rm zar}^{(k)}({\cal D}_{\bul}/S)$. 
That is, 
we have only to prove that the following sequence 
\begin{align*}
0 & \lo {\cal K}_{T_{mn}}\otimes_{{\cal O}_{{\cal X}_m}}
\Om^{\bul}_{{\cal X}_m/S} (\log({\cal Z}_m-{\cal D}_m)) 
\lo {\cal K}_{T_{mn}}\otimes_{{\cal O}_{{\cal X}_m}}
\Om^{\bul}_{{\cal X}_m/S} (\log {\cal Z}_m)
\tag{12.11.7}\label{eqn:ilics} \\ 
& \os{\iota_{m,{\rm zar}}^{(0)*}}{\lo}
{\cal K}_{T_{mn}}\otimes_{{\cal O}_{{\cal X}_m}}
b^{(1)}_{m*}
(\Om^{\bul}_{{\cal D}_m^{(1)}/S}
(\log {\cal Z}_m\vert_{{\cal D}^{(1)}_m})
\os{\iota_{m,{\rm zar}}^{(1)*}}{\lo} \cdots 
\end{align*}
is exact.
We may have cartesian diagrams 
(\ref{cd:dxs}) and (\ref{cd:xxd}) for the SNCD 
${\cal D}_m\cup {\cal Z}_m$ on ${\cal X}_m$; 
we assume that 
${\cal D}_m$ (resp.~${\cal Z}_m$) 
is defined by an equation 
$x_1\cdots x_t=0$ (resp.~$x_{t+1}\cdots x_s=0$). 
Set ${\cal J}_m:=(x_{d+1}, \ldots, x_{d'}){\cal O}_{{\cal X}_m}$. 
Set ${\cal X}'_m:=
\ul{\rm Spec}_{{\cal X}_m}({\cal O}_{{\cal X}_m}/{\cal J}_m)$ 
and ${\cal X}'':=\ul{\rm Spec}_S
({\cal O}_S[x_{d+1}, \ldots, x_{d'}])$. 
Let ${\cal D}'_m$ (resp.~${\cal Z}'_m$) be 
the closed subscheme of 
${\cal X}'_m$ defined by an equation $x_1\cdots x_t=0$ 
(resp.~$x_{t+1}\cdots x_s=0$). 
Let $b'{}^{(k)}_m$ $(k\in {\mab Z}_{>0})$ and 
$\iota'{}^{(k)*}_{m,{\rm zar}}$ $(k\in {\mab Z}_{\geq 0})$ 
be analogous morphisms to $b^{(k)}_m$  and 
$\iota_{m,{\rm zar}}^{(k)*}$, respectively, for ${\cal X}'_m$, 
${\cal D}'_m$ and ${\cal Z}'_m$. 
Then we have an exact sequence 
\begin{equation*}
0 \lo 
\Om^{\bul}_{{\cal X}'_m/S}
(\log({\cal Z}'_m-{\cal D}'_m)) \lo 
\Om^{\bul}_{{\cal X}'_m/S}
(\log {\cal Z}'_m)\os{\iota'{}^{(0)*}_{m,{\rm zar}}}{\lo}
b'{}^{(1)}_{ m*}
(\Om^{\bul}_{{\cal D}_m'{}^{(1)}/S}
(\log {\cal Z}'_m\vert_{{\cal D}'{}^{(1)}_m})
\os{\iota'{}^{(1)*}_{m,{\rm zar}}}{\lo} \cdots. 
\tag{12.11.8}
\label{eqn:ilicsp}
\end{equation*}
As in the proof of (\ref{lemm:eisd}) (2), 
we see that 
\begin{align*}
({\cal K}_S[[x_{d+1},\ldots, x_{d'}]]
[t_{\ul{m}}~\vert~\ul{m} \in {\mab N}^r,\vert\ul{m}\vert =n]
/(x^{\ul{m}}-\pi t_{\ul{m}}))^{\wh{}} 
\end{align*}
is a flat ${\cal K}_S$-module. 
Hence to apply the tensor product  
\begin{align*}
\wh{\otimes}_{{\cal K}_S}
({\cal K}_S[[x_{d+1},\ldots, x_{d'}]]
[t_{\ul{m}}~\vert~\ul{m} \in {\mab N}^r,\vert\ul{m}\vert =n]
/(x^{\ul{m}}-\pi t_{\ul{m}}))^{\wh{}} 
\otimes_{{\cal O}_{{\cal X}''}}\Om^q_{{\cal X}''/S}\quad (q\in {\mab N})
\end{align*} 
to the exact sequence (\ref{eqn:ilicsp}) preserves the exactness. 
Because 
\begin{align*}
&{\cal K}_{T_{mn}}\otimes_{{\cal O}_{{\cal X}_m}}
\Om^{\bul}_{{\cal X}_m/S}
(\log({\cal Z}_m-{\cal D}_m)) \\
&\simeq 
\Om^{\bul}_{{\cal X}'_m/S}
(\log({\cal Z}'_m-{\cal D}'_m))\otimes_{{\cal O}_S} 
{} 
({\cal K}_S[[x_{d+1},\ldots, x_{d'}]]
[t_{\ul{m}}~\vert~\ul{m} \in {\mab N}^r,\vert\ul{m}\vert =n]
/(x^{\ul{m}}-\pi t_{\ul{m}}))^{\wh{}} 
\otimes_{{\cal O}_{{\cal X}''}}\Om^{\bul}_{{\cal X}''/S}
\end{align*} 
and because the similar formula for 
${\cal O}_{T_{mn}}\otimes_{{\cal X}_m}b^{(k)}_{m*}
(\Om^{\bul}_{{\cal D}_m^{(k)}/S}
(\log {\cal Z}_m\vert_{{\cal D}^{(k)}_m})
\otimes_{\mab Z}\vp_{\rm zar}^{(k)}({\cal D}_m/S))$ 
$(k\in {\mab N})$ holds, 
we have the exactness of (\ref{eqn:ilics}). 
\par 
By (\ref{linc}) and (\ref{eqn:lics}), 
we have the following quasi-isomorphism

\begin{align*}
& Q^*_{(X_{\bul},Z_{\bul})/S}
L^{\rm conv}_{(X_{\bul},Z_{\bul})/S}(\Om^{\bul}_{{\cal X}_{\bul}/S}
(\log ({\cal Z}_{\bul}-{\cal D}_{\bul}))\otimes_{\mab Z}{\mab Q}) 
\tag{12.11.9}\label{ali:lncsc} \\ 
&\os{\sim}{\lo} 
[Q^*_{(X_{\bul},Z_{\bul})/S}\{L^{\rm conv}_{(X_{\bul},Z_{\bul})/S}
(\Om^{\bul}_{{\cal X}_{\bul}/S}(\log {\cal Z}_{\bul})\otimes_{\mab Z}{\mab Q})
\otimes_{\mab Z}\vp_{\rm conv}^{(0)}(D_{\bul}/S)\} \\
& \lo (Q^*_{(X_{\bul},Z_{\bul})/S}\{a^{(1){\log}}_{\bul{\rm conv}*}
L^{\rm conv}_{(D^{(1)}_{\bul},Z_{\bul}\vert_{D^{(1)}_{\bul}})/S}
(\Om^{\bul}_{{\cal D}_{\bul}^{(1)}/S}
(\log {\cal Z}_{\bul}\vert_{{\cal D}^{(1)}_{\bul}})\otimes_{\mab Z}{\mab Q})  \\ 
&  
\otimes_{\mab Z}\vp_{\rm conv}^{(1)}(D_{\bul}/S)\}, -d) \lo \cdots]. 
\end{align*}
Applying the direct image 
$R{\pi}_{(X,Z)/S{\rm Rconv}*}$ 
to (\ref{ali:lncsc}),
we obtain the following isomorphism:  
\begin{equation*} 
R{\pi}_{(X,Z)/S{\rm Rconv}*}
Q^*_{(X_{\bul},Z_{\bul})/S}L^{\rm conv}_{(X_{\bul},Z_{\bul})/S}
(\Om^{\bul}_{{\cal X}_{\bul}/S}
(\log({\cal Z}_{\bul} -{\cal D}_{\bul}))\otimes_{\mab Z}{\mab Q})\os{\sim}{\lo} \tag{12.11.10}
\label{ali:bggcsl}
\end{equation*} 
$$
[R{\pi}_{(X,Z)/S{\rm Rconv}*}
Q^*_{(X_{\bul},Z_{\bul})/S}\{L^{\rm conv}_{(X_{\bul},Z_{\bul})/S}
(\Om^{\bul}_{{\cal X}_{\bul}/S}
(\log {\cal Z}_{\bul}))\otimes_{\mab Z}{\mab Q}
\otimes_{\mab Z}\vp_{\rm conv}^{(0)}(D_{\bul}/S)\}  \lo $$ 
\begin{align*}
(R{\pi}_{(X,Z)/S{\rm Rconv}*}
Q^*_{(X_{\bul},Z_{\bul})/S}a^{(1)}_{\bul{\rm conv}*}
& \{L^{\rm conv}_{(D^{(1)}_{\bul},Z_{\bul}\vert_{D^{(1)}})/S}
(\Om^{\bul}_{{\cal D}_{\bul}^{(1)}/S}
(\log({\cal Z}_{\bul}\vert_{{\cal D}_{\bul}^{(1)}}))\otimes_{\mab Z}{\mab Q}) \\ 
& 
\otimes_{\mab Z}\vp_{\rm conv}^{(1)}(D_{\bul}/S)\},-d) 
\lo \cdots ]. 
\end{align*}
%(See (\ref{remadefi:stdfcs}) below.) 
By (\ref{ali:omnx}) we see that 
the isomorphism (\ref{ali:bggcsl}) is nothing 
but an isomorphism (\ref{ali:csct}).
\par
We have to show that the isomorphism 
(\ref{ali:csct}) is independent of the datum (\ref{cd:pfcd}). 
\par
Let the notations be as in \S\ref{sec:rlct} for the data (\ref{cd:pfcd})'s. 
Let 
\begin{align*}
R\eta_{{\rm Rconv}*} \col 
{\rm D}^+{\rm F}(Q^*_{(X_{\bul \bul},Z_{\bul \bul})/S}
(& {\cal K}_{(X_{\bul \bul},Z_{\bul \bul})/S})) \\ 
& \lo
{\rm D}^+{\rm F}
(Q^*_{(X_{\bul},Z_{\bul})/S}({\cal K}_{(X_{\bul},Z_{\bul})/S}))
\end{align*}
be a morphism of filtered derived categories as similarly defined in 
in \S\ref{sec:rlct}. 
Then we have the following commutative diagram by 
the cohomological descent:
\begin{equation*}
\tiny{
\begin{CD}
R{\pi}_{(X,Z)/S{\rm Rconv}*}
Q^*_{(X_{\bul},Z_{\bul})/S}L^{\rm conv}_{(X_{\bul},Z_{\bul})/S}
(\Om^{\bul}_{{\cal X}_{\bul}/S}
(\log({\cal Z}_{\bul} -{\cal D}_{\bul}))\otimes_{\mab Z}{\mab Q})
@>{\sim}>>  \quad \quad \quad \\
@VVV \\ 
R{\pi}_{(X,Z)/S{\rm Rconv}*} 
R\eta_{{\rm Rconv}*}Q^*_{(X_{\bul \bul},Z_{\bul \bul})/S}
L^{\rm conv}_{(X_{\bul \bul},Z_{\bul \bul})/S}
(\Om^{\bul}_{{\cal X}_{\bul \bul}/S}
(\log({\cal Z}_{\bul \bul}-{\cal D}_{\bul \bul}))\otimes_{\mab Z}{\mab Q})
@>{\sim}>> \quad \quad \quad 
\end{CD}}
\end{equation*} 
\begin{equation*}
\tiny{
\begin{CD} 
[R{\pi}_{(X,Z)/S{\rm Rconv}*}
Q^*_{(X_{\bul},Z_{\bul})/S}\{L^{\rm conv}_{(X_{\bul},Z_{\bul})/S}
(\Om^{\bul}_{{\cal X}_{\bul}/S}(\log {\cal Z}_{\bul})
\otimes_{\mab Z}{\mab Q})
\otimes_{\mab Z}
\vp_{\rm conv}^{(0)}(D_{\bul}/S)\} 
\lo\\
@VVV  \\
\{R{\pi}_{(X,Z)/S{\rm Rconv}*}
R\eta_{{\rm Rconv}*}
Q^*_{(X_{\bul \bul},Z_{\bul \bul})/S}\{L^{\rm conv}_{(X_{\bul \bul},Z_{\bul \bul})/S}
(\Om^{\bul}_{{\cal X}_{\bul \bul}/S}
(\log {\cal Z}_{\bul \bul})\otimes_{\mab Z}{\mab Q})
\otimes_{\mab Z}
\vp_{\rm conv}^{(0)}(D_{\bul \bul}/S)\}
\lo
\end{CD}}
\end{equation*}
\begin{equation*}
\tiny{
\begin{CD}
\begin{aligned}
(R({\pi}_{(X,Z)/S{\rm Rconv}}
Q^*_{(X_{\bul},Z_{\bul})/S}a^{(1)}_{\bul{\rm conv}*}
\{L^{\rm conv}_{(D^{(1)}_{\bul},Z_{\bul}\vert_{D^{(1)}})/S}
(& \Om^{\bul}_{{\cal D}_{\bul}^{(1)}/S}
(\log{\cal Z}_{\bul}\vert_{{\cal D}_{\bul}^{(1)}})\otimes_{\mab Z}{\mab Q}) \\ 
& \otimes_{\mab Z}\vp_{\rm conv}^{(1)}(D_{\bul}/S)\},-d) 
\lo 
\cdots\}
\end{aligned}
\\
@VVV \\
\begin{aligned}
(R{\pi}_{(X,Z)/S{\rm Rconv}*} 
R\eta_{{\rm Rconv}*} 
Q^*_{(X_{\bul \bul},Z_{\bul \bul})/S}a^{(1)}_{\bul \bul{\rm conv}*}
& \{L^{\rm conv}_{(D^{(1)}_{\bul \bul},Z_{\bul \bul}\vert_{D^{(1)}})/S}
(\Om^{\bul}_{{\cal D}_{\bul \bul}^{(1)}/S}
(\log {\cal Z}_{\bul \bul}\vert_{{\cal D}_{\bul \bul}^{(1)}})
\otimes_{\mab Z}{\mab Q} \\ 
& \otimes_{\mab Z}
\vp_{\rm conv}^{(1)}(D_{\bul \bul}/S)\},-d) 
\lo \cdots \}.
\end{aligned}
\end{CD}}
\end{equation*}
\parno
Hence the isomorphism (\ref{ali:csct}) is independent of the data
 (\ref{cd:pfcd})'s. 
\end{proof}

\par
Let 
$P^D_{\rm c}:=\{P^{D,k}_{\rm c}\}_{k\in {\mab Z}}$ 
be the stupid filtration on 
$E_{{\rm conv},c}({\cal K}_{(X,D\cup Z)/S})$: 
$$P^{D,k}_{\rm c}E_{{\rm conv},c}
({\cal K}_{(X,D\cup Z)/S})
=a^{(\bul \geq k){\log}}_{{\rm conv}*}
({\cal K}_{(D^{(\bul \geq k)},Z\vert_{D^{(\bul \geq k)}})/S}
\otimes_{\mab Z} 
\vp^{(\bul \geq k)\log}_{\rm conv}(D/S;Z)).$$ 
Then, by (\ref{theo:dcscy}), 
we have a filtered complex  
$(E_{{\rm conv},c}({\cal K}_{(X,D\cup Z)/S}), 
P^{D}_{\rm c}) \in
{\rm D}^+{\rm F}({\cal K}_{(X,Z)/S})$.

\begin{defi}\label{defi:czcs}
We call 
$(E_{{\rm conv},c}
({\cal K}_{(X,D\cup Z)/S}),P^D_{\rm c})$ 
the {\it weight-filtered vanishing cycle convergent complex 
with compact support} of ${\cal K}_{(X, D\cup Z)/S}$ 
(or $(X,D\cup Z)/S$) 
with respect to $D$. 
%Set 
%\begin{equation*} 
%(C^Z_{{\rm conv},c}
%({\cal K}_{(X,D\cup Z)/S}),P^D_{\rm c})
%:=(E_{{\rm conv},c}
%({\cal K}_{(X, D\cup Z)/S}), P^D_{\rm c}). 
%\end{equation*}
%the {\it weight-filtered convergent complex 
%with compact support} of 
%${\cal K}_{(X, D\cup Z)/S}$ 
%(or $(X,D\cup Z)/S$)
%with respect to $D$. 
Set 
\begin{equation*}
(E_{{\rm isozar},c}
({\cal K}_{(X, D\cup Z)/S}),P^D_{\rm c}):=
Ru^{\rm conv}_{(X,Z)/S*}
(E_{{\rm conv},c}
({\cal K}_{(X, D\cup Z)/S}),P^D_{\rm c}). 
\end{equation*}
We call 
$(E_{{\rm isozar},c}
({\cal K}_{(X,D\cup Z)/S}),P^D_{\rm c})$ 
the {\it weight-filtered vanishing cycle isozariskian complex 
with compact support} of ${\cal K}_{(X, D\cup Z)/S}$ 
(or $(X,D\cup Z)/S$) with respect to $D$. 
%If $Z=\emptyset$, we simply call 
%$(E_{{\rm conv},c}({\cal K}_{(X, D)/S}), P^D_{\rm c})
%(=(E_{{\rm conv},c}({\cal K}_{X/S}), P^k_{\rm c}))$ 
%the {\it weight-filtered convergent complex 
%with compact support} of $(X,D)/S$.  
\end{defi}

By the definition of 
$(E_{{\rm isozar},c}
({\cal K}_{(X,D\cup Z)/S}),P^D_{\rm c})$, 
there exists the following canonical  isomorphism 
in $D^+(f^{-1}({\cal K}_S)):$ 
\begin{align*}
E_{{\rm isozar},c}
({\cal K}_{(X,D\cup Z)/S})
& \os{\sim}{\lo}  
\{Ru^{\rm conv}_{(X,Z)/S*}
({\cal K}_{(X,Z)/S}
\otimes_{\mab Z}\vp^{(0)\log}_{\rm conv}(D/S;Z))
 \lo \tag{12.12.1}\label{ali:csz}\\
{} &  a^{(1)}_{{\rm zar}*}
(Ru^{\rm conv}_{(D^{(1)},Z\vert_{D^{(1)}})/S*}
({\cal K}_{(D^{(1)},Z\vert_{D^{(1)}})/S} 
\otimes_{\mab Z}\vp^{(1)\log}_{\rm conv}(D/S;Z)),-d)\\
{} & \lo \cdots\}.
\notag
\end{align*}

\begin{rema-defi}\label{remadefi:stdfcs} 
Because the notation for the right hand side 
of (\ref{ali:csz}) is only suggestive, 
we should give the precise definition of it; 
we can give it as in  \cite[(2.11.8)]{nh2}. 
%Let $I^{\bul \bul}$ be a double complex of 
%${\cal K}_{(X,Z)/S}$-modules such that, for each 
%nonnegative integer $k$, $I^{k\bul}$ is a 
%$u^{\rm conv}_{(X,Z)/S*}$-acyclic resolution of  
%$(a^{(k){\log}}_{{\rm conv}*}
%({\cal K}_{(D^{(k)},Z\vert_{D^{(k)}})/S}
%\otimes_{\mab Z}\vp^{(k)\log}_{\rm conv}(D/S;Z)),(-1)^kd)$. 
%Then the right hand side of (\ref{ali:csz}) is, by definition,  
%an object in ${\rm D}^+(f^{-1}({\cal K}_S))$ 
%which is represented by the single complex of 
%$u^{\rm conv}_{(X,Z)/S*}(I^{\bul \bul})$. 
%Let $P^D_{\rm c}$ 
%be the stupid filtration with respect to the first degree of 
%$u^{\rm conv}_{(X,Z)/S*}(I^{\bul \bul})$. 
%Then $(E_{{\rm isozar},c}
%({\cal K}_{(X,D\cup Z)/S}),P^D_{\rm c})=
%(u^{\rm conv}_{(X,Z)/S*}(I^{\bul \bul}),P^D_{\rm c})$ 
%in ${\rm D}^+{\rm F}(f^{-1}({\cal K}_S))$.  
\end{rema-defi}

\begin{coro}\label{coro:cszdes}
$E_{{\rm isozar},c}
({\cal K}_{(X, D\cup Z)/S})
=Ru^{\rm conv}_{(X,D\cup Z)/S*}
({\cal I}^{D,{\rm conv}}_{(X,D\cup Z)/S})$.
\end{coro}

\par
By applying $Rf_{*}$ to both hand sides of 
(\ref{ali:csz}) (cf.~(\ref{remadefi:stdfcs})), 
we have the following canonical isomorphism
\begin{align*}
Rf_{(X,D\cup Z)/S*,{\rm c}}({\cal K}_{(X,D\cup Z;Z)/S}) 
\os{\sim}{\lo} &
\{Rf_{(X,Z)/S*}({\cal K}_{(X,Z)/S}
\otimes_{\mab Z}\vp_{\rm conv}^{(0)\log}(D/S;Z)) 
\tag{12.14.1}\label{ali:scs}
\end{align*}
$$ \lo
(Rf_{(D^{(1)},Z\vert_{D^{(1)}})/S*}
({\cal K}_{(D^{(1)},Z\vert_{D^{(1)}})/S}
\otimes_{\mab Z}\vp_{\rm conv}^{(1)\log}(D/S;Z)),-d) 
\lo \cdots\}.$$
%Let $P^D_{\rm c}$ be the stupid 
%filtration of (\ref{ali:scs}) (cf.~(\ref{remadefi:stdfcs})):
%\begin{equation*}
%P^D_{\rm c}
%Rf_{(X,D\cup Z)/S*,{\rm c}}({\cal O}_{(X,D\cup Z;Z)/S}) =  
%\tag{18.4.2}\label{ali:stfcss} 
%\end{equation*} 
%$$
%\{(Rf_{(D^{(k)},Z\vert_{D^{(k)}})/S*}
%({\cal O}_{(D^{(k)},Z\vert_{D^{(k)}})/S}
%\otimes_{\mab Z}\vp_{\rm conv}^{(k)}(D/S;Z)),(-1)^kd)
%\lo \cdots \}[-k].$$

\par
Next we give the filtered base change theorem, 
the filtered K\"{u}nneth formula and 
the filtered Poincar\'{e} duality of 
the log convergent cohomology sheaf
with compact support.
\par 
%Assume  that $\pi{\cal V}\subset p{\cal V}$. 
In \cite[(2.11)]{nh2} we have set 
\begin{align*}
E_{\rm zar,c}({\cal O}_{(X,D \cup Z)/S}) 
:= &
\{Ru_{(X,Z)/S*}({\cal O}_{(X,Z)/S}
\otimes_{\mab Z}\vp_{\rm crys}^{(0)}(D/S;Z)) 
\end{align*}
$$ \lo
(Ru_{(D^{(1)},Z\vert_{D^{(1)}})/S*}
({\cal O}_{(D^{(1)},Z\vert_{D^{(1)}})/S}
\otimes_{\mab Z}\vp_{\rm crys}^{(1)}(D/S;Z)),-d) 
\lo \cdots\}.$$
Let $P^D$ be the stupid filtration on 
$E_{\rm zar,c}({\cal O}_{(X,D\cup Z)/S})$.

\begin{theo}[{\bf Comparison isomorphism}]\label{theo:cic}
%Assume  that $\pi{\cal V}\subset p{\cal V}$. 
There exists a canonical isomorphism  
\begin{equation*}  
(E_{\rm isozar,c}({\cal K}_{(X,D \cup Z)/S}), P^D) 
\os{\sim}{\lo}
(E_{\rm zar,c}({\cal O}_{(X,D \cup Z)/S}),P^D)
\otimes_{\mab Z}{\mab Q}
\tag{12.15.1}\label{eqn:ccomp}
\end{equation*} 
in ${\rm D}^+{\rm F}(f^{-1}({\cal K}_S))$. 
\end{theo} 
\begin{proof}  
By (\ref{eqn:coazzpd}) we may assume  that $\pi{\cal V}= p{\cal V}$. 
By using the morphism (\ref{eqn:dyt}) for 
the constant simplicial case 
and the explicit descriptions (\ref{eqn:czzpd}) and 
(\ref{eqn:cozzpd}) for the constant simplicial case, 
we immediately obtain (\ref{theo:cic}) by (\ref{coro:ccco}). 
\end{proof}

\begin{prop}\label{prop:cssp}
Assume that $f \col X \lo S_1$ is proper.
Then 
$$(Rf_{(X,D\cup Z)/S*,{\rm c}}
({\cal K}_{(X,D\cup Z;Z)/S}),P^D_{\rm c})$$ 
is a filteredly strictly perfect complex of ${\cal K}_S$-modules 
in the sense of {\rm \cite[(2.10.8)]{nh2}}.  
\end{prop}
\begin{proof}
%By (\ref{prop:toi}) we may assume  that $\pi{\cal V}=p{\cal V}$. 
(\ref{prop:cssp}) follows from (\ref{theo:cic}) and \cite[(2.11.12)]{nh2}. 
\end{proof}

\begin{prop}\label{prop:bccsic}
Let $Y$ be a quasi-compact smooth scheme over 
$S_1$ $($with trivial log structure$)$. 
Let $f \col (X,D\cup Z) \lo Y$ be 
a morphism of log schemes such that the 
underlying morphism $\os{\circ}{f} \col X\lo Y$ is smooth, 
quasi-compact and quasi-separated 
and such that $D\cup Z$ is a relative SNCD
over $Y$. $($In particular, $D\cup Z$ is 
also a relative SNCD on $X$ over $S_1.)$  
Let $f_{(X,Z)} \col (X,Z) \lo Y$ 
be the induced morphism by $f$. 
Let
\begin{equation*}
\begin{CD}
(X',D'\cup Z') @>{g}>> (X,D\cup Z)\\
 @V{f'}VV  @VV{f}V \\
Y' @>{h}>> Y \\
@V{}VV  @VV{}V \\
S'_1 @>{}>> S_1
\end{CD}
\end{equation*}
be a commutative diagram such that 
the upper rectangle is cartesian,  
where $Y'$ is a quasi-compact smooth scheme over $S'_1$.  
Then the following hold$:$
\par 
$(1)$ 
The natural morphism 
\begin{equation*} 
{\cal O}_{S'}\otimes^L_{{\cal O}_S}
Rf_{(X,Z)/S*}
(E_{{\rm conv},c}
({\cal K}_{(X,D\cup Z)/S}),P^D_{\rm c})
\lo 
\tag{12.17.1}\label{eqn:18.10.?}
\end{equation*}
$$ 
R{f'}_{(X',Z')/S'*}
(E_{{\rm conv},c}({\cal K}_{(X',D'\cup Z')/S'}),
P^{D'}_{\rm c})
$$
is an isomorphism. \par 
$(2)$ 
There exists a natural isomorphism 
\begin{equation*}
{\cal O}_{S'}\otimes^L_{{\cal O}_S} 
Rf_{{\rm conv}*}
({\cal I}^D_{(X,D\cup Z)/S}) 
\lo 
R{f'}_{{\rm conv}*}
({\cal I}^{D'}_{(X',D'\cup Z')/S'}) 
\tag{12.17.2}\label{eqn:18.10.??}
\end{equation*}
which is compatible with the isomorphism 
{\rm (\ref{eqn:18.10.?})}. 
\end{prop}
\begin{proof} 
This follows from (\ref{theo:cic}) and \cite[(2.11.14)]{nh2}. 
\end{proof}

\par
Using the filtered complex  
$(E_{\rm conv,c}
({\cal K}_{(X,D\cup Z)/S}),P^D_{\rm c})$, 
we have the following spectral sequence
\begin{equation*}
E_{1,{\rm c}}^{k,h-k}((X,D\cup Z)/S)=
R^{h-k}f_{(D^{(k)},Z\vert_{D^{(k)}})/S*}
({\cal K}_{(D^{(k)},Z\vert_{D^{(k)}})/S}
\otimes_{\mab Z}\vp_{\rm conv}^{(k){\log}}(D/S;Z)) 
\tag{12.17.3}\label{eqn:spcsz}
\end{equation*}
\begin{equation*}
\Lo
R^hf_{(X,D\cup Z)/S*,{\rm c}}
({\cal K}_{(X,D\cup Z;Z)/S}).
\end{equation*}
%Let $k$ be a fixed integer.
%Set 
%\begin{equation*}
%E_{1,{\rm c}}^{k',h-k'}((X,D\cup Z)/S)= 
%\end{equation*}
%\begin{equation*}
%\begin{cases} 
%R^{h-k'}f_{(D^{(k')},Z\vert_{D^{(k')}})/S*}
%({\cal K}_{(D^{(k')},Z\vert_{D^{(k')}})/S}\otimes_{\mab Z}
%\vp^{(k')\log}_{\rm conv}(D/S;Z)) & (k'\geq k), \\
%0 & (k' < k).
%\end{cases}
%\end{equation*}

\begin{defi}
We call the spectral sequence (\ref{eqn:spcsz}) 
the {\it weight spectral sequence of} 
$(X, D\cup Z)/S$ 
{\it with respect to} $D$ 
{\it for the log convergent cohomology 
with compact support}. If $Z=\emptyset$, then 
we call (\ref{eqn:spcsz}) 
the {\it weight spectral sequence of} 
$(X, D)/S$ 
{\it for the log convergent cohomology 
with compact support}.
\end{defi}

Let $P^{D, \bul}_{\rm c}$ be the filtration on 
$R^hf_{(X,D\cup Z)/S*,{\rm c}}
({\cal K}_{(X,D\cup Z;Z)/S})$ 
induced by the spectral sequence 
(\ref{eqn:spcsz}). 
Since $P^{D,\bul}_{\rm c}$ is a decreasing filtration, 
we also consider the following increasing filtration 
$P^D_{{\bul},{\rm c}}$: 
\begin{equation*}
P^D_{h-{\bul}, {\rm c}}R^hf_{(X,D\cup Z)/S*,{\rm c}}
({\cal K}_{(X,D\cup Z;Z)/S})=
P^{D,\bul}_{\rm c}R^hf_{(X,D\cup Z)/S*,{\rm c}}
({\cal K}_{(X,D\cup Z;Z)/S}).
\tag{12.19.1}\label{eqn:dfpwcs}
\end{equation*}

%\begin{prop}\label{prop:spcscoh}
%Let ${\cal V}$, ${\cal K}$, $\pi$, $S$, $S_1$, 
%${\cal V}'$, ${\cal K}'$, $\pi'$,  $S'$, $S_1'$ and $u$ 
%be as in {\rm (\ref{prop:bchange})}. 
%Then there exists a canonical morphism of 
%spectral sequences
%\begin{equation*}
%\{E_{1, {\rm c}}^{-k,h+k}((X,D\cup Z)/S)
%\otimes_{{\cal K}_S}{\cal K}_{S'} 
%\Lo 
%R^hf_{(X,D\cup Z)/S*, {\rm c}}
%({\cal K}_{(X,D\cup Z;Z)/S})
%\otimes_{{\cal K}_S}{\cal K}_{S'}\} 
%\tag{12.14.1}\label{eqn:csbcsp}
%\end{equation*}
%$$\lo  \{(E_{1, {\rm c}}^{-k,h+k}((X',D'\cup Z')/S') \Lo 
%R^hf_{(X',D'\cup Z')/S'*, {\rm c}}
%({\cal K}_{(X',D'\cup Z';Z')/S'})\}$$
%of ${\cal K}_{S'}$-modules.
%\end{prop}

\begin{prop}\label{prop:csbd}
The boundary morphism 
$d_1^{k,h-k} \col 
E_{1,c}^{k,h-k}((X,D\cup Z)/S) \lo 
E_{1,c}^{k+1,h-k}((X,D\cup Z)/S)$ 
is equal to 
$\iota^{(k){\log}*}_{\rm conv}$.
\end{prop}

\begin{prop}\label{prop:ovfcc}
Let $u$ be the morphism in {\rm (\ref{cd:bpsmlgs})}.  
Let $u_1 \col S_1 \lo S'_1$ be the induced morphism by $u$. 
Let $(X,D\cup Z)$ and $(X',D'\cup Z')$
be  smooth schemes 
with relative SNCD's over $S_1$ and $S'_1$, respectively.
Let 
\begin{equation*}
\begin{CD} 
(X,D\cup Z) @>{g}>> (X',D'\cup Z')\\
@VVV @VVV \\
S_1  @>{u_1}>> S'_1
\end{CD}
\tag{12.20.1}\label{cd:mplgc}
\end{equation*}
be  a commutative diagram of log schemes 
such that the morphism $g$ induces morphisms 
$g^{(k)} \col (D^{(k)},Z\vert_{D^{(k)}})
\lo 
(D'{}^{(k)}, Z'\vert_{D'{}^{(k)}})$ 
of log schemes over $u_1$ 
for all $k\in {\mab N}$.
Then the isomorphism $(\ref{eqn:18.10.?})$ 
and the spectral 
sequence $(\ref{eqn:spcsz})$ are 
functorial with respect to 
$g_{\rm conv}^*$.
\end{prop}

\begin{theo}[{\bf K\"{u}nneth formula}]\label{theo:kcsf} 
Let the notations be as in {\rm (\ref{theo:kcrsp})} 
for the constant case. 
%Assume that $\pi{\cal V}\subset p{\cal V}$. 
Then the following hold$:$ 
\par 
$(1)$  Set 
$$H_{i.{\rm c}}:= Rf_{(X_i, Z_i){\rm conv}*}
(E^{\log,Z_i}_{{\rm conv},{\rm c}}
({\cal K}_{(X_i,D_i\cup Z_i)/S}), P^{D_i}_{\rm c}).$$ 
Then there exists a canonical isomorphism
\begin{equation*}
H_{1,{\rm c}}\otimes_{{\cal K}_{Y/S}}^LH_{2,{\rm c}} 
\os{\sim}{\lo} H_{3,{\rm c}}.
\tag{12.21.1}\label{eqn:flcptkc}
\end{equation*}
%\par
%$(2)$ 
%There exists a canonical isomorphism
%\begin{equation*}
%Rf_{(X_1,D_1\cup Z_1){\rm conv}*}
%({\cal I}^{D_1,{\rm conv}}_{(X_1,D_1\cup Z_1)/S})
%\otimes_{{\cal K}_{Y/S}}^L
%Rf_{(X_2, D_2\cup Z_2){\rm conv}*}
%({\cal I}^{D_2,{\rm conv}}_{(X_2,D_2\cup Z_2)/S}) 
%\tag{12.18.2}\label{eqn:cptkc}
%\end{equation*}
%\begin{equation*}
%\os{\sim}{\lo}
%Rf_{(X_3, D_3\cup Z_3){\rm conv}*}
%({\cal I}^{D_3,{\rm conv}}_{(X_3,D_3\cup Z_3)/S})
%\end{equation*}
%which is compatible with 
%the isomorphism {\rm (\ref{eqn:flcptkc})}.
\par
$(2)$ The isomorphism {\rm (\ref{eqn:flcptkc})} 
%and {\rm (\ref{eqn:cptkc})} 
is compatible 
with the base change isomorphism.
\end{theo}
\begin{proof}
By (\ref{eqn:coazzpd}) we may assume  that $\pi{\cal V}= p{\cal V}$. 
%By (\ref{prop:toi}) we may assume  that $\pi{\cal V}=p{\cal V}$. 
This follows from (\ref{theo:cic}) and \cite[(2.11.19)]{nh2}. 
\end{proof}

\begin{theo}[{\bf Poincar\'{e} duality}]\label{theo:pdci} 
Assume that $X/S$ is projective and that 
the relative dimension of $X/S$ is 
of pure dimension $d$.
Then there exists a perfect pairing 
\begin{equation*}
R^hf^{\rm conv}_{(X,D)/S*,{\rm c}}({\cal K}_{(X,D)/S}) 
\otimes R^{2d-h}f^{\rm conv}_{(X,D)/S*}({\cal K}_{(X,D)/S})
\lo {\cal K}_S(-d),
\tag{12.22.1}\label{eqn:pdpa}
\end{equation*}
which is strictly compatible 
with the weight filtration. That is, 
the natural morphism 
\begin{equation*}
R^hf_{(X,D)/S*,{\rm c}}({\cal K}_{(X,D)/S})
\lo {\cal  H}{\it om}_{{\cal K}_{S}}
(R^{2d-h}f_{(X,D)/S*}({\cal K}_{(X,D)/S}), {\cal K}_{S}(-d))
\tag{12.22.2}\label{eqn:pdism}
\end{equation*} 
is an isomorphism of 
weight-filtered convergent $F$-isocrystals on 
$S/V$.
\end{theo}
\begin{proof}
%By (\ref{prop:toi}) we may assume  that $\pi{\cal V}=p{\cal V}$. 
By (\ref{eqn:coazzpd}) we may assume  that $\pi{\cal V}= p{\cal V}$. 
Now this theorem follows from 
the comparison theorems (\ref{coro:ccco}), (\ref{theo:cic}) and 
\cite[(2.19.1)]{nh2}.
\end{proof}

\bigskip
\bigskip
\parno
Yukiyoshi Nakkajima 
\parno
Department of Mathematics,
Tokyo Denki University,
5 Asahi-cho Senju Adachi-ku,
Tokyo 120--8551, Japan.
%Department of Mathematics, Tokyo Denki University,
%2--2 Nishiki-cho Kanda  Chiyoda-ku, Tokyo 101--8457, 
%Japan. 
\parno
{\it E-mail address\/}: nakayuki@cck.dendai.ac.jp
\bigskip
\bigskip
\parno
Atsushi Shiho
\parno
Graduate School of Mathematical Sciences, 
University of Tokyo,
3-8-1, Komaba Meguro-ku Tokyo 153--8914, Japan. 
\parno
{\it E-mail address\/}: shiho@ms.u-tokyo.ac.jp
\end{document}